\documentclass[11pt,  a4paper]{article}

\RequirePackage{enumerate}
\RequirePackage{verbatim}
\RequirePackage{hyperref}
\hypersetup{pdfpagemode=UseOutlines,colorlinks=false,pdfpagelayout=SinglePage,pdfstartview=FitH,bookmarksopen=false}

\usepackage{latexsym,amscd,amsthm}
\usepackage[all]{xy}
\usepackage{authblk}

\usepackage{mathptmx}       
\usepackage{helvet}         
\usepackage{courier}        
\usepackage{graphicx}        
\usepackage{multicol}        
\usepackage[bottom]{footmisc}
\usepackage{amssymb}
\usepackage{amsmath}
\usepackage{pstricks}
\usepackage[all]{xy}
\usepackage{latexsym}
\usepackage{graphicx}
\usepackage{color}
\usepackage{url}
\usepackage{bm}

\newtheorem{theorem}{Theorem}
\newtheorem{lemma}{Lemma}
\newtheorem{proposition}{Proposition}
\newtheorem{corollary}{Corollary}
\theoremstyle{definition}
\newtheorem{definition}{Definition}
\newtheorem*{acknowledgement}{Acknowledgments}
\newtheorem*{claim}{Claim}

\newenvironment{sketchofproof}{\begin{trivlist}\item[]{\sc Sketch of proof.}}
            { $\Box$ \end{trivlist}}
\theoremstyle{remark}
\newtheorem{example}{Example}
\newtheorem{remark}{Remark}
\newtheorem{question}{Question}

\newcommand{\CH}{\mathcal{CH}}

\newcommand{\R}{\mathbb R}
\newcommand{\hsset}{sSet}
\newcommand{\hssetp}{{sSet_*}}
\newcommand{\Top}{Top}
\newcommand{\hTop}{Top}
\newcommand{\hTopp}{{Top_*}}
\newcommand{\hcdga}{CDGA}
\newcommand{\sset}{sSet}
\newcommand{\cdga}{CDGA}
\newcommand{\hkmod}{\textit{Chain}(k)}
\newcommand{\dgZ}{\textit{Chain}(\mathbb{Z})}
\newcommand{\Mfldn}{\mathrm{Mfld}_n}
\newcommand{\Disk}{\mathrm{Disk}}
\newcommand{\Emb}{\mathrm{Emb}}
\newcommand{\Homeo}{\mathrm{Homeo}}
\newcommand{\RHom}{{\mathbb{R}\text{Hom}}}
\newcommand{\Hom}{{\text{Hom}}}

\newcommand{\com}{\bullet}
\newcommand{\colim}{\mathop{\lim\limits_{\textstyle\longrightarrow}}\limits}
\newcommand{\hocolim}{\mathop{hocolim}\limits}
\newcommand{\holim}{\mathop{holim}\limits}
\newcommand{\It}{{It}}
\newcommand{\Map}{\mathrm{Map}}
\newcommand{\CSS}{\mathcal{CS}e\mathcal{S}p}
\newcommand{\SeSp}{\mathcal{S}e\mathcal{S}p}

\newcommand\leftthreearrows{%
        \mathrel{\vcenter{\mathsurround0pt
                \ialign{##\crcr
                        \noalign{\nointerlineskip}$\leftarrow$\crcr
                        \noalign{\nointerlineskip}$\leftarrow$\crcr
                        \noalign{\nointerlineskip}$\leftarrow$\crcr
                }%
        }}%
}

\let\coprod=\undefined 
\DeclareSymbolFont{cmlargesymbols}{OMX}{cmex}{m}{n}
\DeclareMathSymbol{\coprod}{\mathop}{cmlargesymbols}{"60}

\begin{document}

\title{Notes on factorization algebras, factorization homology and applications}

\author{Gr\'egory Ginot}
\affil{Institut  Math\'ematiques de Jussieu Paris Rive Gauche, \\ Universit\'e Pierre et Marie Curie - Sorbonne Universit\'es} 

%
%

\maketitle

\begin{abstract}These notes are an expanded version of two series of lectures given at the winter school in mathematical physics at les Houches and at the Vietnamese Institute for Mathematical Sciences. They are an introduction to factorization algebras, factorization homology and some of their applications, notably for   studying $E_n$-algebras. 
We give an account of homology theory for manifolds  (and spaces), which give invariant of manifolds but also  invariant of $E_n$-algebras. 
We particularly emphasize  the point of view of factorization algebras (a structure originating from quantum field theory)  which plays, with respect to homology theory for manifolds, the role of sheaves with respect to singular cohomology.  
We mention  some applications to the study of mapping spaces, in particular in  string topology and for (iterated) Bar constructions and study several examples, including some over stratified spaces.
\end{abstract}

\setcounter{tocdepth}{2}
\tableofcontents

\section{Introduction and motivations}

These notes are an introduction to  factorization algebras and factorization homology in the context of topological spaces and manifolds.    The  origin of factorization algebras and factorization homology, as defined by Lurie~\cite{L-HA} and Costello-Gwilliam~\cite{CG}, are to be found in topological quantum field theories and conformal field theories. Indeed, they were largely motivated and influenced by the pioneering work of Beilinson-Drinfeld~\cite{BD} and also  of Segal~\cite{Seg-CFT, Seg-localityCFT}. 
\emph{Factorization homology} is a catchword to describe homology theories \emph{specific} to say oriented manifolds of a fixed dimension $n$. There are also variant  specific to many other classes of structured manifold of fixed dimension. Typically the structure in question would be a  framing\footnote{that is a trivialization of the tangent bundle} or a spin structure or simply no structure at all. 

\emph{Factorization algebras} are algebraic structures which shed many similarities with (co)sheaves and were  introduced to describe  Quantum Field Theories  much in the same way as the structure of a manifold or scheme is described by its sheaf of functions~\cite{BD, CG}. They are related to factorization homology  in the same way as singular cohomology is related to sheaf cohomology.  

 Unlike classical singular homology for which any abelian group can be used as  coefficient of the theory, in order to define factorization homology, one needs a more complicated piece of algebraic data: that of an $E_n$-algebra\footnote{More accurately, $E_n$-algebras are  the piece of data needed in the case of framed manifolds. For other structured manifolds, one needs $E_n$-algebras equipped with additional structure; for instance  an invariance under their natural $SO(n)$-action in the oriented manifold case}. These algebras have been heavily studied in algebraic topology ever since the seventies where they were introduced to study iterated loop spaces and configuration spaces~\cite{BV, May-gils, Seg-ConfSpaces}. 
They have been proved to also have deep significance in  mathematical physics~\cite{K-operad, CG}, string topology~\cite{CS, CV} and (derived) algebraic geometry~\cite{BD, FG, PTTV, L-DAGX}. 
$E_1$-algebras are essentially the same thing as $A_\infty$-algebras, that is homotopy associative algebras. On the other hand, $E_\infty$-algebras are homotopy commutative algebras. In general  the $E_n$- structures  form a hierarchy of  more and more homotopy commutative algebra structures. 
 In fact, an $E_n$-algebra is an homotopy associative algebra whose multiplication $\mu_0$ is commutative up to an homotopy operator $\mu_1$. This operator is itself  commutative  up to an homotopy operator $\mu_2$ and so on up to $\mu_{n-1}$ which is no longer required to be homotopy commutative.

Since factorization homology depends on (some class of) both manifold and $E_n$-algebra, they also give rise to invariant of $E_n$-algebras. These invariants  have proven useful  as we illustrate in \S~\ref{S:Applications}.  For instance, in dimension $n=1$, factorization homology evaluated on a circle is the usual Hochschild homology of algebras (together with its circle action inducing cyclic homology as well).  For $n=\infty$, factorization homology gives rise to an invariant of topological spaces\footnote{and not just manifolds of a fix dimension} (sometimes called higher Hochschild homology~\cite{P}) which we recall in~\S~\ref{S:comAlgebracase}. It is easier to study and  interesting in its own since it is closely related to mapping spaces, their derived analogues and observables of classical topological field theories.  

We give the precise axioms of  homology theory of manifolds in \S~\ref{S:HomologyTheoryMfld}. 
Factorization homology can be computed using 
 \v{C}ech complexes  of  factorization algebras, which, as previously alluded to,  are a kind of \lq\lq{}multiplicative, non-commutative\rq\rq{} analogue of cosheaves. Definitions, properties and many examples of factorization algebras are discussed in~\S~\ref{S:FactAlgebras}. 
Factorization algebras were introduced  to describe observables of Quantum Field Theories~\cite{CG, BD} but they also
 are a very convenient way to encode and study many algebraic structures which arose in algebraic topology and mathematical physics as we illustrate in \S~\ref{S:FactAlgebras}. In particular,  in \S~\ref{S:stratifiedFact}  we study in depth locally constant factorization algebras on stratified spaces and their link with various categories of modules over $E_n$-algebras, giving many examples.
 We also give a detail account of various operations and properties of factorization algebras in \S~\ref{S:OperationsforFact}.
We then (\S~\ref{S:Applications}) review several applications of the formalism of factorization algebras and homology. Notably to cohomology and deformations of $E_n$-algebras, (higher) Deligne conjecture and also  in (higher) string topology and for Bar constructions of iterated loop spaces (and more generally to obtain models for  iterated Bar constructions \emph{with} their algebraic structure).
In \S~\ref{S:ComFactAlgebras} , we consider the case of commutative factorization algebras and prove their theory reduces to the one of cosheaves.  In particular, we cover  the  pedagogical example of classical homology (with twisted coefficient) viewed as factorization homology.

\subsection{Eilenberg-Steenrod axioms for homology theory of spaces}

Factorization homology and factorization algebras generalize  ideas from the axiomatic approach to classical homology  of spaces (and (co)sheaf theory) which we now recall. We then explain how they can be generalized.
The usual (co)homology groups of topological spaces are uniquely determined by a set of axioms.
These are the \emph{Eilenberg-Steenrod axioms} which were formulated in the 40s~\cite{EiSt-axioms}.

Classically they express that an homology theory for spaces   is uniquely determined by (ordinary) functors from the category of pairs $(X,A)$ ($A\subset X$) of spaces to the category of  ($\mathbb{N}$)-graded abelian groups satisfying some axioms.
 Such a functor splits as the direct sum $H_\ast(X,A)=\bigoplus_{i\geq 0} H_i(X,A)$ where $H_i(X,A)$ is the degree $i$ homology groups of the pair.
 This homology group can in fact be defined as the homology of the mapping cone $\textit{cone}(A\hookrightarrow X)$ of the inclusion of the pair. Further, the long exact sequence in homology relating the homology  of the pair to the homology of $A$ and $X$ is induced by a short exact sequence of chain complexes $C_\ast(A)\hookrightarrow C_\ast(X)\twoheadrightarrow C_\ast(X,A)$ (where $C_\ast$ is the singular chain complex). Similarly, the Mayer-Vietoris exact sequence is induced by a short exact sequence of chain complexes.

This suggests that the 
 classical Eilenberg-Steenrod axioms can be lifted at the \emph{chain complex} level. 
That is, we can characterize classical homology as   a functor from the category of spaces (up to homotopy) to the category of chain complexes (up to quasi-isomorphism).

Let us formalize a bit this idea.
A \emph{homology theory} $\mathcal{H}$ for {\em spaces} is  a functor 
$\mathcal{H}: \Top \to \mathop{Chain}(\mathbb{Z})$ from the category $\Top$ of topological spaces\footnote{for simplicity we assume that we consider only spaces homotopy equivalent to CW-complexes}  to the category $\mathop{Chain}(\mathbb{Z})$ of chain complexes over $\mathbb{Z}$ (in other words differential graded abelian groups). 
 This functor has to satisfy the following three axioms.

\begin{enumerate}
\item \emph{(homotopy invariance)} The functor $\mathcal{H}$ shall send homotopies between maps of topological spaces to homotopies between maps of chain complexes.
\item \emph{(monoidal)} The functor $\mathcal{H}$ shall be defined by its value on the connected components of a space. Hence we require it sends disjoint unions of topological spaces to direct sum, that is the canonical map 
$\oplus_{\alpha \in I}\mathcal{F}(X_\alpha) \to
\mathcal{H}(\coprod_{\alpha \in I} X_\alpha)
$ is an homotopy equivalence\footnote{sometimes this map is required to be an actual isomorphism but this is not needed} (here $I$ is any set).
\item \emph{(excision)}
There is another additional property encoding (given the other ones) the classical excision property as well as the Mayer-Vietoris principle. The additional property essentially stipulates the effect of gluing together two CW-complexes along a sub-complex.  Let us formulate it this way:  assume  $i:Z\hookrightarrow X$ and $j:Z\hookrightarrow Y$ are inclusions of  closed sub CW-complex of $X$  and $Y$.
 Let $X\cup_{Z}Y\cong X\coprod Y /( i(z)=j(z), \, z\in Z )$ be the pushout of $X$, $Y$ along $Z$.  The functoriality of $\mathcal{H}$ gives maps $\mathcal{H}(Z) \stackrel{i_*}\to \mathcal{H}(X)$ and
$\mathcal{H}(Z) \stackrel{j_*}\to \mathcal{H}(Y)$; hence a chain complex morphism $\mathcal{H}(Z) \stackrel{i_*-j_*}\to \mathcal{H}(X)\oplus \mathcal{H}(Y)$ . Functoriality also yields a natural map $\mathcal{H}(X)\oplus \mathcal{H}(Y) \to \mathcal{H}(X \cup_{Z} Y)$ whose composition with $i_*-j_*$ is null.  

The \emph{excision axioms requires that the canonical map
$$ \text{cone}\Big(\mathcal{H}(Z) \stackrel{i_*-j_*}\to \mathcal{H}(X)\oplus \mathcal{H}(Y)\Big) \longrightarrow  \mathcal{H}(X\cup_{Z}Y)$$ is an homotopy equivalence}. Here $\text{cone}(f)$ is the mapping cone\footnote{if we know that $f: C_*\to D_*$  is injective, then $cone(f)$ is quasi-isomorphic to the quotient chain complex $D_*/C_*$. See for instance~\cite{Weibel-book} for  mapping cones of general chain maps. } (in $\mathop{Chain}(\mathbb{Z})$) of the map $f$ of chain complexes. 
\end{enumerate}
We can state the following theorem (which follows from Corollary~\ref{C:ordinaryhomologyasFact} and is the (pre-)dual of a result of Mandell~\cite{M1} for cochains).
\begin{theorem}[Eilenberg-Steenrod] \label{T:Eilenberg-Steenrod}
 Let $G$  be an abelian group. Up to natural homotopy equivalence,  there is a unique homology theory for spaces, that is functor $\mathcal{H}:\Top\to \mathop{Chain}(\mathbb{Z})$ satisfying axioms 1, 2 and 3 and further the \emph{dimension} axiom:
$$\mathcal{H}(pt)\stackrel{\simeq}\to G.$$
\end{theorem}
The functor in Theorem~\ref{T:Eilenberg-Steenrod} is of course given by the usual singular chain complex with value in $G$. We can even assume in the theorem that $G$ is any chain complex, in which case we recover  extraordinary homology theories\footnote{in this case, the uniqueness is not necessarily true if one works at the homology level instead of chain complexes.}.  

Theorem~\ref{T:Eilenberg-Steenrod} implies that the category of functors satisfying axioms 1, 2, 3 is (homotopy) equivalent to the category of chain complexes; the equivalence being given by the evaluation of a functor at the point.
 To assign to a chain complex $V_\ast$  an homology theory, one consider the functor $X\mapsto C_\ast(X,\mathbb{Z})\otimes V_\ast$.

For a CW-complex $X$, the singular cohomology  $H^\ast(X,G)$ can be computed as sheaf cohomology of $X$ with value in the constant sheaf $G_X$ of locallly constant functions on $X$ with values in $G$. In particular, the singular cochain complex is naturally quasi-isomorphic to  the derived functor $\mathbb{R}\Gamma (G_X)$ of sections of $G_X$. 
Replacing $G_X$ by a locally constant sheaf (with germs $G$) yields cohomology with local coefficient in $G$\footnote{if $G$ is a linear representation of a group $H$ and $X$ is the classifying space of $H$, then one recovers this way the group (co)homology of $H$ with value in $G$}.
 This point of view realizes singular cohomology (with local coefficient)  as a special case of the theory of sheaf/\v{C}ech cohomology which also has other significance and applications in geometry when allowing more general sheaves.

Note that the homotopy invariance axiom  can be reinterpreted as saying that the functor $\mathcal{H}$ is continuous. 
Indeed, there are natural topologies on the morphism sets of both categories. For instance, one can consider the compact-open topology on the set of maps $\Hom_{\Top}(X,Y)$  (see Example~\ref{Ex:Chains(k)} for $\mathop{Chain}(\mathbb{Z})$).  
Any continuous functor, that is a functor $\mathcal{H}$  such that the maps $\Hom_{\Top}(X,Y)\to \Hom_{\dgZ}(\mathcal{H}(X), \mathcal{H}(Y))$ are continuous,    sends homotopies to homotopies (and homotopies between homotopies to homotopies between homotopies and so on).
 
Note also that excision axiom  really identifies $\mathcal{H}(X\cup_{Z}Y)$ with a homotopy colimit. It is precisely the homotopy coequalizer $\mathop{hocoeq}\Big(\mathcal{H}(Z)\underset{j_*}{\stackrel{i_*}\rightrightarrows}\mathcal{H}(X)\oplus \mathcal{H}(Y)\Big)$ which is computed by the mapping cone $\textit{cone}(i_{*}-j_{*})$. Further, in this axiom, we do not need  $\mathcal{H}$ to be precisely the cone but any natural chain complex quasi-isomorphic to it will do the job. This suggests to actually use a more flexible model than topological categories.
A  convenient way to \emph{deal simultaneously with topological categories, homotopy colimits} (in particular homotopy quotients) \emph{and  
identification of chain complexes up to quasi-isomorphism} is to consider  the $\infty$-categories associated to topological spaces and chain complexes and $\infty$-functors between them  (see Appendix~\ref{S:DKL}, Examples~\ref{E:hsset} and~\ref{Ex:Chains(k)}). The passage from topological categories to $\infty$-categories essentially allows to work in categories in which (weak) homotopy equivalences  have been somehow  \lq\lq{}inverted\rq\rq{} but which still retain enough information of the topology of the initial categories.

Furthermore, in the monoidal axiom, we can replace the direct sum  of chain complexes by any symmetric monoidal structure, for instance by the tensor product $\otimes$ of chain complexes. 
This yields the notion of homology theory for spaces with values in $(\mathop{Chain}(\mathbb{Z}), \otimes)$ see~\S~\ref{SSS:AxiomsforHH}. The latter are not determined by a mere chain complex but by a  (homotopy) commutative algebra $A$. This theory is called factorization homology for spaces and commutative algebras and its main properties are detailled in \S~\ref{S:comAlgebracase}. In fact, already at this level, we see that one needs to replace the cone construction in the excision axiom  by an appropriate derived functor. 

To produce invariant of manifolds which are not invariant of spaces, one needs to replace $\Top$ by another topological category of manifolds. For instance, fixing $n\in \mathbb{N}$,  one can consider the category $\Mfldn^{fr}$ whose objects are \emph{framed manifolds of dimension $n$} and whose morphisms are framed embeddings. In that case, an homology theory is completely determined by an $E_n$-algebra. 
 The precise definitions and variants of homology theories for various classes of structured manifolds (including the local coefficient ones) and the appropriate notion of coefficient  is 
the content of \S~\ref{S:HomologyTheoryMfld}.  

The  (variants of) $E_n$-algebras which arise as coefficient of homology theory for manifolds can be seen as a special case of factorization algebras, and more precisely as locally constant factorization algebras, which are to factorization algebras what (acyclic resolutions of) locally constant sheaves are to sheaves. This point of view is detailed  in~\S~\ref{S:EnasFact} and  extended to stratified spaces in \S~\ref{S:stratifiedFact}. The latter case gives simple\footnote{in the sense that the cosheaf condition satisfied by factorization algebras encodes  some topology which, from the classical $E_n$-operad point of view necessitates  an heavier homotopical machinery.} description of several categories of modules over $E_n$-algebras as well as categories of $E_n$-algebras acting on $E_m$-algebras, which is used in the many applications of \S~\ref{S:Applications}.

\subsection{Notation and conventions}

\begin{enumerate} 
\item Let $k$ be a commutative unital ring.  The $\infty$-category of differential graded $k$-modules (\emph{i.e.} chain  complexes) will be denoted $\hkmod$.
The (\emph{derived}) tensor product over $k$ will be denoted $\otimes$.  The $k$-linear dual of $M\in \hkmod$ will be denoted $M^\vee$.
\item All manifolds are assumed to be Hausdorff, second countable,  paracompact and thus metrizable.
 \item We write $\hTop$ for the $\infty$-category of topological spaces (up to homotopy) and $\hTop^{f}$ for its sub $\infty$-category spanned by the (spaces with the homotopy type of) finite CW-complexes. 
We also denote $\hTopp$, resp.  $\hssetp$, the $\infty$-categories of \emph{pointed} topological spaces and simplicial sets.  
 We write \emph{$C_\ast(X)$ and $C^\ast(X)$ for the singular chain and  cochain complex} of a space $X$. 
We write $\hsset$ for the $\infty$-category of simplicial sets (up to homotopy)  which is equivalent to $\hTop$. 
\item The $\infty$-categories of  unital commutative differential graded algebras (up tho homotopy) will be  denoted by  $\hcdga$. We  simply refer to unital  commutative differential graded algebras as CDGAs.
\item Let $n\in \mathbb{N}\cup\{\infty\}$. By an $E_n$-algebra we mean an algebra over an $E_n$-operad. We write $E_n\textbf{-Alg}$ for the $\infty$-category of (unital) $E_n$-algebras (in $\hkmod$). See Appendix:~\S~\ref{S:EnAlg}. We also write
$E_n\textbf{-Mod}_A$, resp. $E_1\textbf{-LMod}_A$, resp. $E_1\textbf{-RMod}_A$ the $\infty$-categories of $E_n$-$A$-modules, resp. left $A$-modules, resp. right $A$-modules.
\item The $\infty$-category of (small) $\infty$-categories will  be denoted $\infty\textbf{-Cat}$. 
\item We work with a cohomological grading (unless otherwise stated)  for all our (co)homology groups and graded spaces, even when we use subscripts to denote the grading (so that our chain complexes have a funny grading). In particular, all differentials are of degree $+1$, of the form $d:A^i\to A^{i+1}$ and the homology groups $H_i(X)$ of a space $X$ are concentrated in non-positive degree. If $(C^\ast,d_C)\in \hkmod$, we denote $C^\ast[n]$ the chain complex given by $(C^\ast[n])^i:= C^{i+n}$ with differential $(-1)^n d_C$.
\item We will denote  $\textbf{PFac}_X$, resp. $\textbf{Fac}_X$, resp. $\textbf{Fac}^{lc}_X $  the $\infty$-categories of prefactorization algebras, resp. factorization algebras, resp. locally constant factorization algebras over $X$. See Definition~\ref{D:CatofFacAlg}.
\item Usually, if $C$ a (topological or simplicial or model) category, we will use the boldface letter $\mathbf{C}$ to denote  the $\infty$-category associated to $C$ (see Appendix~\ref{S:operad}). This is for instance
 the case for the categories of topological spaces or chain complexes or CDGAs mentionned above.
 \item  Despite their names, the values of Hochschild or factorization (co)homology  will be  (co)chain complexes (up to equivalences), \emph{i.e.} objects of  $\hkmod$, or  objects of another $\infty$-category such as $E_\infty\textbf{-Alg}$.
\end{enumerate}

These notes deal mainly with applications of factorization algebras in algebraic topology and homotopical algebra. However, there are very interesting applications to mathematical physics as described in the work of Costello \emph{et al}~\cite{CG, Co, Co-Houches}, \cite{GrGw, Gw-Thesis} and also beautiful applications in algebraic geometry and geometric representation theory, for instance see~\cite{BD, FG, Ga1, Ga2}.

We almost always refer to the existing literature for proofs; though there are some exceptions to this rule, mainly  in  sections \S~\ref{S:OperationsforFact}, \S~\ref{S:stratifiedFact} and \S~\ref{S:ComFactAlgebras}, where we treat several new (or not detailed in the literature) examples and results related to factorization algebras. To help the reader browsing through the examples in \S~\ref{S:OperationsforFact} and \S~\ref{S:stratifiedFact}, the longer proofs are postponed to a dedicated appendix, namely \S~\ref{S:SomeProofs}.
Some other references concerning factorization algebras and factorization homology include~\cite{L-HA, L-TFT} and~\cite{An, AFT, Ca-HDR, F, F2, GTZ2, GTZ3, T-Houches}.

\paragraph{\textbf{About $\infty$-categories:}}
We use $\infty$-categories as a convenient framework for homotopical algebra and in particular as higher categorical \emph{derived categories}.  
In our context, they will typically arise when one  consider a topological category or a category $\mathcal{M}$ with a  notion of (weak homotopy)  equivalence. The $\infty$-category associated to that case  will be a lifting of the homotopy category $Ho(\mathcal{M})$ (the category obtained by formally inverting the equivalences). It has  \emph{spaces} of morphisms and composition and associativity laws are defined up to coherent homotopies. 
We recall some basic examples and definitions in the Appendix: \S~\ref{S:DKL}. 

\emph{Topological  categories and continuous functors between them are actually a model for $\infty$-categories and $(\infty$-)functors between them}. By a topological category we mean a category $\mathcal{C}$ endowed with a \emph{space} of morphisms $\Map_{\mathcal{C}}(x,y)$  between objects such that the composition $\Map_{\mathcal{C}}(x,y)\times \Map_{\mathcal{C}}(y,z) \to \Map_{\mathcal{C}}(x,z)$ is continuous.  
A continuous functor $F: \mathcal{C}\to \mathcal{D}$ between topological categories is a functor (of the underlying categories) such that for all objects $x$, $y$, the map $ \Map_{\mathcal{C}}(x,y)\stackrel{F}\to \Map_{\mathcal{D}}(F(x),F(y))$ is continuous. 
In fact,  every $\infty$-category admits
a strict model, in other word is equivalent to a topological category (though finding a strict model can be hard in practice).
One can also replace, in the previous paragraph, topological categories by simplicially enriched categories, which are the same thing as topological categories where spaces are replaced by simplicial sets  (and continuous maps by maps of simplicial sets).    In practice many topological categories we consider are geometric realization of simplicially enriched categories. 

The reader can thus substitute topological category  to $\infty$-category in every statement of these notes (or also simplicially enriched or even differential graded\footnote{in this  case, we refer to~\cite{To-dgCat} for the needed homotopy categorical framework on dg-categories} ),  but modulo the   fact that one may have to replace  the topological or $\infty$-category in question by  another equivalent topological one. The same remark applies to  functors between $\infty$-categories. Furthermore, many constructions   involving factorization algebras are actually carried out in (topological) categories (which provide concrete models to homotopy equivalent  (derived) $\infty$-category of some algebraic structures). 

If $\mathcal{C}$ is an $\infty$-category, we  will denote $\Map_ {\mathcal{C}}(x,y)$ its space of morphisms from $x$ to $y$ while we will simply write $\Hom_ {\mathcal{D}}(x,y)$ for the morphism \emph{set} of an ordinary category $\mathcal{D}$ (that is a topological category whose space of morphisms are discrete).

Many derived functors of homological algebra have natural extensions to the setting of $\infty$-categories. In that case we will use the usual derived functor notation to denote their canonical lifting to $\infty$-category and to emphasize that they can be computed using the\emph{ usual resolutions} of homological algebra.
 For instance, we will denote $(M,N)\mapsto M\mathop{\otimes}^{\mathbb{L}}_{A} N$ for the  functor $E_1\textbf{-RMod}_A\times E_1\textbf{-LMod}_A \to \hkmod$ lifting the usual tensor product of left and right modules to their $\infty$-categories. 

There is a slight exception to this notational rule. 
We  denote $M\otimes N$ the \emph{derived} tensor products of complexes in $\hkmod$. We do not use a derived tensor product notation since it will be too cumbersome and since in practice it will often be applied in the case where $k$ is a field or $M$, $N$ are projective over $k$.

\begin{acknowledgement}
The author thanks the \emph{Newton Institute for Mathematical Sciences} for its stimulating environment and hospitality during the time when most of these notes were written. The author was partially supported by ANR grant HOGT.

The author deeply thanks Damien Calaque for all the numerous conversations we had about  factorization stuff and also wish to thank 
Kevin Costello, John Francis, Owen Gwilliam, Geoffroy Horel, Fr\'ed\'eric Paugam, Hiro-Lee Tanaka, Thomas Tradler and Mahmoud Zeinalian for several enlightening  discussions related to the topic of theses notes. 
\end{acknowledgement}


\section{Factorization homology for commutative algebras and spaces and derived higher Hochschild homology}\label{S:comAlgebracase}
Factorization homology restricted to commutative algebras is also known as \emph{higher Hochschild homology} and has been studied (in various guise) since at least the end of the 90s (see the approach of~\cite{EKMM, MCSV} to topological Hochschild homology, or, the work of Pirashvili~\cite{P} which is closely related to $\Gamma$-homology). Though its axiomatic description is an easy corollary of the description of $\hTop$ as a symmetric monoidal category with pushouts, it has a lot of nice properties and appealing combinatorial description in characteristic zero (related to rational homotopy theory \`a la Sullivan). We review some of its main properties in this Section. 

\subsection{Homology theory for spaces and derived Hochschild homology}\label{SS:AxiomsforHH}
\subsubsection{Axiomatic presentation}\label{SSS:AxiomsforHH}
Let us first start by defining the axioms of an homology theory for spaces with values in the symmetric monoidal $\infty$-category $(\hkmod,\otimes)$ (instead of $(\text{Chain}(\mathbb{Z}), \oplus)$). The (homotopy) commutative monoids in $(\hkmod, \otimes)$ are the $E_\infty$-algebras (Definition~\ref{D:EinftyAlg}). 
In characteristic zero, one can restrict to   differential graded commutative algebras since the natural functor $\hcdga \to E_\infty\textbf{-Alg}$ is an (homotopy) equivalence.

\smallskip

The $\infty$-category $\hTop$ has a symmetric monoidal structure given by disjoint union of spaces $X\coprod Y$, which  is also the coproduct of $X$ and $Y$ in $\hTop$. The identity map $id_X:X\to X$ yields a canonical map $ X\coprod X \stackrel{\coprod id_X}\longrightarrow X$ which is associative and commutative  in the ordinary category of topological spaces. Hence $X$ is \emph{canonically} a commutative algebra object in $(\hTop, \coprod)$. And so is its image by a symmetric monoidal functor. We thus have:

\begin{lemma} \label{L:monoidalimpliesCom}Let $(\mathcal{C},\otimes)$ be a symmetric monoidal $\infty$-category. Any symmetric monoidal functor $F:\hTop\to \mathcal{C}$ has a canonical lift $\tilde{F}: \hTop \to E_\infty\textbf{-Alg}(\mathcal{C})$. 
\end{lemma}
In particular, any homology theory $\hTop\to \hkmod$ shall have a canonical factorization $\hTop  \to E_\infty\textbf{-Alg}$. This motivates the following definition.
\begin{definition} \label{D:axioms} An homology theory for spaces with values in the symmetric monoidal $\infty$-category $(\hkmod, \otimes)$ is an $\infty$-functor 
$\mathcal{CH}: \hTop\times E_\infty\textbf{-Alg} \to E_\infty\textbf{-Alg}$ (denoted $(X,A)\mapsto {CH}_X(A)$ on the objects),
 satisfying the following axioms: 
\begin{enumerate} 
\item[i)] \textbf{(value on a point)} there is a natural equivalence ${CH}_{pt}(A)\stackrel{\simeq}\to A$ in ${E_\infty}\textbf{-Alg}$;
\item[ii)] \textbf{(monoidal)}   the canonical maps (induced by universal property of coproducts)   $$\bigotimes_{i\in I}{CH}_{X_i}(A)   \stackrel{\simeq}\longrightarrow {CH}_{\coprod_{i\in I} X_i}(A) $$ are  equivalences (for any set $I$);
\item[iii)]\textbf{(excision)} The functor $\mathcal{CH}$ commutes with homotopy pushout of spaces, \emph{i.e.}, the canonical maps (induced by the universal property of derived tensor product) 
$${CH}_{X}(A)\mathop{\otimes}\limits_{{CH}_{Z}(A)}^{\mathbb{L}} {CH}_{Y}(A)\stackrel{\simeq} \longrightarrow  {CH}_{X\cup^h_{Z}Y}(A) $$ are natural equivalences.   
\end{enumerate}
\end{definition}
\begin{remark}
If one replace $\hTop$ by $\hTop^{f}$ then axiom \emph{ii)} is equivalent to saying that, the functors $X\mapsto CH_{X}(A)$ are symmetric monoidal. 
\end{remark}
\begin{remark}[\textbf{Homology theory for spaces in an arbitrary symmetric monoidal $\infty$-category}]
{If $(\mathcal{C},\otimes)$ is a symmetric monoidal $\infty$-category, we define an homology theory with values in $(\mathcal{C},\otimes)$ in the same way, simply replacing $\hkmod$ by $\mathcal{C}$ (and thus $E_\infty\textbf{-Alg}$ by $E_\infty\textbf{-Alg}(\mathcal{C})$). 

All results (in particular the existence and uniqueness Theorem~\ref{T:derivedfunctor}) in Section~\ref{SS:AxiomsforHH}, Section~\ref{SS:coHHEinfty} and Section~\ref{SS:HHforCDGAs} still hold by just replacing the monoidal structure of $\hkmod$ by the one of $\mathcal{C}$, provided that $\mathcal{C}$ has colimits and that its monoidal structure commutes with geometric realization.}
\end{remark}

\begin{theorem}\label{T:derivedfunctor} \begin{enumerate}
\item There is an unique\footnote{up to contractible choices} homology theory for spaces (in the sense of Definition~\ref{D:axioms}).  
\item This homology theory is given by \emph{derived Hochschild chains}, \emph{i.e.},   there are natural equivalences 
 $$A\boxtimes X \cong {CH}_X(A)$$ where $A\boxtimes X$ is the tensor of the $E_\infty$-algebra $A$ with the space $X$ (see Remark~\ref{R:Einftyistensored}).  In particular,
\begin{equation}\label{eq:CH=tensored}\Map_{\hTop}\big({X},\Map_{E_\infty\textbf{-Alg}}(A,B)\big) \cong \Map_{E_\infty\textbf{-Alg}}(CH_{X}(A),B). \end{equation}
\item \textbf{(generalized uniqueness)} Let $F:E_\infty\textbf{-Alg}\to E_\infty\textbf{-Alg}$ be a functor. There is an unique functor  $ \hTop\times E_\infty\textbf{-Alg} \to E_\infty\textbf{-Alg}$ satisfying axioms \textbf{ii), iii)} in Definition~\ref{D:axioms} and whose value on a point is  $F(A)$. This functor is  $(X,A)\mapsto {CH}_{X}(F(A))$. 
\end{enumerate}
\end{theorem}
\begin{remark}
 Theorem~\ref{T:derivedfunctor} still holds with $\hTop^f$ instead of $\hTop$ (and where in axiom \emph{ii)} one restricts to finite sets $I$).
 In that case, it  can be rephrased  as follows: 
\begin{proposition} \label{P:derivedfunctorfinite}The  functor $F \mapsto F(pt)$ from the category  of symmetric monoidal functors $\hTop^{f}\to \hkmod$ satisfying excision\footnote{by Lemma~\ref{L:monoidalimpliesCom}, the excision axiom makes sense for any such functor} to the category of $E_\infty$-algebras is a natural equivalence.
\end{proposition}
Similarly, Theorem~\ref{T:derivedfunctor}  can be rephrased in the following way:

\begin{minipage}[c]{11cm}\lq\lq{}\emph{The  functor $F \mapsto F(pt)$ from the category  of  functors $\hTop\to \hkmod$ preserving arbitrary coproducts  and satisfying excision  to the category of $E_\infty$-algebras is an natural equivalence.}\rq\rq{}\end{minipage}
\end{remark}

An immediate consequence of $A\boxtimes X \cong {CH}_X(A)$ and the identity~\eqref{eq:CH=tensored} is the following natural equivalence
\begin{equation}\label{eq:CHprod=prodCH}
 CH_{X\times Y}(A) \cong CH_{X}\big(CH_{Y}(A)\big)
\end{equation}
in $E_\infty\textbf{-Alg}$. This is the \lq\lq{}exponential law\rq\rq{} for derived Hochschild homology. 

Another interesting consequence of~\eqref{eq:CH=tensored} is that,   for any spaces $K$ and $X$ and $E_\infty$-algebra $A$,  the identity map in $\Map_{E_\infty\textbf{-Alg}}(CH_{X\times K}(A), CH_{X\times K}(A))$ yields a canonical element in $\Map_{\hTop}\big(K,\Map_{E_\infty\textbf{-Alg}}(CH_{X}(A),CH_{X\times K}(A))\big)$
hence a canonical map of chain complexes
\begin{equation}\label{eq:canmapCHinproduct}
 \text{tens}:\, C_\ast\big(K\big)\otimes CH_X(A)\longrightarrow CH_{K\times X}(A).
\end{equation}
 Similarly, let  $f:K\times X\to Y$ 
 be a map of topological spaces, then we get a canonical continuous map 
 $K\longrightarrow Map_{E_\infty\text{-}Alg}(CH_X(A),CH_Y(A))$ or equivalently 
 a chain map 
 $f_*:C_\ast(K)\otimes CH_X(A)\longrightarrow CH_Y(A)$ in $\hkmod$ which is just the composition 
 $$C_\ast(K)\otimes CH_X(A)\stackrel{\text{tens}}\longrightarrow CH_{K\times X}(A)\stackrel{f_*}\longrightarrow CH_{Y}(A)$$ where the last map is by functoriality of $\CH$ with respect to maps of topological spaces.

 \begin{remark}[\textbf{Group actions on derived Hochschild homology}] \label{R:groupactionsonCH}Since $\CH$ is a functor of both variables, $CH_X(A)$ has a natural action of the topological monoid $Map_{\Top}(X,X)$ (and thus of the group $Homeo(X)$), \emph{i.e.}, there is a monoid\footnote{here, monoid  means an homotopy monoid, that is  an $E_1$-algebra in the symmetric monoidal category $(\hTop, \times)$} map $Map_{\hTop}(X,X)\to Map_{E_\infty\textbf{-Alg}}(CH_X(A), CH_X(A))$.
 By adjunction, we get a chain map\footnote{and higher homotopy coherences} $C_\ast\big(Map_{\hTop}(X,X)\big)\otimes CH_X(A)\to CH_X(A)$ 
  which exhibits $CH_X(A)$ as a module over  $Map_{\hTop}(X,X)$ in $E_\infty\textbf{-Alg}$ .
\end{remark}

\subsubsection{Derived functor interpretation}\label{SSS:derivedfunctor}
We now explain a derived functor interpretation of derived Hochschild homology.  
Recall (Example~\ref{E:singularchainasEinfty}) that the singular chain functor of a space $X$ has a natural structure of $E_\infty$-coalgebras. In other words, it  is an object  (abusively denoted $C_\ast(X)$) of $\textbf{Fun}^{\otimes}({\mathbf{Fin}}^{op},\hkmod)$ the category of \emph{contravariant} symmetric monoidal functor from finite sets to chain complexes. 

 We can identify an  $E_\infty$-coalgebra $C$, resp. an $E_\infty$-algebra $A$, respectively,  with a  \emph{right} module, resp. \emph{left} module over the ($\infty$-)operad $\mathbb{E}_\infty$; or equivalently with   contravariant, resp.  covariant, symmetric monoidal functors from ${\mathbf{Fin}}$ to $\hkmod$). We can thus   form their (derived) tensor products $C\mathop{\otimes}\limits_{\mathbb{E}_\infty}^\mathbb{L} A \in \hkmod$ which is computed as a (homotopy) coequalizer:
$$C\mathop{\otimes}_{\mathbb{E}_\infty^{\otimes}}^\mathbb{L} A \cong  \mathop{hocoeq} \Big( \coprod_{f:\{1,\dots, q\} \to \{1,\dots, p\}} C^{\otimes p} \otimes \mathbb{E}_\infty(q,p)\otimes A^{\otimes q} \rightrightarrows  \coprod_{n} C^{\otimes n}\otimes A^{\otimes n}\Big) $$
where the maps $f:\{1,\dots, q\} \to \{1,\dots, p\}$ are maps of sets. The upper map  in the coequalizer is induced by the maps $f^*:C^{\otimes p} \otimes \mathbb{E}_\infty(q,p)\otimes A^{\otimes q} \to C^{\otimes q} \otimes A^{\otimes q}$ obtained from the coalgebra structure of $C$ and the lower map is induced by the maps $f_*:C^{\otimes p} \otimes \mathbb{E}_\infty^{\otimes}(q,p)\otimes A^{\otimes q} \to C^{\otimes p} \otimes A^{\otimes p}$ induced by the algebra structure.
One can define similarly $C\mathop{\otimes}\limits^{\mathbb{L}}_{{\mathbf{Fin}}} A$ the derived tensor product of a covariant and contravariant ${\mathbf{Fin}}$-modules. 
\begin{proposition}\label{P-CH-coeq} Let $X$ be a space and $A$ be an $E_\infty$-algebra. There is a natural equivalence (in $\hkmod$)
 $$CH_{X}(A) \;\cong\;   C_\ast(X) \mathop{\otimes}_{\mathbb{E}_\infty}^\mathbb{L} A. $$
If $A$ has a structure of CDGA, then we further have $CH_{X}(A) \;\cong\;   C_\ast(X) \mathop{\otimes}\limits_{{\mathbf{Fin}}}^\mathbb{L} A$
\end{proposition}
\begin{proof}
Note that the $E_\infty$-coalgebra structure on $C_\ast(X_\com)$ is given by the functor ${\mathbf{Fin}}^{op}\to  \hkmod$ defined by $I\mapsto
 k\big[Hom_{Fin}(I, X_\bullet)\big]$. The rest of the proof  is the same as in~\cite[Proposition~4]{GTZ2}. 
\qed \end{proof}

\begin{remark}[Factorization homology of commutative algebras as derived mapping stacks]\label{R:FactasRMap}
There is another  nice interpretation of derived Hochschild homology    in terms of  derived (or homotopical) algebraic geometry. 
Let $\mathbf{dSt}_k$ be the $\infty$-category of derived stacks over the ground ring $k$ described in details 
in~\cite[Section 2.2]{ToVe}. This category admits internal Hom's that we denote by $\mathbb{R}\mathop{Map}(F,G)$ 
following~\cite{ToVe, ToVe2} and further  
 is also an enrichment of the homotopy category of spaces. Indeed, any 
simplicial set $X_\com$  yields a constant simplicial presheaf $E_\infty\text{-}Alg \to \sset$ defined by $R\mapsto X_\com$ 
which, in turn, can be stackified.  We denote $\mathfrak{X}$ the associated stack, \emph{i.e.} the stackification of 
$R\mapsto X_\com$, which depends only on the (weak) homotopy type of $X_\bullet$.
For a (derived) stack $\mathfrak{Y}\in \mathbf{dSt}_k$, we denote $\mathcal{O}_{\mathfrak{Y}}$ its functions, \emph{i.e.}, $\mathcal{O}_{\mathfrak{Y}}:=\mathbb{R}\underline{Hom}(\mathfrak{Y},\mathbb{A}^1)$, (see~\cite{ToVe}). A direct application of Theorem~\ref{T:derivedfunctor} is:
\begin{corollary}[\cite{GTZ2}]\label{C:mappingstack} Let $\mathfrak{R}=\mathbb{R}\mathop{Spec}(R)$ be an affine derived stack (for instance an 
affine stack)~\cite{ToVe} and $\mathfrak{X}$ be the stack associated to a space $X$. Then the Hochschild chains over $X$ with 
coefficients in $R$ represent the mapping stack $\mathbb{R}\mathop{Map}(\mathfrak{X}, \mathfrak{R})$. That is, there are canonical 
equivalences $$\mathcal{O}_{\mathbb{R}\mathop{Map}(\mathfrak{X},\mathfrak{R})}\; \cong \; CH_{X}(R), 
\qquad \mathbb{R}\mathop{Map}(\mathfrak{X},\mathfrak{R}) \;\cong \; \mathbb{R}\mathop{Spec}\big( CH_{X}(R) \big)$$
\end{corollary}
If a group  $G$ acts on $X$, the natural action of $G$ on $CH_{X}(A)$ (Remark~\ref{R:groupactionsonCH}) identifies with the natural one  on $\mathbb{R}\mathop{Map}(\mathfrak{X},\mathfrak{R})$ under the equivalence given by Corollary~\ref{C:mappingstack}.
\end{remark}

\subsection{Pointed spaces and higher Hochschild cohomology}\label{SS:coHHEinfty}

In order to have a dual and relative versions of the construction of \S~\ref{SS:AxiomsforHH}, we consider the ($\infty$-)category $\hTopp$ of \emph{pointed} spaces. 
  Let $\tau:pt\to X$ be a base point of $X \in \hTopp$. 
The map $\tau$ yields a map of $E_\infty$-algebras $ A\cong CH_{pt}(A) \stackrel{\tau_*}\longrightarrow CH_{X}(A)$ 
and thus makes   $CH_{X}(A)$ an $A$-module.  
Let $M$  be an $E_\infty$-module over $A$; for instance,  take $M$ to be a  module  over  a CDGA $A$.  Note that $M$  has  induced left and right modules structures\footnote{Note also that there is an equivalence of $\infty$-categories $E_1\textbf{-LMod}_A\cong E_1\textbf{-RMod}_A$ if $A\in E_\infty\textbf{-Alg}$} over $A$.

\begin{definition} \label{D-coHoch} Let $A$ be an $E_\infty$-algebra and $M$ be an $E_\infty$-module over $A$. 
\begin{itemize}
\item The (derived) \emph{Hochschild cochains} of $A$  with values in $M$ over 
(a pointed topological space) $X$ is given by $$CH^{X}(A,M):= \RHom^{left}_{A}(CH_{X}(A), M),$$
the (derived) chain complex of homomorphisms of underlying  left $E_1$-modules over $A$ (Definition~\ref{D:LandRMod}). 
\item \label{D:HochChainModitem}
The (derived) \emph{Hochschild homology} of  $A$ with values in $M$ over (a pointed space) 
$X$ is defined as $$CH_{X}(A,M):= M\mathop{\otimes}^{\mathbb{L}}_{A} CH_{X_\bullet}(A)$$ the relative tensor product of (a left and a right) $E_1$-modules over $A$.\end{itemize}
\end{definition}
The two definitions above depend on the choice of the base point even though we do not write it explicitly in the definition.

\begin{remark} One can also use the relative tensor products of $E_\infty$-modules over $A$ (as defined, for instance, in~\cite{L-HA, KM}) for defining the Hochschild homology $CH_X(A,M)$. This does not change the computation (and makes Lemma~\ref{L:HochModstructure} below trivial) according to 
 Proposition~\ref{P:EinftyasFact} (or~\cite{L-HA, KM}). The same remark applies to the definition of derived  Hochschild cohomology.
\end{remark}

Since the based point map $\tau_*: A\to CH_{X}(A)$ is a map of $E_\infty$-algebras, the canonical module structure of $CH_{X}(A)$ over itself induces a $CH_{X}(A)$-module structure on $CH_{X}(A,M)$   after tensoring by $A$ (see \cite[Part V]{KM}, \cite{L-HA}):
\begin{lemma}\label{L:HochModstructure}
 Let $M$ be in $E_\infty\textbf{-Mod}_A$, that is, $M$ is an $E_\infty$-$A$-module. Then $CH_{X}(A,M)$ is canonically a  $E_\infty$-module over $CH_{X}(A)$.
\end{lemma}
The Lemma is obvious when $A$ is a CDGA. 

We have the $\infty$-category $E_\infty\textbf{-Mod}$ of pairs $(A,M)$ with $A$ an $E_\infty$-algebra and $M$ an $A$-module (Definition~\ref{DEnModoverA}). Let $\pi_{E_\infty}:  E_\infty\textbf{-Mod}\to E_\infty\textbf{-Alg}$ be the canonical functor.
\begin{proposition}[\cite{GTZ2, GTZ3}]\label{P-CHfunctorMod} 
\begin{itemize} \item The derived Hochschild chains (Definition~\ref{D:HochChainModitem})
 induces a functor of $\infty$-categories   
 $CH: (X,M) \mapsto CH_{X}(\pi_{E_\infty}(M),M)$ from $\hTopp \times E_\infty\textbf{-Mod}$ to $ E_\infty\textbf{-Mod}$ which fits into a commutative diagram 
 $$\xymatrix{ \hTopp \times E_\infty\textbf{-Mod}   \ar[d]_{for\times \pi_{E_\infty}} \ar[r]^{\;\;\;\;CH}  &  E_\infty\textbf{-Mod}   \ar[d]^{\pi_{E_\infty}} \\
\hTop \times E_\infty\textbf{-Alg} \ar[r]^{\;\;\CH} & E_\infty\textbf{-Alg}. }$$Here $for: \hTopp \to \hTop$ forget the base point.
\item
The derived Hochschild cochains  (Definition~\ref{D-coHoch})
 induces a functor of $\infty$-categories   
 $(X,M) \mapsto CH^{X_\bullet}(A,M)$ from $(\hTopp)^{op} \times E_\infty\textbf{-Mod}_A$ to $ E_\infty\textbf{-Mod}_A$, 
which is further contravariant with respect to $A$.
\end{itemize}
\end{proposition}
In particular, if $M=A$, then we have an natural equivalence $CH_{X}(A,A)\cong CH_X(A)$ in $E_\infty\textbf{-Mod}$\footnote{here we implicitly use the canonical functor $E_\infty\textbf{-Alg} \to  E_\infty\textbf{-Mod}$ which sees an $A$-lagebra as a module over itself}.

\begin{remark}[\textbf{Functor homology point of view}]
There is also a derived functor interpretation of the above functors as in~\S~\ref{SSS:derivedfunctor}.
 Let ${\mathbf{Fin}_*}$ be the $\infty$-category associated to the category of pointed finite sets (Example~\ref{ex:discreteasinfty}).
 If $X$ is pointed, then we have a functor $\tilde{C}_\ast(X):{\mathbf{Fin}_*}^{op}\to \hkmod$ which sends a finite pointed set $I$ to $ C_\ast(\Map_{pointed}(I,X))$ the singular chain on the space of pointed maps from $I$ to $X$.
 Further,  let $M$ be an $E_\infty$-module. 
 Similarly to~\S~\ref{SSS:derivedfunctor} we find a symmetric monoidal functor $\tilde{M}:{\mathbf{Fin}_*}\to \hkmod$. 
 When $M$ is a module over a CDGA $A$, denoting $*$ the base point,  this is simply the functor $\tilde{M}(\{*\}\coprod J) = M\otimes A^{\otimes J}$, see~\cite{G-HHnote}.  This functor actually factors through $E_1\textbf{-LMod}_{A}$.
 
We have a dual version of $\tilde{M}$, that we denote $\mathcal{H}(A,M):  {\mathbf{Fin}_*}^{op}\to \hkmod$, defined as $ \mathcal{H}(A,M)(J):=\Hom_{A}(\tilde{A}(J), M)$ (where $J\mapsto \tilde{A}(J)$ is the functor ${\mathbf{Fin}_*}\to E_1\textbf{-LMod}_{A}$ defined by the canonical $E_\infty$-module structure of $A$). See~\cite{G-HHnote} for an explicit construction when $A$ is a CDGA and $M$ a module. 

A proof similar to the one of Proposition~\ref{P-CH-coeq} yields:
\begin{proposition}\label{P:derivedfunctorversionforcoHH}
There are natural equivalences
$$CH_X(A,M) \cong   \tilde{C}_\ast(X)\mathop{\otimes}_{{\mathbf{Fin}_*}}^{\mathbb{L}} \tilde{M}, \qquad CH^{X}(A,M) \cong \RHom_{{\mathbf{Fin}_*}}(\tilde{C}_\ast(X), \mathcal{H}(A,M)).$$
\end{proposition}
\end{remark}

\subsection{Explicit model for derived Hochschild chains}\label{SS:explicitsimplicialmodelforHH}\label{SS:HHforCDGAs}
Following Pirashvili~\cite{P}, one can construct rather simple explicit  chain complexes computing  derived Hochschild chains when the input is a CDGA. We mainly deal with the unpointed case, the pointed one being similar and left to the reader.

\emph{In this section, we consider only CDGAs}.
Note that, if we assume $k$ is of characteristic zero, $E_\infty$-algebras are always homotopy equivalent to a CDGA so that we do not loose much generality.  
This construction, using  simplicial sets as models for topological spaces,  provides explicit semi-free resolutions for $CH_X(A)$ which makes them combinatorially appealing.

Let $(A=\bigoplus_{i\in \mathbb{Z}} A^i, d, \mu)$ be a differential graded, associative, 
commutative algebra and 
let $n_+$ be the  set $n_+:=\{0,\dots, n\}$. We define 
$CH_{n_+}(A):=A^{\otimes n+1} \cong  A^{\otimes n_+}$.  Let $f: k_+\to \ell_+$ be any set map,
we denote by $f_*: A^{\otimes k_+}\to A^{\otimes \ell_+}$, the linear map given by
\begin{equation}\label{eq:f*}
 f_*(a_0\otimes a_1\otimes \dots\otimes a_{k})=(-1)^{\epsilon}\cdot b_0\otimes b_1\otimes \dots\otimes b_{\ell},
 \end{equation}
  where $b_{j}=\prod_{i\in f^{-1}(j)} a_i$ (or $b_j=1$ if $f^{-1}(j)=\emptyset$) for $j=0,\dots,{\ell}$. 
The sign $\epsilon$ in equation \eqref{eq:f*} is determined by the usual Koszul sign rule of
 $(-1)^{|x|\cdot |y|}$ whenever $x$ moves across $y$.  In particular, $n_+\mapsto CH_{n_+}(A)$ is functorial. 
Extending the construction by colimit we obtain a 
well-defined functor 
\begin{equation}\label{eq:CHdefinedoverset} Y\mapsto CH_Y(A):= \colim_{Fin \ni K\to Y}  CH_K(A)\end{equation} 
from  sets to differential graded commutative algebras (since the tensor products of CDGAs is a CDGA).
 Now, if  $Y_\bullet$ is a simplicial set, we get  a simplicial CDGA
$CH_{Y_\bullet}(A)$ and by the Dold-Kan construction  a CDGA whose product is induced by the \emph{shuffle product} which is defined (in simplicial degree $p$, $q$) as the composition
\begin{multline}
sh: CH_{Y_p}(A)\otimes CH_{Y_q}(A) \stackrel{sh^{\times}}\longrightarrow CH_{Y_{p+q}}(A)\otimes CH_{Y_{p+q}}(A)\cong CH_{Y_{p+q}}(A\otimes A)\\  \stackrel{\mu_*}\longrightarrow CH_{Y_{p+q}}(A).
\end{multline}
Here $\mu:A\otimes A\to A$ denotes the multiplication in $A$ (which is a map of algebras) and, denoting $s_i$ the degeneracies of the simplicial structure in $CH_{Y_\bullet}(A)$,  
 $$ sh^{\times}(v\otimes w)=\sum_{(\mu,\nu)} sgn(\mu,\nu) (s_{\nu_q}\dots s_{\nu_1}(v)\otimes s_{\mu_p}\dots s_{\mu_1}(w)), $$
where $(\mu,\nu)$ denotes a $(p,q)$-shuffle, i.e. a permutation of $\{0,\dots,p+q-1\}$ mapping $0\leq j\leq p-1$ to $\mu_{j+1}$ and $p\leq j\leq p+q-1$ to $\nu_{j-p+1}$, such that $\mu_1<\dots<\mu_p$ and $\nu_1<\dots<\nu_q$. 
The differential $D: CH_{Y_\bullet}(A)\to  CH_{Y_\bullet}(A)[1]$ is given as follows. The tensor products of chain complexes   $A^{\otimes Y_i}$  have an internal differential which we abusively denote as $d$ since it is induced by the  differential $d: A\to A[1]$. Then,  the differential on $CH_{Y_\bullet}(A)$ is given by the formula:
\begin{eqnarray*}
D\big(\bigotimes_{i\in Y_i} a_{i}\big)&:=& (-1)^i d\big(\bigotimes_{i\in Y_i} a_{i}\big)+\sum_{r=0}^i (-1)^r (d_r)_* \big(\bigotimes_{i\in Y_i} a_{i}\big),
\end{eqnarray*}
where the $(d_r)_*:CH_{Y_i}(A)\to CH_{Y_{i-1}}(A)$ are induced by the corresponding faces $d_r: Y_i\to Y_{i-1}$ of the simplicial set $Y_\bullet$. 

\begin{definition}\label{D:Hoch}Let $Y_\bullet$ be a  simplicial set.   The Hochschild chains over $Y_\bullet$ of $A$ is the commutative differential graded algebra 
$(CH_{Y_\bullet}(A), D, sh)$.
\end{definition}
The rule $(Y_\bullet,A)\mapsto(CH_{Y_\bullet}(A), D, sh)$ is a bifunctor from the ordinary discrete categories of simplicial sets and CDGA to the ordinary discrete category of CDGA. 

If $Y_\bullet$ is a \emph{pointed} simplicial set, we have a canonical CDGA map $A\stackrel{\sim}\to CH_{pt_\bullet}(A)\to CH_{Y_\bullet}(A)$.  This allows to mimick Definition~\ref{D-coHoch}:
\begin{definition}\label{D:HochChainMod}Let $Y_\bullet$ be a  simplicial set, $A$ a CDGA and $M$ an $A$-module (viewed as a symmetric bimodule).
\begin{itemize}
\item The \emph{ Hochschild chains} of  $A$ with values in  $M$ over 
$Y_\bullet$ are: $$CH_{X_\bullet}(A,M):= M\mathop{\otimes}_{A} CH_{X_\bullet}(A).$$
\item  The \emph{Hochschild cochains} of  $A$ with values in $M$ over 
 $Y_\bullet$ are: $$CH^{X_\bullet}(A,M)= Hom_{A}(CH_{X_\bullet}(A), M).$$\end{itemize}
\end{definition}

The above definition computes the derived Hochschild homology of Theorem~\ref{T:derivedfunctor}.
Indeed,  we have the adjunction $|-|: \sset \stackrel{\sim}{\underset{\sim}{\rightleftarrows}} \Top :\Delta_\bullet(-)$ given by the geometric realization $|Y_\bullet|$ of a simplicial set and the singular set functor: $n\mapsto \Delta_n(X):=\Hom_{\Top}(\Delta^n, X)$ (where $\Delta^n\in \Top$ is the standard $n$-simplex).
 This adjunction is a Quillen adjunction hence induces an equivalence of $\infty$-categories. Further (by unicity Theorem~\ref{T:derivedfunctor}) we have a  commutative diagram (in $\text{Fun}(\hsset\times \hcdga,  E_\infty\textbf{-Alg})$)
  \begin{equation}
\xymatrix{\hsset \times \hcdga    \ar[rrr]^{(Y_\bullet,A)\mapsto CH_{Y_\bullet}(A)}  \ar[d]_{|-|}^{\simeq} &&& \hcdga \ar[d] \\
\hTop\times \hcdga \ar[r] &  \hTop\times E_\infty\textbf{-Alg} \ar[rr]^{\mathcal{CH}} && E_\infty\textbf{-Alg}, }
\end{equation}see~\cite{GTZ2, GTZ3} for more details. From there, we get
\begin{proposition} One has 
 natural equivalences $CH_{X_\bullet}(A) \cong CH_{|X_\bullet|}(A)$  of $E_\infty$-algebras  as well as equivalences
$$ CH_{X_\bullet}(A,M) \cong CH_{|X_\bullet|}(A, M), \qquad CH^{X_\bullet}(A,M) \cong CH^{|X_\bullet|}(A,M)$$ of $ CH_{|X_\bullet|}(A)$-modules.
\end{proposition}

We now demonstrate the above combinatorial definitions in a few examples (in which we assume, for simplicity, that $A$ has a trivial differential).
\begin{example}[The point and the interval]\label{ex:ptinterval} The point  has a trivial simplicial model  given by the constant simplicial set $pt_n=\{pt\}$. Hence $$(CH_{pt_\bullet}(A),D):=A\stackrel{0}\leftarrow   A\stackrel{id}\leftarrow   A\stackrel{0}\leftarrow   A\stackrel{id}\leftarrow   A\cdots$$ which is a deformation  retract of $A$ (as a CDGA).
A (pointed) simplicial model for the interval $I=[0,1]$ is given by
$I_n=\{\underline{0},1\cdots,n+1\}$, hence in simplicial degree $n$, $CH_{I_n}(A,M)= M\otimes A^{\otimes n+1}$ and  the simplicial face maps are $$d_i(a_0\otimes \cdots a_{n+1})=a_0\otimes \cdots \otimes (a_i a_{i+1}) \otimes\cdots \otimes a_{n+1}.$$ An easy computation shows that $CH_{I_\bullet}(A,M)=Bar(M,A,A)$ is the standard  Bar construction\footnote{the two-sided one, with values in the two $A$-modules $A$ and $M$} which is  quasi-isomorphic to $M$.
\end{example}

\begin{example}[The circle]\label{ex:simplicialS1}
The circle  $S^1\cong I/(0\sim 1)$ has (by Example~\ref{ex:ptinterval}) a simplicial model $S^1_\bullet$ which is the quotient  $S^1_n=I_n/(0\sim n+1) \cong \{0,\dots,n\}$.  One computes that the face maps $d_i:S^1_{n}\to S^1_{n-1}$, for $0\leq i \leq n-1$ are given by $d_i(j)$ is equal to $j$ or $j-1$ depending on $j=0,\dots,i$ or $j=i+1,\dots,n$ and  $d_n(j)$ is equal to $j$ or $0$ depending on $j=0,\dots, n-1$ or $j=n$. For $i=0,\dots,n$, the degeneracies $s_i(j)$ is equal to $j$ or $j+1$ depending on $j=0,\dots,i$ or $j=i+1,\dots, n$. This is the standard simplicial model of $S^1$ cf. \cite[6.4.2]{Lo-book}.
Thus, $CH_{S^1_\bullet}(A)= \bigoplus_{n\geq 0} A\otimes A^{\otimes n}$ and the differential agrees with
the \emph{usual one on the Hochschild chain complex} $C_\bullet(A)$ of $A$ (see~\cite{Lo-book}). 

It can be proved that the $S^1$ action on $CH_{S^1}(A)$ given by Remark~\ref{R:groupactionsonCH} agrees with the canonical mixed complex structure of $CH_{S^1_\bullet}(A)$ (see~\cite{ToVe-DR}).
\end{example}

\begin{example}[The torus]\label{ex:torus} The torus $\mathbb{T}$ is the product $S^1\times S^1$. Thus, by Example~\ref{ex:simplicialS1},  it has a simplicial model given by
 $(S^1\times S^1)_\bullet$  the diagonal simplicial set associated to the bisimplicial set $S^1_\bullet\times S^1_\bullet$, i.e. $(S^1\times S^1)_k=S^1_k\times S^1_k =\{0,\dots, k\}^{2}$.
We may write $(S^1\times S^1)_k=\{(p,q)  \,|\,\, p,q=0,\dots,k\}$ which we equipped with the lexicographical ordering. The face maps $d_i:(S^1\times S^1)_k\to (S^1\times S^1)_{k-1}$ and degeneracies $s_i:(S^1\times S^1)_k\to (S^1\times S^1)_{k+1}$, for $i=0,\dots,k$, are given as the products of the differentials and degeneracies of $S^1_\bullet$, i.e. $d_i(p,q)=(d_i(p), d_i(q))$ and $s_i(p,q)=(s_i(p), s_i(q))$.

We obtain $CH_{(S^1\times S^1)_\bullet}(A,A)=\bigoplus_{k\geq 0} A\otimes A^{\otimes (k^2+2k)}$. 
The face maps $d_i$  can be described more explicitly, when placing the tensor $a_{(0,0)}\otimes \dots\otimes a_{(k,k)}$ in a $(k+1)\times(k+1)$ matrix. For $i=0,\dots,k-1$, we obtain $d_i(a_{(0,0)}\otimes\dots \otimes a_{(k,k)})$ by multiplying the $i$th and $(i+1)$th rows and the $i$th and $(i+1)$th columns simultaneously, \emph{i.e.}, $d_i(a_{(0,0)}\otimes\dots \otimes a_{(k,k)})$ is equal to:
\begin{equation*}
\begin{matrix}
 a_{(0,0)}& \mathbf{\dots }& \mathbf{ (a_{(0,i)}a_{(0,i+1)})}&  \mathbf{\dots }& \mathbf{ a_{(0,k)}} \\  
 \mathbf{ \vdots} && \vdots &&\vdots \\
 \mathbf{ a_{(i-1,0)}} &\dots & (a_{(i-1,i)}a_{(i-1,i+1)})& \dots &\otimes a_{(i-1,k)} \\ 
  \mathbf{(a_{(i,0)} a_{(i+1,0)})}& \dots & (a_{(i,i)}a_{(i,i+1)}a_{(i+1,i)}a_{(i+1,i+1)})& \dots & (a_{(i,k)}a_{(i+1,k)}) \\  
 \mathbf{ a_{(i+2,0)}}& \dots & (a_{(i+2,i)}a_{(i+2,i+1)})& \dots & a_{(i+2,k)} \\  
 \mathbf{ \vdots} && \vdots &&\vdots \\
 \mathbf{ a_{(k,0)}}& \dots & (a_{(k,i)}a_{(k,i+1)})& \dots & a_{(k,k)} \\  
\end{matrix} 
\end{equation*}

The differential $d_k$ is obtained by multiplying the $k$th and $0$th rows and the $k$th and $0$th columns simultaneously, \emph{i.e.}, $d_k(a_{(0,0)}\otimes\dots \otimes a_{(k,k)})$  equals
\begin{equation*}
\begin{matrix}
(a_{(0,0)}a_{(0,k)}a_{(k,0)} a_{(k,k)}) &\mathbf{(a_{(0,1)}a_{(k,1)})} & \mathbf{\dots}& \mathbf{ (a_{(0,k-1)}a_{(k,k-1)})}\\
\mathbf{(a_{(1,0)}a_{(1,k))}} & a_{(1,1)} & \dots &  a_{(1,k-1)}\\
 \mathbf{\vdots} & \vdots & & \vdots\\
\mathbf{(a_{(k-1,0)}a_{(k-1,k)})} & a_{(k-1,1)} & \dots &  a_{(k-1,k-1)}
\end{matrix}
\end{equation*}
\end{example}

\begin{example}[The Riemann sphere $S^2$]\label{ex:2sphere}
The sphere $S^2$ has a simplicial model   $S^2_\bullet=  I^2_\bullet/\partial I^2_\bullet$ i.e. $S^2_n= \{(0,0)\}\coprod \{{1}\cdots n\}^2$.
Thus  $CH_{S^2_\bullet}(A)= \bigoplus_{n\geq 0} A\otimes A^{\otimes n^2}$.

Here the face and degeneracies maps are the diagonal ones as for $(S^1\times S^1)_\bullet$ in  Example~\ref{ex:torus}.
In particular, the $i$th differential is also obtained from the previous examples by setting $d_i^{S^2_\bullet}(p,q)=(0,0)$ in the case that $d_i(p)=0$ or $d_i(q)=0$ (where $d_i$ is the $i^{\text{th}}$-face map of $S^1_\bullet$), or setting otherwise $d_i(p,q)=(d_i(p), d_i(q))$. For $i\leq n-1$, we obtain $d_i(a_{(0,0)}\otimes\dots \otimes a_{(k,k)})$ is equal to:
\begin{equation*}
\begin{matrix}
 a_{(0,0)}&& &  \\  
&  (a_{(i-1,i)}a_{(i-1,i+1)})& \dots & a_{(i-1,n)} \\ 
 &( a_{(i,i)}a_{(i,i+1)}a_{(i+1,i)}a_{(i+1,i+1)})& \dots & (a_{(i,n)}a_{(i+1,n)}) \\  
 & (a_{(i+2,i)}a_{(i+2,i+1)})& \dots & a_{(i+2,n)} \\  
 & \vdots &&\vdots \\
 & (a_{(n,i)}a_{(n,i+1)})& \dots & a_{(n,n)} \\  
\end{matrix} 
\end{equation*}
which is similar to the one of Example~\ref{ex:torus} without the \lq\lq{}boldface\rq\rq{} tensors.
\end{example}

\begin{example}[higher spheres]\label{ex:higherspheres} Similarly to $S^2$, we have the \emph{standard model} $S^d_{\bullet}:= (I_\bullet)^{d}/\partial (I_\bullet)^d \cong S^1_\bullet \wedge \dots
\wedge S^1_\bullet$ ($d$-factors) for the sphere $S^d$. Hence $S^d_n\cong \{\underline{0}\}\coprod \{{1}\cdots n\}^d$  and $CH_{S^d_\bullet}(A)=\bigoplus_{n\geq 0} A\otimes A^{\otimes n^d}$.  
The face operators are similar to those of Example~\ref{ex:2sphere} (except that, instead of a matrix, we have a dimension $d$-lattice) and face maps are obtained by simultaneously multiplying each $i^{\text{th}}$-hyperplane with $(i+1)^{\text{th}}$-hyperplane in each dimension. The last face $d_n$ is obtained by multiplying all tensors of all   $n^{\text{th}}$-hyperplanes with $a_{0}$.
 
We also have the \emph{small model}  ${S^d_{sm}}_\bullet$ which is the simplicial set with exactly two \emph{non-degenerate} simplices, one in
  degree 0 and one in degree $d$. Then ${S^d_{sm}}_n\cong \{1,\dots,   \binom{n}{d}\}$.   Using this model, it is straightforward to check the following computation of the first homology groups of $CH_{S^d}(A)$:
$$H_n(CH_{S^d}(A)) \cong H_n(CH_{{S^d_{sm}}_\bullet}(A)) =\left\{\begin{matrix} = A & \mbox{ if } n=0\\
= 0 & \mbox{ if } 0< n<d\\ 
= \Omega^1_A & \mbox{ if } n=d \end{matrix}\right.$$ where $\Omega^1_A$ is the $A$-module of K\"ahler differentials (see~\cite{Lo-book, Weibel-book}). 
\end{example}

\begin{example}[Hochschild-Kostant-Rosenberg]\label{Ex:HKR}
Let $A$ be a smooth commutative algebra. The classical Hochschild-Kostant-Rosenberg Theorem states that its  (standard) Hochschild homology is given by the algebra of K\"ahler forms $\wedge^{\bullet}_A(\Omega^1_A)\cong S^{\bullet}_A(\Omega^1_A[1])$, where $\Omega^1_A$ is the $A$-module of K\"ahler differentials; here a $i$-form is viewed as having cohomological degree $-i$ and $S_A^\bullet$ is the free graded commutative algebra functor (in the category of graded $A$-modules). This theorem extends to Hochschild homology over all spheres:

\begin{theorem}[Generalized HKR]\label{T:HKRhigher} Let $A$ be a smooth algebra and $X$ be an affine smooth scheme or a smooth manifold. Let $n\geq 1$  and $\Sigma^g$ be a genus $g$ surface.
\begin{enumerate}
\item \textbf{(Pirashvili~\cite{P})} There is a quasi-isomorphism of CDGAs:  $CH_{S^n}(A)\cong S^{\bullet}_A\big(\Omega^1_A[n]\big)$.
\item \textbf{(\cite{GTZ})}  There is an equivalence  $CH_{\Sigma^g}(A)\cong S^{\bullet}_A\big(\Omega^1_A[2]\oplus (\Omega^1_{A}[1])^{\oplus 2g}\big)$ of CDGAs.
\item There  are equivalences $CH_{S^n}(\mathcal{O}_X) \cong  S^{\bullet}_{\mathcal{O}_X}\big(\Omega^1_{X}[n]\big)$ and $$CH_{\Sigma^g}(\mathcal{O}_X)\cong S^{\bullet}_{\mathcal{O}_X}\big(\Omega^1_{X}[2]\oplus (\Omega^1_{X}[1])^{\oplus 2g}\big)$$ of sheaves of CDGAs\footnote{here the differentials on the right hand sides are zero; they are \emph{not} the de Rham differential}.
\end{enumerate}
\end{theorem}
The third assertion  in Theorem~\ref{T:HKRhigher} follows from 1 and 2 after sheafifying in an appropriate way the derived Hochschild chains.
\end{example}


\subsection{Relationship with mapping spaces}\label{SS:MappingSpaces}

We have seen the relationship between derived Hochschild chains and derived mapping spaces (Remark~\ref{R:FactasRMap}). It is also classical that the usual Hochschild homology of de Rham forms of a simply connected  manifold $M$ is a model for the de Rham forms on the free loop space $LM:=\Map(S^1,M)$ of $M$ (see~\cite{Ch}). There is a generalization of this result for spaces where the forms are replaced by the singular cochains with their $E_\infty$-algebra structure (\cite{Jo}). These two results extend to derived Hochschild chains in general to provide algebraic models of mapping spaces.

First, we sketch  a generalization of Chen iterated integrals (studied in~\cite{GTZ}).
Let $M$ be a compact, oriented manifold, denote by $\Omega_{dR}^\bullet (M)$ the differential graded algebra of differential forms on $M$, and let  $Y_\bullet$ be a simplicial set with geometric realization $Y:=|Y_\bullet|$.
Denote  $M^Y:=\Map_{sm}(Y,M)$ the space of continuous maps from $Y$ to $M$, which are smooth on the interior of each simplex in $Y$.
Recall from Chen~\cite[Definition 1.2.1]{Ch}, that  a differentiable structure on $M^Y$ is specified by the set of plots $\phi:U\to M^Y$, where $U\subset \R^n$ for some $n$, which are those maps whose adjoint  $\phi_\sharp:U\times Y\to M$ is continuous on $U\times Y$, and smooth on the restriction to the interior of each simplex of $Y$, i.e. $\phi_\sharp|_{U\times (\text{simplex of }Y)^\circ}$ is smooth.
Following \cite[Definition 1.2.2]{Ch}, a $p$-form $\omega\in \Omega^p_{dR}(M^Y)$ on $M^Y$ is given by a $p$-form $\omega_\phi\in \Omega^p_{dR}(U)$ for each plot $\phi:U\to M^Y$, which is invariant with respect to smooth transformations of the domain.

We now define the space of Chen (generalized) \emph{iterated integrals} $\mathcal{C}\mathit{hen} (M^Y)$ of the mapping space $M^Y$. Let $\eta : Y_\bullet \to\Delta_\bullet |Y_\bullet|$ be the  canonical simplicial map (induced by adjunction)  which is given for $i\in Y_k$ by maps $\eta(i):\Delta^k \to Y$ in the following way,
\begin{equation}\label{Chen-map}
\eta (i)(t_1\leq\cdots\leq t_k):= [(t_1\leq\cdots\leq t_k)\times \{i\}]\in \left(\coprod \Delta^\bullet \times Y_\bullet /\sim\right) =Y.
\end{equation}
The map $\eta$ allows to define,  
for any plot $\phi:U\to M^Y$, a map $\rho_\phi:=ev\circ(\phi\times id)$,
\begin{equation}\label{rho_phi}
\rho_\phi: U\times \Delta^k \stackrel{\phi\times id} \longrightarrow M^Y \times \Delta^k\stackrel {ev} \longrightarrow M^{Y_{k}},
\end{equation}
where $ev$ is defined as the evaluation map,
\begin{equation}\label{eval}
ev(\gamma:  Y\to M, \underline{t})(i)= \gamma\big(\eta(i) (\underline{t}) \big).
\end{equation}
Now, if we are given   a form 
 $\bigotimes\limits_{y\in Y_k}a_y \in \big(\Omega_{dR}(M)\big)^{\otimes Y_k}$ (with  only finitely many $a_i\neq 1$), 
  the pullback $(\rho_\phi)^*\big(\bigotimes\limits_{y\in Y_k}a_y\big)\in \Omega^\bullet (U\times \Delta^k)$, may be integrated along the fiber $\Delta^k$, and is denoted by
\begin{equation*}
\left(\int_\mathcal{C} \bigotimes\limits_{y\in Y_k}a_y \right)_{\phi}:=\int_{\Delta^k} (\rho_\phi)^*\big(\mathop{\otimes}\limits_{y\in Y_k}a_y\big)\quad \in \Omega_{dR}^\bullet (U).
\end{equation*}
The resulting $p=(\sum_i deg(a_i) -k)$-form $\int_\mathcal{C} \Big(\mathop{\otimes}\limits_{y\in Y_k}a_y \Big)\in \Omega_{dR}^p(M^Y$ is called the (generalized) {\em  iterated integral} of $a_0,\dots, a_{y_k}$.
The subspace of the space of De Rham forms $\Omega^\bullet(M^Y)$ generated by all iterated integrals is denoted by $\mathcal{C}\mathit{hen}(M^Y)$. In short, we may picture an iterated integral as the pullback composed with the integration along the fiber $\Delta^k$ of a form in $M^{Y_k}$,
\begin{equation*}
\xymatrix{ M^Y&
  M^Y\times \Delta^k \ar[r]^{ev} \ar[l]_{\int_{\Delta^k}} & M^{Y_k}  }
\end{equation*}
\begin{definition}\label{D:Iteratedintegrals}
 We define ${\mathcal{I}t}^{Y_\bullet}:CH_{Y_\bullet}(\Omega_{dR}^\bullet(M))\cong (\Omega_{dR}^\bullet(M))^{\otimes Y_\bullet}\to \mathcal{C}\mathit{hen}(M^Y)$ by \begin{equation} \label{eq:chen-map}{\mathcal{I}t}^{Y_\bullet}\left(\bigotimes\limits_{y\in Y_k} a_y\right):=\int_\mathcal{C} \Big(\mathop{\otimes}\limits_{y\in Y_k}a_y \Big).\end{equation}
\end{definition}
Interesting applications of iterated integrals to study gerbes and higher holonomy are given in~\cite{AbWa, TWZ}.
\begin{theorem}[\cite{GTZ}]\label{T:IteratedIntegrals}
 The iterated integral map ${\mathcal{I}t}^{Y_\bullet}:CH_{Y_\bullet}(\Omega_{dR}^\bullet(M))\to \Omega_{dR}^\bullet(M^Y)$ is a (natural)  map of CDGAs.
 
 Further, assume that $Y=|Y_\bullet|$ is $n$-dimensional, i.e. the highest degree of any non-degenerate simplex is $n$, and assume that $M$ is $n$-connected. Then,
 ${\mathcal{I}t}^{Y_\bullet}$ is a quasi-isomorphism.
\end{theorem}
There is also a purely topological and characteristic free analogue of this result using singular cochains instead of forms.
\begin{theorem}[\cite{F2, GTZ3}]\label{T:IteratedIntegralsTop} Let $X,Y$ be topological spaces. There is a natural map of $E_\infty$-algebras 
 $$CH_{Y}(C^\ast(X)) \longrightarrow C^\ast(\Map(Y,X))$$ which is an equivalence when $Y=|Y_\bullet|$ is $n$-dimensional and $X$ is connected, nilpotent with finite homotopy groups in degree less or equal to $n$ (for instance when $X$ is $n$-connected).  
\end{theorem}
\begin{example}We compute the iterated integral map~\eqref{eq:chen-map}  in the case of $S^1_\bullet$ (Example~\ref{ex:simplicialS1}) and $\mathbb{T}$ (Example~\ref{ex:torus}). 
Since $S^{1}$ is the interval $I=[0,1]$ where the endpoints $0$ and $1$ are identified,  the map $\eta(i):S^1_k=\{0,1\dots, k\}\to \Delta_k(S^1)=\Map(\Delta^k,S^1)$ defined via~\eqref{Chen-map} is given by $\eta(i)(0\leq t_1\leq \dots\leq t_k\leq 1)=t_i$, where we have set $t_0=0$. Thus, the evaluation map~\eqref{eval} becomes $$ev(\gamma:S^1\to M, t_1\leq\dots\leq t_k)=(\gamma(0),\gamma(t_1),\dots,\gamma(t_k))\in M^{k+1}.$$ Furthermore, this recovers the classical Chen iterated integrals $\It^{S^1_\bullet}:CH_\bullet(A,A)\to \Omega^\bullet(M^{S^1})$ as follows. For a plot $\phi:U\to M^{S^1}$ we have,
\begin{multline*}
\It^{S^1_\bullet}(a_0\otimes\dots\otimes a_k)_\phi=\left(\int_\mathcal{C} a_0\dots a_k\right)_{\phi} = \int_{\Delta^k} (\rho_\phi)^*(a_0\otimes\dots\otimes a_k) \\
=(\pi_0)^*(a_0)\wedge \int_{\Delta^k} (\widetilde {\rho_\phi})^*(a_1\otimes\dots\otimes a_k)= (\pi_0)^*(a_0)\wedge \int a_1\dots a_k,
\end{multline*}
where $\widetilde {\rho_\phi}: U\times \Delta^k\stackrel {\phi\times id}\longrightarrow M^{S^1}\times \Delta^k\stackrel {\widetilde{ev}} \to M^k$ is the classical Chen  integral $\int a_1\dots a_k$ from \cite{Ch} and $\pi_0:M^{S^1} \to M$ is the evaluation at the base point $\pi_0:\gamma \mapsto \gamma(0)$.

\smallskip

In the case of the torus   $\mathbb{T}=S^1\times S^1$, the map $\eta(p,q):(S^1\times S^1)_k\to \Map(\Delta^k,S^1\times S^1)$ is given by $\eta(p,q)(0\leq t_1\leq\dots\leq t_k\leq 1)=(t_p,t_q)\in S^1\times S^1$, for $p,q=0,\dots,k$ and $t_0=0$. Thus, the evaluation map~\eqref{eval} becomes
\begin{equation*}
ev(\gamma: \mathbb T\to M, t_1\leq\cdots\leq t_k)=\quad
\left(\begin{matrix}
 \gamma(0,0),\, \gamma(0,t_1),\, \cdots,\, \gamma( 0,t_k),\\ \quad \gamma (t_1,0),\gamma(t_1,t_1),\cdots, \gamma(t_1,t_k),\\ \quad ...\\ \quad
\gamma (t_k,0),\gamma(t_k,t_1),\cdots, \gamma(t_k,t_k)
\end{matrix}\right) \in M^{(k+1)^2}
\end{equation*} 
According to definition \ref{D:Iteratedintegrals}, the iterated integral ${It}^{S^1\times S^1}(a_{(0,0)}\otimes \dots\otimes a_{(k,k)})$ is given by a pullback under the above map $M^{S^1\times S^1}\times \Delta^k\stackrel {ev}\longrightarrow M^{(k+1)^2}$, and integration along the fiber $\Delta^k$.
\end{example}
\subsection{The \emph{wedge product} of higher Hochschild cohomology}

Let $A\stackrel{f}\to B$ be a map of CDGAs. Note that it makes $B$  into an $A$-algebra as well as an $A\otimes A$-algebra (since the multiplication $A\otimes A\to A$ is an algebra morphism).
The excision axiom (Theorem~\ref{T:derivedfunctor}) implies 
\begin{lemma} \label{L:wedge} Let $M$ be an $A$-module and $X,Y$ be pointed topological spaces.
There is a natural equivalence $$\mu:\, Hom_{A\otimes A}\left( CH_{X}(A)\otimes CH_{Y}(A), M\right)\stackrel{\simeq}\longrightarrow CH^{X\vee Y}(A,M)$$
\end{lemma}
We use Lemma~\ref{L:wedge} to  obtain
\begin{definition}[\cite{G-HHnote}] \label{D:muvee}The  \emph{wedge product} of (derived) Hochschild cochains  is the linear map 
\begin{multline}\label{eq:muvee}
\mu_{\vee}: CH^{X}(A,B) \otimes CH^{Y}(A,B) \longrightarrow Hom_{A\otimes A}\Big( CH_{X}(A) \otimes CH_{Y}(A), B\otimes B\Big) \\
\stackrel{(m_B)_*}\longrightarrow Hom_{A\otimes A}\Big( CH_{X}(A) \otimes CH_{Y}(A), B\Big) \cong CH^{X\vee Y}(A,B)
\end{multline}
where the first map is the obvious one: $f\otimes g \mapsto \big(x\otimes y\mapsto f(x)\otimes g(y)\big)$.
\end{definition} 
\begin{example}If  $X_\bullet, Y_\bullet$ are finite simplicial sets models of $X,Y$, the map $\mu_{\vee}$ can be combinatorially described as the composition of the linear map $\tilde{\mu}$
given, for any $f\in CH^{X_n}(A,B)\cong Hom_A(A^{\otimes \# X_n},B)$, $g\in CH^{X_n}(A,B)\cong Hom_A(A^{\otimes
  \# Y_n},B))$ by $$\tilde{\mu}(f,g)(a_0,a_2,\dots a_{\# X_n},b_2,\dots,
b_{\#Y_n})=a_0.f(1,a_2,\dots a_{\# X_n}).g(1,b_2,\dots, b_{\#Y_n})$$ (where $a_0$ corresponds to the element indexed by the base point of $X_n\vee Y_n$) with the Eilenberg-Zilber quasi-isomorphism from $CH^{X_\bullet}(A,B) \otimes CH^{Y_\bullet}(A,B)$ to the chain complex associated to the diagonal cosimplicial space 
$\big(CH^{X_n}(A,B) \otimes CH^{Y_n}(A,B)\big)_{n\in \mathbb{N}}$. 
\end{example}
\begin{proposition}\label{P:Defcupproduct}The wedge product (of Definition~\ref{D:muvee}) is associative\footnote{precisely, it means that $\mu_{\vee}$ makes $X\mapsto CH^X(A,B)$ into a lax monoidal functor   $((\hTopp)^{op}, \vee)\to (\hkmod,\otimes)$}. 
In particular, if there is a diagonal  $X\stackrel{\delta}\to X\vee X$ making $X$ an $E_1$-coalgebra (in $(\hTopp,\vee)$), then $(CH^X(A,B), \delta^*\circ \mu_{\vee})$ is an $E_1$-algebra.
\end{proposition}
\begin{example}\label{ex:cupsphere} A standard example of space with a diagonal is a sphere $S^d$.
For $d=1$, we obtain a cup product on the usual Hochschild cochain complex which is (homotopy) equivalent to the standard cup-product for Hochschild cochains from~\cite{Ge}. 
\end{example}
The little $d$-dimensional little cubes operad
$\text{Cube}_d$ acts continuously on $S^d$ by the pinching map
\begin{equation}\label{eq:pinchcube} pinch: \text{Cube}_d(r) \times S^d \longrightarrow  \bigvee_{i=1\dots r}\, S^d.\end{equation}
given, for any $c\in \text{Cube}_d(r)$,  by the map $pinch_{c}:S^d\to \bigvee S^d$
 collapsing the complement of the interiors of
the $r$ rectangles to the base point.  
We thus get a map 
\begin{equation}\label{eq:pinchSd} \tilde{pinch}: \text{Cube}_d(r)  \longrightarrow \Map_{\hcdga}(CH_{S^d}(A,B), CH_{\bigvee S^d}(A,B))\end{equation}

Applying the contravariance of Hochschild cochains and the wedge  product  (Definition~\ref{D:muvee}), we get, for all $d\geq 1$, a morphism
\begin{multline}\label{eq:pinchSr}
pinch_{S^d,r}^*: C_{\ast}\big(\text{Cube}_d(r)\big) \otimes \left(CH^{S^d}(A,B)\right)^{\otimes r}\\ 
\stackrel{(\mu_{\vee})^{(d-1)}}\longrightarrow  C_{\ast}\big(\text{Cube}_d(r)\big) \otimes CH^{\bigvee_{i=1}^r S^d}\big(A, B \big)
\stackrel{\tilde{pinch}^*}\longrightarrow CH^{S^d}(A,B).
\end{multline}
The map~\eqref{eq:pinchSr} has a  canonical extension to the case of  $E_\infty$-algebras.  We find 
\begin{proposition}[\cite{G-HHnote, GTZ3}]\label{T:EdHoch}
Let  $A\stackrel{f}\to B$ be a CDGA (or $E_\infty$-algebra) map. The collection of maps $(pinch_{S^d,k})_{k\geq 1}$ makes 
$CH^{S^d}(A,B)$ an $E_d$-algebra. 
\end{proposition}
The algebra structure is natural with respect to CDGA maps, meaning that given a commutative diagram {\small $\xymatrix{ A \ar[r]^{f}  & B \ar[d]^{\varphi}\\ 
A' \ar[u]^{\psi}\ar[r]^{f'} & B'  }$}, the canonical map $h'\mapsto \varphi \circ h'\circ \psi$ is an $E_d$-algebras morphism $CH^{S^d}(A',B') \to CH^{S^d}(A,B)$. 
\begin{remark} If $f:A\to B$ is a CDGA map, it is possible to describe this $E_d$-algebra structure by giving an explicit action of the filtration $F_{d}\text{BE}$ of the Barrat-Eccles operad  on $CH^{S^d_\bullet}(A,B)$ using the standard simplicial model of $S^d$ (Example~\ref{ex:higherspheres}).
\end{remark}
\begin{example}\label{ex:EnHHSn}
If $A=k$, 
$CH^{S^n}(k,B)\cong B$ (viewed as $E_n$-algebras).
If $B=k$, then the $E_n$-algebra structure of $CH^{S^n}(A,k)$ is the dual of the $E_n$-coalgebra structure given by the $n$-times iterated Bar construction  $Bar^{(n)}(A)$, see~\S~\ref{S:BarConstruction}. 
\end{example}


\section{Homology Theory for manifolds} \label{S:HomologyTheoryMfld}
\subsection{Categories of structured manifolds and variations on $E_n$-algebras}
In order to specify what is a homology theory for manifolds, we need to specify an interesting category of manifolds. 
\begin{definition}\label{D:Mfldn}
 Let $\Mfldn$ be the $\infty$-category associated\footnote{by Example~\ref{ex:topologicalasinfty}} to the topological category with objects  topological manifolds of dimension $n$ and  with morphism space
 $$\Map_{\Mfldn}(M,N):= \Emb (M,N)$$ the space of all embeddings of $M$ into $N$ (viewed as a subspace of the space $\Map(M,N)$ of all continuous maps from $M$ to $N$ endowed with the compact-open topology).
\end{definition}
In the above definition, the manifolds can be closed or open, but  have no boundary\footnote{though  homology theory for manifolds can be extended to stratified manifolds, see~\cite{AFT}}.
\begin{remark}
 It is important to consider embeddings instead of smooth maps.
 Indeed,  the category of all manifolds and all (smooth) maps is weakly homotopy equivalent to $\Top$ so that, in that case,  one would obtain a homology theory which extends to spaces.
\end{remark}
\begin{remark}[smooth manifolds] \label{R:smoothmafld} One can also restrict to \emph{smooth} manifolds in which case it makes sense to equip $\Emb(M,N)$ with the weak Whitney $C^\infty$-topology; this gives us \emph{the $\infty$-category $\Mfldn^{un}$ of smooth manifolds of dimension $n$}. This latter category embeds in $\Mfldn$ and this embedding is an equivalence  onto the full subcategory of $\Mfldn$ spanned by the smooth manifolds.
\end{remark}

One can also consider categories of  more structured manifolds, such as oriented, spin or framed manifolds, as follows.  Let $E\to X$ be a topological $n$-dimensional vector bundle, which is the same as a space $X$ together with a (homotopy class of) map $e:X\to B \Homeo(\R^n)$ from $X$ to the classifying space of the group of homeomorphisms of $\R^n$. \emph{An $(X,e)$-structure on a manifold $M\in \Mfldn$} is a map $f:M\to X$ such that $TM$ is the pullback $f^*(E)$ which is the same as a factorization $M\stackrel{f}\to X \stackrel{e}\to B\Homeo(\R^n)$ of the map $M \stackrel{e_M}\longrightarrow B\Homeo(\R^n)$ classifying the tangent (micro-)bundle of $M$. 
\begin{definition}\label{D:XMfldn}  Let $\Mfldn^{(X,e)}$ be the (homotopy) pullback (in $\infty\textbf{-Cat}$)
 $$\Mfldn^{(X,e)}:= \Mfldn \times^{h}_{{\hTop}_{/B\Homeo(\R^n)}} {\hTop}_{/X}. $$
\end{definition}
In other words $\Mfldn^{(X,e)}$ is the $\infty$-category with objects $n$-dimensional topological manifolds with an $(X,e)$-structure and with morphism the embeddings preserving the $(X,e)$-structure. The latter morphisms are made into a topological space by identifying them with the homotopy pullback space 
 \begin{multline*}\Map_{\Mfldn^{(X,e)}}(M,N):= \Emb^{(X,e)} (M,N)\\ \cong \Emb (M,N) \times^{h}_{\Map_{/B\Homeo(\R^n)}(M,N)} \Map_{/X}(M,N).\end{multline*}
\begin{example}\label{E:mainExofXstructure} We list our main examples of study.
\begin{itemize}
 \item Let $X=pt$, then $E$ is trivial (here $e$ is induced by the canonical base point of $B\Homeo(\R^n)$) and an $(X,e)$-structure on $M$ is a framing, that is, a trivialization of the tangent (micro-)bundle of $M$. 
 In that case, \emph{we denote $\Mfldn^{fr}:= \Mfldn^{(pt,e)}$ the $\infty$-category of framed manifolds}. Note that this $\infty$-category is equivalent to the one  associated to the topological category with objects the framed manifolds of dimension $n$ and morphism spaces from $M$ to $N$ the framed embeddings, that is the pairs $(f,h)$ where $f\in \Emb(M,N)$ and $h$ is an homotopy between the two trivialisation of $TM$ induced by the framing of $M$ and the framing of $N$ pulled-back along $f$. 
 \item Let $X=BO(n)$ and  $BO(n)\stackrel{e}\to B\Homeo(\R^n)$  be the canonical map. Then $\Mfldn^{(BO(n),e)}$ is (equivalent to) the \emph{$\infty$-category of smooth $n$-manifolds} of Remark~\ref{R:smoothmafld}. This essentially follows because the map $O(n)\to \text{Diffeo}(\R^n)$ is a deformation retract and the characterization of smooth manifolds in terms of their micro-bundle structure~\cite{KiSi}.
 \item Let $X=BSO(n)$ and  $BSO(n)\stackrel{e}\to B\Homeo(\R^n)$  be the canonical map  induced by the inclusion of $SO(n)\hookrightarrow \Homeo(\R^n)$. 
 Then a $(BSO(n),e)$-structure on $M$ is an orientation of $M$. We denote \emph{$\Mfldn^{or}:=\Mfldn^{(BSO(n),e)}$ the $\infty$-category of oriented smooth $n$-manifolds}. Similarly to the framed case, it has a straightforward description as the $\infty$-category associated a topological category with morphisms the space of oriented embeddings.
\item If $X$ is a $n$-dimensional manifold, we can take $e_X: X\to B\Homeo(\R^n)$ to be the map corresponding to the tangent bundle $TX\to X$ of $X$. We simply denote  $\Mfldn^{(X,TX)}$ the associated $\infty$-category of manifolds. Every open subset of $X$ is canonically an object of $\Mfldn^{(X, TX)}$.
\end{itemize}
\end{example}

The (topological) coproduct $M\coprod N$ (that is disjoint union) of two $(X,e)$-manifolds $M,N$ has a canonical structure of $(X,e)$-manifold  (given by $M\coprod N \stackrel{f\coprod g}\to X$ where $M\stackrel{f}\to X$ and $N\stackrel{g}\to X$ define the $(X,e)$-structures). Note that in general there are no embeddings $M\coprod M \to M$ so that the disjoint union of manifolds is \emph{not} a coproduct (in the sense of category theory) in $\Mfldn^{(X,e)}$. Nevertheless
\begin{lemma} $(\Mfldn^{(X,e)}, \coprod)$ is a symmetric monoidal $\infty$-category. 
\end{lemma}

There is a canonical choice of framing of $\R^n$ which induces a canonical $(X,e)$-structure on $\R^n$ for any pointed space $X$. Unlike other manifolds, there are interesting framed embeddings $\coprod \R^i \to \R^i$ for any integer $i$. Indeed, in view of Example~\ref{E:mainExofXstructure} and Definition~\ref{D:EnasDisk}, the space of embeddings $\Emb^{fr}(\coprod_{\{1,\dots,r\}} \R^n, \R^n)=\Disk_n^{fr}(r,1)$ is homotopy equivalent to $\text{Cube}_n(r)$ the arity $r$ space of the  little cube operad and thus is homotopy equivalent to the configuration space of $r$ unordered points in $\R^n$. 

This motivates the following $(X,e)$-structured version of $E_n$-algebras. 
\begin{definition} \label{D:DiskXstructured} Let $\Disk^{(X,e)}_n$ be the full subcategory of $\Mfldn^{(X,e)}$ spanned by disjoints union of standard euclidean disks $\R^n$.  The $\infty$-category of $\Disk^{(X,e)}_n$-algebras\footnote{which we also referred to as the category of $(X,e)$-structured $E_n$-algebras} is the category 
$$\textbf{Fun}^{\otimes}(\Disk^{(X,e)}_n,\hkmod)$$ of symmetric monoidal ($\infty$-)functors from $(\Disk^{(X,e)}_n, \coprod)$ to $(\hkmod, \otimes)$. 

The \emph{underlying object} of a  $\Disk^{(X,e)}_n$-algebra $A$ is its value $A(\R^n)$ on a single disk $\R^n$. We will often abusively denote in the same way the $\Disk^{(X,e)}_n$-algebra and its underlying object.
\end{definition}
We denote $\Disk^{(X,e)}\textbf{-Alg}$ \emph{the $\infty$-category of $\Disk^{(X,e)}_n$-algebras} and $\Disk^{(X,e)}\textbf{-Alg}(\mathcal{C})$ the one of    $\Disk^{(X,e)}_n$-algebras with values in a symmetric monoidal category $\mathcal{C}$ (whose definition are the same as Definition~\ref{D:DiskXstructured} with $(\hkmod,\otimes)$ replaced by $(\mathcal{C}, \otimes)$). 
The underlying object induces a functor  
\begin{equation}\label{eq:Defunderlying} \Disk^{(X,e)}\textbf{-Alg}(\mathcal{C})\longrightarrow \mathcal{C} \end{equation}

\begin{example}\label{E:DisknAlg}
 \begin{itemize}
  \item For $X=pt$,  the category  of $\Disk^{(X,e)}_n$-algebras will be \emph{denoted $\Disk^{fr}_n\textbf{-Alg}$}. It is equivalent to the usual category of $E_n$-algebras and corresponds to the case of framed manifolds.
  \item The category  of $\Disk^{(BSO(n),e)}_n$-algebras is equivalent to the category of algebras over the operad $\big(\text{Cube}_n(r)\ltimes SO(n)^r\big)_{r\geq 1}$ introduced in~\cite{SW} (and called the framed little disk operad). 
Since these algebras corresponds to the case of oriented manifolds, we call them \emph{oriented $E_n$-algebras} and we simply write $\Disk^{or}_n$ for $\Disk^{(BSO(n),e)}_n$.
 It can be shown that $\Disk^{or}_n$-algebras are homotopy fixed points of the $E_n$-algebras with respect to the action of $SO(n)$ on the operad $\Disk_n^{fr}$.
  \item Similarly, the category of $\Disk^{(BO(n),e)}_n$-algebras is equivalent to the category of algebras over the operad $\big(\text{Cube}_n(r)\ltimes O(n)^r\big)_{r\geq 1}$. We also call them \emph{unoriented $E_n$-algebras} and simply write $\Disk^{un}_n$ for $\Disk^{(BO(n),e)}_n$. 
\item Let  $U\cong \R^n$ be a disk in $X$. By restriction to sub-disks of $U$, we have a canonical functor $\Disk_n^{(X,TX)}\textbf{-Alg}\to \Disk_n^{(\R^n,T\R^n)}\textbf{-Alg}\cong E_n\textbf{-Alg}$ (see Theorem~\ref{P:En=Fact} below). It follows that a $\Disk_n^{(X,TX)}$-algebra  is simply a family of  $E_n$-algebras over $X$. 
 \end{itemize}
\end{example}

\begin{example}[{Commutative algebras as $\Disk_n^{(X,e)}$-algebras}]\label{E:ComisDisk}
 The canonical functor $\Disk_n^{(X,e)}\to \mathbf{Fin}$ (where $\mathbf{Fin}$ is the $\infty$-category associated to the category of finite sets) shows that any $E_\infty$-algebra (Definition~\ref{D:EinftyAlg}) has a canonical structure of $\Disk_n^{(X,e)}$-algebras. Thus we have canonical functors $$\hcdga \longrightarrow E_{\infty}\textbf{-Alg}\longrightarrow \Disk_n^{(X,e)}\textbf{-Alg}.$$
 
 For $A$  a differential graded commutative algebra, this structure is the symmetric monoidal functor defined by $A(\coprod_{i\in I} \R^n):= A^{\otimes I}$ and, for an $(X,e)$-preserving embedding $\coprod_{I} \R^n \hookrightarrow \R^n$, by the (iterated) multiplication $A^{\otimes I}\to A$. 
\end{example}

\begin{example}[Opposite of an $E_n$-algebra] \label{R:OppositeAlgebra}
 There is a canonical $\mathbb{Z}/2\mathbb{Z}$-action on $E_n\textbf{-Alg}$ induced by the antipodal map $\tau: \R^n \to \R^n$, $x\mapsto -x$ acting on the source of  $\textbf{Fun}^\otimes(\text{\emph{Disk}}^{fr}_n,\,\hkmod)$. If $A$ is an $E_n$-algebra, then the result of this action $A^{op}:=\tau^*(A)$ is its opposite algebra. If $n=\infty$, the antipodal map is homotopical to the identity so that $A^{op}$ is equivalent to $A$ as an $E_\infty$-algebra. 
\end{example}

\subsection{Factorization homology of manifolds}
We now explain what is a Homology Theory for Manifolds (Definition~\ref{D:HomologyforMfld}) in a way parallel to the presentation of the Eilenberg-Steenrod axioms. 
We first need an analogue of Lemma~\ref{L:monoidalimpliesCom} for monoidal functors out of manifolds (instead of spaces) in order to formulate the correct excision axiom.

 \smallskip

Observe that $\R^n$ is canonically an $E_n$-algebra object in $\Mfldn$. Let
  $N$ be an $(n-s)$-dimensional manifold such that $N\times \R^s$ has an $(X,e)$-structure. Then, similarly, $N\times \R^s$ is also an $E_s$-algebra object in $\Mfldn^{(X,e)}$. Let us describe more precisily this structure: for finite sets $I,J$, we have  continuous maps $$\gamma_{I,J}^N: \Emb^{fr}\Big(\coprod_{I} \R^s,\coprod_{J}\R^s\Big) \to \Emb^{(X,e)}\Big(\coprod_{I} (N\times \R^s), \coprod_{J} (N\times \R^s)\Big)$$ induced by the composition
$$\coprod_{I} (N\times \R^s) \cong N\times \Big(\coprod_{I} \R^s \Big)\stackrel{id_N\times f} \to  N\times  \Big(\coprod_{J} \R^s \Big)\cong \coprod_{J} (N\times \R^s)$$ for any $f\in \Emb^{fr}\big(\coprod_{I} \R^s,\coprod_{J}\R^s\big) $.
In particular,  taking $s=n$, $N=pt$, the above maps induce a canonical map of operads $$\gamma: \Disk_n^{fr}\to \Disk_n^{(X,e)}$$ and thus we have an underlying  functor $\gamma^*:\Disk_n^{(X,e)}\textbf{-Alg}\longrightarrow E_n\textbf{-Alg}.$ And more generally we obtain functors: 
$(\gamma^{N})^*:\Disk_n^{(N\times \R^s,T(N\times \R^s))}\textbf{-Alg}\longrightarrow E_s\textbf{-Alg}.$ 

\smallskip

The main consequence is that any symmetric monoidal functor (from $\Mfldn^{(X,e)}$) maps  $N\times \R^s$  to an $E_s$-algebra object of the target category. More precisely: 
\begin{lemma}\label{L:homthMfldisEn}
 Let  $ \big(\Mfldn^{(X,e)},\coprod\big)\stackrel{\mathcal{F}}\longrightarrow (\hkmod,\otimes)$ be a symmetric monoidal functor. \begin{enumerate}
\item For any manifold $N\times \R^s$ with an $(X,e)$-structure, $\mathcal{F}(N\times \R^s)$  has a canonical $E_s$-algebra structure. 
 \item  Let $M$ be  an $(X,e)$-structured manifold with an end\footnote{\emph{i.e.} open boundary component} trivialized as    $N\times \R$ (where $N$ is of codimension 1 and the open part of $M$ lies in the neighborhood of $N\times \{-\infty\}$, see Figure~\ref{fig:ModuleTCH}). Then $\mathcal{F}(M)$ has a canonical   left\footnote{If $N\times \R$ is trivialized so that the open part of $M$ is in the neighborhood of $N\times \{+\infty\}$, then $\mathcal{F}(M)$ has a canonical \emph{right} module structure.} module structure over the $E_1$-algebra $\mathcal{F}(N\times \R)$.
\item $\mathcal{F}(\R^n)$ has a natural structure of $\Disk_n^{(X,e)}$-algebra.\end{enumerate}
\end{lemma}

\begin{figure}[b]\begin{center}
 \includegraphics[scale=0.65]{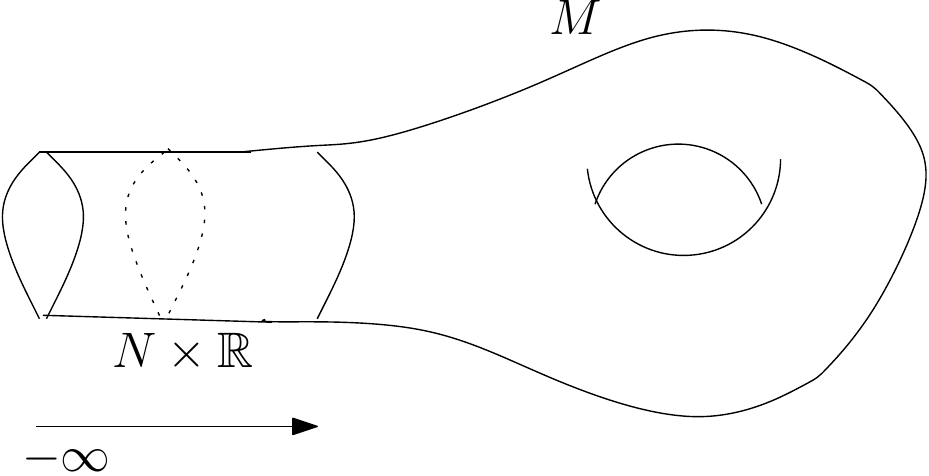}
\caption{A manifold $M$ with a trivialization $N\times \mathbb{R}$ of its open boundary.}\label{fig:ModuleTCH} \end{center}
\end{figure}
\begin{proof} Endow $\R^n$ with its canonical framing. It automatically inherits an $(X,e)$-structure for every connected component\footnote{In practice, $X$ will almost always be connex so that the structure will be canonical} of $X$ (since every vector bundle is locally trivial). 
Now, since $\mathcal{F}$ is  symmetric monoidal, then $\mathcal{F}(\R^n)$ also has an induced structure of $\Disk_n^{(X,e)}$-algebra. Let us describe the structure mentioned in 1. and 2.
  \begin{enumerate}
\item For any manifold $N\times \R^s$ with an $(X,e)$-structure,  the $E_s$-algebra structure on $\mathcal{F}(N\times \R^i)$  is given by the structure maps
\begin{multline*} \Emb^{fr}\Big(\coprod_{I} \R^s,\R^s\Big) \times \big(\mathcal{F}(N\times \R^s)\big)^{I}\\ \stackrel{\gamma_{I,pt}^N}\longrightarrow \Emb^{(X,e)}\Big(\coprod_{I} (N\times \R^i),N\times \R^s\Big)\times \big(\mathcal{F}(N\times \R^s)\big)^{I} \\ \longrightarrow \Emb^{(X,e)}\Big(\coprod_{I} (N\times \R^s),N\times \R^s\Big)\times \mathcal{F}\Big(\coprod_I (N\times \R^s)\Big) \stackrel{\mathcal{F}(\Emb^{(X,e)}(-,-))}\longrightarrow \mathcal{F}(N\times \R^s).\end{multline*}
The fact that $\mathcal{F}$ is monoidal ensures it defines an $E_s$-algebra structure.

From the definition, it is clear that $\gamma^*(\mathcal{F}(\R^n))\cong \mathcal{F}(\R^n)$ as an $E_n$-algebra (where the two structures are given by 1. and 3.).
 \item  Now, let $M$ be  an $(X,e)$-structured manifold with an end trivialized as    $N\times \R$; $\mathcal{F}(N\times \R)$ is an $E_1$-algebra by 1.  The left module structure of $\mathcal{F}(M)$ is given       
  by the maps
\begin{multline*} \Emb^{fr}\Big((\coprod_{I} \R)\coprod (0,1],(0,1]\Big) \times \Big(\mathcal{F}(N\times \R)\Big)^{I}\times \mathcal{F}(M) \\
\stackrel{\gamma_{I\coprod \{*\},pt}^N}\longrightarrow \Emb^{(X,e)}\Big(\big(\coprod_{I} N\times \R\big)\coprod M,M\Big)\times \big(\mathcal{F}(N\times \R^i)\big)^{I} \times \mathcal{F}(M)
\\ \longrightarrow \Emb^{(X,e)}\Big(\coprod_{I} (N\times \R)\coprod M ,M\Big)\times \mathcal{F}\Big(\big(\coprod_I (N\times \R)\big)\coprod M \Big)
\\ \stackrel{\mathcal{F}(\Emb^{(X,e)}(-,-))}\longrightarrow \mathcal{F}(N\times \R^i).\end{multline*}
\end{enumerate}
\end{proof}

The above lemma is crucial in order to formulate the excision property. 
\begin{definition} \label{D:HomologyforMfld} An homology theory for $(X,e)$-manifolds (with values in the symmetric monoidal $\infty$-category $(\hkmod, \otimes)$) is a functor 
$$\mathcal{F}: \Mfldn^{(X,e)} \times \Disk_n^{(X,e)}\textbf{-Alg} \to \hkmod$$ (denoted $(M,A)\mapsto \mathcal{F}_M(A)$)
 satisfying the following axioms: 
\begin{enumerate} 
\item[i)] \textbf{(dimension)} there is a natural equivalence $\mathcal{F}_{\R^n}(A)\cong A$ in $\hkmod$;
\item[ii)] \textbf{(monoidal)} the functor $M\mapsto \mathcal{F}_M(A)$ is  symmetric lax-monoidal and, for any set $I$,  the following  induced  maps are  equivalences (naturally in $A$) \begin{equation}\label{eq:monoidalaxiom}\bigotimes_{i\in I}\mathcal{F}_{M_i}(A)\stackrel{\simeq}\longrightarrow \mathcal{F}_{\coprod_{i\in I} M_i}(A) . \end{equation} 
\item[iii)] \textbf{(excision)} Let  $M$ be an $(X,e)$-manifold.  Assume   there is a codimension $1$ submanifold  $N$ of $M$  with a trivialization $N\times \R$ of its neighborhood  such that $M$ is decomposable as $M=R\cup_{N\times \R}L$ where $R, L$ are submanifolds  of $M$ glued along  $N \times \R$.  
By Lemma~\ref{L:homthMfldisEn}, $\mathcal{F}_{N\times \R}(A)$ is an $E_1$-algebra and $\mathcal{F}_{R}(A)$, $\mathcal{F}_{L}(A)$ are respectively right and left modules. The \emph{excision axiom}\footnote{or Mayer-Vietoris principle} is that the  canonical map
$$ \mathcal{F}_{L}(A)\mathop{\otimes}\limits^{\mathbb{L}}_{\mathcal{F}_{N\times \R}(A)} \mathcal{F}_{R}(A)\stackrel{\simeq}\longrightarrow \mathcal{F}_M(A)$$
 (induced by the universal property of the right hand side)  is an equivalence.
\end{enumerate}
\end{definition}
\begin{remark}\label{R:explainaxioms}
The symmetric lax-monoidal condition in axiom ii) means that there are natural (in $A$, $M$) transformations like~\eqref{eq:monoidalaxiom} compatible with composition for any finite $I$ and invariant under the action of permutations. The axiom ii) thus implies that $M\mapsto \mathcal{F}_{M}(A)$ is symmetric monoidal. When $I$ is not finite, the right hand side in~\eqref{eq:monoidalaxiom} is the colimit  $\colim_{F\to I} \bigotimes_{j\in F}\mathcal{F}_{M_j}(A)$ over all finite sets $F$ and the map is induced by the universal property of the colimit and the lax monoidal property. 
\end{remark}
\begin{theorem}[Francis~\cite{F2}]\label{T:UniquenessFactHom} 
 There is an unique\footnote{up to contractible choices} homology theory for $(X,e)$-manifolds (in the sense of Definition~\ref{D:HomologyforMfld}), which is called factorization homology\footnote{the name comes from the fact that it satisfies the factorization property (Remark~\ref{R:factpointed}). Another name is topological chiral homology.}.  
\end{theorem}
Factorization homology is defined in \cite{L-HA} and \emph{its value on a $(X,e)$-manifold $M$ and $\text{Disk}_n^{(X,e)}$-algebra $A$ is denoted $\int_{M}A$}. 
\begin{remark}[\textbf{other coefficients}]\label{R:factgeneral} In Definition~\ref{D:HomologyforMfld} and Theorem~\ref{T:UniquenessFactHom}, one can replace the symmetric monoidal category $(\hkmod, \otimes)$ by any symmetric monoidal $\infty$-category $(\mathcal{C},\otimes)$ which has all colimits and whose monoidal structure commutes with geometric realization and filtered colimits, see~\cite{F, F2, AFT}.
\end{remark}
\begin{remark}[finite variant]
If one restricts to the full subcategory of $\Mfldn^{(X,e)}$ spanned by the manifolds which have finitely many connected components which are the interior of closed manifolds, then axiom ii) becomes equivalent to asking $\mathcal{F}$ to be naturally symmetric monoidal and Theorem~\ref{T:UniquenessFactHom} stills holds in this context.
\end{remark}

Note that factorization homology depends on the $(X,e)$-structure not the underlying topological manifold structure of $M$ in general. For instance, if $M=\R$ is equipped with its standard framing and $A$ is an associative algebra (hence $E_1$), then $\int_{\R} A\cong A$ as an $E_1$-algebra. However, if $N=\R$ is equipped with the opposite framing (pointing toward $-\infty$), then $\int_{N} A\cong A^{op}$ (where $A^{op}$ is the algebra with opposite multiplication) as an $E_1$-algebra (see~\cite{F, L-HA} for more general statements).
\begin{remark}\label{R:facthomwithcoeff}
 In particular,  Theorem~\ref{T:UniquenessFactHom} implies that the functor $\mathcal{F} \mapsto \mathcal{F}_{\R^n}$ from, the category of symmetric monoidal functors $\Mfldn^{(X,e)}\to \hkmod$ satisfying excision\footnote{i.e. axiom iii) in Definition~\ref{D:HomologyforMfld}}, to the category of $\Disk_n^{(X,e)}$-algebras (which is a  well defined functor by Lemma~\ref{L:homthMfldisEn}) is a natural equivalence.

\begin{definition}\label{D:facthomwithcoeff}
Let $A\in \Disk_n^{(X,e)}\textbf{-Alg}$. The homology theory for $(X,e)$-manifolds defined\footnote{in the sense of Remark~\ref{R:facthomwithcoeff}} by $A$ will be called factorization homology (or homology theory) with coefficient in $A$.
\end{definition}
\end{remark}

\begin{example}[Hochschild homology]\label{Ex:FacS1=Hoch}Let $A$ be a differential graded associative algebra (or even an $A_\infty$-algebra) and choose a framing of $S^1=SO(2)$ induced by its Lie group structure. We can use excision to evaluate the factorization homology with value in $A$ on the framed manifold $S^1$.  
Here, we see the circle as being obtained by gluing two intervals: $S^1= \R\cup_{\{1,-1\}\times \R} \R$, see Figure \ref{fig:HHasFact} .
\begin{figure}[b]
\begin{minipage}[c]{0.5\textwidth}
\includegraphics[scale=0.6]{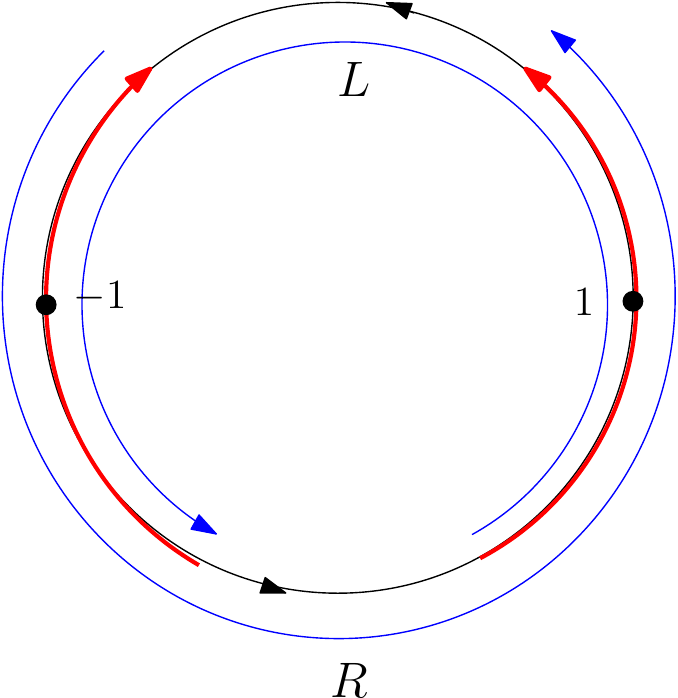}\end{minipage} \hfill \begin{minipage}[c]{0.48\textwidth}
\caption{The decomposition of the circle $S^1$ into 2 intervals (pictured in blue just across the circle) $L\cong \R$ and $R\cong \R$ along a trivialization $\{1,-1\}\times \R$ (pictured in red). The arrows are indicating the orientations/framing of the circle and other various pieces of the decomposition.}\label{fig:HHasFact}\end{minipage}
\end{figure}

Note that the induced framing on $\{1,-1\}\times \R$ correspond to the standard framing of $\R$ on the component $\{1\}\times \R$ and the opposite framing on the component $\{-1\}\times \R$ so that $\int_{\{-1\}\times \R} A = A^{op}$ (see Example~\ref{ex:E1frombasis}).  Thus, by excision we find that
\begin{equation}\label{eq:HHaschiral} \int_{S^1}A \cong \int_{\R}A \mathop{\otimes}\limits^{\mathbb{L}}_{\int_{\{1,-1\}\times \R} A} \int_{\R} A \cong A \mathop{\otimes}\limits^{\mathbb{L}}_{A\otimes A^{op}} A \cong HH(A) \end{equation}
where $HH(A)$ is the usual Hochschild homology\footnote{at least if $A$ is projective over $k$; if $A$ is not projective over $k$, there are several variants of Hochschild homology, the one we are considering is the derived version and correspond to what is sometimes called shukla homology \cite{Shukla, Qu-cohomology}} of $A$ with value in itself. 
\end{example}
\begin{example}\label{ex:FactofConf}
Let $\text{Free}_n$ be the free $E_n$-algebra on $k$, which is naturally a $\Disk_n^{un}$-algebra (Example~\ref{E:DisknAlg}). It can thus be evaluated on any manifold.
\begin{proposition}[\cite{AFT}]\label{P:FactofConf} Let $M$ be a manifold. Then $\int_{M} \text{Free}_n \, \cong \, C_{\ast}\big(\coprod_{n\in \mathbb{N}}\text{Conf}_n(M)\big)$ where $\text{Conf}_n(M)$ is the space of configurations of $n$ -unordered points in $M$.
\end{proposition}
In particular, factorization homology is not an homotopy invariant of manifolds (since configurations spaces of unordered points are not, see~\cite{LS}). 
By considering configuration spaces of points with labels, one has a similar result for the free $E_n$-algebra $\text{Free}_n(V)$ associated to $V\in \hkmod$, see~\cite{AFT}.
\end{example}

\begin{example}[\textbf{Non-abelian Poincar\'e duality}]
Let us now mention another important example of computation of factorization homology. 
Let $(Y,y_0)$ be a pointed space and $\Omega^n(Y):=\{ f: [0,1]^n \to Y, \, f(\partial [0,1]^n)=y_0\}$ be its $n$-fold based loop space. Then the singular chains  $C_*(\Omega^n(Y))$ has a natural structure of unoriented $E_n$-algebra.
\begin{theorem}[non-abelian Poincar\'e duality, Lurie~\cite{L-HA}] \label{T:nonabelianPoincareduality}
If $M$ is a manifold  of dimension $n$ and $Y$ an  $n-1$-connective pointed space, then $$\int_{M}C_*(\Omega^n(Y)) \cong C_*(\Map_c(M,Y))$$ where
$\Map_c(M,Y)$ is the space of compactly supported maps from $M$ to $Y$.
\end{theorem}
If  $n=1$ and $Y$ is connected,   Theorem~\ref{T:nonabelianPoincareduality} reduces to Goodwillie's quasi-isomorphism~\cite{Go-loop} $HH(C_*(\Omega(Y))) \cong C_*(LY)$ where $LY=\Map(S^1,Y)$ is the free loop space of $Y$.
\end{example}

\begin{remark}[\textbf{Derived functor definition}]
One possible way for defining factorization homology is similar to the one of \S~\ref{SSS:derivedfunctor}. 
Indeed, let $A$ be a $\Disk_n^{(X,e)}$-algebra. Then $A$ defines a covariant functor $\Disk_n^{(X,e)}\to \hkmod$. 
Similarly,  if $M$ is  in $\Mfldn^{(X,e)}$, then it defines a contravariant functor $E_{M}^{(X,e)}:\big(\Disk_n^{(X,e)}\big)^{op}\to \hTop$, given by the formula  $$E_{M}^{(X,e)}\Big(\coprod_{i\in I}\R^n\Big) := \Emb^{(X,e)}\Big(\coprod_{i\in I}\R^n, M\Big).$$ 
The data of $A$ and $M$ thus gave a functor 
$$E_{M}^{(X,e)}\otimes A:\big(\Disk_n^{(X,e)}\big)^{op}\times \Disk_n^{(X,e)} \stackrel{E_{M}^{(X,e)}\times A}\longrightarrow \Top \times \hkmod \stackrel{\otimes}\longrightarrow \hkmod.$$
Here $\Top \times \hkmod \stackrel{\otimes}\to \hkmod$ means the tensor of a space with a chain complex which is equivalent to 
$(X, D_\ast)\mapsto C_\ast(X)\otimes D_\ast$ where $C_\ast(X)$ is the singular chain functor of $X$ (with value in $k$).
\begin{proposition}[\cite{F2}]The factorization homology $\int_{M}A$ is the (homotopy) coend  of $E_{M}^{(X,e)}\otimes A$. In other words:
\begin{multline*}
\int_{M}A \cong E_{M}^{(X,e)}\mathop{\otimes}_{\Disk_n^{(X,e)}}^{\mathbb{L}} A\\  \cong
 \hocolim \left( \coprod_{f:\{1,\dots, q\} \to \{1,\dots, p\}} C_\ast\big(E_{M}^{(X,e)} (\R^n)\big)^{\otimes p} \otimes \Disk_n^{(X,e)}(q,p)\otimes A^{\otimes q}\right. \\ \left. \rightrightarrows  \coprod_{m} C_\ast\big(E_{M}^{(X,e)}(\R^n)\big)^{\otimes m}\otimes A^{\otimes m}\right) \end{multline*}
\end{proposition}
The Proposition remains true with $(\hkmod,\otimes)$ replaced by any symmetric monoidal $\infty$-category satisfying the assumptions of Remark~\ref{R:factgeneral}.
\end{remark}

\section{Factorizations algebras}\label{S:FactAlgebras}
In this section we will give a \v{C}ech type construction of Factorization homology which plays for Factorization homology the same role as sheaf cohomology plays for singular cohomology\footnote{singular cohomology of a paracompact space $X$ can be computed as the cohomology of the constant sheaf $\mathbb{Z}_X$ on $X$ while singular cohomology with twisted coefficient is computed by sheaf cohomology with value in a locally constant sheaf}. This analogue of cosheaf theory is given by factorization algebras which we describe in length here. 

\subsection{The category of factorization algebras} 
We start by describing various categories of (pre)factorization algebras (including the locally constant ones).

Following Costello Gwilliam~\cite{CG}, given a topological space $X$, a \emph{prefactorization} algebra over $X$  is an algebra over the colored operad whose objects are open subsets of $X$  and whose morphisms from $\{U_1, \cdots, U_n\}$ to $V$ are empty unless  when $U_i$'s are mutually disjoint subsets of $U$, in which case they are singletons.  Unfolding the definition, we find 
\begin{definition}\label{D:Prefact} A prefactorization algebra  on $X$ (with value in chain complexes) is a rule that assigns to any open set $U$ a chain complex $\mathcal{F}(U)$ and, to any finite family of pairwise disjoint open sets $U_1,\dots, U_n \subset V$ included in an open $V$, a chain map
 $$\rho_{U_1,\dots, U_n,V}: \mathcal{F}(U_1)\otimes\cdots\otimes \mathcal{F}(U_n) \longrightarrow \mathcal{F}(V).$$ These structure maps are required to satisfy obvious \emph{associativity and symmetry} conditions (see~\cite{CG}): the map $\rho_{U_1,\dots, U_n,V}$ is invariant with respect to the action of the symmetric group $S_n$ by permutations of the factors on its domain (in other words, the map $\rho_{U_1,\dots, U_n,V}$ depends only of the collection $U_1, \dots, U_n, V$ not on the particular choice of ordering of the open sets) and $\rho_{U,U}$ is the identity\footnote{we could weaken this condition to be only a weak-equivalence or actually just a chain map. In the latter case, we will obtain a (homotopy) strictly weaker notion of prefactorization algebras; however, this will \emph{not} change the notion of factorization algebras since the condition of being a factorization algebra (Definition~\ref{D:FactAlgebra}) will imply that $\rho_{U,U}$ is an equivalence as $U$ is always a factorizing cover of itself; since it is idempotent by associativity, we will get that it is homotopy equivalent to the identity} of $\mathcal{F}(U)$.
 Further, the associativity condition is that: for any finite collection of pairwise disjoint open subsets $(V_j)_{j\in J}$ lying in an open subset $W$ together with, for all $j\in J$,  a finite collections $(U_{i,j})_{i\in I_j}$ of pairwise disjoint open subset   lying in $V_j$, the following diagram 
\begin{equation}\label{eq:associativityPFact}
\xymatrix{
\mathop{\bigotimes}\limits_{(i,j)\in \coprod_{r} I_r}\mathcal{F}\left(U_{ij}\right) \ar[rd]_{\bigotimes\limits_{j} \rho_{(U_{ij})_{i\in I_j},V_j}} \ar[rr]^{\rho_{(U_{ij}),W}} & & 
\mathcal{F}\left(W\right) \\
& \bigotimes\limits_{j\in J}\mathcal{F}\left(V_{j}\right) \ar[ru]_{\rho_{(V_j)_{j\in J},W}} & 
}\end{equation}
is commutative.

If $\mathcal{U}$ is an open cover of $X$, we define a \emph{prefactorization algebra on $\mathcal{U}$}, also denoted a \emph{$\mathcal{U}$-prefactorization algebra}, to be the same thing as a prefactorization algebra except that  $\mathcal{F}(U)$ is defined only for $U\in \mathcal{U}$.  
\end{definition}

\begin{remark}One can define a prefactorization algebra with value in any symmetric monoidal category  $(\mathcal{C},\otimes)$ by replacing chain complexes by objects of $\mathcal{C}$.
\end{remark}
\begin{remark}\label{R:prefactpointed}
Prefactorization algebras are \emph{pointed} since the inclusion $\emptyset \hookrightarrow U$ of the empty set in any open induces a canonical map $\mathcal{F}(\emptyset) \to \mathcal{F}(U)$. Further, the structure maps of a prefactorization algebra  exhibit $\mathcal{F}(\emptyset) $  as a  commutative algebra in $(\mathcal{C},\otimes)$ (non necessarily unital) and $\mathcal{F}(U)$ as a $\mathcal{F}(\emptyset) $-module. 
\end{remark}

There is  a \emph{\v{C}ech-complex} associated to a cover $\mathcal{U}$ of an open set $U$. Denoting $P\mathcal{U}$ the set of finite pairwise disjoint open subsets $\{U_1,\dots,U_n \, |\, U_i\in \mathcal{U}\}$ ($n$ is not fixed),  it is, by definition the realization of the
simplicial chain complex
$$\check{C}_\bullet(\mathcal{U},\mathcal{F})= \bigoplus_{\alpha \in P\mathcal{U}}\left( \bigotimes_{U \in \alpha}\mathcal{F}(U)\right)   \leftleftarrows \bigoplus_{(\alpha, \beta)\in P\mathcal{U} \times P\mathcal{U}} \left(\bigotimes_{(U,V) \in \alpha \times \beta}\mathcal{F}(U\cap V) \right)\leftthreearrows \cdots$$
where the horizontal arrows are induced by the natural inclusions as for the usual \v{C}ech complex of a cosheaf (see~\cite{CG}). 

 Let us describe the simplicial structure more precisely. 
In simplicial degree $i$, we get the chain complex $\check{C}_i(\mathcal{U},\mathcal{F}):=\bigoplus\limits_{\alpha\in P\mathcal{U}^{i+1}} \mathcal{F}(\alpha)$  where, for $\alpha =(\alpha_0,\dots, \alpha_i)\in P\mathcal{U}^{i}$,
we denote $\mathcal{F}(\alpha)$ the tensor product of chain complexes (with its natural differential) :
$$\mathcal{F}(\alpha)= \bigotimes_{U_j\in \alpha_j} \mathcal{F}\Big(\bigcap_{j=0}^i U_j \Big).$$ 
We write  $d_{in}: \bigoplus_{\alpha \in P\mathcal{U}^{i+1}} \mathcal{F}(\alpha)\to \bigoplus_{\alpha \in P\mathcal{U}^{i+1}} \mathcal{F}(\alpha)$ the induced differential. 
The face maps $\partial_s: \bigoplus\limits_{\alpha\in P\mathcal{U}^{n+1}} \mathcal{F}(\alpha)\longrightarrow \bigoplus\limits_{\beta\in P\mathcal{U}^{n}} \mathcal{F}(\beta)$  ($s=0\dots n$)  are the direct sum of  maps $\widehat{\rho}_{\alpha}^s:\mathcal{F}(\alpha)\to \mathcal{F}(\widehat{\alpha}^s)$ where $\widehat{\alpha}^s=(\alpha_0,\dots, \alpha_{s-1}, \alpha_{s+1}, \dots, \alpha_n)$ is obtained by discarding the $s^{\text{th}}$-collection of opens in $P\mathcal{U}^{n+1}$. 
Precisely $\widehat{\rho}_{\alpha}^s$ is the tensor product 
$$\bigotimes_{U_j\in \alpha_j} \mathcal{F}\Big(\bigcap_{j=0}^n U_j \Big) \longrightarrow  \bigotimes_{\scriptsize\begin{array}{l} U_k\in \alpha_k,\\ k\neq s\end{array}} \mathcal{F}\Big(\bigcap_{\scriptsize\begin{array}{l} k=0\\  k\neq s\end{array}}^n U_k \Big) $$ of the structure maps associated to the inclusion of opens $\bigcap\limits_{j=0\dots n} U_j$ into $\bigcap\limits_{j\neq s} U_j $. The degeneracies are similarly given by  operations $(\alpha_0,\dots, \alpha_n)\mapsto (\alpha_0, \dots, \alpha_j, \alpha_j, \dots, \alpha_{n})$ doubling a set $\alpha_j$.

\smallskip

The simplicial chain-complex $\check{C}_\bullet(\mathcal{U},\mathcal{F})$ can be made into a chain complex (which is the total complex of a  bicomplex):
$$\check{C}(\mathcal{U},\mathcal{F})= \bigoplus_{\alpha\in P\mathcal{U}} \mathcal{F}(\alpha)  \leftarrow \bigoplus_{\beta \in P\mathcal{U} \times P\mathcal{U}} \mathcal{F}(\beta)[1]  \leftarrow \cdots$$ where the horizontal arrows are induced by the alternating sum of the faces $\partial_j$ in the standard way.
 In other words,  $\check{C}(\mathcal{U},\mathcal{F}) = \bigoplus\limits_{i\geq 0} \Big(\bigoplus\limits_{\alpha\in P\mathcal{U}^{i+1}} \mathcal{F}(\alpha)[i]\Big)$ with differential the sum of  $ \mathcal{F}(\alpha)[i] \stackrel{(-1)^i d_{in}}\longrightarrow \mathcal{F}(\alpha)[i] $ and $$\sum_{j=0}^{n} (-1)^{j}\partial_j: \bigoplus\limits_{\alpha\in P\mathcal{U}^{n+1}} \mathcal{F}(\alpha)[n] \to  \bigoplus\limits_{\beta\in P\mathcal{U}^{n}} \mathcal{F}(\beta)[n-1].$$

\begin{remark} If a cover $\mathcal{U}$ is stable under finite intersections, we only need  $\mathcal{F}$ to be  a prefactorization algebra on $\mathcal{U}$,  to define the \v{C}ech-complex $\check{C}(\mathcal{U},\mathcal{F})$.
\end{remark}

If $\mathcal{U}$ is a cover of an open set $U$, then the structure maps of $\mathcal{F}$ yield canonical maps $\mathcal{F}(\alpha)\to \mathcal{F}(U)$ which commute with the simplicial maps. Thus, we get a natural map of simplicial chain complexes 
$\big(\check{C}_i(\mathcal{U},\mathcal{F})\to \mathcal{F}(U)\big)_{i\geq 0}$ to the constant simplicial chain complex $\big(\mathcal{F}(U)\big)_{i\geq 0}$. Passing to geometric realization, we obtain a canonical chain complex homomorphism:
\begin{equation}\label{eq:canonicalmapFact}
\check{C}(\mathcal{U},\mathcal{F})\longrightarrow \mathcal{F}(U).
\end{equation}
\begin{remark}[\textbf{\v{C}ech complexes in $(\mathcal{C}, \otimes)$}]
If $(\mathcal{C}, \otimes)$ is a symmetric monoidal category with coproducts, we define the \v{C}ech complex of a prefactorization algebra with values in $\mathcal{C}$  in the same way, replacing the direct sum by the coproduct in order to get a simplicial object
$\check{C}_\bullet(\mathcal{U},\mathcal{F})$ in $\mathcal{C}$. If further, $\mathcal{C}$ has a geometric realization, then we obtain the \v{C}ech complex $\check{C}(\mathcal{U},\mathcal{F})\in \mathcal{C}$ exactly as  for chain complexes above and the canonical map~\eqref{eq:canonicalmapFact} is also well defined.
\end{remark}

\begin{definition}\label{D:FactAlgebra} An open cover of $\mathcal{U}$  is \emph{factorizing} if,  for all finite collections $x_1,\dots, x_n$ of distinct points in $U$, there are pairwise disjoint open subsets $U_1,\dots, U_k$ in $\mathcal{U}$ such that $\{x_1,\dots, x_n\} \subset \bigcup_{i=1}^k U_i$.

A prefactorization algebra $\mathcal{F}$ on $X$  is said to be a \emph{homotopy\footnote{we can also say \emph{derived factorization algebra}. Unless otherwise specified, the word factorization algebra will always mean a homotopy factorization algebra in these notes} factorization algebra} if, for all open subsets $U\in Op(X)$ and for every factorizing cover $\mathcal{U}$ of $U$, the canonical map
$\check{C}(\mathcal{U},\mathcal{F})\to \mathcal{F}(U)$ is a quasi-isomorphism (see~\cite{Co,CG}). 
\end{definition}
 Note that we do \emph{not} consider the lax version of homotopy factorization algebra defined in~\cite{CG}. 
\begin{remark}[\textbf{Factorization property}]\label{R:factpointed}
If $\mathcal{F}$ is a factorization algebra and $U_1, \dots, U_i$ are disjoint open subsets of $X$, the factorization condition implies that the structure map \begin{equation}\label{eq:factorizationproperty}\mathcal{F}(U_1)\otimes \cdots \otimes \mathcal{F}(U_i) \longrightarrow \mathcal{F}(U_1\cup \dots \cup U_i)\end{equation} is a quasi-isomorphism. In particular, this implies that our definition of factorization algebra agrees with the one of~\cite{CG}. 

Another consequence is that  $\mathcal{F}(\emptyset)\cong k$ (or, more generally, is the unit of the symmetric monoidal category $\mathcal{C}$ if $\mathcal{F}$ has values in $\mathcal{C}$).

The fact that the map~\eqref{eq:factorizationproperty} is an equivalence is called the \emph{factorization property} in the terminology of Beilinson-Drinfeld~\cite{BD}, in the sense that the value of $\mathcal{F}$ on disjoint opens factors through its value on each connected component. 
\end{remark}
\begin{example}[the trivial factorization algebra]\label{ex:trivialFact}
 The \emph{\textbf{trivial}} prefactorization algebra $k$ is the \emph{constant} prefactorization algebra given by the rule $U\mapsto k(U):=k$, with structure maps given by multiplication. It is  a (homotopy) factorization algebra. It is in particular  locally constant over any stratified space $X$ (Definitions~\ref{D:locallyconstant} and~\ref{D:lcFacStratified}).
 
 One defines similarly the trivial factorization algebra over $X$ with values in a symmetric monoidal $\infty$-category  $(\mathcal{C}, \otimes)$ by the rule $U\mapsto \textbf{1}_{\mathcal{C}}$ where  $\textbf{1}_{\mathcal{C}}$ is the unit of the monoidal structure.
\end{example}

\begin{remark}[genuine factorization algebras] \label{R:NaiveFacAlg}The notion of homotopy (or derived) factorization algebra in Definition~\ref{D:FactAlgebra} is a homotopy version of a more naive, un-derived, version of factorization algebra. This  version is a prefactorization algebra such that the following sequence
$$ \Big(\bigoplus\limits_{\alpha\in P\mathcal{U}^{2}} \mathcal{F}(\alpha)\Big) \underset{\partial_1}{\stackrel{\partial_0}\rightrightarrows}  \Big(\bigoplus\limits_{\beta\in P\mathcal{U}} \mathcal{F}(\beta)\Big) \to \mathcal{F}(U) $$
is (right) exact for any factorizing cover $\mathcal{U}$ of $U$. In other words we ask for a similar condition as in Definition~\ref{D:FactAlgebra} but with the truncated \v{C}ech complex. We refer to prefactorization algebras satisfying  this condition as \emph{genuine factorization algebras} (they are also called \emph{strict} in~\cite{CG}). Note that a genuine factorization algebra is \emph{not} a (homotopy) factorization algebra in general.  Homotopy factorization algebras are to genuine factorization algebras what homotopy cosheaves are to cosheaves; that is they are obtained by replacing the naive version by an acyclic resolution.
\end{remark}

When $X$ is a manifold we have the notion of locally constant factorization algebra which roughly means that the structure maps do not depend on the size of the open subsets but only their relative shapes:
\begin{definition}\label{D:locallyconstant}  Let $X$ be a topological manifold of dimension $n$. We say that an open subset $U$ of $X$ is a \emph{disk} if $U$ is homeomorphic to a standard euclidean disk $\R^n$.
A  (pre-)factorization algebra over $X$  is  \emph{locally constant} if for any inclusion of open disks $U\hookrightarrow V$  in $X$, the structure map $\mathcal{F}(U) \to \mathcal{F}(V)$ is a quasi-isomorphism. 
\end{definition}
Let us mention that a locally constant prefactorization algebra is automatically a (homotopy) factorization algebra, see Remark~\ref{R:Ndisk(M)}.

\begin{definition}\label{D:CatofFacAlg}
A \emph{morphism}  $\mathcal{F}\to \mathcal{G}$ of (pre)factorization algebras over $X$ is the data of chain complexes morphisms $\phi_U:\mathcal{F}(U)\to \mathcal{G}(U)$ for every open set $U\subset X$ which commute with the structures maps; that is the following diagram 
$$\xymatrix{\mathcal{F}(U_1)\otimes \cdots \otimes \mathcal{F}(U_i) \ar[d]_{\otimes \phi_j(U_j)}\ar[rrr]^{\rho_{U_1,\dots,U_i,V}}&&& \mathcal{F}(V)\ar[d]^{\phi_V}\\ \mathcal{G}(U_1)\otimes \cdots \otimes \mathcal{G}(U_i) \ar[rrr]^{\rho_{U_1,\dots,U_i,V}} &&& \mathcal{G}(V)} $$
is commutative for any pairwise disjoint finite family $U_1,\dots, U_i$ of open subsets of an open set $V$.
Morphisms of (pre)factorization algebras are naturally enriched over topological space. Indeed, we have   \emph{mapping spaces $\Map(\mathcal{F}, \mathcal{G})$} defined as the geometric realization of the simplicial set $$n\mapsto \Map(\mathcal{F}, \mathcal{G})_n:=\{\text{prefactorization algebras morphisms from } \mathcal{F} \text{ to }C^{\ast}(\Delta^n)\otimes\mathcal{G}   \} $$
where $ C^{\ast}(\Delta^n)\otimes\mathcal{G}$ is the prefactorization algebra whose value on an open set $U$ is $ C^{\ast}(\Delta^n)\otimes\mathcal{G}(U)$. We obtain in this way $\infty$-categories of (pre)factorization algebras (as in~\S~\ref{S:DKL}, Example~\ref{ex:topologicalasinfty}).

The $\infty$-\emph{category of prefactorization algebras over $X$ is denoted $\textbf{PFac}_X$} and similarly we write
 \emph{$\textbf{Fac}_X$ for the $\infty$-categories of factorization algebras over $X$}  and  \emph{$\textbf{Fac}^{lc}_X$ for the  locally constant ones} (which is a full subcategory). Also if $(\mathcal{C},  \otimes)$ is a symmetric monoidal ($\infty$-)category (with coproducts and geometric realization), 
we will denote  $\textbf{Fac}_X(\mathcal{C})$,  $\textbf{Fac}_X^{lc}(\mathcal{C})$ the $\infty$-categories of \emph{factorization algebras in $\mathcal{C}$}.
\end{definition}
 Note that the embedding $\textbf{Fac}_X\to \textbf{PFac}_X$ is a fully faithful embedding.

The underlying tensor product\footnote{recall our convention that if $k$ is not a field, the tensor product really means   derived tensor product}  of chain complexes induces a tensor product of factorization algebras which is computed pointwise: for $\mathcal{F}, \mathcal{G}\in \textbf{PFac}_X$ and an open set $U$,  we have
\begin{equation}\label{eq:tensorproductFact}\big(\mathcal{F}\otimes  \mathcal{G}\big) (U) := \mathcal{F}(U)\otimes \mathcal{G}(U)\end{equation}
 and the structure maps are just the tensor product of the structure maps. If $(\mathcal{C},\otimes)$ is symmetric monoidal,  the same construction yields a monoidal structure on  $\textbf{PFac}_X(\mathcal{C})$. Its unit is the trivial factorization algebra with values in $\mathcal{C}$ (Example~\ref{ex:trivialFact}).
 
\begin{proposition}[Costello-Gwilliam~\cite{CG}] The ($\infty$-)categories $\textbf{PFac}_X(\mathcal{C})$, $\textbf{Fac}_X(\mathcal{C})$, $\textbf{Fac}^{lc}_X(\mathcal{C})$\footnote{the latter is defined when $X$ is a manifold} are symmetric monoidal  with tensor product given by~\eqref{eq:tensorproductFact}.
\end{proposition}
\begin{remark}[\textbf{Restrictions}]If $Y\subset X$ is an open subspace, then we have natural restriction functors $\textbf{PFac}_X(\mathcal{C})\to \textbf{PFac}_Y(\mathcal{C})$, $\textbf{Fac}_X(\mathcal{C})\to \textbf{Fac}_Y(\mathcal{C})$. When $X$ is a manifold, the same holds for locally constant factorization algebras.
\begin{definition}If $U$ is an open subset of $X$ and $\mathcal{A}\in \textbf{PFac}(X)$, we  write \emph{$\mathcal{A}_{|U}\in \textbf{PFac}_{U}$ for the restriction 
of $\mathcal{A}$ to $U$} and similarly for (possibly locally constant) factorization algebras. \end{definition}
A homeomorphism  $f:X\stackrel{\simeq} \to Y$  induces  isomorphisms $\textbf{PFac}_X\cong \textbf{PFac}_Y$ and
$\textbf{Fac}_X\cong \textbf{Fac}_Y$ (or $\textbf{Fac}^{lc}_X\cong \textbf{Fac}^{lc}_Y$ when $X$ is a manifold) realized by the functor $f_*$ see~\S~\ref{SS:pushforward}.
\end{remark}

\subsection{Factorization homology and locally constant factorization algebras}\label{S:EnasFact}
 We now explain the relationship between the \v{C}ech complex of a factorization algebras and factorization homology. 
 We first start to express $\text{Disk}^{(X,TX)}_n$-algebras in terms of factorization algebras. For simplicity, we assume in  \S~\ref{S:EnasFact} that manifolds are smooth. For topological manifolds one obtains the same result as below by  replacing geodesic convex neighborhoods by  families of embeddings $\R^n\to M$ wich preserves the $(M,TM)$-structure and whose images form a basis of open of $M$. 
 
Let $M$ be a manifold with an $(X,e)$-structure. Every open subset $U$ of $M$ inherits a canonical $(X,e)$-structure  given by the factorization  $U\hookrightarrow M\stackrel{f}\to X \stackrel{e}\to B\Homeo(\R^n)$ of the map $e_U: U\to B\Homeo(\R^n)$ classifying the tangent bundle of $U$. This construction extends canonically into a  functor $$f_*:\Disk_n^{(M,TM)} \longrightarrow  \Disk_n^{(X,e)}$$ and (by Definition~\ref{D:DiskXstructured}) we have 
\begin{lemma} An $(X,e)$-structure $M\stackrel{f}\to X \stackrel{e}\to B\Homeo(\R^n)$ on a manifold $M$ induces a
  functor $f^*:\Disk_n^{(X,e)}\textbf{-Alg} \to \Disk_n^{(M,TM)}\textbf{-Alg}$.
\end{lemma}

Now let $A$ be a $\Disk_n^{(X,e)}$-algebra and choose a metric on $M$. A family of  pairwise disjoint open convex geodesic neighborhoods $U_1,\dots, U_i$ which lies in a convex geodesic neighborhood\footnote{which is thus canonically homeomorphic to an euclidean disk} $V$,  defines an $(X,e)$-structure preserving embedding $i_{U_1,\dots, U_i,V}\in \Emb^{(X,e)}\big(\coprod_{\{1,\dots, i\}} \R^n, \R^n\big)$ so that the $\Disk_n^{(X,e)}$-algebra structure of $A$  yields a structure map 
$$\mu_{U_1,\dots, U_i, V} : A^{\otimes i} \stackrel{i_{U_1,\dots, U_i,V}} \longrightarrow  \Emb^{(X,e)}\Big(\coprod_{\{1,\dots, i\}} \R^n, \R^n\Big)\otimes A^{\otimes i} \to A. $$
This allows us to define a \emph{prefactorization algebra $\mathcal{F}_A$ on open convex geodesic subsets} by the formula $\mathcal{F}_A(V):=A$. Since, the convex geodesic   neighborhoods form a basis of open which is stable by intersection,  for any open set $U\subset M$, we have the \v{C}ech complex\footnote{the construction is actually the extension of a factorization algebra on $\mathcal{CV}(U)$ as in Section~\ref{S:OperationsforFact} } 
$$\check{C}(\mathcal{CV}(U), \mathcal{F}_A)$$ where $\mathcal{CV}(U)$ is the factorizing cover of $U$ given by the geodesic convex open subsets of $U$. The following result shows that $\mathcal{F}_A$ is actually (the restriction of) a factorization algebra and computes factorization homology. 
 
\begin{theorem}[{\cite{GTZ2}}]\label{T:Theorem6GTZ2}  Let $A$ be a $\Disk_n^{(X,e)}$-algebra. 
\begin{itemize}\item The rule $M\mapsto \check{C}(\mathcal{CV}(M), \mathcal{F}_A)$ is a  homology theory for $(X,e)$-manifolds. In particular the \v{C}ech complex  is independent of the choice of the metric and computes factorization homology of $M$: $$\check{C}(\mathcal{CV}(M), \mathcal{F}_A)\simeq \int_M A.$$ 
\item  The functor $(U,A)\mapsto\check{C}(\mathcal{CV}(U), \mathcal{F}_A)$ induces an equivalence of $\infty$-categories  $\text{Disk}^{(M,TM)}_n\textbf{-Alg}\stackrel{\simeq}\longrightarrow \textbf{Fac}^{lc}_M$.
\end{itemize}
\end{theorem}

Since we have a preferred choice of framing for $\R^n$, the projection map $\R^n\to  pt$ induces an equivalence of $\infty$-categories $\text{Disk}^{(\R^n,T\R^n)}_n\stackrel{\simeq} \to \text{Disk}^{(pt,e)}_n$ and thus  equivalences $\text{Disk}^{(X,TX)}_n\textbf{-Alg} \cong \text{Disk}^{fr}_n\textbf{-Alg}\cong E_n\textbf{-Alg}$ (see Example~\ref{E:DisknAlg}).
Hence  Theorem~\ref{T:Theorem6GTZ2} is a slight generalization of  the following beautiful result.
\begin{theorem}[Lurie \cite{L-HA}]\label{P:En=Fact}
There is a natural equivalence of $\infty$-categories $$E_n\textbf{-Alg} \;\cong \;\textbf{Fac}^{lc}_{\R^n}.$$
The functor $\textbf{Fac}^{lc}_{\R^n}\to E_n\textbf{-Alg}$ is given by the global section (i.e. the pushforward $p_*$ where $p:\R^n\to pt$, see \S~\ref{SS:pushforward}) and the inverse functor is precisely given by factorization homology.
\end{theorem}
Locally constant factorization algebras on $\R^n$  are thus \emph{a model for $E_n$-algebras}. More generally, locally constant factorization algebras are a model for $\Disk_n^{(X,TX)}$-algebras in which the cosheaf property replaces\footnote{note that factorization algebras  are described by operads in \emph{discrete} space together with the \v{C}ech condition, see Remark~\ref{R:Ndisk(M)}} some of the  higher homotopy machinery needed for studying these algebras (at the price of working with \lq\lq{}lax\rq\rq{} algebras).

\begin{remark}\label{R:En=Fact} Theorem~\ref{P:En=Fact} is the key example of the relationship between factorization algebras and factorization homology so we now explain the equivalence in more depth.
 Recall that $E_n\textbf{-Alg}$ is the $\infty$-category
of algebras over the operad $\text{Cube}_n$ of little cubes.
 It is  equivalent to the $\infty$-category $\Disk_n^{fr}\textbf{-Alg}$ since we have an equivalence of operads $\text{Cube}_n \stackrel{\simeq}\longrightarrow \Disk_n^{fr}$  induced by a choice of diffeomorphism $\theta:(0,1)^n\cong \R^n$.
Consider the open cover $\mathcal{D}$ of $\R^n$ 
consisting of all open disks and denote $\textbf{PFac}_{\mathcal{D}}^{lc}$\footnote{Note that the category of $\mathcal{D}$-prefactorization algebras is the category of algebras over the colored operad $N(\Disk(\R^n))$, see Remark~\ref{R:Ndisk(M)}.}
 the category of $\mathcal{D}$-prefactorization algebras which satisfy the locally constant condition (Definition~\ref{D:locallyconstant} and Definition~\ref{D:Prefact}). Evaluation of a $\Disk_n^{fr}$-algebra on an open disk yields a functor
$\Disk_n^{fr}\textbf{-Alg} \to \textbf{PFac}_{\mathcal{D}}^{lc}$ which is an equivalence by~\cite[\S 5.2.4]{L-HA}. 

Similarly, let $\mathcal{R}$ be the cover of $(0,1)^n$ by open rectangles and $\textbf{PFac}_{\mathcal{R}}^{lc}$ be the category of locally constant  $\mathcal{R}$-prefactorization algebras. Evaluation of a   $\text{Cube}_n$-algebra on rectangles yields a functor
$E_n\textbf{-Alg}=\text{Cube}_n\textbf{-Alg} \to \textbf{PFac}_{\mathcal{R}}^{lc}$. Denote $\textbf{Fac}_{\mathcal{D}}^{lc}$,  $\textbf{Fac}_{\mathcal{R}}^{lc}$ the category of locally constant factorization algebras over the covers $\mathcal{D}$, $\mathcal{R}$ respectively (see \S~\ref{SS:extensionfrombasis}).
We have two commutative diagram and an equivalence between them induced by the diffeomorphism $\theta$: 
\begin{equation}\label{eq:diagEn=Fact1}\xymatrix{ 
\Disk_n^{fr}\textbf{-Alg} \ar[rr]^{\simeq} \ar@{.>}[rrd] &  &\textbf{PFac}_{\mathcal{D}}^{lc} & &\textbf{Fac}_{\R^n}^{lc} \ar[ll]
\ar[dll]^{\simeq}  \\
& &\textbf{Fac}_{\mathcal{D}}^{lc} \ar@{^{(}->}[u] && }\end{equation}
\begin{equation}\label{eq:diagEn=Fact2}\xymatrix{
E_n\textbf{-Alg} \ar@{.>}[rrd] \ar[rr]^{\simeq} &&  \textbf{PFac}_{\mathcal{R}}^{lc} && \textbf{Fac}_{(0,1)^n}^{lc} \ar[lld]^{\simeq}\ar[ll] \\
&& \textbf{Fac}_{\mathcal{R}}^{lc} \ar@{^{(}->}[u] & &
} \end{equation}
where the dotted arrows exists by Theorem~\ref{T:Theorem6GTZ2} and the diagonal right equivalences are given by Proposition~\ref{P:extensionfrombasis}. 
Since the embedding of  factorization algebras in  prefactorization algebras (over any cover or space) is fully faithful, we obtain that all maps in Diagrams~\eqref{eq:diagEn=Fact1} and~\eqref{eq:diagEn=Fact2} are equivalences so that we recover the equivalence $E_n\textbf{-Alg} \cong \textbf{Fac}^{lc}_{\R^n}\cong \textbf{Fac}^{lc}_{(0,1)^n}$ of Theorem~\ref{P:En=Fact}, also see~\cite{Ca-HDR}.
\end{remark}

\begin{example}[Constant factorization algebra on framed manifolds]\label{ex:framedisEn}
Let $M$ be a framed manifold of dimension $n$. By Theorem~\ref{T:Theorem6GTZ2} or~\cite{L-HA, GTZ2},  any $E_n$-algebra $A$ yields a locally constant factorization algebra $\mathcal{A}$ on $M$ which is defined by assigning to any geodesic disk $D$ the chain complex $\mathcal{A}(D) \cong A$. We call such a factorization algebra the \emph{constant factorization algebra} on $M$ associated to $A$ since it satisfies the property that there is a (globally defined) $E_n$-algebra $A$ together with  natural (with respect to the structure map of the factorization algebra) quasi-isomorphism $\mathcal{A}(D) \stackrel{\simeq}\to A$ for every disk $D$.

 In particular, for $n=0,1,3,7$,  there is a faithful embedding of $E_n$-algebras into   constant factorization algebras over the $n$-sphere $S^n$. 
\end{example}

If a manifold $X$ is not framable, we can obtain constant factorization algebras on $X$ by using  (un)oriented $E_n$-algebras instead of plain $E_n$-algebras: 
\begin{example}[Constant factorization algebra on (oriented) smooth manifolds]
Let $A$ be an unoriented $E_n$-algebra (i.e. a $\Disk_n^{un}$-algebra, Example~\ref{E:DisknAlg}). Then $A$ yields  a (locally) constant factorization $\mathcal{A}$ algebra on any smooth manifold of dimension $n$  which is defined by assigning to any geodesic disk $D$ the chain complex $\mathcal{A}(D) \cong A$. 

Similarly, an oriented $E_n$-algebra (i.e. a $\Disk_n^{or}$-algebra) yields a (locally) constant factorization algebra on any oriented dimension $n$ manifold.
\end{example}

\begin{example}[Commutative factorization algebras]\label{E:EinftygivesNDisk}
The canonical functor  $E_\infty\textbf{-Alg} \to \Disk_n^{(X,e)}\textbf{-Alg}$ (see Example~\ref{E:ComisDisk}) shows that any $E_\infty$-algebras induces a canonical structure of (locally) constant factorization algebra on any (topological) manifold $M$. In that case, the factorization homology reduces to the derived Hochschild chains according to Theorem~\ref{T:Theorem6GTZ2} and Theorem~\ref{T:Fact=CH} below. See~\S~\ref{SS:ComFactAlgebras} for more details. 
\end{example}

\begin{theorem}\label{T:Fact=CH}
If $A$ is an $E_\infty$-algebra\footnote{for instance a (differential graded) commutative algebra}, then, for every topological manifold $M$, there is an natural equivalence ${CH}_M(A)\cong \int_M A$. In particular, factorization homology of $E_\infty$-algebras extends uniquely as an homology theory for spaces (see Definition~\ref{D:axioms}).
\end{theorem}
\begin{proof}This  is proved in~\cite{GTZ2}, also see~\cite{F}.  The result essentially follows by uniqueness of the  homology theories (Theorem~\ref{T:UniquenessFactHom}). Namely,  if $\mathcal{H}_A$ is an homology theory for spaces whose value on a point  is $A$, then $\mathcal{H}_A(\R^n)=A$ ($\R^n$ is contractible) and further $\mathcal{H}_A$ satisfies the monoidal and excision axioms of a homology theory for manifolds.  
\qed \end{proof}

\begin{example}[pre-cosheaves]\label{ex:Precosheafyeildsprefact} Let $\mathcal{P}$ be a pre-cosheaf on $X$ (with values in vector spaces or chain complexes).  For any open  $U \subset X$,  set $\mathcal{F}(U):= S^\bullet(\mathcal{P}(U))=\bigoplus_{n\geq 0} (\mathcal{P}(U)^{\otimes n})_{S_n}$ where $S^\bullet$ is the free (differential graded) commutative algebra functor. Then  $\mathcal{F}$ is a prefactorization algebra with structure maps given by the algebra structure of $S^{\bullet}(\mathcal{P}(V))$: 
\begin{multline*}\mathcal{F}(U_1)\otimes \cdots \otimes \mathcal{F}(U_i) \cong  S^\bullet(\mathcal{P}(U_1))\otimes \cdots \otimes  S^\bullet(\mathcal{P}(U_i)) \\ \stackrel{\otimes S^\bullet(\mathcal{P}(U_i\to V))}\longrightarrow S^\bullet(\mathcal{P}(V))\otimes \cdots \otimes  S^\bullet(\mathcal{P}(V)) \longrightarrow S^\bullet(\mathcal{P}(V))=\mathcal{F}(V).\end{multline*}
\begin{proposition}[\emph{cf}.~\cite{CG}]\label{P:Precosheafyeildsprefact}
If $\mathcal{F}$ is a homotopy cosheaf, then $\mathcal{F}$ is a factorization algebra (not necessarily locally constant). 

In characteristic zero, if $\mathcal{P}$ is a homotopy cosheaf, then  $\mathcal{F}$ is a factorization algebra. \end{proposition}
\end{example}
\begin{example}[Observables] \label{ex:ObservablesasFact} Several examples of (pre-)factorization algebras arising from theoretical physics (more precisely from perturbative quantum field theories) are described in the beautiful work~\cite{CG, Co}. They  arose as deformations of those obtained as in the previous Example~\ref{ex:Precosheafyeildsprefact}.
For instance, let $E\to X$ be a (possibly graded) vector bundle over a smooth manifold $X$. Let $\mathcal{E}$ be the sheaf of smooth sections of  $E$ (which may be endowed with a differential which is a differential operator) and $\mathcal{E}^{'}$ be  its associated distributions. 
The above construction yields a (homotopy) factorization algebra $U\mapsto S(\mathcal{E}^{'}(U))=\bigoplus_{n\geq 0} \text{Hom}_{\mathcal{O}_X(U)}(\mathcal{E}(U)^{\otimes n}, \R)$.  In~\cite{CG}, Costello-Gwilliam have been refining this example to equip the classical observables of a classical field theory with the structure of a factorization algebra (with values in $P_0$-algebras, Example~\ref{Ex:PnAlg}). Their construction is a variant of the classical AKSZ formalism~\cite{AKSZ}. Related constructions are studied in~\cite{Paugam-book}.

Further, the quantum observables of a quantization of the classical field theory, when they exist, also form a factorization algebra (not necessarily locally constant). A very nice example of this procedure arises when $X$ is an elliptic curve, see~\cite{Co}.

These factorization algebras (with values in lax $P_0$-algebras) encode the algebraic structure governing observables of the field theories (in the same way as the observables of classical mechanics are described by the algebra of smooth functions on a manifold together with its Poisson bracket).
 Very roughly speaking, the \emph{locally constant factorization algebras correspond to observables of topological field theories}. 
\end{example}

\begin{example}[Enveloping factorization algebra of a  dg-Lie algebra]
 Let $\mathcal{L}$ be a homotopy cosheaf of differential graded Lie algebras on a Hausdorff space $X$, over a characteristic zero ring.
 For instance, $\mathcal{L}$ can be the cosheaf of \emph{compactly supported} forms $\Omega^{\bullet, c}_{dR,M}\otimes \mathfrak{g}$ with value in $\mathfrak{g}$ where $\mathfrak{g}$ is a differential graded Lie algebra and $\Omega^{\bullet}_{dR,M}$ is the (complex of) sheaf on a manifold $M$ given by the de Rham complex. If $M$  is a complex manifold, another interesting example is obtained by substituting the Dolbeaut complex to the de Rham complex. 
 
 \smallskip
 
 For any open $U\subset X$, we can form the Chevalley-Eilenberg  chain complex $C^{CE}_\bullet\big(\mathcal{L}(U)\big)$ of the (dg-)Lie algebra $\mathcal{L}(U)$.
 Its  underlying $k$-module (see~\cite{Weibel-book}) is 
given by 
 $$C^{CE}_\bullet\big(\mathcal{L}(U)\big):= S^\bullet\big(\mathcal{L}(U)[1]\big) $$ and its differential is induced by the Lie bracket and inner differential of $\mathcal{L}$. 
 The structure maps of Example~\ref{ex:Precosheafyeildsprefact} (applied to $\mathcal{F}=\mathcal{L}[1]$) are maps of chain complexes (since $\mathcal{L}$ is a precosheaf of dg-Lie algebras), hence make $C^{CE}_\bullet\big(\mathcal{L}(-)\big)$ a prefactorization algebra over $X$, which we denote $C^{CE}(\mathcal{L})$ (note that this construction only requires $\mathcal{L}$ to be a precoheaf of dg-Lie algebras).
 As a corollary of Proposition~\ref{P:Precosheafyeildsprefact}, one obtains
 \begin{corollary}[Theorem 4.5.3,~\cite{Gw-Thesis}]If $\mathcal{L}$ is a homotopy cosheaf of dg-Lie algebras, the prefactorization algebra 
  $C^{CE}(\mathcal{L})$ is a (homotopy) factorization algebra.
 \end{corollary}
 The above corollary extends  to homotopy cohseaves of $L_\infty$-algebras as well.
 
 This construction actually generalizes the construction of the \emph{universal enveloping algebra} of a Lie algebra which corresponds to the case $\mathcal{L}=\Omega^{\bullet,c}_{dR,\mathbb{R}}\otimes \mathfrak{g}$ (\cite{Gw-Thesis}). 
 
 \smallskip
 
 More generally the observables of \emph{Free Field Theories} can be obtained this way, see~\cite{Gw-Thesis} for many examples. 
\end{example}

\begin{remark}[Algebras over disks in $X$]\label{R:Ndisk(M)} Assume $X$ has a cover by euclidean neighborhoods.
One can define a colored operad whose objects are open subsets of $X$ that are homeomorphic to $\mathbb R^n$ and whose morphisms from $\{U_1, \cdots, U_n\}$ to $V$ are empty except  when the $U_i$'s are mutually disjoint subsets of $V$, in which case they are singletons. 
We can take the monoidal envelope of this operad  (as in Appendix~\ref{S:EnAlg} or \cite[\S 2.4]{L-HA}) to get a symmetric monoidal $\infty$-category $\text{Disk}(X)$ (also see \cite{L-HA}, Remark 5.2.4.7).  
For any symmetric monoidal $\infty$-category $\mathcal{C}$, we thus get the $\infty$-category  $\text{Disk}(X)\textbf{-Alg}:=\textbf{Fun}^{\otimes}(\text{Disk}(X), \mathcal{C})$ of $\text{Disk}(X)$-algebras. 
Unfolding the definition we find  that a $\text{Disk}(X)$-algebra  is precisely a $\mathcal{D}_{\textit{isk}}$-prefactorization algebra over $X$ where $\mathcal{D}_{\textit{isk}}$ is the set of all open disks in $X$.

A $\text{Disk}(X)$-algebra is \emph{locally constant} if for any inclusion of open disks $U\hookrightarrow V$  in $X$, the structure map $\mathcal{F}(U) \to \mathcal{F}(V)$ is a quasi-isomorphism (see~\cite{L-HA}). By Theorem~\ref{T:Theorem6GTZ2} and~\cite[\S 5.2.4]{L-HA},  locally constant $N(\text{Disk}(M))$-algebras are the same as locally constant factorization algebras. Hence we have
\begin{proposition} \label{P:Ndisk(M)}A locally constant $\mathcal{D}_{\textit{isk}}$-prefactorization algebra has an unique extension as a locally constant homotopy factorization algebra. In fact,  the functor $\textbf{Fac}^{lc}_X \to \textbf{PFac}^{lc}_{\mathcal{D}_{\textit{isk}}} \cong \text{Disk}(X)\textbf{-Alg}^{lc}$ is an equivalence. \end{proposition}
In particular, a \emph{locally constant prefactorization algebra $\mathcal{F}$ has an unique extension as a locally constant factorization algebra} $\mathcal{F}^{\circ}$ taking the same values as $\mathcal{F}$ on any disk.
\begin{example} The unique extension $\mathcal{F}^{\circ}$ of $\mathcal{F}$ as a factorization algebra can have different values than $\mathcal{F}$ and two different prefactorization algebras on $X$ can have the same values on the open cover $\mathcal{D}_{\textit{isk}}$. 
As a trivial example, let $X=\{x,y\}$ be a two  points discrete set. Then a  $\mathcal{D}_{\textit{isk}}$-prefactorization algebra is given by two pointed chain complexes $\mathcal{F}(\{x\})=V_x$, $\mathcal{F}(\{y\})=V_y$. It is locally constant and 
 the associated factorization algebra is given by $\mathcal{F}^{\circ}(\{x, y\})\cong V_{x}\otimes  V_{y}$. However, if $W$ is any chain complex with pointed maps $V_{x}\to W$, $V_{y}\to W$ then we have a prefactorization algebras $\mathcal{G}$ defined by $\mathcal{G}(X)=W$ and is otherwise the same as $\mathcal{F}^{\circ}$.  In particular, $\mathcal{G}$ is a prefactorization algebra on $X$ with the same values as $\mathcal{F}^{\circ}$ on the disks of $X$ but is different from $\mathcal{F}^{\circ}$. 
\end{example}
\end{remark}

The notion of being locally constant for a \emph{factorization algebra} is indeed a local property (though its definition is about all disks) as proved by the following result.
\begin{proposition}\label{P:locallocallyconstant}
Let $M$ be a topological manifold and $\mathcal{F}$ be a factorization algebra on $M$. Assume that there is an open  cover $\mathcal{U}$ of $M$ such that for any $U\in \mathcal{U}$  the restriction $\mathcal{F}_{|U}$ is locally constant.
Then $\mathcal{F}$ is locally constant on $M$.
\end{proposition}
See \S~\ref{S:SomeProofsGen} for a proof.

\begin{remark}[\textbf{Ran space}] Factorization algebras on $X$ can be seen as a certain kind of cosheaf on the Ran space of $X$. 
This definition is actually the correct one to deal with factorization algebras in the algebraic geometry context (see~\cite{BD, FG}).   
The  \emph{Ran space $\text{Ran}(X)$ of a manifold $X$} is the space of finite non-empty subsets of $X$. 
Its topology\footnote{this topology is closely related (but different) to  the final topology on $\text{Ran}(X)$ making the canonical applications $X^{n}\to \text{Ran}(X)$ ($n>0$) continuous} is the coarsest topology on $\text{Ran}(X)$ for which  the sets $\text{Ran}(\{U_i\}_{i\in I})$ are open for every non-empty finite collection of pairwise disjoint opens subsets  $U_i$ ($i\in I$) of $X$. 
Here, the set  $\text{Ran}(\{U_i\}_{i\in I})$  is the collection of all finite subsets $\{x_j\}_{j \in J }\subset X$ such that $\{x_j\}_{j \in J }\cap U_i$ is non empty for every $i\in I$.  

Two subsets $U$, $V$ of $\text{Ran}(X)$ are said to be independent if  the two subsets $\big(\bigcup_{S\in U} S \big)\subset X$ and $\big(\bigcup_{T\in V} T \big)\subset X$ are disjoint (as subsets of $X$). If $U$, $V$ are subsets of $\text{Ran}(X)$, one denotes 
$U\star V$ the subset $\{ S\cup T, \, S\in U, \, T\in V\}$ of $\text{Ran}(X)$.

It is proved in~\cite{L-HA} that a factorization algebra on $X$ is the same thing as a constructible cosheaf $F$  on $\text{Ran}(X)$ which satisfies in addition the \emph{factorizing condition}, that is,  that satisfies that, for every family of pairwise independent open subsets,  the canonical map
$$F(U_1)\otimes \cdots \otimes F(U_n) \longrightarrow F(U_1\star \cdots \star U_n) $$ is an equivalence (the condition is similar to Remark~\ref{R:factpointed}). 

This characterization explain why there is a similarity between cosheaves and factorization algebras. However, the factorization condition is \emph{not} a purely cosheaf condition and is not compatible with every operations on cosheaves. 
\end{remark}

\section{Operations for factorization algebras}
\label{S:OperationsforFact}
In this section we review many properties and operations available for factorization algebras.
\subsection{Pushforward} \label{SS:pushforward} If $\mathcal{F}$ is a prefactorization algebra on $X$, and $f:X\to Y$ is a continuous map, one can define the \emph{pushforward $f_*(\mathcal{F})$} by the formula $f_*(\mathcal{F})(V) =\mathcal{F}(f^{-1}(V))$. 
If $\mathcal{F}$ is an (homotopy) factorization algebra then so is $f_*(\mathcal{F})$,  see~\cite{CG}.  
\begin{proposition}\label{P:pushforwardFac}The pushforward is a symmetric monoidal functor  $f_*: \textbf{Fac}_X\to \textbf{Fac}_Y$ and further $(f\circ g)_*=f_*\circ g_*$.
\end{proposition}
\smallskip

 Let us abusively \emph{denote $\mathcal{G}$ for the global section  $\mathcal{G}(pt)$ of a factorization algebra over the point $pt$}.
Let $p:X\to pt$ be the canonical map. By Theorem~\ref{T:Theorem6GTZ2}, when $X$ is a manifold and $\mathcal{F}_A$ is a locally constant factorization algebra associated to a $\Disk_n^{(X,TX)}$-algebra, then the factorization homology of $M$ is 
\begin{equation}\label{eq:facthom=Factalg}\int_X A \,\cong\, \mathcal{F}_A(A)\,\cong \, p_*(\mathcal{F}_A) .\end{equation}
This analogy legitimates to call $p_*(\mathcal{F})$, that is the (derived) global sections of $\mathcal{F}$, its factorization homology:  
\begin{definition}\label{D:FacthomologyofFactalgebra}
The \emph{factorization homology\footnote{recall that we are considering homotopy factorization algebras, which are already derived objects. If $\mathcal{G}$ is a genuine factorization algebra (Remark~\ref{R:NaiveFacAlg}), then its factorization homology would be $\mathbb{L}p_*(\mathcal{G}):=p_*(\tilde{G})$ where $\tilde{G}$ is an acyclic resolution of $\mathcal{G}$ (as genuine factorization algebra)} of a factorization algebra} $\mathcal{F}\in \textbf{Fac}_X$ is $p_*(\mathcal{F})$ and is also denoted $\int_X\mathcal{F}$.
\end{definition}

\begin{proposition}\label{P:pushforwardlc} Let $f:X\to Y$ be a locally trivial fibration between smooth manifolds. If $\mathcal{F}\in \textbf{Fac}_X$ is locally constant, then $f_*(\mathcal{F}) \in \textbf{Fac}_Y$ is locally constant. 
\end{proposition}
\begin{proof}
Let $U\hookrightarrow V$ be an inclusion of an open sub-disk $U$ inside an open disk $V\subset Y$. Since $V$ is contractible, it can be trivialized so  we can assume $f^{-1}(V)= V\times F$ with $F$ a smooth manifold. Taking a stable by finite intersection and factorizing cover $\mathcal{V}$ of $F$ by open disks,  we have a factorizing cover $\{V\}\times \mathcal{V}$ of $f^{-1}(V)$ consisting of open disks in $X$. Similarly $\{U\}\times \mathcal{V}$ is a factorizing cover of $f^{-1}(U)$ consisting of open disks. In particular for any $D\in \mathcal{V}$, the structure map $\mathcal{F}(U\times D) \to \mathcal{F}(V\times D)$ is a quasi-isomorphism since $\mathcal{F}$ is locally constant.
Thus the induced map $\check{C}(\{U\}\times \mathcal{V},\mathcal{F}) \to \check{C}(\{V\}\times \mathcal{V},\mathcal{F})$ is a quasi-isomorphism as well which implies that $f_*(\mathcal{F})(U)\to f_*(\mathcal{F})(V)$ is a quasi-isomorphism. 
\qed \end{proof}

\begin{example}{Locally constant factorization algebras induced on a submanifold}\label{Ex:bundleFac} 
Let  $i:X\hookrightarrow \R^n$ be  an embedding of a manifold $X$ into $\R^n$ and   $NX$ be an open tubular neighborhood of $X$ in $\R^n$. We write $q:NX\to X$ for the bundle map. If $A$ is an $E_n$-algebra, then it defines a factorisation algebra $\mathcal{F}_A$ on $\R^n$. Then the pushforward   $q_*(\mathcal{F_{A}}_{|NX})$ is a locally constant (by Proposition~\ref{P:pushforwardlc}) factorization algebra on $X$, which is \emph{not} constant in general if the normal bundle $NX$ is not trivialized.
\end{example}

Since a continuous map $f: X\to Y$ yields a factorization $X\stackrel{f}\to Y \to pt$ of $X\to pt$,  Proposition~\ref{P:pushforwardFac} and the equivalence~\eqref{eq:facthom=Factalg} imply the following pushforward formula of factorization homology.
\begin{proposition}[pushforward formula] Let $X\stackrel{f}\to Y$ be  continuous and $\mathcal{F}$ be in $\textbf{Fac}_X$. The factorization homology of $\mathcal{F}$ over $X$ is the same as the factorization homology of $f_*(\mathcal{F})$ over $Y$:
\begin{equation}\label{eq:pushforwardfacthomology}
\int_{X} \mathcal{F} \,\cong \,p_*(\mathcal{F}) \,\cong\,  p_* \big(f_*(\mathcal{F})\big) \, \cong \, \int_{Y} f_*(\mathcal{F}).
\end{equation}
\end{proposition}


\subsection{Extension from a basis}\label{SS:extensionfrombasis} Let $\mathcal{U}$ be a \emph{basis} stable by finite intersections  for the topology of a space $X$ and which is also a \emph{factorizing cover}.  Let $\mathcal{F}$ be a (homotopy) \emph{$\mathcal{U}$-factorization algebra}, that is a $\mathcal{U}$-prefactorization algebra (Definition~\ref{D:Prefact}) such that, for any  $U\in \mathcal{U}$ and factorizing cover  $\mathcal{V}$ of $U$   consisting of open sets in $\mathcal{U}$, the canonical map $\check{C}(\mathcal{V},\mathcal{F}) \to \mathcal{F}(U) $ is a quasi-isomorphism.

\begin{proposition}[Costello-Gwilliam~\cite{CG}]\label{P:extensionfrombasis}
There is an unique\footnote{up to contractoble choices} (homotopy) factorization algebra $i_*^{\mathcal{U}}(\mathcal{F})$  on $X$ extending $\mathcal{F}$ (that is equipped with a quasi-isomorphism of $\mathcal{U}$-factorization algebras $i_*^{\mathcal{U}}(\mathcal{F})\to \mathcal{F}$). 

Precisely, for any open set $V\subset X$, one has  
$$i_*^{\mathcal{U}}(\mathcal{F})(V) \,:=\, \check{C}(\mathcal{U}_V, \mathcal{F})$$ 
where $\mathcal{U}_V$ is the open cover of $V$ consisting of all open subsets of $V$ which are in $\mathcal{U}$. 
\end{proposition}
Note that the uniqueness is immediate since, if $\mathcal{G}$ is a factorization algebra on $X$, then for any open $V$ the canonical map $\check{C}(\mathcal{U}_V, \mathcal{G})\to \mathcal{G}(V)$ is a quasi-isomorphism and, further,  the \v{C}ech complex   $\check{C}(\mathcal{U}_V, \mathcal{G})$  is computed using only open subset in $\mathcal{U}$.

\smallskip

 Proposition~\ref{P:extensionfrombasis} gives a way to construct (locally constant) factorization algebras as we now demonstrate.
\begin{example}\label{ex:E1frombasis}
By Example~\ref{ex:framedisEn}, we know that an associative unital algebra (possibly differential graded) gives a locally constant factorization algebra on the interval $\R$. It can be explicitly given by using extension along a basis. Indeed, the collection $\mathcal{I}$  of intervals $(a,b)$ ($a<b$)  is a factorizing basis of opens, which is stable by finite intersections. Then one can set a $\mathcal{I}$-prefactorization algebra $\mathcal{F}_A$ by setting $\mathcal{F}_A((a,b)):=A$.
 For pairwise disjoints open interval $I_1,\dots, I_n\subset I$, where the indices are chosen so that $\sup(I_i) \leq \inf(I_{i+1})$, the structure maps  are given by
\begin{eqnarray}\label{eq:structureinterval}
A^{\otimes n}=\mathcal F_A(I_1)\otimes\cdots\otimes\mathcal F_A(I_n)& \longrightarrow & \mathcal F_A(I)=A \\  
a_1\otimes\cdots\otimes a_n & \longmapsto & a_1\cdots a_n\,.
\end{eqnarray} To extend this construction to a full homotopy factorization algebra on $\R$, one needs to check that $\mathcal{F}_A$ is a $\mathcal{I}$-factorization algebra which is the content of Proposition~\ref{P:BarInterval} below. 
 
In the construction, we have chosen an implicit orientation of $\R$; namely, in the structure map~\eqref{eq:structureinterval}, we have  decided to multiply the elements $(a_i)$ by choosing to order the intervals in increasing order from left to right.

Let $\tau: \R \to \R$ be the antipodal map $x\mapsto -x$ reversing the orientation.
 One can check  that  $\tau_*(\mathcal{F}_A)= \mathcal{F}_{A^{op}}$ where $A^{op}$ is the algebra $A$ with opposite multiplication $(a,b)\mapsto b\cdot a$.
 In other words, \emph{choosing  the opposite orientation} (that is the decreasing one) \emph{of $\R$} amounts to \emph{replacing the algebra by its opposite algebra}.
\end{example}
\begin{example}[back to the circle] \label{Ex:S1frombasis} The circle $S^1$ also has a (factorizing) basis given by the open (embedded) intervals (of length less than half of the perimeter of the circle in order to be stable by intersection). Choosing an orientation on the circle, one can define a (homotopy) factorization algebra $\mathcal{S}_A$ on the circle using again the structure maps~\eqref{eq:structureinterval}. 
This gives an explicit construction of the factorization algebra associated to a framing of $S^1$ from Example~\ref{ex:framedisEn} (since they agree on a stable by finite intersection basis of open sets). 
The global section of  $\mathcal{S}_A$ are thus the Hochschild chains of $A$ by the computation~\eqref{eq:HHaschiral}. 

Similarly to Example~\ref{ex:E1frombasis}, choosing the opposite framing on the circle amounts to considering the factorization algebra $\mathcal{S}_{A^{op}}$.
However, \emph{unlike on $\R$}, there is an equivalence  $\mathcal{S}_A\cong \mathcal{S}_{A^{op}}$ induced by the fact that there is an orientation preserving diffeomorphism between the to possible orientations of the circle and that further, the value of $\mathcal{S}_A$ on any interval is constant. 
\end{example}

\subsection{Exponential law: factorization algebras on a product}

Let $\pi_1: X\times Y \to X$ be the canonical projection. By Proposition~\ref{P:pushforwardlc}, we have the \emph{pushforward} functor ${\pi_1}_*:\textbf{Fac}^{lc}_{X\times Y}\to \textbf{Fac}^{lc}_{X}$.
 This functor has an natural lift into $\textbf{Fac}^{lc}(Y)$. Indeed, if $U$ is  open in $X$ and   $V$ is  open  in $Y$, we have a chain complex: $$\underline{{\pi}_1}_*(\mathcal{F})(U,V):= \mathcal{F}({\pi_1}_{|X\times V}^{-1}(U))=\mathcal{F}(U\times V).$$ Let $V_1, \dots, V_i$ be pairwise disjoint open subsets in an open set $W\subset Y$ and consider
\begin{multline}\label{eq:structureexpFact} \underline{\pi_1}_*(\mathcal{F})(U,V_1)\otimes \cdots \otimes \underline{\pi_1}_*(\mathcal{F})(U,V_i)  \cong \mathcal{F}(U\times V_1) \otimes \cdots \otimes \mathcal{F}(U\times V_i) \\ \stackrel{\rho_{U\times V_1,\dots, U\times V_i, U\times W}}\longrightarrow 
\mathcal{F}(U\times W)= \underline{\pi_1}_*(\mathcal{F})(U,W).
\end{multline}
The map~\eqref{eq:structureexpFact} makes $\underline{\pi_1}_*(\mathcal{F})(U)$ a prefactorization algebra on $Y$. If $U_1,\dots, U_j$ are pairwise disjoint open inside an open  $O\subset X$,  the collection of structure maps 
\begin{multline}\label{eq:structureexpFact2} \underline{\pi_1}_*(\mathcal{F})(U_1,V)\otimes \cdots \otimes \underline{\pi_1}_*(\mathcal{F})(U_j,V)  \cong \mathcal{F}(U_1\times V) \otimes \cdots \otimes \mathcal{F}(U_j\times V) \\ \stackrel{\rho_{U_1\times V,\dots, U_j\times V, O\times V}}\longrightarrow 
\mathcal{F}(O\times V)= \underline{\pi_1}_*(\mathcal{F})(O,V)
\end{multline}
indexed by opens $V\subset Y$ is  a map $ \underline{\pi_1}_*(\mathcal{F})(U_1)\otimes \cdots  \otimes\underline{\pi_1}_*(\mathcal{F})(U_j) \longrightarrow \underline{\pi_1}_*(\mathcal{F})(O)$ of  prefactorization algebras over $Y$.

Combining the two constructions, we find that
the structure maps~\eqref{eq:structureexpFact} and~\eqref{eq:structureexpFact2} make ${\pi_1}_*(\mathcal{F})$ a prefactorization algebra over $X$ with values in the category of prefactorization algebras over $Y$. In other words we have just defined a functor:
\begin{equation}\label{eq:defpi1}
  \underline{\pi_1}_*: \textbf{PFac}_{X\times Y} \longrightarrow \textbf{PFac}_{X}(\textbf{PFac}_{Y})
\end{equation}
 fitting into a commutative diagram 
 $$\xymatrix{\textbf{PFac}_{X\times Y} \ar[r]^{\underline{\pi_1}_*} \ar[rd]_{(\pi_1)_*} & \textbf{PFac}_{X}(\textbf{PFac}_{Y}) \ar[d]^{\textbf{PFac}_X(p_*)} \\ &  \textbf{PFac}_{X}}$$ where $p_*$ is given by Definition~\ref{D:FacthomologyofFactalgebra}.
\begin{proposition}\label{L:expFact} Let $\pi_1: X\times Y \to X$ be the canonical projection.
The pushforward~\eqref{eq:defpi1} by $\pi_1$ induces  a functor $$\underline{\pi_1}_*: {\textbf{Fac}}_{X\times Y} \longrightarrow {\textbf{Fac}}_X(\textbf{Fac}_Y)$$ and, if $X$, $Y$ are smooth manifolds, an equivalence $\underline{\pi_1}_*: {\textbf{Fac}}^{lc}_{X\times Y} \stackrel{\simeq}\longrightarrow {\textbf{Fac}}^{lc}_X(\textbf{Fac}_Y^{lc})$  of $\infty$-categories .
\end{proposition}
See \S~\ref{Proof:expFact} for a proof.

\smallskip

The above Proposition is a slight generalization of (and relies on) the  following $\infty$-category version of the beautiful  Dunn's Theorem~\cite{Du}  proved under the following form by Lurie~\cite{L-HA} (see~\cite{GTZ2} for the pushforward interpretation):
\begin{theorem}[Dunn's Theorem]\label{T:Dunn} There is an equivalence  of $\infty$-categories $$E_{m+n}\textbf{-Alg} \stackrel{\simeq}\longrightarrow E_m\textbf{-Alg}(E_n\textbf{-Alg}).$$ 
 Under the equivalence $E_n\textbf{-Alg} \cong \textbf{Fac}_{\R^n}^{lc}$ (Theorem~\ref{P:En=Fact}), the above equivalence is realized by 
the pushforward $\underline{\pi}_*:\textbf{Fac}_{\R^m\times \R^n}^{lc} \to \textbf{Fac}_{\R^m}^{lc}(\textbf{Fac}^{lc}_{\R^n})$ associated to the canonical projection $\pi:\R^m\times \R^n\to \R^m$.
\end{theorem}
\begin{example}[PTVV construction] There is a derived geometry variant of the AKSZ formalism introduced recently in~\cite{PTTV} which leads to factorization algebras similarly to Example~\ref{ex:ObservablesasFact}.  
We briefly sketch it. The main input is a (derived Artin) stack $X$ (over a characteristic zero field) which is assumed to be compact and equipped with an orientation (write $d_X$ for the dimension of $X$) and a stack $Y$ with an $n$-shifted symplectic structure $\omega$ (for instance take $Y$ to be the shifted cotangent complex $Y=T^*[n]Z$ of a scheme $Z$).
   The natural evaluation map $ev: X\times \mathbb{R}\mathop{Map}(X,Y)$ allows to pullback the  symplectic structure on the space of fields $\mathbb{R}\mathop{Map}(X,Y)$. 
Precisely, $\mathbb{R}\mathop{Map}(X,Y)$ carries an natural $(n-d_X)$-shifted symplectic structure roughly given by the integration
$\int_{[X]} ev^*(\omega)$ of the pullback of $\omega$ on the fundamental class of $X$.
 It is expected that the observables $\mathcal{O}_{ \mathbb{R}\mathop{Map}(X,Y)}$ carries a structure of $P_{1+n-d_X}$-algebra. 
 Assume further that $X$ is a Betti stack, that is in the essential image of $\hTop \to \textbf{dSt}_k$.
It will then follow from Corollary~\ref{C:mappingstack} and Example~\ref{E:EinftygivesNDisk} that
 $\mathcal{O}_{ \mathbb{R}\mathop{Map}(X,Y)}$ belongs to $\textbf{Fac}^{lc}_{X}(P_{1+n-d_X}\textbf{-Alg})$. Using the formality of the little disks operads in dimension $\geq 2$, 
  Proposition~\ref{L:expFact} then will give
$\mathcal{O}_{ \mathbb{R}\mathop{Map}(X,Y)}$ the structure of a locally constant factorization algebra on $X\times \R^{1+n-d_X}$ when $n>d_X$. 
It is also expected that the quantization of such shifted symplectic stacks shall carries canonical locally constant factorization algebras structures.

Note that the Pantev-To\"en-Vaqui\'e-Vezzosi construction was recently extended by Calaque~\cite{CaMappingStack} to add boundary conditions. The global observables of such \emph{relative} mapping stacks shall be naturally endowed with the structure of locally constant factorization algebra on stratified spaces (as defined in \S~\ref{S:stratifiedFact}).
\end{example}

Proposition~\ref{L:expFact} and Theorem~\ref{T:Theorem6GTZ2} have the following consequence
\begin{corollary}[Fubini formula \cite{GTZ2}]
\label{C:FubiniTCH} Let $M$, $N$ be manifolds of respective dimension $m$, $n$ and let  $A$ be a $\Disk_{n+m}^{(M\times N, T(M\times N))}$-algebra. Then, $\int_{N} A$ has a canonical structure of $\Disk_{m}^{(M,TM)}$-algebra and further, 
$$\int_{M\times N} A \; \cong \; \int_M\Big(\int_N A\Big). $$
\end{corollary}

\begin{example}
Let $A$ be a smooth  commutative algebra. By Hochschild-Kostant-Rosenberg theorem (see Example~\ref{Ex:HKR}),  $CH_{S^1}(A) \stackrel{\simeq} \to  S_A^\bullet(\Omega^1(A)[1])$; this  algebra is also smooth. 
Since $A$ is commutative it defines a factorization algebra on the torus $S^1\times S^1$. By Corollary~\ref{C:FubiniTCH}, we find that
$$\int_{S^1\times S^1} A \cong \int_{S^1} \big(S_A^\bullet(\Omega^1(A)[1])\big) \,\cong\, S_A ^\bullet\Big(\Omega^1(A)[1]\oplus \Omega^1(A)[1] \oplus \Omega^1(A)[2]\Big).$$
\end{example}

\subsection{Pullback along open immersions and equivariant factorization algebras} 
Let $f: X\to Y$  be an open immersion and let $\mathcal{G}$ be a factorization algebra on $Y$. Since $f:X\to Y$ is an open immersion, the set $$\mathcal{U}_f:=\{U \text{ open in $X$ such that }f_{|U}: U\to Y \text{ is an homeomorphism}\}$$ is an open cover of $X$ as well as a factorizing basis. 
For $U\in \mathcal{U}_{f}$, we define $f^*(\mathcal{G})(U):=\mathcal{G}(f(U))$. The structure maps of $\mathcal{G}$ make $f^*(\mathcal{G})$ a $\mathcal{U}_f$-factorization algebra in a canonical way. Thus by Proposition~\ref{P:extensionfrombasis}, 
$i_*^{\mathcal{U}_f}(f^*(\mathcal{G}))$ is the factorization algebra on $X$ extending $f^*(\mathcal{G})$ .
We (abusively) denote  $f^*(\mathcal{G}):=i_*^{\mathcal{U}_f}(f^*(\mathcal{G}))$ and call it the pullback along $f$ of the factorization algebra $\mathcal{G}$. 
\begin{proposition}[\cite{CG}]The pullback along open immersion is a functor $f^*: \textbf{Fac}_Y\to \textbf{Fac}_X$. If $f: X\to Y$ and $g: Y\to Z$ are open immersions, then $(g\circ f)^*= f^* \circ g^*$.
\end{proposition}
\smallskip

If $U\in \mathcal{U}_f$, then $U$ is an open subset of the open set $f^{-1}(f(U))$.
 Thus if $\mathcal{F}$ is a factorization algebra on $X$, we have the natural map 
 \begin{equation}\label{eq:NatTransf}\mathcal{F}(U)\stackrel{\rho_{U, f^{-1}(f(U))}}\longrightarrow \mathcal{F}(f^{-1}(f(U)))\cong f^{*}(f_{*}(\mathcal{F}))(U)\end{equation}                                                                                                                                                                which is a map of $\mathcal{U}_f$-factorization algebras. Since $\mathcal{F}$ and $f^{*}(f_{*}(\mathcal{F}))$ are factorization algebras on $X$, the above map extends uniquely into a map of factorization algebras on $X$. We have proved:
\begin{proposition}
Let $f: X\to Y$  be an open immersion. There is an natural transformation $\text{Id}_{\textbf{Fac}_X} \to  f^{*}f_{*}$ induced by the maps~\eqref{eq:NatTransf}.
\end{proposition}
\begin{example}
Let $X=\{c,d\}$ be a discrete space with two elements and consider the projection $f:X\to pt$.  A factorization algebra $\mathcal{G}$ on $pt$ is just the data of a chain complex $G$ with a distinguished cycle $g_0$ while  a factorization algebra $\mathcal{F}$ on $X$ is given by two chain complexes $C$, $D$  (with distinguished cycles $c_0$, $d_0$) and  the rule $\mathcal{F}(\{c\})=C$, $\mathcal{F}(\{d\})=D$, $\mathcal{F}(\{c\})=C\stackrel{id\otimes \{d_0\}}\longrightarrow C\otimes D=\mathcal{F}(X)$ and $\mathcal{F}(\{d\})=C\stackrel{\{c_0\}\otimes id}\longrightarrow C\otimes D=\mathcal{F}(X)$. 
In that case we have that $$f^*(f_*(\mathcal{F})(\{x\})=C\otimes D =f^*(f_*(\mathcal{F})(\{y\})=f^*(f_*(\mathcal{F})(X) $$  while $f_*(f^*(\mathcal{G}))(pt)=F\otimes F$. Note that there are no natural transformation of chain complexes $F\otimes F\to F$  in general; in particular \emph{$f^*$ and $f_*$ are not adjoint}.

 In fact $f_*$ does \emph{not} have any adjoint in general; indeed as the above example of $F:\{c,d\}\to pt$ demonstrates, $f_*$ does not commute with coproducts nor products.
\end{example}

We now turn on to a descent property of factorization algebras.
Let $G$ be a discrete group acting on a space $X$. For $g\in G$, we write $g: X\to X$ the homeomorphism $x\mapsto g\cdot x$ induced by the action.
\begin{definition}\label{D:GequivariantFacAlg} A $G$-equivariant factorization algebra on $X$ is a factorization algebra $\mathcal{G}\in \textbf{Fac}_X$ together with, for all $g\in G$, (quasi-)isomorphisms of factorization algebras 
$$\theta_g: g^{*}(\mathcal{G}) \stackrel{\simeq} \to\mathcal{G}$$   such that $\theta_{1}=\text{id}$ and $$ \theta_{gh}=\theta_{h}\circ h^*(\theta_g): h^*(g^*(\mathcal{G}))\to \mathcal{G}.$$
We write $\textbf{Fac}_X^{G}$ for the category of $G$-equivariant factorization algebras over $X$.
\end{definition}

Assume $G$ acts properly discontinuously and $X$ is Hausdorff so that the quotient map $q: X\to X/G$ is an open immersion. If $\mathcal{F}$ is a factorization algebra over $X/G$, then $q^{*}(\mathcal{F})$ is $G$-equivariant (since $q(g\cdot x)=q(x)$). We thus have a functor $q^*:\textbf{Fac}_{X/G} \longrightarrow \textbf{Fac}_X^{G}$.
\begin{proposition}[Costello-Gwilliam~\cite{CG}] \label{P:descentGequiv} If the discrete group $G$ acts properly discontinuously on $X$, then the functor $q^*:\textbf{Fac}_{X/G} \longrightarrow \textbf{Fac}_X^{G}$ is an equivalence of categories.
\end{proposition}
The proof essentially relies on considering the factorization basis given by trivialization of the principal $G$-bundle $X\to X/G$ to define an inverse to $q^*$.
\begin{proposition}\label{P:GequivFactlc}
 If the discrete group $G$ acts properly discontinuously on a smooth manifold $X$, then the equivalence $q^*:\textbf{Fac}_{X/G} \longrightarrow \textbf{Fac}_X^{G}$ factors as an equivalence $q^*:\textbf{Fac}_{X/G}^{lc} \longrightarrow (\textbf{Fac}_X^{lc})^{G}$ between the subcategories of locally constant factorization algebras.
\end{proposition}
\begin{proof}
Let $U$ be an open set such that $q_{|U}: U\to X/G$ is an homeomorphism onto its image. Then, for every open subset $V$ of $U$, $q^*(\mathcal{F})(V)=\mathcal{F}(q(V))$. Thus, if $\mathcal{F}$ satisfies the condition of being locally constant for disks included in $U$, then so does $q^*(\mathcal{F})$ for disks included in $q(U)$. Hence, by Proposition~\ref{P:locallocallyconstant}, $q^*(\mathcal{F})$ is locally constant if $\mathcal{F}$ is locally constant. 

Now, assume $\mathcal{G}\in \textbf{Fac}_{X}^{G}$ is locally constant. Then $(q^*)^{-1}(\mathcal{G})$ is the factorization algebra defined on every section $X/G\supset U\to \sigma(U)\subset X$ of $q$ (with $U$ open) by $(q^*)^{-1}(\mathcal{G})(U)=\mathcal{G}(\sigma(U))$. Since every disk $D$  is contractible, we always have a section $D\to X$ of $q_{|D}$. Thus, if $\mathcal{G}$ is   locally constant, then so is $(q^*)^{-1}(\mathcal{G})$.
\qed \end{proof}
\begin{remark}[General definition of equivariant factorization algebras]
Definition~\ref{D:GequivariantFacAlg} can be easily generalized to topological groups as follows. Indeed, if $G$ acts continuously on $X$, then the rule $(g,\mathcal{F})\mapsto g^*(\mathcal{F})$ induces a right action of $G$ on $\textbf{Fac}_X$\footnote{that is a map of $E_1$-algebras (in $\hTop$) from $G^{op}$ to $\textbf{Fun}(\textbf{Fac}_X, \textbf{Fac}_X)$}. 

The $\infty$-category of
  $G$-equivariant factorization algebras is the $\infty$-category of homotopy $G$-fixed points of $\textbf{Fac}_X$: $$\textbf{Fac}_X^G:=(\textbf{Fac}_X)^{hG}.$$
This $\infty$-category is equivalent to the one of Definition~\ref{D:GequivariantFacAlg} for discrete groups. It is the $\infty$-category consisting  of a factorization algebra $\mathcal{G}$ on $X$ together with
quasi-isomorphisms of factorization algebras $\theta_g: g^*(\mathcal{G})\to \mathcal{G}$ (inducing a $\infty$-functor $BG\to \textbf{Fac}_X$, where $BG$ is the $\infty$-category associated to the topological category   with a single object and mapping space of morphisms given by $G$) and equivalences $\theta_{gh} \sim \theta_h \circ h^*(\theta_g)$   satisfying some higher coherences.
\end{remark}

\subsection{Example: locally constant factorization algebras over the circle}
Let $q:\R \to S^1=\R/\mathbb{Z}$ be the universal cover of $S^1$ and let  $\mathcal{F}$ be a locally constant factorization algebra on $S^1$. By Proposition~\ref{P:GequivFactlc},  $\mathcal{F}$ is equivalent to the data of a $\mathbb{Z}$-equivariant  locally constant factorization algebra  on $\R$ which is the same as a locally constant factorization algebra over $\R$ together with an 
equivalence of factorization algebras,  the equivalence being given by $ \theta_1: 1^*(q^{*}(\mathcal{F})) \stackrel{\simeq}\to q^*(\mathcal{F})$. 
By Theorem~\ref{P:En=Fact}, the category of locally constant factorization algebras on $\R$ is the same as the category of $E_1$-algebras, and thus is equivalent to its full subcategory of \emph{constant} factorization algebras. 
It follows that we have a canonical equivalence $1^*(\mathcal{G})\cong \mathcal{G}$ for $\mathcal{G}\in \textbf{Fac}^{lc}_{\R}$ (in particular, for any open interval $I$, the structure map $1*(\mathcal{G}(I))=\mathcal{G}(1+I)\to \mathcal{G}(\R)$ is a quasi-isomorphism).
\begin{definition} We denote $\mathop{mon}: q^{*}(\mathcal{F})\cong 1^*(q^{*}(\mathcal{F}))\stackrel{\theta_1}{\to} q^{*}(\mathcal{F})$ the self-equivalence of $q^*(\mathcal{F})$ induced by $\theta_1$ and call it the \emph{monodromy} of $\mathcal{F}$. 
\end{definition}
 We thus get the following result
\begin{corollary}\label{C:locallyconstantS1}
The category $\textbf{Fac}_{S^1}^{lc}$ of locally constant factorization algebra on the circle is equivalent to
the $\infty$-category $\textbf{Aut}(E_1\textbf{-Alg})$ of  $E_1$-algebras equipped with a self-equivalence.
\end{corollary}
\begin{remark} Using Proposition~\ref{L:expFact}, it is easy to prove similarly that $\textbf{Fac}^{lc}_{S^1\times S^1}$ is equivalent to the category of $E_2$-algebras equipped with two commuting monodromies (i.e. self-equivalences).
 
It seems harder to describe the categories of locally constant factorization algebras over the spheres $S^3$, $S^7$ in terms of $E_3$ and $E_7$-algebras (due to the complicated homotopy groups of the spheres). However, for $n=3,7$, there shall be an embedding of the categories of $E_n$-algebras equipped with an $n$-gerbe\footnote{by an $n$-gerbe over $A\in E_n\textbf{-Alg}$,  we mean a monoid map $\mathbb{Z}\to \Omega^{n-1}\Map_{E_n\textbf{-Alg}}(A,A)$.} into $\textbf{Fac}_{S^n}^{lc}$.
\end{remark}

Let $\mathcal{F}$ be a locally constant factorization algebra on $S^1$  (identified with the unit sphere in $\mathbb{C}$). We wish to compute the \emph{global section of} $\mathcal{F}$ (i.e. its factorization homology $\int_{S^1} \mathcal{F}$).
Let $\mathcal{B}\cong \mathcal{F}(S^1\setminus\{1\})$ be its underlying  $E_1$-algebra (with monodromy  $\mathop{mon}:\mathcal{B}\stackrel{\simeq}\to \mathcal{B}$). 
We use the orthogonal  projection $\pi:S^1\to [-1,1]$ from $S^1$ to the real axis. 
The equivalence~\eqref{eq:pushforwardfacthomology} yields 
 $\mathcal{F}(S^1)\cong \pi_*(\mathcal{F})([-1,1])$ and by  Proposition~\ref{P:pushforwardlc} and Proposition~\ref{P:BarInterval}, we are left to compute the $E_1$-algebra $ \pi_*(\mathcal{F}) \big((-1,1)\big)$ and left and right modules 
$ \pi_*(\mathcal{F} \big((-1,1]\big)$,  $ \pi_*(\mathcal{F}) \big([-1,1)\big)$.  From Example~\ref{ex:E1frombasis}, we get  $ \pi_*(\mathcal{F}) \big((-1,1)\big)\cong \mathcal{B}\otimes \mathcal{B}^{op}$.
 Further, $$ \pi_*(\mathcal{F}) \big([-1,1)\big)\cong \mathcal{F}(S^1\setminus\{1\})=\mathcal{B} \quad \mbox{ (as a $\mathcal{B}\otimes \mathcal{B}^{op}$-module).} $$ Similarly 
 $ \pi_*(\mathcal{F}) \big((-1,1]\big)\cong \mathcal{B}^{\mathop{mon}}$, that is   $\mathcal{B}$ viewed as a $\mathcal{B}\otimes \mathcal{B}^{op}$-module through the monodromy. 
 When $\mathcal{B}$ is actually a differential graded algebra, then the bimodule stucture of $\mathcal{B}^{\mathop{mon}}$ boils down to  $a\cdot x \cdot b = \mathop{mon}(b)\cdot m \cdot a$. 
This proves the following which is also asserted in~\cite[\S 5.3.3]{L-HA}. 
\begin{corollary}\label{C:Hochtwistedmonodromy} Let $\mathcal{B}$ be a locally constant factorization algebra on $S^1$. Let $B$ be  a differential graded algebra and $\mathop{mon}: B\stackrel{\simeq}\to B$ be a quasi-isomorphism of algebras so that $(B, \mathop{mon})$ is a model for the underlying $E_1$-algebra of $\mathcal{B}$ and its monodromy. 
Then the factorization homology $$\int_{S^1} \mathcal{B}\,\cong\, B\mathop{\otimes}\limits^{\mathbb{L}}_{B\otimes B^{op}} B^{\mathop{mon}}\, \cong \, HH(B, B^{\mathop{mon}})$$
is computed by the (standard) Hochschild homology\footnote{in particular by the standard Hochschild complex (see \cite{Lo-book}) $C_*(B,B^{\mathop{mon}})$ when $B$ is flat over $k$} $HH(B)$ of $B$ with value in $B$ twisted by the monodromy.
\end{corollary}

\begin{example}[the circle again] Let $p:\R\to S^1$ be the universal cover of $S^1$. By Example~\ref{ex:E1frombasis}, an unital associative algebra $A$ defines a locally constant factorization algebra, denoted  $\mathcal{A}$, on $\R$. By Proposition~\ref{P:pushforwardlc},  the pushforward $p_*(\mathcal{A})$ is a locally constant factorization algebra on $S^1$, which, on any interval $I\subset S^1$ is given by $p_*(\mathcal{A})(I) =\mathcal{A}(I\times \mathbb{Z}) = A^{\otimes \mathbb{Z}}$. It is however \emph{not a constant} factorization algebra since the global section of $p_*(\mathcal{A})$ is different from the Hochschild homology of $A$:
$$ p_*(\mathcal{A})(S^1)=\mathcal{A}(\R)\cong A  \not\simeq HH(A^{\otimes \mathbb{Z}})$$ 
(for instance if $A$ is commutative the Hochschild homology of $A^{\otimes \mathbb{Z}}$ is $A^{\otimes \mathbb{Z}}$ in degree $0$.) Indeed, the monodromy of $\mathcal{A}$ is given by the automorphism  $\sigma$ of $A$ which sends the element $a_i$ in the tensor index by an integer $i$ into the tensor factor indexed by $i+1$, that is $\sigma\big( \bigotimes_{i\in \mathbb{Z}} a_i\big) =  \bigotimes_{i\in \mathbb{Z}}a_{i-1}$.

 However, by Corollary~\ref{C:Hochtwistedmonodromy}, we have that, for any $E_1$-algebra $A$,   $$HH\big(A^{\otimes \mathbb{Z}}, \big(A^{\otimes \mathbb{Z}}\big)^{\mathop{mon}}\big)\cong A.$$
\end{example}

\subsection{Descent}
There is a way to glue together factorization algebras provided they satisfy some descent conditions which we now explain.

Let $\mathcal{U}$ be an open cover of a space $X$ (which we assume to be equipped with a factorizing basis).  We also assume that all   intersections of infinitely many different opens in $\mathcal{U}$ are empty. 
 For every finite
subset $\{U_i\}_{i\in I}$ 
of $\mathcal{U}$, let $\mathcal{F}_{I}$ be a factorization algebra on $\bigcap_{i\in I} U_i$. 
For any $i\in I$, we have an inclusion 
$s_i: \bigcap_{i\in I} U_i \hookrightarrow \bigcap_{j\in I\setminus\{i\}} U_j$.
\begin{definition} A gluing data is a collection, for all finite  subset $\{U_i\}_{i\in I}\subset \mathcal{U}$ and $i\in I$,  of quasi-isomorphisms $r_{I,i}: \mathcal{F}_I \longrightarrow  \big(\mathcal{F}_{I\setminus\{i\}}\big)_{|\mathcal{U}_I}$ such that, for all $I$, $i,j\in I$, the following diagram commutes:
$$ \xymatrix{ \mathcal{F}_I \ar[d]_{r_{I,i}} \ar[rr]^{r_{I,j}}  &&   \big(\mathcal{F}_{I\setminus\{j\}} \big)_{|\mathcal{U}_I}
\ar[d]^{r_{I\setminus\{j\},i}} \\      \big(\mathcal{F}_{I\setminus\{i\}}\big)_{|\mathcal{U}_I}\ar[rr]^{r_{I\setminus\{i\},j}} &&  
 \big(\mathcal{F}_{I\setminus\{i,j\}}\big)_{|\mathcal{U}_I}.}$$
\end{definition}
Given a gluing data, one can define a factorizing basis $\mathcal{V}_{\mathcal{U}}$ given by the family of all opens which lies in 
some $U\in \mathcal{U}$. 
For any $V\in \mathcal{V}_{\mathcal{U}}$, set $\mathcal{F}(V) =\mathcal{F}_{I_V}(V)$ where $I_V$ is the largest subset of $I$ such that $V\in \bigcap_{j\in I_V} U_j$.  The maps $R_{I,i}$ induce a structure of $\mathcal{V}_{\mathcal{U}}$-prefactorization algebra.
\begin{proposition}[\cite{CG}] Given a gluing data, the  $\mathcal{V}_{\mathcal{U}}$-prefactorization  $\mathcal{F}$ extends uniquely 
into a factorization algebra $\mathcal{F}$ on $X$ whose restriction $\mathcal{F}_{|\mathcal{U}_I}$ on each $\mathcal{U}_{I}$ is canonically equivalent to $\mathcal{F}_I$.
\end{proposition}
Note that if the $\mathcal{F}_I$ are the restrictions to $\mathcal{U}_I$ of  a factorization algebra $\mathcal{F}$, then the collection of the  $\mathcal{F}_I$ satisfy the condition of a gluing data.

\section{Locally constant factorization algebras on stratified spaces and categories of modules} \label{S:stratifiedFact}
There is an interesting variant of locally constant factorization algebras over (topologically) stratified spaces which can be used to \emph{encode categories of $E_n$-algebras and their modules} for instance. Note that by Remark~\ref{R:factpointed}, all our categories of modules will be pointed, that is coming with a preferred element. We gave the definition and several examples in this Section. An analogue of Theorem~\ref{T:Theorem6GTZ2} for stratified spaces shall provide the link between the result in this section and results of~\cite{AFT}.

\subsection{Stratified locally constant factorization algebras}\label{SS:stratifiedFact}

In this paper, by a \emph{stratified space} of dimension $n$, we mean a Hausdorff paracompact topological space $X$,   which is filtered as the union of a sequence of closed subspaces $\emptyset=X_{-1}\subset X_0\subset X_1\subset \cdots \subset X_n=X$ such that any point $x\in X_i\setminus X_{i-1}$  has a neighborhood $U_x\stackrel{\phi}{\simeq} \R^i \times C(L)$ in $X$ where $C(L)$ is the (open) cone on a stratified space of dimension $n-i-1$ and the homeomorphism preserves the filtration\footnote{that is $\phi(U_x\cap X_{i+j+1})= \R^i\times C(L_j)$ for  $0\leq j\leq n-i-1$}. We further require that $X\setminus X_{n-1}$ is dense in $X$. 
In particular, a stratified space of dimension $0$ is simply a topological manifold of dimension $0$ and  $X_i\setminus X_{i-1}$ is a topological manifold of dimension $i$ (possibly empty or non-connected). 

The connected components of $X_i\setminus X_{i-1}$ are called the dimension $i$-strata of $X$. We always assume that $X$ has at most countable strata. 

\begin{definition} \label{D:lcFacStratified} An open subset $D$ of $X$ is called a (stratified) \emph{disk} if it is homeomorphic to $\R^i\times C(L)$ with $L$ stratified of dimension $n-i-1$, the homeomorphism preserves the filtration and further $D\cap X_i \neq \emptyset$ and $D\subset X\setminus X_{i-1}$. We call $i$ the \emph{index of} the (stratified disk) $D$. It is the \emph{smallest integer} $j$ such that $D\cap X_j\neq \emptyset$ 

We say that a (stratified) disk $D$ is a \emph{good neighborhood at} $X_i$  if $i$ is the index of $D$   and $D$ intersects only one connected component of $X_i\setminus X_{i-1}$. 

A  \emph{factorization algebra} $\mathcal{F}$ over a stratified space $X$  is  called \emph{locally constant} if for any inclusion of  (stratified) disks $U\hookrightarrow V$  such that both $U$ and $V$ are good neighborhoods at $X_i$ (for the same $i\in \{0,\dots,n\}$)\footnote{in other words are good neighborhoods of same index}, the structure map $\mathcal{F}(U) \to \mathcal{F}(V)$ is a quasi-isomorphism. 
\end{definition}
The underlying space of almost all examples of stratified spaces $X$ arising in these notes will be a manifold (with boundary or corners). In that cases, all (stratified) disk are   homeomorphic to  to a standard euclidean (half-)disk $\R^{n-j}\times [0,+\infty)^{j}$.

\smallskip

Let $X$ be a manifold (without boundary) and let $X^{str}$ be the same manifold  endowed with some stratification. A locally constant factorization algebra on $X$ is also locally constant with respect to the stratification. Thus, we have a fully faithful embedding 
\begin{equation}\label{eq:FactoFacstratified}
 \textbf{Fac}_{X}^{lc} \longrightarrow \textbf{Fac}_{X^{str}}^{lc}.
\end{equation}

Several general results on locally constant factorization algebras from \S~\ref{S:FactAlgebras} have analogues in the stratified case. We now list three useful ones.
\begin{proposition}\label{P:locallocallyconstantStrat}
Let $X$ be a stratified manifold and $\mathcal{F}$ be a factorization algebra on $X$ such  that there is an open  cover $\mathcal{U}$ of $X$ such that for any $U\in \mathcal{U}$  the restriction $\mathcal{F}_{|U}$ is locally constant.
Then $\mathcal{F}$ is locally constant on $X$.\end{proposition}

The functor~\eqref{eq:FactoFacstratified} generalizes to inclusion of any stratified subspace. 
\begin{proposition}
 Let $i: X\hookrightarrow Y$ be a stratified (that is filtration preserving) embedding of stratified spaces in such a way that $i(X)$ is a reunion of strata of $Y$. Then the pushforward along $i$  preserves locally constantness, that is lift as a functor
 $$ \textbf{Fac}^{lc}_X\to \textbf{Fac}^{lc}_Y.$$
\end{proposition}
\begin{proof}  Let $\mathcal{F}$ be in $\textbf{Fac}^{lc}_X$ and $U\subset D$ be  good disks of index $j$ at a neigborhood of a strata in $i(X)$.
 The preimage $i^{-1}(U) \cong i(X)\cap U$ is a good disk of index $j$ in $X$ and so is $i^{-1}(D)$. Hence $i_*(\mathcal{F}(U))\to i_*(\mathcal{F}(D))$ is a quasi-isomorphism. 
 On the other hand, if $V\subset Y\setminus i(X)$ is a good disk, then $i_*(\mathcal{F}(V))\cong k$. Since the constant factorization algebra with values $k$ (example~\ref{ex:trivialFact}) is locally constant on every stratified space, the result follows. 
\end{proof}

Let $f:X\to Y$ be a locally trivial fibration between stratified spaces. We say that $f$ is \emph{adequatly stratified} if  $Y$ has an open cover by trivializing (stratified) disks $V$ which are good neighborhoods satisfying that: \begin{itemize} \item  $f^{-1}(V) \stackrel{\psi}\simeq V\times F$ has a cover  by (stratified) disks of the form $\psi^{-1}(V\times D)$ which are good neighborhoods in $X$;
\item for  sub-disks $T\subset U$ which are good neighborhoods (in $V$) with the same index, then    $\psi^{-1}(T\times D)$ is a good neighborhood of $X$ of same index as  $\psi^{-1}(U\times D)$.                                                                                                                                                                          \end{itemize}
Obvious examples of adequatly stratified maps are given by locally trivial stratified fibrations; in particular by proper stratified submersions according to Thom first isotopy lemma~\cite{Thom-stratified, GMP-stratified}.   
\begin{proposition}\label{P:pushforwardlcStrat}  Let $f:X\to Y$ be adequatly stratified. If $\mathcal{F}\in \textbf{Fac}_X$ is locally constant, then $f_*(\mathcal{F}) \in \textbf{Fac}_Y$ is locally constant. 
\end{proposition}
\begin{proof}   Let $U\hookrightarrow V$ be an inclusion of open disks which are both good neighborhoods at $Y_i$; by proposition~\ref{P:locallocallyconstantStrat}, we may assume $V$ lies in one of the good trivializing disk in the definition of an adequatly stratified map so that we have a cover by opens homeomorphic to  $(\psi^{-1}(V\times D_j))_{j\in J}$ which are good neigborhood such that $(\psi^{-1}(U\times D_j))_{j\in J}$  is also a good neighborhood of same index. 
This reduces the proof to the same argument as  the one of Proposition~\ref{P:pushforwardlc}.
\end{proof}

If $X$, $Y$ are stratified spaces with finitely many strata, there is a natural stratification on the product $X\times Y$, given by $(X\times Y)_k:=\bigcup_{i+j=k} X_i\times Y_j \subset X\times Y$. The natural projections on $X$ and $Y$ are adequatly stratified. 
\begin{corollary}\label{L:expFactStrat} Let $X$, $Y$ be stratified spaces with finitely many strata.
The pushforward $\underline{\pi_1}_*: {\textbf{Fac}}_{X\times Y} \longrightarrow {\textbf{Fac}}_X(\textbf{Fac}_Y)$ (see Proposition~\ref{L:expFact}) induces a functor $\underline{\pi_1}_*: {\textbf{Fac}}^{lc}_{X\times Y} \longrightarrow {\textbf{Fac}}^{lc}_X(\textbf{Fac}_Y^{lc})$.
\end{corollary}
We conjecture that $\underline{\pi_1}_*$ is an equivalence under  rather weak conditions on $X$ and $Y$. We will give a couple of examples.
\begin{proof}Since the projections are adequatly stratified, the result follows from Proposition~\ref{L:expFact} together with Proposition~\ref{P:pushforwardlcStrat} applied to both projections (on $X$ and $Y$). 
\end{proof}
\begin{remark} Let $\mathcal{F}$ be a stratified locally constant factorization algebra on $X$.
 Let $U\subset V$ be stratified disks of same index $i$, but not necessarily good neighborhoods at $X_i$. Assume all connected component of $V\cap X_i$ contains exactly one connected component of $U\cap X_i$. Then the structure maps
 $\mathcal{F}(U) \to \mathcal{F}(V)$ is a quasi-isomorphism. Indeed, we can take a factorizing cover $\mathcal{V}$ of $V$ by good neighborhoods $D$ such that $D\cap U$ are good neigborhoods. Then, $\mathcal{F}(D\cap U)\to \mathcal{F}(D)$ is a quasi-isomorphism and thus we get a quasi-isomorphism 
 $\check{C}(\mathcal{F}, \mathcal{V}\cap U) \stackrel{\simeq}\to \check{C}(\mathcal{F}, \mathcal{V})$. 
\end{remark}

\subsection{Factorization algebras on the interval and (bi)modules} \label{SS:Fachalfline} Let us consider an important example: the closed interval $I=[0,1]$ viewed as a stratified space\footnote{note that this stratification is just given by looking at $[0,1]$ as a manifold with boundary}  with two dimension $0$-strata given by $I_0=\{0,1\}$. 

The  disks at $I_0$ are  the half-closed intervals $[0,s)$ ($s<1$) and $(t,1]$ ($0<t$) and the disks at $I_1$ are the open  intervals $(t,u)$ ($0<t<u<1$).  The disks of (the stratified space) $I$ form a (stable by finite intersection) factorizing basis denoted $\mathcal{I}$. 

An example of stratified locally constant factorization algebra on $I$ is obtained as follows.
 Let $A$ be a differential graded associative unital algebra, $M^r$ a pointed differential graded right $A$-module (with distinguished element denoted $m^r\in M^r$) and $M^\ell$ a pointed differential graded left $A$-module (with distinguished element  $m^\ell\in M^\ell$).  We define a $\mathcal{I}$-prefactorization algebra by setting, for any interval $J\in \mathcal{I}$ 
$$
\mathcal F(J):=\begin{cases}
M^r\textrm{ if }0\in J \\
M^\ell\textrm{ if }1\in J \\
A\textrm{ else.}
\end{cases}
$$

We define its structure maps to be given by the following\footnote{as in Example~\ref{ex:E1frombasis}, we use the implicit  orientation of $I$ given by increasing numbers}:
\begin{itemize}
\item$\mathcal{F}(\emptyset)\to \mathcal{F}([0,s))$ is given by $k\ni 1 \mapsto m^r$, $\mathcal{F}(\emptyset)\to \mathcal{F}((t,s))$ is given by $ 1 \mapsto 1_A$ and  $\mathcal{F}(\emptyset)\to \mathcal{F}((t,1]))$ is given by $k\ni 1 \mapsto m^\ell$; 
\item For $0<s<t_1<u_1<\cdots <t_i<u_{i}<v<1$ one sets  
\begin{eqnarray*}
M^r\otimes A^{\otimes i}=\mathcal F([0,s))\otimes\mathcal F((t_1,u_1))\otimes \cdots  \otimes \mathcal F((t_i,u_i))& \longrightarrow & \mathcal F([0,v))=M^r \\
m\otimes a_1\otimes \cdots \otimes a_i & \longmapsto & m\cdot a_1\cdots a_i\,;
\end{eqnarray*}
\begin{eqnarray*}
 A^{\otimes i}\otimes M^{\ell}=\mathcal F((t_1,u_1))\otimes \cdots  \otimes\mathcal F((t_i,u_i))\otimes \mathcal F((v,1])& \longrightarrow & \mathcal F((s,1])=M^\ell \\
 a_1\otimes \cdots \otimes a_i \otimes n & \longmapsto & a_1\cdots a_i\cdot n\,;
\end{eqnarray*}
and also 
\begin{eqnarray*}
 A^{\otimes i}=\mathcal F((t_1,u_1))\otimes \cdots  \otimes\mathcal F((t_i,u_i)) & \longrightarrow & \mathcal F((s,v))=A \\
a_1\otimes \cdots \otimes a_i & \longmapsto &  a_1\cdots a_i\,.
\end{eqnarray*}
\end{itemize}
It is straightforward to check that $\mathcal{F}$ is a $\mathcal{I}$-prefactorization algebra and,
by definition, it satisfies the locally constant condition.

Proposition~\ref{P:BarInterval} shows that $\mathcal{F}$ is indeed a locally constant factorization algebra on the closed interval $I$. 
Further, \emph{any} locally constant factorization algebra on  $I$ is (homotopy) equivalent to such a factorization algebra.

\smallskip

Also note that the $\mathcal{I}$-prefactorization algebra  induced by $\mathcal{F}$ on the open interval $(0,1)$ is precisely the  $\mathcal{I}$-prefactorization algebra constructed in Example~\ref{ex:E1frombasis} (up to an identification of $(0,1)$ with $\R$). We denote it $\mathcal{F}_A$. 
\begin{proposition}\label{P:BarInterval} Let $\mathcal{F}$, $\mathcal{F}_A$\ be defined as above. \begin{enumerate} \item 
The $\mathcal{I}$- prefactorization algebra $\mathcal{F}$ is an $\mathcal{I}$-factorization algebra hence extends  uniquely into a factorization algebra (still denoted) $\mathcal{F}$ on the stratified closed interval $I=[0,1]$;
\item  in particular, $\mathcal{F}_A$ also extends uniquely into a factorization algebra (still denoted) $\mathcal{F}_A$ on $(0,1)$.
\item There is an equivalence $\int_{[0,1]}\mathcal{F}=\mathcal{F}([0,1]) \cong M^{r}\mathop{\otimes}\limits^{\mathbb{L}}_{A} M^{\ell}$ in $\hkmod$.
\item Moreover, any locally constant factorization algebra $\mathcal{G}$ on $I=[0,1]$ is equivalent\footnote{more precisely, taking $A$, $M^{\ell}$, $M^r$  be strictification of  $\mathcal{A}$, $\mathcal{M}^{\ell}$ and $\mathcal{M}^{r}$, there is a quasi-isomorphism of factorization algebras from $\mathcal{F}$ (associated to $A, M^{\ell}, M^{r}$) to $\mathcal{G}$.} to $\mathcal{F}$ for some $A, M^{\ell}, M^r$, that is, it is uniquely determined by an $E_1$-algebra $\mathcal{A}$ and pointed left module $\mathcal{M}^{\ell}$ and pointed right module $\mathcal{M}^{r}$ satisfying $$\mathcal{G}([0,1))\cong\mathcal{M}^{r}, \quad \mathcal{G}((0,1])\cong\mathcal{M}^{ \ell}, \quad \mathcal{G}((0,1))\cong \mathcal{A}$$ with structure maps given by the $E_1$-structure similarly to those of $\mathcal{F}$. 
\end{enumerate}
\end{proposition}
For a proof, see \S~\ref{Proof:BarInterval}.
The last statement restricted to the open interval $(0,1)$ is just Theorem~\ref{P:En=Fact} (in the case $n=1$). 

\begin{example}\label{ex:halfline} We consider the closed half-line $[0,+\infty)$ as a stratified manifold, with strata  $\{0\}\subset [0,+\infty)$ given by its boundary.
Namely, it has a $0$-dimensional strata given by $\{0\}$ and thus one dimension 1 strata $(0,+\infty)$. Similarly, there is a stratified closed half-line $(-\infty, 0]$.
From Proposition~\ref{P:BarInterval} (and its proof)  we also deduce 
\begin{proposition} \label{P:halfline} There is an equivalence of $\infty$-categories between locally constant factorization algebra on the closed half-line $[0,+\infty)$ 
and the category $E_1\textbf{-RMod}$ of (pointed) right modules over $E_1$-algebras\footnote{which, informally, is the category of pairs $(A, M^r)$ where $A$ is an $E_1$-algebra and $M^r$ a (pointed) right $A$-module}. This equivalence sits in a commutative diagram 
$$ \xymatrix{\textbf{Fac}^{lc}_{[0,+\infty)} \ar[rr]^{\cong} \ar[d] && E_1\textbf{-RMod} \ar[d] \\
\textbf{Fac}^{lc}_{(0,+\infty)} \ar[rr]^{\cong} && E_1\textbf{-Alg}  }$$
where the left vertical functor is given by restriction to the open line  and the lower horizontal functor is given by Theorem~\ref{P:En=Fact}.

There is a  similar equivalence (and diagram) of $\infty$-categories between locally constant factorization algebra on the closed half-line $(-\infty, 0]$ 
and the category $E_1\textbf{-LMod}$ of (pointed) left modules over $E_1$-algebras.
\end{proposition} 

Let $X$ be a manifold and consider the stratified manifold $X\times [0,+\infty)$, with a $\dim(X)$ open strata $X\times \{0\}$. Using Corollary~\ref{L:expFactStrat} and Proposition~\ref{P:halfline} one can get
\begin{corollary}\label{C:FacXhalfline}
The pushforward along the projection $X\times [0,+\infty) \to [0,+\infty)$ induces an equivalence
$$\textbf{Fac}^{lc}_{X\times [0,+\infty)} \stackrel{\simeq}\longrightarrow E_1\textbf{-RMod}(\textbf{Fac}_X^{lc}).$$
\end{corollary}
\end{example}
\subsection{Factorization algebras on pointed disk and $E_n$-modules} \label{SS:EnModasFac}

In this section we relate $E_n$-modules\footnote{according to our convention in Appendix~\ref{S:EnAlg},  all $E_n$-modules are pointed by definition} and factorization algebras over the pointed disk.  

Let \emph{$\R^n_*$ denote the pointed disk} which we see as a stratified manifold with one  $0$-dimensional strata given by the point $0\in \R^n$  and  $n$-dimensional strata given by the complement $\R^n\setminus\{0\}$. 
\begin{definition}We denote $\textbf{Fac}_{\R^n_*}^{lc}$ the \emph{$\infty$-category of locally constant factorization algebras on the pointed disk $\R^n_*$} (in the sense of Definition~\ref{D:lcFacStratified}).
\end{definition} Recall  the  functor~\eqref{eq:FactoFacstratified} giving the obvious embedding  $\textbf{Fac}^{lc}_{\R^n} \longrightarrow \textbf{Fac}^{lc}_{\R^n_*}$.

Locally constant factorization algebras on $\R^n_*$ are related to those on the closed half-line (\S~\ref{SS:Fachalfline}) as follows: let $N:\R^n\to [0,+\infty)$ be the euclidean norm map $x\mapsto \|x\|$. We have the pushforwards $N_*:\textbf{Fac}_{\R^n_*}\to \textbf{Fac}_{[0,+\infty)}$ and $(-N)_*:\textbf{Fac}_{\R^n_*}\to \textbf{Fac}_{(-\infty,0]}$.
\begin{lemma}\label{L:Normpreserveslc} If $\mathcal{F}\in \textbf{Fac}_{\R^n_*}^{lc}$, then $N_*(\mathcal{F})\in \textbf{Fac}_{[0,+\infty)}^{lc}$ and $(-N)_*(\mathcal{F})\in \textbf{Fac}_{(-\infty,0]}^{lc}$.
\end{lemma}
\begin{proof} 
For $0<\varepsilon<\eta$, the structure map  
$N_*(\mathcal{F})\big([0,\varepsilon) \big) \cong \mathcal{F}\big(N^{-1}([0,\varepsilon)) \big)\to  \mathcal{F}\big(N^{-1}([0,\eta)) \big)\cong N_*(\mathcal{F})\big([0,\eta) \big)  $ is an equivalence since $\mathcal{F}$ is locally constant and $N^{-1}([0,\alpha)$ is a euclidean disk centered at $0$. Further, by Proposition~\ref{P:pushforwardlc}, $N_*(\mathcal{F}_{|\R^n\setminus\{0\}})$ is locally constant  from which we deduce that $N_*(\mathcal{F})$ is locally constant on the stratified half-line $[0,+\infty)$. The case of $(-N)_*$ is the same. \qed
\end{proof}

Our next task is to define a functor $E_n\textbf{-Mod} \to \textbf{Fac}^{lc}_{\R^n_*}$ from (pointed) $E_n$-modules (see~\S~\ref{S:EnAlg}) to (locally constant) factorization algebras on the pointed disk. It is enough to associate (functorially), to 
any $M\in  E_n\textbf{-Mod}$,  a $\mathcal{CV}(\R^n)$-factorization algebra $\mathcal{F}_M$ where $\mathcal{CV}(\R^n)$ is the (stable by finite intersection) factorizing basis of $\R^n$ of convex open subsets. 
It turns out to be easy: since any convex subset $C$ is canonically an embedded framed disk, we can set $\mathcal{F}_M(C) :=M(C)$. In other words, we assign the module to a convex neighborhood of $0$ and the algebra to a convex neigborhood which does not contain the origin.  

Then set the structure maps $\mathcal{F}_M(C_1)\otimes \cdots \otimes\mathcal{F}_M(C_i) \to \mathcal{F}_M(D)$, for any pairwise disjoint convex subsets $C_k$ of $D$, to be given by  the map $M(C_1)\otimes \cdots\otimes M(C_i)\to M(D)$ associated to the framed embedding $\coprod_{k=1\dots i} \R^n \cong \bigcup_{k=1\dots i} C_k \hookrightarrow D \hookrightarrow \R^n$. 

\begin{theorem}\label{P:EnModasFact} The rule $M\mapsto \mathcal{F}_M$ induces a fully faithful functor $\psi: E_n\textbf{-Mod} \to \textbf{Fac}_{\R^n_*}^{lc}$   which fits  in a commutative diagram 
$$\xymatrix{ E_n\textbf{-Alg} \ar[rr]^{\simeq} \ar[d]_{can} & & \textbf{Fac}^{lc}_{\R^n}\ar[d] \\
E_n\textbf{-Mod} \ar[rr]^{\psi} & & \textbf{Fac}_{\R^n_*}^{lc} .} $$ Here $can:E_n\textbf{-Alg} \to E_n\textbf{-Mod}$ is given by the canonical module structure of an algebra  over itself.
\end{theorem}

We will now identify  $E_n$-modules in terms of factorization algebras on $\R^n_*$; that is the essential image of the functor $\psi:E_n\textbf{-Mod}\to \textbf{Fac}^{lc}_{\R^n_*}$ given by Theorem~\ref{P:EnModasFact}. Recall the functor   $\pi_{E_n}: E_n\textbf{-Mod}\to E_n\textbf{-Alg}\cong \textbf{Fac}_{\R^n}^{lc}$ which, to a module $M\in E_n\textbf{-Mod}_A$, associates $\pi_{E_n}(M)=A$.

By restriction to  the open set $\R^n\setminus\{0\}$ we get that the   two compositions of functors
$$E_n\textbf{-Mod}\stackrel{\psi}\longrightarrow \textbf{Fac}^{lc}_{\R^n_*}\to \textbf{Fac}_{\R^n\setminus\{0\}}^{lc}\;\; \mbox{ and } \; \;
E_n\textbf{-Mod}\stackrel{\pi_{E_n}}\longrightarrow  E_n\textbf{-Alg}\cong\textbf{Fac}^{lc}_{\R^n}\to \textbf{Fac}_{\R^n\setminus\{0\}}^{lc}$$  are equivalent. Hence, we get a factorization of $(\psi, \pi_{E_n})$ to the pullback
$$ (\psi, \pi_{E_n}):E_n\textbf{-Mod}\longrightarrow \textbf{Fac}^{lc}_{\R^n_*}\times^{h}_{\textbf{Fac}^{lc}_{\R^n\setminus\{0\}}} \textbf{Fac}^{lc}_{\R^n}.$$
  Informally,   $\textbf{Fac}^{lc}_{\R^n_*}\times^{h}_{\textbf{Fac}^{lc}_{\R^n\setminus\{0\}}}\textbf{Fac}^{lc}_{\R^n}$ is simply the ($\infty$-)category  of pairs $(\mathcal{A}, \mathcal{M})\in \textbf{Fac}_{\R^n}^{lc}\times \textbf{Fac}^{lc}_{\R^n_*}$
 together with a quasi-isomorphism  $f: \mathcal{A}_{|\R^n\setminus\{0\}} \to \mathcal{M}_{|\R^n\setminus\{0\}}$ of factorization algebras. 
\begin{corollary}\label{C:EnModasFact}
The functor $ E_n\textbf{-Mod}\stackrel{(\psi, \pi_{E_n})}\longrightarrow \textbf{Fac}^{lc}_{\R^n_*}\times^{h}_{\textbf{Fac}^{lc}_{\R^n\setminus\{0\}}} \textbf{Fac}^{lc}_{\R^n}$ is an equivalence.
\end{corollary}
Now, if $A$ is an $E_n$-algebra, we can see it as a factorization algebra on $\R^n$ (Theorem~\ref{P:En=Fact}) and taking the (homotopy) fiber  at $\{A\}$ of the right hand side of  the equivalence in Corollary~\ref{C:EnModasFact}, we get
\begin{corollary}\label{C:AEnModasFact}
The functor $ E_n\textbf{-Mod}_{A}\stackrel{\psi}\longrightarrow \textbf{Fac}^{lc}_{\R^n_*}\times^{h}_{\textbf{Fac}^{lc}_{\R^n\setminus\{0\}}} \{A\}$ is an equivalence. 
\end{corollary}

\begin{example}[Locally constant factorization algebra on the pointed line $\R_*$]\label{ex:pointeddisk} The case of the pointed line $\R_*$ is slightly special since it has two (and not one) strata of maximal dimension. 
The two embeddings  $j_+:(0,+\infty)\hookrightarrow \R$ and $j_{-}:(-\infty,0)\hookrightarrow \R$ 
yields two restrictions functors: $j_{\pm}^*:\textbf{Fac}_{\R_*}^{lc} \to \textbf{Fac}_{\R}^{lc}$.
Hence   $\mathcal{F}\in \textbf{Fac}_{\R_*}^{lc}$  determines two $E_1$-algebras $R\cong \mathcal{F}\big((0,+\infty)\big)$ and $L\cong \mathcal{F}\big((-\infty, 0)\big)$. 

By Lemma~\ref{L:Normpreserveslc},  the pushforward $(-N)_*: \textbf{Fac}_{\R_*}^{lc}\to \textbf{Fac}^{lc}_{(-\infty,0]}$ along $-N:x\mapsto -|x|$ is well defined.
Then Proposition~\ref{P:halfline} implies that  $(-N)_*(\mathcal{F})$ is  determined by a left module $M$ over the $E_1$-algebra  $A\cong (-N)_*(\mathcal{F})\big((0,+\infty)\big)\cong L\otimes R^{op}$, \emph{i.e.} by a $(L,R)$-bimodule. 

We thus have a functor $(j_{\pm}^*, (-N)_*):\textbf{Fac}_{\R_{*}}^{lc} \to \textbf{BiMod}$ where $\textbf{BiMod}$ is the $\infty$-category of bimodules (in $\hkmod$) defined in~\cite[\S 4.3]{L-HA} (\emph{i.e.}   the $\infty$-category  of triples $(L,R,M)$ where $L$, $R$ are $E_1$-algebras and $M$ is a $(L,R)$-bimodule). 
\begin{proposition}\label{P:BimodasFact}
The functor  $(j_{\pm}^*, (-N)_*):\textbf{Fac}_{\R_{*}}^{lc} \cong \textbf{BiMod}$  is an equivalence. 
\end{proposition}

\end{example}

\begin{example}[From pointed disk to the $n$-dimensional annulus: $S^{n-1}\times \R$]
 Let $\mathcal{M}$ be  a locally constant factorization algebra on the pointed disk $\R^n_*$.   By the previous results in this Section  (specifically Lemma~\ref{L:Normpreserveslc} and Proposition~\ref{P:halfline}), 
   for any $\mathcal{A}\in \textbf{Fac}^{lc}_{S^{n-1}\times \R}$,  the pushforward along the euclidean norm $N: \R^n\to [0,+\infty)$ factors as a functor $$N_*: \textbf{Fac}^{lc}_{\R^n_*}\times_{\textbf{Fac}^{lc}_{\R^n\setminus\{0\}}} \{\mathcal{A}\}  \longrightarrow E_1\textbf{-RMod}_{\mathcal{A}(S^{n-1}\times \R)}$$
from the category of locally constant factorization algebras on the pointed disk whose restriction to $\R^n\setminus \{0\}$ is (quasi-isomorphic to) $\mathcal{A}$ to the category of right modules over the $E_1$-algebra\footnote{the $E_1$-structure is given by Lemma~\ref{L:homthMfldisEn}} $\mathcal{A}({S^{n-1}\times \R})\cong \int_{S^{n-1}\times \R} \mathcal{A}$.
Similarly, we have the functor 
$(-N)_*: \textbf{Fac}^{lc}_{\R^n_*}\times_{\textbf{Fac}^{lc}_{\R^n\setminus\{0\}}} \{\mathcal{A}\}  \longrightarrow E_1\textbf{-LMod}_{\mathcal{A}(S^{n-1}\times \R)}$.
\begin{proposition}\label{P:envelopFac} 
Let $\mathcal{A}$ be in $\textbf{Fac}^{lc}_{S^{n-1}\times \R}$.
\begin{itemize} \item The functor $N_*: \textbf{Fac}^{lc}_{\R^n_*}\times_{\textbf{Fac}^{lc}_{\R^n\setminus\{0\}}} \{\mathcal{A}\}  \longrightarrow E_1\textbf{-RMod}_{\mathcal{A}(S^{n-1}\times \R)}$ is an equivalence of $\infty$-categories.
\item   $(-N)_*: \textbf{Fac}^{lc}_{\R^n_*}\times_{\textbf{Fac}^{lc}_{\R^n\setminus\{0\}}} \{\mathcal{A}\}  \longrightarrow E_1\textbf{-LMod}_{\mathcal{A}(S^{n-1}\times \R)}$ is an equivalence.
\end{itemize}
\end{proposition}
See~\S~\ref{SS:AppEnModasFac} for a Proof.

The Proposition allows to reduce a category  of\lq\lq{}modules\rq\rq{} over $\mathcal{A}$ to a category of modules over an (homotopy) dg-\emph{associative} algebra. 

We will see in section~\ref{SS:enveloping} that it essentially gives a concrete description of the enveloping algebra of an $E_n$-algebra.
\end{example}


\subsection{More examples}
We now give, without going into any details, a few more examples of locally constant factorization algebras on stratified spaces which can be studied along the same way as we did previously.

\begin{example}[towards quantum mechanics] One can make the following   variant of the interval example of Proposition~\ref{P:BarInterval} (see~\cite{CG} for details). Fix a one parameter group  $(\alpha_t)_{t\in[0,1]}$ of invertible elements in an associative (topological) algebra $A$. Now, we define a factorization algebra on the basis $\mathcal{I}$ of open disks of $[0,1]$ in the same way as in Proposition~\ref{P:BarInterval} except that we add the element $\alpha_d$ to any hole of length $d$ between two intervals (corresponding to the inclusion of the empty set: $\mathcal{F} (\emptyset)\to \mathcal{F}(I)$ into any open interval).
Precisely, we  set:
\begin{itemize}
\item the structure map $k=\mathcal{F}(\emptyset)\to \mathcal{F}\big(s,t)\big)=A$ is given by the element  
 $\alpha_{t-s}$,  the structure map $k=\mathcal{F}(\emptyset)\to \mathcal{F}\big[0,t)\big)=M^r$ $m^r\cdot \alpha_{t}$ and the structure map $k=\mathcal{F}(\emptyset)\to \mathcal{F}\big(u,1]\big)=M^{\ell}$ $\alpha_{1-u}\cdot m^{\ell}$.
\item For $0<s<t_1<u_1<\cdots <t_i<u_{i}<v<1$ one sets  the map
$$M^r\otimes A^{\otimes i}=\mathcal F([0,s))\otimes\mathcal F((t_1,u_1))\otimes \cdots  \mathcal F((t_i,u_i)) \longrightarrow  \mathcal F([0,v))=M^r$$ to be given by
$$m\otimes a_1\otimes \cdots \otimes a_i \longmapsto  m\cdot \alpha_{t_1-s} a_1\alpha_{t_2-u_1}a_2\cdots \alpha_{t_i-u_{i-1}}a_i \alpha_{v-u_i},$$
the map $$
 A^{\otimes i}\otimes M^{\ell}=\mathcal F((t_1,u_1))\otimes \cdots  \mathcal F((t_i,u_i))\otimes \mathcal F((v,1]) \longrightarrow  \mathcal F((s,1])=M^\ell $$ to be given by
$ a_1\otimes \cdots \otimes a_i \otimes n  \longmapsto  \alpha_{t_1-s}a_1\alpha_{t_2-u_1}a_2\cdots \alpha_{t_i-u_{i-1}}a_i \alpha_{v-u_i}n,$
and also the map 
$ A^{\otimes i}=\mathcal F((t_1,u_1))\otimes \cdots  \otimes \mathcal F((t_i,u_i))  \longrightarrow  \mathcal F((s,v))=A $ to be given by  $$a_1\otimes \cdots \otimes a_i  \longmapsto  \alpha_{t_1-s}a_1\alpha_{t_2-u_1}a_2\cdots \alpha_{t_i-u_{i-1}}a_i \alpha_{v-u_i}.$$
\end{itemize}
One can check that these structure maps define a $\mathcal{I}$-factorization algebra and thus a factorization algebra on $I$.

One can replace chain complexes with topological  $\mathbb{C}$-vector spaces (with monoidal structure the completed tensor product) and take $V$ to be an Hilbert space. Then one can choose $A=\left(End^{cont}(V)\right)^{op}$ and $M^r=V=M^\ell$ where the left $A$-module structure on $M^{\ell}$ is given by the action of adjoint operators. Let 
 $\alpha_t:=e^{{\rm i}t\varphi}$ where $\varphi\in End^{cont}(V)$. 
One has $\mathcal F([0,1])=\mathbb{C}$. One can think of $A$ as the algebra of observables where $V$ is the space of states and 
 $e^{it\varphi}$ is the time evolution operator. Now, given two consecutives
measures $O_1,O_2$ 
made during the time intervals  $]s,t[$ and $]u,v[$ ($0<s<t<u<v<1$), the probability amplitude that
the system goes from an initial  state $v^r$ to a final state
$v^\ell$ is  the image of $O_1\otimes O_2$ under the structure map
$$
A\otimes A=\mathcal F(]s,t[)\otimes\mathcal F(]u,v[)\longrightarrow\mathcal F([0,1])=\mathbb{C}\,,
$$
and is usually denoted $\langle v^\ell|e^{i(1-v)\varphi}O_2e^{i(u-t)\varphi}O_1e^{is\varphi}|v^r\rangle$. 
\end{example}

\begin{example}[the upper half-plane] \label{ex:upperhalfplane} Let $H=\{z=x+i y \in \mathbb{C}, y\geq 0\}$ be the closed upper half-plane, viewed as a stratified space with $H_0=\emptyset$, $H_1=\R$ the real line as dimension $1$ strata and $H_2=H$.
 The orthogonal  projection $\pi: H\to \R$ onto the imaginary axis $\R i \subset \mathbb{C}$ induces, by Proposition~\ref{L:expFactStrat} (and Proposition~\ref{P:halfline}),  an equivalence 
$$\textbf{Fac}_{H}^{lc} \stackrel{\simeq}\longrightarrow \textbf{Fac}^{lc}_{[0,+\infty)}(\textbf{Fac}_\R^{lc})\cong E_1\textbf{-RMod}(E_1\textbf{-Alg}).$$
Using Dunn Theorem~\ref{T:Dunn},  we have (sketched a proof of the fact) that 
\begin{proposition}
The  $\infty$-category of stratified locally constant factorization algebra on $H$ is equivalent to the $\infty$-category of algebras over the swiss cheese operad, that is the $\infty$-category consisting of triples $(A,B,\rho)$ where $A$ is an $E_2$-algebra, $B$ an $E_1$-algebra and $\rho: A\to \RHom_{B}^{E_1}(B,B)$ is an action of $A$ on $B$ compatible with all the multiplications\footnote{that is the map $\rho$ is a map of $E_2$-algebras where $ \RHom_{B}^{E_1}(B,B)$ is the $E_2$-algebra given by the (derived) center of $B$. In particular, $\rho$ induces a map $\rho(1_B):A\to B$}. 
\end{proposition} 
One can also consider another stratification $\tilde{H}$ on $\R\times [0,+\infty)$ given by adding a $0$-dimensional strata to $H$, given by the point $0\in \mathbb{C}$.
There is now  4 kinds of (stratified) disks in $\tilde{H}$: the half-disk containing $0$, the half disk with a connected boundary component lying on $(-\infty,0)$, the half-disk containing $0$, the half disk with a connected boundary component lying on $(0, +\infty)$ and the open disks in the interior $\{x+iy, \, y>0\}$ of $H$. 

One proves similarly
\begin{proposition}\label{P:halfplanepted}Locally constant factorization algebras on the stratified space $\tilde{H}$ are the same as the category given by quadruples $(M, A,B, \rho_A, \rho_B, E)$ where $E$ is an $E_2$-algebra, $(A, \rho_A)$, $(B, \rho_B)$ are $E_1$-algebras together with a compatible action of $E$ and $(M, \rho_M)$ is a $(A,B)$-bimodule  together with a compatible action\footnote{precisely, this means  the choice of a factorization  $A\otimes B^{op} \longrightarrow  A\mathop{\otimes}\limits^{\mathbb{L}}_{E} B^{op} \stackrel{\rho_{M}} \longrightarrow  \RHom(M,M) $ (in $E_1\textbf{-Alg}$) of the $(A,B)$-bimodule structure of $M$} of $E$. 
\end{proposition}
Examples of such factorization algebras occur in deformation quantization in the presence of two branes, see~\cite{CFFR}.

Note that the norm is again adequatly stratified so that, if $\mathcal{F}\in \textbf{Fac}^{lc}_{\tilde{H}}$,  then $(-N)_*(\mathcal{F})\in \textbf{Fac}^{lc}\big((-\infty,0]\big)\cong E_1\textbf{-LMod}$. Using the argument of Proposition~\ref{P:envelopFac} and Proposition~\ref{P:BarInterval}, we see that if $\mathcal{F}$ is given by a tuple $(M, A,B, \rho_A, \rho_B, E)$, then the underlying $E_1$-algebra of $(-N)_*(\mathcal{F})$ is (the two-sided Bar construction) $A\mathop{\otimes}\limits^{\mathbb{L}}_{E} B^{op}$.

\end{example}

\begin{example}[the unit disk in $\mathbb{C}$]
Let $D=\{z\in \mathbb{C}, \, |z|\leq 1\}$ be the closed unit disk (with dimension 1 strata given by its boundary).  By Lemma~\ref{L:Normpreserveslc}, the norm  gives us a pushforward functor 
$N_*: \textbf{Fac}^{lc}_{D} \to \textbf{Fac}^{lc}_{[0,1]}$. A proof similar to the one of Corollary~\ref{C:FacXhalfline} and Theorem~\ref{P:EnModasFact} shows that the two restrictions of $N_*$ to $D\setminus \partial D$ and $D\setminus \{0\}$  induces an equivalence
$$\textbf{Fac}^{lc}_{D} \stackrel{\simeq} \longrightarrow \textbf{Fac}^{lc}_{D\setminus \partial D} \times^{h}_{\textbf{Fac}^{lc}_{(0,1)}}E_1\textbf{-RMod}\big(\textbf{Fac}^{lc}_{S^1}\big). $$ 
By Corollary~\ref{C:locallyconstantS1} and Theorem~\ref{P:En=Fact}, one obtains:
\begin{proposition}\label{P:FaconclosedRiemannDisk}
 The $\infty$-category of locally constant factorization algebra on the (stratified) closed unit disk $D$ is equivalent to the $\infty$-category  consisting of quadruples $(A,B,\rho,f)$ where $A$ is an $E_2$-algebra, $B$ an $E_1$-algebra, $\rho: A: \RHom_{E_1\textbf{-Alg}}(B,B)$ is an action of $A$ on $B$ compatible with all the multiplications and $f:B\to B$ is a monodromy compatible with $\rho$.
\end{proposition}
One can also consider a variant of this construction with dimension $0$ strata given by the center of the disk. In that case, one has to add a $E_2$-$A$-module to the data in Proposition~\ref{P:FaconclosedRiemannDisk}. 
\end{example}

\begin{example}[the closed unit disk in $\R^n$]\label{Diskandaugmented}
Let $D^n$ be the closed unit disk of $\R^n$ which is stratified with a single  strata of dimension $n-1$ given by its boundary $\partial D^n =S^{n-1}$. 
We have restriction functors $\textbf{Fac}^{lc}_{D^n}\longrightarrow \textbf{Fac}^{lc}_{D^n\setminus \partial D^n}\cong E_n\textbf{-Alg}$ (Theorem~\ref{P:En=Fact}),    $$E_n\textbf{-Alg} \cong  \textbf{Fac}^{lc}_{D^n\setminus \partial D^n}\to \textbf{Fac}^{lc}_{S^{n-1}\times (0,1)}\cong E_1\textbf{-Alg}(\textbf{Fac}^{lc}_{S^{n-1}}) \quad \mbox{(by Proposition~\ref{L:expFact})}$$
 and $\textbf{Fac}^{lc}_{D^n}\longrightarrow \textbf{Fac}^{lc}_{D^n\setminus \{0\}}$.  From Corollary~\ref{C:FacXhalfline}, we deduce
\begin{proposition}\label{P:FaconDn}
The above restriction functors induce an equivalence $$\textbf{Fac}^{lc}_{D^n}\stackrel{\simeq}\longrightarrow E_n\textbf{-Alg}\times_{E_1\textbf{-Alg}\big(\textbf{Fac}^{lc}_{S^{n-1}}\big)} E_1\textbf{-LMod}\big(\textbf{Fac}^{lc}_{S^{n-1}}\big).$$
\end{proposition}
Let $f:A\to B$ be an $E_n$-algebra map. Since $S^{n-1}\times \R$ has a canonical framing, both $A$ and $B$ carries a structure of locally constant factorization algebra  on $S^{n-1}$, which is induced by pushforward along the projection $q:S^{n-1}\times \R\to S^{n-1}$ (see Example~\ref{Ex:bundleFac}). 
Further, $B$ inherits an $E_n$-module structure over $A$ induced by $f$ and, by restriction, $q_*(B_{|S^{n-1}\times \R})$ is an $E_1$-module over $q_*(A_{|S^{n-1}\times \R})$.  
Thus, Proposition~\ref{P:FaconDn} yields a factorization algebra $\omega_{D^n}(A\stackrel{f}\to B)\in \textbf{Fac}^{lc}_{D^n}$.  
\begin{proposition}\label{P:EnAlgmapgivesfactonDn}
The map $\omega_{D^n}$ induces a faithful functor $\omega_{D^n}:\Hom_{E_n\textbf{-Alg}}\longrightarrow \textbf{Fac}^{lc}_{D^n}$ where $\Hom_{E_n\textbf{-Alg}}$ is the $\infty$-groupoid of $E_n$-algebras maps.  
\end{proposition}

By the relative higher Deligne conjecture (Theorem~\ref{T:Deligne}), we also have the $E_n$-algebra $HH_{E_n}(A, B_f)$ and an $E_n$-algebra map $HH_{E_n}(A, B_f)\to B$ given as the composition 
$$HH_{E_n}(A, B_f)\cong \mathfrak{z}(f)\stackrel{id \otimes 1_A}\longrightarrow \mathfrak{z}(f)\otimes \mathfrak{z}(k\stackrel{1_A}\to A) \longrightarrow \mathfrak{z}(k\stackrel{1_B}\to B) \cong B.$$ 
Thus by Proposition~\ref{P:EnAlgmapgivesfactonDn}, the pair $(HH_{E_n}(A, B_f), B)$ also defines a factorization algebra on $D^n$. 

\smallskip

Let $\tau: D^n\to \R^n\cup \{\infty\}$ be the map collapsing $\partial D^n$ to the point $\infty$. It is an adequatly stratified map if we see $\R^n \cup \{\infty\}$  as being stratified with one dimension $0$ strata given by $\infty$. Denote $\widehat{\R^n}$ this stratified manifold. The composition of $\tau_*$ with $\omega_{D^n}$ gives us:
\begin{corollary}\label{C:EnAlgmapgivesfactonDnhat}
 There is a faithful functor $\widehat{\omega}:\Hom_{E_n\textbf{-Alg}}\longrightarrow \textbf{Fac}^{lc}_{\widehat{\R^n}}$. 
\end{corollary}
\end{example}

\begin{example}[The Square]\label{Ex:FacOnSquare}
 Consider the square $I^2:=[0,1]^{2}$, the product of the interval (with its natural stratification of  \S~\ref{SS:Fachalfline}) with itself. This stratification agrees with the one given by seeing $I^2$ as a manifold with corners. Thus, there are four $0$-dimensional strata corresponding to the vertices of $I^2$ and four $1$-dimensional strata corresponding to the edges. Let us denote $\{1,2,3,4\}$ the set of vertices and $[i,i+1]$ the corresponding edge linking the vertices $i$ and $i+1$ (ordered cyclically\footnote{in other words we have the 4 edges $[1,2]$, $[2,3]$, $[3,4]$ and $[4,1]$}). See Figure~\ref{fig:FacOnSquare}.
 
 We can construct a factorization algebra on $I^2$ as follows. 
For every edge $[i,j]$ of $I^2$, let $A_{[i,j]}$ be an $E_1$-algebra; also  let $A$ be an $E_2$-algebra.  For every vertex $i$, let $M_i$ be a $(A_{[i-1,i]}, A_{[i, i+1]})$-bimodule. Finally, assume that $A$ acts on each $A_{[i,j]}$ and $M_{k}$ in a  compatible way with algebras and module structures. Compatible, here, means  that for
every vertex $i$, the data $(M_i,A_{[i-1,i]}, A_{[i, i+1]}, A)$ (together with the various module structures) define a locally constant stratified factorization algebra on a neighborhood\footnote{for instance, take  the complement of the closed sub-triangle of $I^2$ given by the three other vertices, which is isomorphic as a stratified space to the pointed half plane $\tilde{H}$ of Example~\ref{ex:upperhalfplane}} of the vertex $i$ as given by Proposition~\ref{P:halfplanepted}.

\begin{figure}[b] \begin{minipage}[c]{0.48\textwidth}
 \includegraphics[scale=0.67]{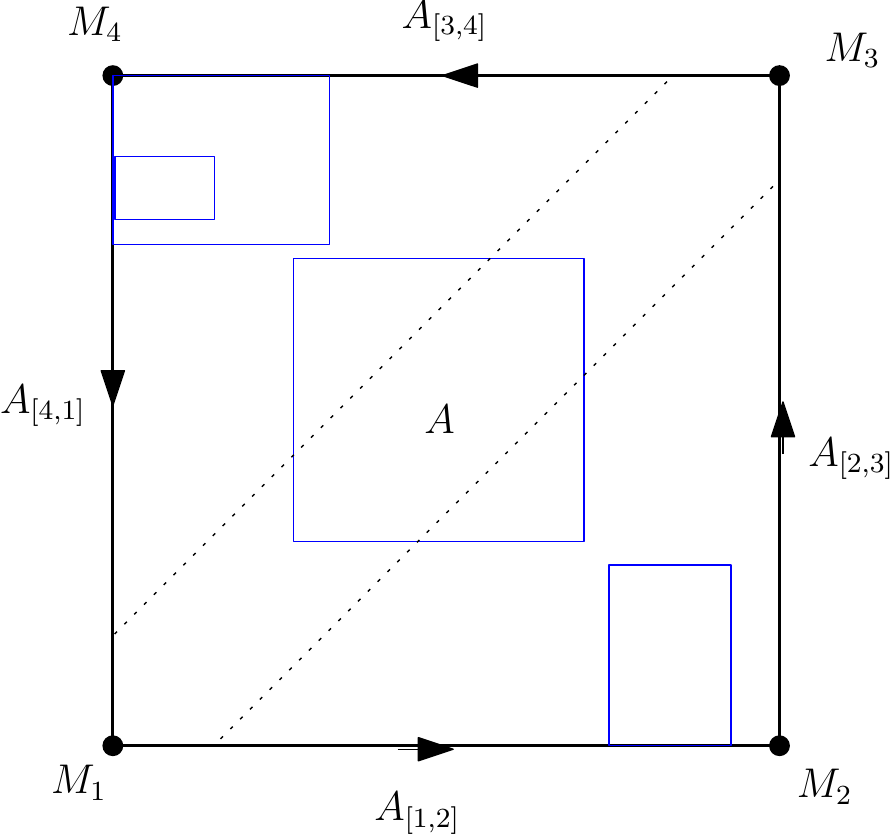}\end{minipage} \hfill \begin{minipage}[c]{0.48\textwidth}
\caption{The stratification of the square. The (oriented) edges and vertices are decorated with their associated algebras and modules. Also 4 rectangles which are good neighborhoods at respectively $1$, $[4,1]$, $[1,2]$ and the dimension $2$-strata  are depicted in blue. The band $D_{13}$ is delimited by the dotted lines while the annulus $T$ is the complement of the interior blue square.}\label{fig:FacOnSquare}\end{minipage}
\end{figure}

Using Propositions~\ref{P:BarInterval} and~\ref{P:BimodasFact} (and their proofs) and Remark~\ref{R:En=Fact}, we see that we obtain a locally constant factorization algebra $\mathcal{I}$ on $I^2$ whose value on an open rectangle $R\subset I^2$ is given by
\begin{equation}\label{eq:FaconSquare} \mathcal{I}(R):= \left\{\begin{array}{ll} M_i & \mbox{if $R$ is a good neighborhood at the vertex $i$;} \\ A_{[i,i+1]} & \mbox{if $R$ is a good neighborhood at the edge $[i,i+1]$;}\\ A & \mbox{if $R$ lies in $I^2\setminus \partial I^2$.} \end{array}\right.\end{equation}
The structure maps being given by the various module and algebras structure\footnote{note that we  orient  the edges accordingly to the ordering of the edges}.

 By Corollary~\ref{L:expFactStrat}, we get the functor $\underline{\pi_1}_*:\textbf{Fac}^{lc}_{[0,1]^{2}}\to \textbf{Fac}^{lc}_{[0,1]}(\textbf{Fac}^{lc}_{[0,1]})$. Combining the proof of Proposition~\ref{L:expFact} and Proposition~\ref{P:BarInterval} (similarly to the proof of Corollary~\ref{C:FacXhalfline}), we get:
 \begin{proposition} The functor $\underline{\pi_1}_*: \textbf{Fac}^{lc}_{[0,1]^{2}}\to \textbf{Fac}^{lc}_{[0,1]}(\textbf{Fac}^{lc}_{[0,1]})$ is an equivalence of $\infty$-categories.
 
 Moreover, any stratified locally constant factorization algebra  on $I^2$ is quasi-isomorphic to a factorization algebra associated to a tuple $(A, A_{[i,i+1]}, M_i, i=1\dots 4)$ as in the rule~\eqref{eq:FaconSquare} above.
 \end{proposition}
Let us give some examples of computations of the global sections of $\mathcal{I}$ on various opens.
\begin{itemize}
 \item The band $I^2\setminus ([4,1] \cup [2,3])$ is isomorphic (as a stratified space) to $(0,1)\times [0,1]$. Then from the definition of factorization homology and Proposition~\ref{P:BarInterval}.3, we get
\begin{eqnarray*}\mathcal{I}\big(I^2\setminus ([4,1] \cup [2,3])\big) &\cong& p_*\big(\mathcal{I}_{|I^2\setminus ([4,1] \cup [2,3])}\big)  \cong  p_*(\underline{\pi_1}_*(\mathcal{I}_{|I^2\setminus ([4,1] \cup [2,3])}) \\&=& \int_{[0,1]}\underline{\pi_1}_*(\mathcal{I}_{|I^2\setminus ([4,1] \cup [2,3])} \cong  A_{[1,2]}\,\mathop{\otimes}^{\mathbb{L}}_{A} \,(A_{[3,4]})^{op}. \end{eqnarray*}In view of Corollary~\ref{C:FacXhalfline} in the case $X=\R$ (and Theorem~\ref{P:En=Fact}); this is an equivalence of $E_1$-algebras.
\item Consider a tubular neighborhood of the boundary, namely the complement\footnote{note that the obvious radial projection $T\to \partial I^2$ is \emph{not}  adequatly stratified} $T:=I^2\setminus [1/4,3/4]^2$ (see Figure~\ref{fig:FacOnSquare}). The argument of Proposition~\ref{P:BarInterval} and Example~\ref{Ex:FacS1=Hoch} shows (by projecting the square on $[0,1]\times \{1/2\}$) that
$$\mathcal{I}\big(T\big) \cong  
\Big(M_4\mathop{\otimes}^{\mathbb{L}}_{A_{[4,1]}} M_1\Big) \mathop{\bigotimes}^{\mathbb{L}}_{A_{[1,2]}\otimes (A_{[3,4]})^{op}} \Big(M_2\mathop{\otimes}^{\mathbb{L}}_{A_{[2,3]}} M_3\Big) .$$
\item Now, consider a diagonal band, say a tubular neigborhood $D_{13}$ of the diagonal linking the vertex $1$ to the vertex $3$ (see Figure~\ref{fig:FacOnSquare}). Then projecting onto the diagonal $[1,3]$ and using again Proposition~\ref{P:BarInterval}, we find,
$$\mathcal{I}\big(D_{13}) \cong M_1 \mathop{\otimes}^{\mathbb{L}}_{A} M_3.$$
\end{itemize}
Iterating the above constructions to higher dimensional cubes, one finds 
\begin{proposition} Let $[0,1]^n$ be the stratified cube. 
 The pushforward along the canonical projections is an equivalence  $\textbf{Fac}^{lc}_{[0,1]^{n}}\stackrel{\simeq}\to \textbf{Fac}^{lc}_{[0,1]}\big(\cdots (\textbf{Fac}^{lc}_{[0,1]})\cdots \big)$.
\end{proposition}
 In other words, $\textbf{Fac}^{lc}_{[0,1]^{n}}$ is a \emph{tractable model} for an $\infty$-category   consisting of the data of an $E_n$-algebra $A_{n}$ together with  $E_{n-1}$-algebras $A_{n-1,i_{n-1}}$ ($i_{n-1}= 1\dots 2n$) equipped with an action of $A_n$,  $E_{n-2}$-algebras $A_{n_2,i_{n-2}}$ ($i_{n-2}=1\dots 2n(n-1)$) each equipped with a structure of bimodule over $2$ of the $A_{n-1, j}$  compatible with the $A$ actions, \dots,  $E_{k}$-algebras $A_{k,i_{k}}$ ($i_k=1\dots 2^{n-k} \binom{n}{k}$) equipped with structure of $n-k$-fold modules over ($n-k$ many of) the $ A_{k+1,j}$ algebras, compatible with the previous actions, and so on \dots.
 
 Similarly to the previous example~\ref{Diskandaugmented}, we also have a faithful functor 
 $$E_n\textbf{-Alg} \cong \textbf{Fac}^{lc}_{\R^n} \longrightarrow \textbf{Fac}^{lc}_{[0,1]^{n}}.$$
\end{example}

\begin{example}[Iterated categories of (bi)modules] We consider the $n$-fold product $(\R_*)^{n}$ of the pointed line (see  Example~\ref{ex:pointeddisk}) with its induced stratification. It has one $0$-dimensional strata given by the origins, $2n$ many $1$-dimensional strata given by  half of the coordinate axis, \dots, and $2^{n}$ open strata. 

Locally constant (stratified) factorization algebras on $(\R_*)^{n}$ are a model for iterated categories of bimodules objects.
Indeed, by Corollary~\ref{L:expFactStrat}, the iterated first projections on $\R_*$ yields a functor
$\underline{\pi}_*: \textbf{Fac}^{lc}_{(\R_*)^{n}} \longrightarrow  \textbf{Fac}^{lc}_{\R_*} \Big(\textbf{Fac}^{lc}_{\R_*}\big( \dots (\textbf{Fac}^{lc}_{\R_*}) \dots\big)\Big).$
 From Proposition~\ref{P:BimodasFact} (and its proof) combined with the arguments of the proofs of Corollary~\ref{C:FacXhalfline} and Proposition~\ref{P:envelopFac}, we get
 \begin{corollary}\label{C:iteratedBiMod} The functor \begin{multline*}\underline{\pi}_*: \textbf{Fac}^{lc}_{(\R_*)^{n}} \longrightarrow  \textbf{Fac}^{lc}_{\R_*} \Big(\textbf{Fac}^{lc}_{\R_*}\big( \dots (\textbf{Fac}^{lc}_{\R_*}) \dots\big)\Big) \\ \cong \textbf{BiMod}\Big(\textbf{BiMod}\big( \dots (\textbf{BiMod}) \dots\big)\Big)\end{multline*}
  is an equivalence.
 \end{corollary}
An interesting consequence arise  if we assume that the restriction to $\R^n\setminus\{0\}$ of $\mathcal{F}\in \textbf{Fac}^{lc}_{(\R_*)^{n}}$ is constant, that is in the essential image of the functor $\textbf{Fac}^{lc}_{\R^{n}}\to \textbf{Fac}^{lc}_{(\R_*)^{n}}$. In that case,  $\mathcal{F}$ belongs to the essential image of $\textbf{Fac}^{lc}_{\R_*^{n}}\to \textbf{Fac}^{lc}_{(\R_*)^{n}}$. Thus Corollary~\ref{C:AEnModasFact} and Corollary~\ref{C:iteratedBiMod} imply
\begin{corollary}There is a natural equivalence $$E_n\textbf{-Mod} \cong E_1\textbf{-Mod}\Big(E_1\textbf{-Mod}\big( \dots (E_1\textbf{-Mod}) \dots\big)\Big)$$ inducing an  equivalence
 $$E_n\textbf{-Mod}_A \cong E_1\textbf{-Mod}_{A}\Big(E_1\textbf{-Mod}_{A}\big( \dots (E_1\textbf{-Mod}_{A}) \dots\big)\Big)$$  between the relevant $\infty$-subcategories\footnote{the $A$-$E_1$-module structure on the right hand side are taken along the various underlying $E_1$-structures of $A$ obtained by projecting on the various component $\R$ of $\R^n$}. 
\end{corollary}

Let us describe now informally $\textbf{Fac}^{lc}_{(\R_*)^{2}}$. A basis of (stratified) disks is given by the convex open subsets. Such a subset is a good disk of index $0$ if it is a neigborhood of the origin, of index $1$ if it is in $\R^2\setminus \{0\}$ and intersects one and only one half open coordinate axis.  It is a good disk  of index $2$ if it lies in the complement of the coordinate axis.  
 We can construct a factorization algebra on $(\R_*)^{2}$ as follows. Let $E_{i}$, $i=1\dots 4$ be four $E_2$-algebras (labelled by the cyclically ordered four quadrants of $\R^2$).
 Let $A_{1,2}$, $A_{2,3}$, $A_{3,4}$ and $A_{4,1}$ be four $E_1$-algebras endowed, for each $i\in \mathbb{Z}/4\mathbb{Z}$, 
 with a compatible $(E_i, E_{i+1})$-bimodule structure\footnote{that is a map of $E_2$-algebras $E_i\otimes (E_{i+1})^{op} \to \RHom_{A_{i, i+1}}^{E_1}(A_{i, i+1},A_{i, i+1})$ } on $A_{i, i+1}$. 
 Also let $M$ be a  $E_2$-module over each $E_i$ in a compatible way. Precisely, this means that $M$ is endowed with a right action\footnote{In particular, it implies that the tuple $(M,A_{i-1,i}, A_{i, i+1}, E_i)$ defines a stratified locally constant factorization algebra on $\tilde{H}$ by Proposition~\ref{P:halfplanepted}} of the $E_1$-algebra $ \big(A_{4,1}\mathop{\otimes}\limits^{\mathbb{L}}_{E_1} A_{1,2} \big)\mathop{\bigotimes}\limits^{\mathbb{L}}_{E_2\otimes (E_4)^{op}}  \big(A_{2,3}\mathop{\otimes}\limits^{\mathbb{L}}_{E_3}A_{3,4} \big)$. 
 
 Similarly to previous examples (in particular example~\ref{Ex:FacOnSquare}), 
 we see that we obtain a locally constant factorization algebra $\mathcal{M}$ on $(\R_*)^{2}$ whose value on an open rectangle  is given by
\begin{equation}\label{eq:FaconQuadrant} \mathcal{M}(R):= \left\{\begin{array}{ll} M & \mbox{if $R$ is a good neighborhood of the origin;} \\ A_{i,i+1} & \mbox{if $R$ is a good neighborhood of index $1$} \\& \mbox{ intersecting the quadrant labelled  $i$ and $i+1$;}\\ E_i & \mbox{if $R$ lies in the interior of the $i^{\mbox{th}}$-quadrant.} \end{array}\right.\end{equation}
The structure maps are given by the various module and algebras structures. 
 As in example~\ref{Ex:FacOnSquare}, we get
 \begin{proposition}
  Any stratified locally constant factorization algebra  on $(\R_*)^2$ is quasi-isomorphic to a factorization algebra associated to a tuple $(M, A_{i,i+1}, E_i, i=1\dots 4)$ as in the rule~\eqref{eq:FaconQuadrant} above.
 \end{proposition}
\end{example}

\begin{example}[Butterfly]
 Let us give one of the most simple singular stratified example. Consider the \lq\lq{}(semi-open) butterfly\rq\rq{} that is the  subspace $B:=\{(x,y)\in \R^2|\, |y|< |x|\} \cup \{(0,0)\}$ of $\R^2$ .
 $B$ has a dimension $0$ strata given by the origin and two open strata $B_{+}$, $B_{-}$ of dimension $2$ given respectively by restricting to those points $(x,y)\in B$ such that $x>0$, resp.  $x<0$.  
 
The restriction to $B_{+}$  of a stratified locally constant factorization algebra  on $B$ is  locally constant factorization over $B_{+}\cong \mathbb{R}^2$ (hence is determined by an $E_2$-algebra). This way, we get the restriction functor 
$\text{res}^*:\textbf{Fac}^{lc}_{B} \to \textbf{Fac}^{lc}_{\R^2} \times \textbf{Fac}^{lc}_{\R^2}$. 

Let $\pi: \R^2\to \R$ be the projection $(x,y)\mapsto x$ on the first coordinate. Then, by Corollary~\ref{L:expFactStrat}, we get a functor  $\pi_*: \textbf{Fac}^{lc}_{B} \to \textbf{Fac}^{lc}_{\R_*}$ where $\R_*$ is the pointed line (example~\ref{ex:pointeddisk}). The restriction of $\pi_*$ to $B_{+}\cup B_{-}$ thus yields a the functor $\textbf{Fac}^{lc}_{\R^2} \times \textbf{Fac}^{lc}_{\R^2}\to  \textbf{Fac}^{lc}_{\R \setminus \{0\}}$.

A proof similar (and slightly easier) to the one of Proposition~\ref{P:Barofenveloping} shows that the induced functor
$$(\pi_*, \text{res}^*): \textbf{Fac}^{lc}_{B} \longrightarrow \textbf{Fac}^{lc}_{\R_*} \times^{h}_{\textbf{Fac}^{lc}_{\R \setminus \{0\}}}  (\textbf{Fac}^{lc}_{\R^2} \times \textbf{Fac}^{lc}_{\R^2}) $$ is an equivalence.
From  Proposition~\ref{P:BimodasFact}, we then deduce
\begin{proposition}
 There is an equivalence $$\textbf{Fac}^{lc}_{B}\cong \textbf{BiMod} \times^{h}_{(E_1\textbf{-Alg} \times E_1\textbf{-Alg})} \,\big(E_2\textbf{-Alg} \times E_2\textbf{-Alg}\big).$$  In other words, locally constant factorization algebras on the (semi-open) butterfly are equivalent to the $\infty$-category  of triples $(A, B, M)$ where $A$, $B$ are $E_2$-algebras and $M$ is a left $A\otimes B^{op}$-module (for the underlying $E_1$-algebras structures of $A$ and $B$). 
\end{proposition}
 Let $\overline{B}=\{(x,y)\in \R^2|\,  |y|\leq |x|\}$ be the closure of the butterfly. It has four additional dimension $1$ strata, given by the boundary of $B_+$ and $B_{-}$. The above argument yields an equivalence
 $$ \textbf{Fac}^{lc}_{\overline{B}} \stackrel{\simeq}\longrightarrow \textbf{Fac}^{lc}_{\R_*} \times^{h}_{\textbf{Fac}^{lc}_{\R \setminus \{0\}}}  \big(\textbf{Fac}^{lc}_{\tilde{H}} \times \textbf{Fac}^{lc}_{\tilde{H}}\big) $$
 where $\tilde{H}$ is the pointed half-plane, see example~\ref{ex:upperhalfplane}. From Propositions~\ref{P:halfplanepted},~\ref{P:BarInterval} and~\ref{P:BimodasFact} we see that a locally constant factorization algebra on the closed butterfly  $\overline{B}$ is equivalent to the data of two $E_2$-algebras $E_{+}$, $E_{-}$, four $E_1$-algebras $A_+$,  $B_+$, $A_{-}$, $B_{-}$, equipped respectively with left or right modules structures over $E_{+}$ or $E_{-}$, and a left $\big(A_{+}\otimes_{E_{+}}^{\mathbb{L}} B_{+}\big) \otimes \big(A_{-}\otimes_{E_{-}}^{\mathbb{L}} B_{-}\big)^{op}$-module $M$.  
\end{example}

\begin{example}[Homotopy calculus]\label{ex:calculus}
The canonical action of $S^1=SO(2)$ on $\R^2$ has the origin for fixed point.
 It follows that 
 $S^1$ acts canonically on $\textbf{Fac}_{\R^2_*}$, the category of factorization algebras over the pointed disk. If $\mathcal{F}\in \textbf{Fac}_{\R^2_*}^{lc}$, its restriction to $\R^2\setminus \{0\}$ determines 
an $E_2$-algebra $A$ with monodromy by Corollary~\ref{C:locallyconstantS1}, Proposition~\ref{L:expFact} and Theorem~\ref{T:Dunn}. If $\mathcal{F}$ is $S^1$-equivariant, then its monodromy is trivial and it follows that the global section $\mathcal{F}(\R^2)$ is an $E_2$-module over $A$. The $S^1$-action yields an $S^1$-action on $\mathcal{F}(\R^2)$ which, algebraically, boils down to an additional differential of homological degree $1$ on   $\mathcal{F}(\R^2)$. We believe that the techniques in this section suitably extended to the case of compact group actions on factorization algebras allow to prove
\begin{claim}\emph{
The category $\big(\textbf{Fac}^{lc}_{\R^2_*}\big)^{S^1}$ of $S^1$-equivariant locally constant factorization algebras on the pointed disk $\R^2_*$ is equivalent to the category of homotopy calculus of Tamarkin-Tsygan~\cite{TT} describing homotopy Gerstenhaber algebras acting on a $\text{BV}$-module.}
\end{claim}
There are nice examples of homotopy calculus arising in algebraic geometry~\cite{BeFa}.
\end{example}

\begin{example} Let $K: S^1\to \R^3$ be a knot, that is a smooth embedding of $S^1$ inside $\R^3$. Then we can consider the stratified manifold $\R^3_{K}$ with a  1 dimensional open strata given by the image of $K$, and another $3$-dimensional open strata given by the knot complement $\R^3\setminus K(S^1)$. Then the category of locally constant factorization algebra on $\R^3_{K}$ is equivalent to the category of quadruples $(A,B, f,\rho)$ where $A$ is an $E_3$-algebras, $B$ is an $E_1$-algebra, $f:B\to B$ is a monodromy and $\rho: A\otimes B\to B$  is an action of $A$ onto $B$ (compatible with all the structures (that is making $B$ an object of $E_2\textbf{-Mod}_{A}(E_1\textbf{-Alg})$. We refer to~\cite{AFT} and~\cite{T-Houches} for details on the invariant of knots produced this way.
\end{example}

\section{Applications of factorization algebras and homology} \label{S:Applications}
\subsection{Enveloping algebras of $E_n$-algebras and Hochschild cohomology of $E_n$-algebras}\label{SS:enveloping}
In this section we describe the universal enveloping algebra of an $E_n$-algebra  in terms of factorization algebras, and apply it to describe $E_n$-Hochschild cohomology.

\smallskip

Given an $E_n$-algebra $A$, we get a factorization algebra on $\R^n$ and thus on its submanifold $S^{n-1}\times \R$ (equipped with the induced framing); see Example~\ref{ex:framedisEn}. We can use the results of \S~\ref{S:stratifiedFact} to study the category of $E_n$-modules over $A$.

In particular, from Corollary~\ref{C:EnModasFact}, and Proposition~\ref{P:envelopFac}, we obtain the following corollary which was first proved by J. Francis~\cite{F}.
\begin{corollary}\label{C:envelopEn} Let $A$ be an $E_n$-algebra.
The functor $N_*: E_n\textbf{-Mod}_A \to E_1\textbf{-RMod}_{\int_{S^{n-1}\times \R} A}$ is an equivalence.

Similarly, the functor $(-N)_*: E_n\textbf{-Mod}_A \to E_1\textbf{-LMod}_{\int_{S^{n-1}\times \R} A}$ is an equivalence.
\end{corollary}
We call \emph{$\int_{S^{n-1}\times \R} A$ the universal ($E_1$-)enveloping algebra of the $E_n$-algebra $A$}. 

A virtue of Corollary~\ref{C:envelopEn} is that it reduces the homological algebra aspects in the category of $E_n$-modules to standard homological algebra in the category of modules over a differential graded algebra  (given by any strict model of $\int_{S^{n-1}\times \R}A$).
\begin{remark}\label{R:EinftyModasFact}
Corollary~\ref{C:EnModasFact} remains true for $n=\infty$ (and follows from the above study, see~\cite{GTZ3}), in which case, since $S^\infty$ is contractible, it boils down to the following result:
\begin{proposition}[\cite{L-HA}, \cite{KM}]\label{P:EinftyasFact} Let $A$ be an $E_\infty$-algebra. There is an natural equivalence of $\infty$-categories $E_\infty\textbf{-Mod}_A \cong E_1\textbf{-RMod}_{A}$ (where in the right hand side, $A$ is identified with its underlying  $E_1$-algebra).
\end{proposition}
\end{remark}

\begin{example}
Let $A$ be a smooth commutative algebra (or the sheaf of functions of a smooth scheme or manifold) viewed as an $E_n$-algebra. Then, by Theorem~\ref{T:HKRhigher}  and Theorem~\ref{T:Fact=CH},   we have 

\begin{proposition}  For $n\geq 2$, there is an equivalence $$ E_n\textbf{-Mod}_A \; \cong \; E_1\textbf{-RMod}_{S^\bullet_{A}(\Omega_A^1[n-1])}.$$
\end{proposition}
The  right hand side is just a category of modules over a graded commutative algebra. If $A=\mathcal{O}_X$, then one thus has an equivalence between
 $E_n\textbf{-Mod}_{\mathcal{O}_X}$ and right graded modules over ${\mathcal{O}_{T_X[1-n]}}$, the functions on the graded tangent space of $X$.
\end{example}

\begin{example}
Let $A$ be an $E_n$-algebra. It  is  canonically a $E_n$-module  over itself; thus by Corollary~\ref{C:envelopEn} it has a structure of right module over $\int_{S^{n-1}\times \R} A$.  
The later  has an easy geometrict description.  Indeed, by the dimension axiom, $A\cong \int_{\R^n} A$. The euclidean norm gives the trivialization $S^{n-1}\times (0,+\infty)$ of the end(s) of $\R^n$ so that, by Lemma~\ref{L:homthMfldisEn}, $\int_{\R^n}A$ has a canonical structure of right module over $\int_{S^{n-1}\times (0,+\infty)} A$.
\end{example}

 Let us consider the example of an $n$-fold loop space. Let $Y$ be an $n$-connective pointed space ($n\geq 0$) and let $A=C_*(\Omega^n(Y))$ be the associated $E_n$-algebra. By non-abelian Poincar\'e duality (Theorem~\ref{T:nonabelianPoincareduality}) we have an equivalence
$$ \int_{S^{n-1}\times \R}A \,\cong\, C_*\big( \Map_c(S^{n-1}\times \R, Y)\big)\,\cong \, C_*\Big(\Omega \Big( Y^{S^{n-1}}\Big)\Big).$$ 
By Corollary~\ref{C:envelopEn}  we get
\begin{corollary}\label{C:Enmodfornfoldloop}
The category of $E_n$-modules over $C_*(\Omega^n(Y))$ is equivalent to the category of right modules over $C_*\Big(\Omega\Big(Y^{S^{n-1}}\Big)\Big)$.
\end{corollary}
The algebra  $C_*\Big(\Omega\Big(Y^{S^{n-1}}\Big)\Big)$ in Corollary~\ref{C:Enmodfornfoldloop} is computed by the cobar construction of the differential graded coalgebra $C_*\big(\Map(S^{n-1},Y)\big)$. If $Y$ is of finite type, $n-1$-connected and the ground ring $k$ is a field of characteristic zero, the latter is  the linear dual of the commutative differential graded algebra $CH_{S^{n-1}}(\Omega_{dR}^*(Y))$ where $\Omega_{dR}^*(Y)$ is  (by Theorem~\ref{T:IteratedIntegralsTop}) the differential graded algebra of Sullivan polynomial forms on $Y$.
 In that case, the structure can be computed using rational homotopy techniques.
\begin{example} Assume $Y=S^{2m+1}$, with $2m\geq n$. Then, $Y$ has a Sullivan model given by the CDGA $S(y)$, with  $|y|=2m+1$. 
By Theorem~\ref{T:HKRhigher} and Theorem~\ref{T:IteratedIntegralsTop}, 
$C_*(\Map(S^{n-1},S^{2m+1}))$ is equivalent to the cofree cocommutative coalgebra $S(u,v)$ with $|u|=-1-2m$ and $|v|=n-2m-2$. 
By Corollary~\ref{C:Enmodfornfoldloop} and Corollary~\ref{C:iteratedloopbar} we find that the category of $E_n$-modules over $C_*(\Omega^n(S^{2m+1})$ is equivalent to the category of right modules over the graded commutative algebra $S(a,b)$ where $|u|=-2m-2$ and $|v|=n-2m-3$. 
\end{example}

There is an natural notion of cohomology  for $E_n$-algebras which generalizes Hochschild cohomology of associative algebras. It plays the same role with respect to deformations of $E_n$-algebras as Hochschild cohomology plays with respect to deformations of associatives algebras.
\begin{definition}
\label{D:EnHoch}
Let $M$ be an $E_n$-module over an $E_{n}$-algebra $A$.
The $E_n$-Hochschild cohomology\footnote{which is an object of $\hkmod$} of $A$ with values in $M$, denoted
by $HH_{E_n}(A, M)$, is by definition (see~\cite{F}) 
$\RHom_{A}^{E_n}(A, M)$ (Definition~\ref{DEnModoverA}).
\end{definition}

\begin{corollary}\label{P:coHH=coTCH} 
Let $A$ be an $E_n$-algebra, and $M$, $N$ be $E_n$-modules over $A$. 
\begin{enumerate}\item There is a canonical equivalence $$\RHom_{A}^{E_n}(M,N)\cong \RHom_{\int_{S^{n-1}\times \R} A}^{\text{left}}(M,N)$$ where the right hand side are  homomorphisms of left modules (Definition~\ref{D:LandRMod}).
\item In particular $HH_{E_n}(A,M)\cong  \RHom_{\int_{S^{n-1}\times \R} A}^{\text{left}}(A,M)$.
\item If $A$ is a CDGA (or $E_\infty$-algebra) and $M$ is a left module over $A$, then  $$HH_{E_n}(A,M)\cong \RHom_{A}^{E_n}(A,M)\cong CH^{S^n}(A,M).$$
\end{enumerate}
\end{corollary}
\begin{proof}The first two points follows from from Corollary~\ref{C:envelopEn}.
The last one follows from Theorem~\ref{T:Fact=CH} which yields equivalences
\begin{multline}\label{eq:equivCHSnHHEn} 
 \RHom_{\int_{S^{n-1}\times \R}A}^{\text{left}}(A, M)
\,\cong \, \RHom_{CH_{S^{n-1}}(A)}^{left}\left({CH_{\R^n}(A)}, M\right) \\
\, \cong \,  \RHom_A^{left}\left( CH_{\R^n}(A)\otimes_{CH_{S^{n-1}}(A)}^{\mathbb{L}} A   , M\right)\\
\, \cong\,  \RHom_A^{left}\left( CH_{S^n}(A), M\right)
\, \cong \, CH^{S^n}(A,M)
\end{multline} when $A$ is an $E_\infty$-algebra.
\qed\end{proof}
\begin{example}
Let $A$ be a smooth commutative algebra and $M$ a symmetric $A$-bimodule. By the HKR Theorem  (see Theorem~\ref{T:HKRhigher}), one has $CH_{S^d}(A)\cong S_{A}^\bullet(\Omega^1_A[d])$  which is a projective $A$-module since $A$ is smooth. Thus,  Corollary~\ref{P:coHH=coTCH} implies $$HH_{E_d}(A,M)\cong S_{A}^\bullet(\text{Der}(A, M)[-d]).$$
\end{example}
 We now explain  the relationship in between  $E_n$-Hochschild cohomology and deformation of $E_n$-algebras.  
Denote $E_n\textbf{-Alg}_{|A}$ the $\infty$-category  of $E_n$-algebras over $A$.  The bifunctor 
 of $E_n$-derivations $\text{Der}:(E_n\textbf{-Alg}_{|A})^{op}\times E_n\textbf{-Mod}_A\to \hkmod$ is defined as  $$ \text{Der}(R,N):= \Map_{E_n\textbf{-Alg}_{|A}}(R, A\oplus N).$$
The (absolute) \emph{cotangent complex} of $A$ (as an $E_n$-algebra) is the value on $A$ of the left adjoint of the split square zero extension functor  $E_n\textbf{-Mod}_A\ni M\mapsto A\oplus M \in E_n\textbf{-Alg}_{|A}$. 
 In other words, there is an natural equivalence $$\Map_{E_n\textbf{-Mod}_A}(L_A,-) \stackrel{\simeq}\longrightarrow \text{Der}(A,-)$$ of functors.
The (absolute) \emph{tangent complex of $A$} (as an $E_n$-algebra) is the dual of $L_A$ (as an $E_n$-module):
$$T_A:= \RHom_{A}^{E_n}(L_A, A) \cong  \RHom_{\int_{S^{n-1}\times \R}A}^{\text{left}}(L_A,A).$$
The tangent complex has a structure of an (homotopy) Lie algebra which controls the deformation of $A$ as an $E_n$-algebra (that is, its deformations  are precisely given by the solutions of Maurer-Cartan equations in $T_A$). Indeed, 
 Francis~\cite{F} has proved the following beautiful result which solve (and generalize) a conjecture of Kontsevich~\cite{K-operad}.
 His proof relies heavily on factorization homology and in particular on the excision property to identify $E_n\textbf{-Mod}_A$ with the $E_{n-1}$-Hochschild homology of $ E_1\textbf{-LMod}_A$ which is a $E_{n-1}$-monoidal category.
 \begin{theorem}[\cite{F}]\label{T:tangentcomplexEn}Let $A$ be an $E_n$-algebra and $T_A$ be its tangent complex. There is a fiber sequence of non-unital $E_{n+1}$-algebras
$$ A[-1]\longrightarrow T_A[-n] \longrightarrow HH_{E_n}(A)$$ inducing a fiber sequence of (homotopy) Lie algebras
$$ A[n-1]\longrightarrow T_A \longrightarrow HH_{E_n}(A)[n]$$ after suspension.
\end{theorem}

\subsection{Centralizers and  (higher) Deligne conjectures}\label{S:maincentralizers}
We will here sketch applications of factorization algebras to study \emph{centralizers} and solve the (relative and higher) Deligne conjecture.

The following definition is due to Lurie~\cite{L-HA} 
(and generalize the notion of center of a category due to Drinfeld). 
\begin{definition}\label{D:center}
The (derived) centralizer of an $E_n$-algebra map $f:A\to B$ is the \emph{universal} $E_n$-algebra $\mathfrak{z}_n(f)$ equipped  with a map of $E_n$-algebras $e_{\mathfrak{z}_n(f)}: A\otimes \mathfrak{z}_{n}(f)\to B$  making the following diagram
\begin{equation}
\label{eq:Defcenter} \xymatrix{ & A\otimes \mathfrak{z}_{n}(f) \ar@{.>}[rd]^{e_{\mathfrak{z}_{n}(f)}} & \\
A \ar[ru]^{id\otimes 1_{\mathfrak{z}_{n}(f)}} \ar[rr]^{f} && B  }
\end{equation}
commutative in $E_n$-Alg. The \emph{(derived) center} of an $E_n$-algebra $A$ is the centralizer $\mathfrak{z}_{n}(A):=\mathfrak{z}_{n}(A\stackrel{id}\to A)$ of the identity map.
\end{definition}
The existence of the derived centralizer $\mathfrak{z}_{n}(f)$ of   an $E_n$-algebra map $f:A\to B$ is a \emph{non-trivial} result of Lurie~\cite{L-HA}. 
The universal property means that if $C\stackrel{\varphi}\to B$ is an $E_n$-algebra map fitting inside a commutative diagram   
\begin{equation}\label{eq:Csatisfiescenter}
\xymatrix{ & A\otimes C \ar[rd]^{\varphi} & \\
A \ar[ru]^{id\otimes 1_{C}} \ar[rr]^{f} && B  }
\end{equation}
then there is a unique\footnote{up to a contractible space of choices} factorization $\varphi: A\otimes C\stackrel{id\otimes \kappa}\dashrightarrow A\otimes \mathfrak{z}_{n}(f) \stackrel{e_{\mathfrak{z}_{n}(f)}}\longrightarrow B$ of $\varphi$ by an $E_n$-algebra map $\kappa:C\to \mathfrak{z}_{n}(f)$. In  particular,  the commutative diagram 
$$ \xymatrix{ && A\otimes  \mathfrak{z}_{n}(f) \otimes \mathfrak{z}_{n}(g) \ar[rd]^{e_{\mathfrak{z}_{n}(f)}\otimes id} && \\
& A\otimes \mathfrak{z}_{n}(f)\ar[ru]^{id\otimes 1_{\mathfrak{z}_{n}(g)}} \ar[rd]^{e_{\mathfrak{z}_{n}(f)}} &   & B\otimes \mathfrak{z}_{n}(g) \ar[rd]^{e_{\mathfrak{z}_{n}(g)}} & \\
A \ar[ru]^{id\otimes 1_{\mathfrak{z}_{n}(f)}} \ar[rr]^{f} && B \ar[ru]^{id\otimes 1_{\mathfrak{z}_{n}(g)}} \ar[rr]^{g} && C  } $$
induces natural
maps of $E_n$-algebras 
\begin{equation}
\label{eq:z(circ)} \mathfrak{z}_{n}(\circ): \mathfrak{z}_{n}(f)\otimes \mathfrak{z}_{n}(g) \longrightarrow \mathfrak{z}_{n}(g\circ f) .
\end{equation}

\begin{example} Let $M$ be a  monoid (for instance a group). That is an $E_1$-algebra in the (discrete) category of sets with cartesian product for monoidal structure. Then $\mathfrak{z}_1(M)=Z(M)$ is the usual  center $\{m\in M, \forall n\in M, n\cdot m =m\cdot n\}$ of $M$. 
Let $f: H\hookrightarrow G$ be the inclusion of a subgroup in a group $G$. Then $\mathfrak{z}_1(f)$ is the usual centralizer of the subgroup $H$ in $G$. This examples explain the name centralizer.

Similarly, let $k\textit{-Mod}$ be the (discrete) category of $k$-vector spaces over a field $k$. Then an $E_1$-algebra in $k\textit{-Mod}$ is an associative algebra and $\mathfrak{z}_1(A)=Z(A)$ is its usual (non-derived) center. 
However, if one sees $A$ as an $E_1$-algebra in the $\infty$-category $\hkmod$ of chain complexes, then $\mathfrak{z}_1(A)\cong \RHom_{A\otimes A^{op}}(A,A)$ is (computed by) the usual Hochschild cochain complex (\cite{Lo-book}) in which  the usual center embeds naturally, but  is  different from it even when $A$ is commutative. 
\end{example}

 Let $A\stackrel{f}\to B$ be an $E_n$-algebra map. 
Then $B$ \emph{ inherits a} canonical \emph{structure of $E_n$-$A$-module, denoted $B_f$}, which is the pullback along $f$ of the tautological $E_n$-$B$-module structure on $B$.

The relative Deligne conjecture claims that the centralizer is computed by  $E_n$-Hochschild cohomology. 
\begin{theorem}[Relative Deligne conjecture]\label{T:RelativeDeligne} Let $A\stackrel{f}\to B$ and $B\stackrel{g}\to C$ be  maps of $E_n$-algebras.
\begin{enumerate}
\item There is an $E_n$-algebra structure on $HH_{E_n}(A,B_{f})\cong \RHom_{A}^{E_n}(A,B_{f})$ which makes $HH_{E_n}(A,B_{f})$ the centralizer $\mathfrak{z}_n(f)$ of $f$ (in particular, $\mathfrak{z}(f)$ exists);
\item the diagram
$$\xymatrix{    \mathfrak{z}_{n}(f) \otimes \mathfrak{z}_{n}(g) \ar[rr]^{\mathfrak{z_n}(\circ)} &&  \mathfrak{z}_{n}(g\circ f)    \\
 \RHom_{A}^{E_n}(A,B_{f})\otimes  \RHom_{B}^{E_n}(B,C_{g}) \ar[u]^{\cong} \ar[rr]^{\circ} &&  \RHom_{A}^{E_n}(A,C_{g\circ f})  \ar[u]_{\cong}} $$ is commutative in $E_{n}\textbf{-Alg}$ (where the lower arrow is induced by composition of maps  in $E_n\textbf{-Mod}$).
\item If $A\stackrel{f}\to B$ is a map of $E_\infty$-algebras, then there is an equivalence of $E_n$-algebras $\mathfrak{z}_n\cong CH^{S^n}(A, B_f)$ where $CH^{S^n}(A,B)$ is endowed with the structure given by Proposition~\ref{T:EdHoch}.
\end{enumerate}
\end{theorem}
\begin{sketchofproof} This result is proved in~\cite{GTZ3} (the techniques of~\cite{F} shall also give an independent proof) and we only briefly sketch the main point of the argument. We first define an $E_n$-algebra structure on $\RHom_{A}^{E_n}(A,B_{f})$. 
By Corollary~\ref{C:EnModasFact} and Theorem~\ref{P:En=Fact}, we can assume that   $A$, $B$ are  factorization algebras  on $\R^n$ and that a morphism of modules is a map of the underlying (stratified) factorization algebras. 
We are left to prove that there is a locally constant factorization algebra structure on $\R^n$ whose global sections are $HH_{E_n}(A,B_{f})\cong \RHom_{A}^{E_n}(A,B_{f})$. 
It is enough to define it on the basis of convex open subsets $\mathcal{CV}$ of $\R^n$ (by Proposition~\ref{P:extensionfrombasis}).
 To any convex open set $U$ (with central point $x_U$), we associate the chain complex $\RHom_{A_{|U}}^{\textbf{Fac}_U}(A_{|U}, {B_f}_{|U})$ of factorization algebras morphisms from the restrictions $A_{|U}$ to $B_{|U}$ which, on the restriction to $U\setminus \{x_U\}$ are given by $f$. Note that since $U$ is convex, there is a quasi-isomorphism  
\begin{equation}\label{eq:definerhocenter} \RHom_{A_{|U}}^{\textbf{Fac}_U}(A_{|U}, {B_f}_{|U})\cong  \RHom_{A}^{E_n}(A,B_{f})= \RHom_{A_{|\R^n}}^{\textbf{Fac}_{\R^n}}(A_{|\R^n}, {B_f}_{|\R^n}).\end{equation} Now, given convex sets $U_1,\dots, U_r$ which are pairwise disjoint inside a bigger convex $V$, we define a map  
\begin{equation*}\rho_{U_1,\dots, U_r, V}:\bigotimes_{i=1\dots r}\RHom_{A_{|U_i}}^{\textbf{Fac}_{U_i}}(A_{|U_i}, {B_f}_{|U_i}) \longrightarrow \RHom_{A_{|V}}^{\textbf{Fac}_V}(A_{|V}, {B_f}_{|V})\end{equation*} as follows. To define $\rho_{U_1,\dots, U_r, V}(g_1,\dots, g_r)$, we need to define a factorization algebra map on $V$ and for this, it is enough to do it on the open set consisting of convex subsets of $V$ which either are included in one of the $U_i$ and contains $x_{U_i}$ or else does not contains any $x_{U_i}$. To each open set $x_{U_i}\in D_i\subset U_i$ of the first kind, we define $$\rho_{U_1,\dots, U_r, V}(g_1,\dots, g_r)(D_i): A(D_i) \stackrel{g_i}\longrightarrow B_f(D_i)$$ to be given by $g_i$, while for any open set $D\subset V\setminus\{x_{U_1},\dots, x_{|U_r}\}$, 
 we define $$\rho_{U_1,\dots, U_r, V}(g_1,\dots, g_r)(D): A(D) \stackrel{f}\longrightarrow B_f(D)$$ 
to be given by $f$. The conditions that $g_1,\dots, g_r$ are maps of $E_n$-modules over $A(U_i)$ ensures that $\rho_{U_1,\dots, U_r, V}$ define the structure maps of a factorization algebra which is further locally constant since  $\rho_{U,\R^n}$ is the equivalence~\eqref{eq:definerhocenter}. 

 The construction  is roughly described in Figure~\ref{fig:algebramap}.
\begin{figure}[b]
\begin{minipage}[c]{0.6\textwidth} \includegraphics[scale=0.45]{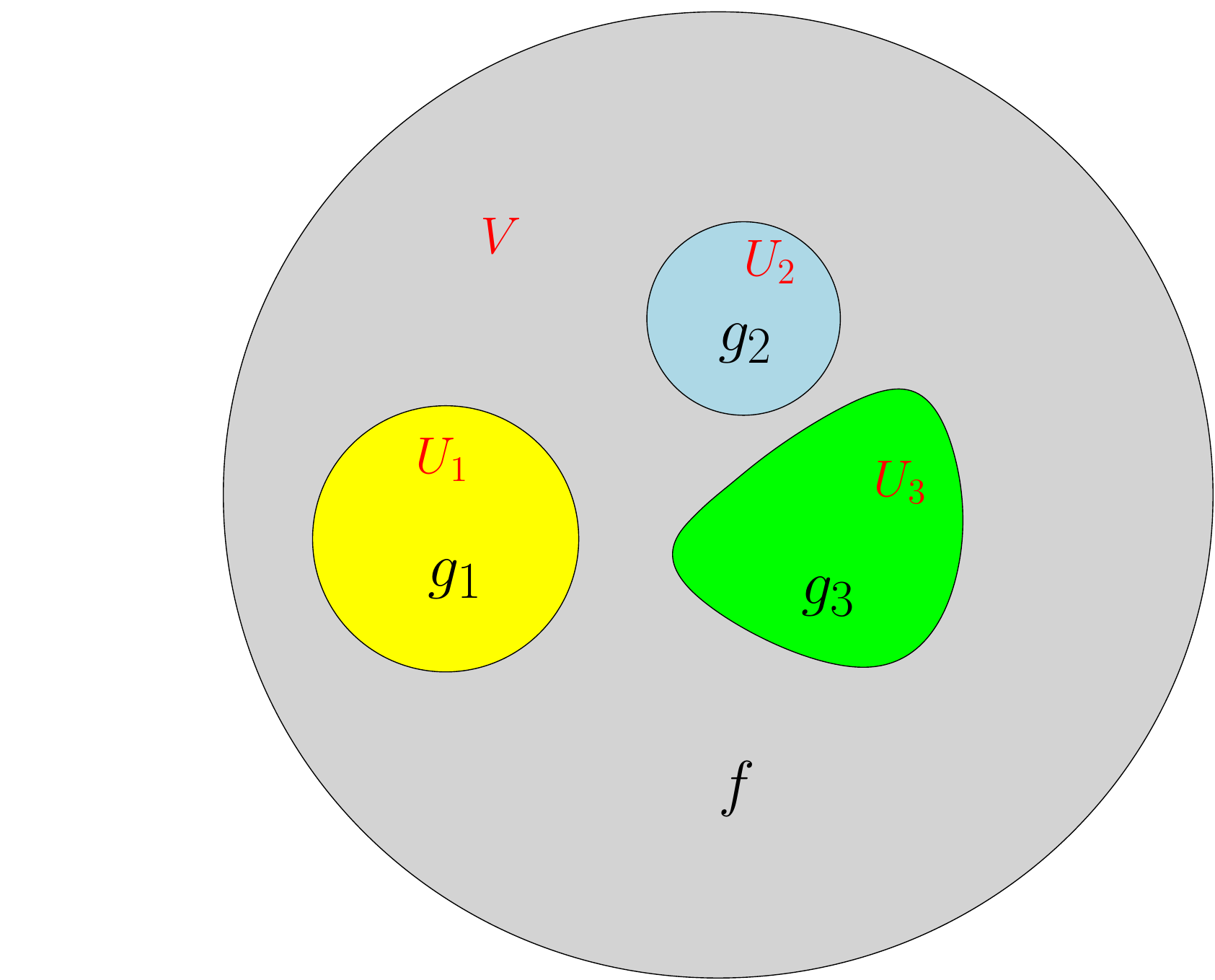}\end{minipage} \hfill \begin{minipage}[c]{0.37\textwidth}
\caption{The factorization algebra map $A_{|V} \to B_{|V}$ obtained by applying the relevant maps of  modules $g_1, g_2, g_3$ (viewed as maps of factorizations algebras) and the $E_n$-algebra map $f:A\to B$ on the respective regions}\label{fig:algebramap}\end{minipage}
\end{figure}
One can check that  the natural evaluation map
 $\mathop{eval}: A\otimes\RHom_{A}^{E_n}(A,B_{f})\to B$ is a map of $E_n$-algebras.  

 Now let  $C$ be an $E_n$-algebra fitting in the commutative diagram~\eqref{eq:Csatisfiescenter}, which we again identify with a factorization algebra map (over $\R^n$). 
By  adjunction (in $\hkmod$), the map $\varphi: A\otimes C\to B$ has a (derived) adjoint 
$\theta_{\varphi}: C\to RHom(A,B)$. Since $\varphi$ is a map of factorization algebras 
and diagram~\eqref{eq:Defcenter}  is commutative, one check that $\theta_{\varphi}$ factors through a map 
\begin{equation}\label{eq:factorthetaphi}
 \tilde{\theta_{\varphi}} : C \longrightarrow RHom_{A}^{{E}_n}(A,B)\cong CH^{S^n}(A,B).
\end{equation}
which can be proved to be a map of factorization algebras. Further, by definition of $\theta_{\varphi}$,  the identity 
$$\mathop{eval} \circ \big(id_A\otimes \theta_{\varphi} \big)\;=\; \varphi$$ holds.
The uniqueness of the map $\tilde{\theta_\varphi}$ follows  from the fact that the composition 
\begin{equation*}
 RHom_{A}^{{E}_n}(A,B) 
\stackrel{1_{RHom_A^{{E}_n}(A,A)}\otimes id}\longrightarrow 
 RHom_A^{{E}_n}\Big(A,A\otimes RHom_{A}^{{E}_n}(A,B)\Big)
\stackrel{ev_*}\longrightarrow RHom_{A}^{{E}_n}(A,B)
\end{equation*}
is the identity map. 

Finally the equivalence between $\mathfrak{z}_n(f)$ and $CH^{S^{n}}(A,B_f)$ in the commutative case follows from the string of equivalences~\eqref{eq:equivCHSnHHEn} which can be checked to be an equivalence of $E_n$-algebras using diagram~\eqref{eq:diagEn=Fact2} connecting algebras over the operads of little rectangles of dimension $n$ and factorization algebras.
\qed\end{sketchofproof}

\begin{example}\label{R:derivedcompositionasusualcomposition} 
Let $A\stackrel{f}\to B$ be a map of CDGAs.  then by Proposition~\ref{P:coHH=coTCH}  we have an equivalence
$HH_{E_n}(A,B) \cong \RHom_{CH_{\partial I^n}(A)}^{left}(CH_{I^n}(A), CH_{I^n}(B))$ where $I=[0,1]$ and $\partial I^n\cong S^{n-1}$ is the boundary of the unit cube $I^n$. We have a simplicial model $I_\bullet^n$ of $I^n$ where $I_\bullet$ is the standard model of  Example~\ref{ex:ptinterval}; its  boundary $\partial I_\bullet^n$ is a simplicial model for $S^{n-1}$. 
Then Theorem~\ref{T:RelativeDeligne} identifies the derived composition~\eqref{eq:z(circ)} as the usual composition (of left dg-modules)
\begin{multline*} Hom_{CH_{\partial I_\bullet^n}(A)}^{left}\left(CH_{ I^n_\bullet}(A), CH_{ I^n_\bullet}(B)\right)\otimes Hom_{CH_{\partial I_\bullet^n}(B)}^{left}\left(CH_{ I_\bullet^n}(B), CH_{ I_\bullet^n}(C)\right)\\
 \stackrel{\circ}\longrightarrow  Hom_{CH_{\partial I_\bullet^n}(A)}^{left}\left(CH_{ I_\bullet^n}(A), 
CH_{ I_\bullet^n}(C)\right). \end{multline*}
\end{example}

\smallskip

The relative Deligne conjecture implies easily the standard one and also the Swiss cheese conjecture. 
Indeed, 
 Theorem~\ref{T:RelativeDeligne} implies that the multiplication
$\mathfrak{z}_n(A)\otimes \mathfrak{z}_n( A) \stackrel{\mathfrak{z}_n(\circ)}\longrightarrow \mathfrak{z}_n(A)$
makes $\mathfrak{z}_n(A)$ into an $E_1$-algebra in the $\infty$-category $E_n\textbf{-Alg}$.
The $\infty$-category version of Dunn Theorem (Theorem~\ref{T:Dunn}) gives an equivalence
 $E_1\textbf{-Alg}\big(E_n\textbf{-Alg}\big)\cong E_{n+1}\textbf{-Alg}$. This yields the following solution to the higher Deligne conjecture, see~\cite{GTZ3, L-HA}.
\begin{corollary}[Higher Deligne Conjecture]\label{T:Deligne}  
 Let $A$ be an $E_n$-algebra. The $E_n$-Hochschild cohomology $HH_{E_n}(A,A)$ has an $E_{n+1}$-algebra structure lifting the Yoneda product  which further lifts the $E_n$-algebra structure of the centralizer $\mathfrak{z}(A\stackrel{id}\to A)$.
\end{corollary} 
In particular, if $A$ is commutative,  there is a natural  $E_{n+1}$-algebra structure on $CH^{S^n}(A,A)\cong HH_{E_n}(A,A)$ whose underlying  $E_n$-algebra structure is the one given by Theorem~\ref{T:EdHoch}. Hence the underlying
$E_1$-algebra structure is given by the cup-product (Example~\ref{ex:cupsphere}). 

\begin{example} Let $\mathcal{C}$ be a monoidal (ordinary) category. Then the center $\mathfrak{z}_1(\mathcal{C})$ is in $E_2\textbf{-Alg}(\textit{Cat})$, that is a braided monoidal category. One can prove that  $\mathfrak{z}_1(\mathcal{C})$ is actually the Drinfeld center of $\mathcal{C}$, see~\cite{L-HA}.
\end{example}
\begin{remark}
Presumably, the $E_{n+1}$-structure on $HH_{E_n}(A)$ given by Corollary~\ref{T:Deligne} shall be closely related to the one given by Theorem~\ref{T:tangentcomplexEn}.
\end{remark}
\begin{example}\label{ex:z(1A)}
Let $1_A: k\to A$ be the unit of an $E_n$-algebra $A$. Then  $\mathfrak{z}_n(1_A)\cong A$ as an $E_n$-algebra.
The derived composition~\eqref{eq:z(circ)} yields canonical map\footnote{which is equivalent to $e_{\mathfrak{z}(1_A)}$} of $E_n$-algebras
$\mathfrak{z}_n(A)\otimes \mathfrak{z}_n(1_A) \longrightarrow \mathfrak{z}_n(1_A)$ which exhibits $A\cong \mathfrak{z}_n(1_A)$ as a right $E_1$-module over $\mathfrak{z}_n(A)\cong HH_{E_n}(A,A)$ in the category of $E_n$-algebras (by Theorem~\ref{T:Deligne}).
 Hence, in view of Example~\ref{ex:upperhalfplane}, we obtain, as an immediate corollary (see~\cite{Ca-HDR}), a proof of the Swiss-Cheese version of Deligne conjecture\footnote{originally proved, for a slightly different variant of the swiss cheese operad, in~\cite{Do-SC, Vo-SwissCheese}}:
\begin{corollary}[Deligne conjecture with action]\label{C:ActionDeligne} Let $A$ be an $E_n$-algebra. Then the pair $(HH_{E_n}(A,A), A)$ is canonically an object of $E_1\textbf{-RMod}(E_n\textbf{-Alg})$, that is $A$ has an natural action of the $E_{n+1}$-algebra $HH_{E_n}(A,A)$.
\end{corollary}
\end{example}

\begin{example}[(Higher homotopy) calculus again~\cite{Ca-HDR}]
 Let $A$ be an $E_n$-algebra.  Assume $n=0,1,3,7$, so that  $A$ defines canonically a   (locally constant) factorization algebra $A_{S^n}$ on the framed manifold $S^n$ (see   Example~\ref{ex:framedisEn}).
 Similarly  the $E_{n+1}$-algebra given by the higher Hochschild cohomology $HH_{E_n}(A,A)$ defines canonically a   (locally constant) factorization algebra on the manifold $S^n \times (0,\infty)$ endowed with the product framing. 
 
 The Deligne conjecture with action  (Corollary~\ref{C:ActionDeligne}) shows that $A$ is also  a left module over $HH_{E_n}(A,A)$.   
Thus, according to Proposition~\ref{P:envelopFac}, the pair    $(HH_{E_n}(A,A),A)$ yields a stratified locally constant factorization algebra $\mathcal{H}$ on  $D^{n+1}\setminus \{0\}$, the closed disk in which we have removed the origin.  

By Theorem~\ref{T:Theorem6GTZ2}, we have that $(A_{S^n})(\partial D^{n+1})\cong \int_{S^n} A$. Collapsing the boundary $\partial D^{n+1}$ to a point yields an adequatly stratified  map $\tau:D^{n+1}\setminus \{0\}\to \R^{n+1}_*$ so that $\mathcal{A}:=\tau_*(\mathcal{H})$ is stratified locally constant on $ \R^{n+1}_*$. 
It is further
$SO(n+1)$-equivariant. Together with example~\ref{ex:calculus}, 
 the above paragraph thus sketches a proof of the following fact:
 \begin{corollary}
  Let $A$ be an $E_n$-algebra and $n=0,1,3,7$. Then $A$ gives rise to an  $SO(n+1)$-equivariant  stratified locally constant factorization algebra $\mathcal{A}$ on the pointed disk $\R^{n+1}_*$ such that $\mathcal{A}(\R^{n+1})\cong \int_{S^n} A$ and for any sub-disk $D\subset \R^{n+1}\setminus\{0\}$, there is an natural (with respect to disk inclusions) equivalence of $E_{n+1}$-algebras $\mathcal{A}(D)\stackrel{\simeq}\to HH_{E_n}(A,A)$.
  
  In particular, for $n=1$,  we recover that the  pair $(HH_{E_1}(A,A), CH_{S^1}(A))$, given by Hochschild cohomology and Hochschild homology of an associative or $A_\infty$-algebra $A$, defines an homotopy calculus (see~\cite{KS, TT} or Example~\ref{ex:calculus}).
 \end{corollary}
 This corollary is proved in details (using indeed factorization homology techniques) in the interesting paper~\cite{Hor} along with many other examples in which $D^n$ is replaced by other framed manifold.
\end{example}

\subsection{Higher string topology}
The formalism of factorization homology for CDGAs and higher Deligne conjecture was applied in~\cite{G-HHnote, GTZ3} to higher string topology which we now explain briefly. We also refer to the work~\cite{Hu, HKV} for a related approach.

Let $M$ be a closed oriented manifold,  equipped with a Riemannian metric. String topology is about the algebraic structure of the chains and homology of the free loop space $LM:=\Map(S^1, M)$ and its higher free sphere spaces $M^{S^n}:= \Map(S^n, M)$. These spaces  have Fr\'echet manifold structures and there is a submersion $ev: M^{S^n}\to M$ given by evaluating at a chosen base point in $S^n$. 
The canonical embedding $  Map(S^n\vee S^n,M)\stackrel{\rho_{in}} \longrightarrow
   Map(S^n,M)\times Map(S^n,M) $ has an oriented normal bundle\footnote{which can be obtained as a pullback along $ev$ of the normal bundle of the diagonal $M\to M\times M$}.
It follows that there is a Gysin map 
$(\rho_{in})_!: H_\ast\Big(M^{S^n\coprod S^n} \Big) \longrightarrow H_{\ast -\dim(M)}\Big( M^{S^n\vee S^n}\Big) $.
The pinching map $\delta_{S^n}:S^n\to S^n\vee S^n$ (obtained by collapsing the equator to a point) yields 
the map $\delta_{S^n}^*: \Map(S^n\vee S^n, M) \longrightarrow M^{S^n}$.
The \emph{sphere product} is the composition 
\begin{multline}\label{D:sphereproduct}\star_{S^n}: H_{\ast+\dim(M)}\Big(M^{S^n}\Big)^{\otimes 2}\to H_{\ast+2\dim(M)}\Big(M^{S^n\coprod S^n}\Big)\\
\stackrel{(\rho_{in})_!}\longrightarrow H_{\ast+\dim(M)}\Big(M^{S^n\vee S^n}\Big)\stackrel{(\delta_{S^n}^*)_*}\longrightarrow H_{\ast+\dim(M)}\Big(M^{S^n}\Big). \end{multline} 
The circle action on itself induces an action $\gamma: LM\times S^1\to LM$.
\begin{theorem} \label{T:overviewstring}\begin{enumerate}
\item \textbf{(Chas-Sullivan~\cite{CS})} Let $\Delta: H_\ast(LM)\stackrel{\times [S^1]}\longrightarrow H_{\ast+1}(LM\times S^1) \stackrel{\gamma_*}\longrightarrow H_{\ast+1}(LM)$ be induced by the $S^1$-action. Then $(H_{\ast+\dim(M)}(LM), \star_{S^1}, \Delta)$ is a Batalin Vilkoviski-algebra and in particular a $P_2$-algebra.
\item \textbf{(Sullivan-Voronov~\cite{CV})}  $(H_{\ast+\dim(M)}(LM), \star_{S^n})$ is a graded commutative algebra.
\item \textbf{(Costello~\cite{Co-CY}, Lurie~\cite{L-TFT})} If $M$ is simply connected, the chains $C_{\ast}(LM)[\dim(M)]$ have a structure of $E_2$-algebra (and actually of $\Disk_2^{or}$-algebra) which, in characteristic $0$, induces Chas-Sullivan $P_2$-structure in homology (by~\cite{FT}). 
\end{enumerate}
\end{theorem}
In~\cite{CV}, Sullivan-Voronov also sketched a proof of the fact that 
 $H_{\bullet+\dim(M)}(M^{S^n})$ is an algebra 
over the homology $H_{\ast}(\Disk_{n+1}^{or})$ of the operad $\Disk_{n+1}^{or}$ and in particular has a $P_{n+1}$-algebra structure (see Example~\ref{Ex:PnAlg}). Their work and the aforementioned work for $n=1$ (Theorem~\ref{T:overviewstring}.3) rise the following
\begin{question} 
Is there a natural $E_{n+1}$-algebra (or even $\Disk_{n+1}^{or}$-algebra)  on the \emph{chains} $C_\ast\Big(M^{S^n}\Big)[\dim(M)]$ which induces Sullivan-Voronov product in homology ?
\end{question} 
Using the solution to the higher Deligne conjecture and the relationship between factorization homology and mapping spaces, one obtains a positive solution to the above question for sufficiently connected manifolds. 
\begin{theorem}[\cite{GTZ3}]\label{T:BraneChain} Let $M$ be an $n$-connected\footnote{it is actually sufficient to assume that $M$ is nilpotent, connected and has finite homotopy groups $\pi_i(M,m_0)$ for $1\leq i \leq n$} Poincar\'e duality space.
The shifted chain complex $C_{\bullet+\dim(X)} (X^{S^n})$ has a
  natural\footnote{with respect to maps of 
Poincar\'e duality spaces}  $E_{n+1}$-algebra structure which induces the sphere product $\star_{S^n}$ (given by the map~\eqref{D:sphereproduct}) of Sullivan-Voronov~\cite{CV}
$$H_{p}\big(X^{S^n}\big)\otimes H_{q}\big(X^{S^n}\big) \to H_{p+q-\dim(X)}\big(X^{S^n}\big)$$ in homology when $X$ is an oriented closed manifold.
\end{theorem}
\begin{sketchofproof} We only gives the key steps of the proof following~\cite{GTZ3, G-HHnote}.
\begin{itemize}
\item Let $[M]$ be the fundamental class of $M$.   The Poincar\'e duality map $\chi_M: x\mapsto x\cap [M]$
is a map of left modules and thus, by Proposition~\ref{P:EinftyasFact} (and since, by assumption, the  biduality homomorphism $C_\ast(X)\to (C^\ast(X))^{\vee}$ is a quasi-isomorphism), 
$$\chi_M: C^{\ast}(X) \to C_{\ast}(X)[\dim(X)] \; \cong\; 
\big(C^{\ast}(X)\big)^{\vee}[\dim(X)]$$ has an natural lift as an $E_\infty$-module. And thus as an $E_n$-module as well.
\item From the previous point we deduce that there is an equivalence 
\begin{multline}\label{eq:equivBrane}HH_{E_n}(C^\ast(X),C^\ast(X)) \cong 
\RHom_{C^\ast(X)}^{E_n}\big(C^\ast(X), C^\ast(X)\big)\\ \stackrel{(\chi_M)*}\longrightarrow \RHom_{C^\ast(X)}^{E_n}\big(C^\ast(X), \big(C^{\ast}(X)\big)^{\vee}\big)[\dim(M)]\\
\cong\RHom_{\int_{S^{n-1}\times \R}C^\ast(X)}^{left}\big(C^\ast(X), \big(C^{\ast}(X)\big)^{\vee}\big)[\dim(M)]\\
\cong \RHom_{C^\ast(X)}^{left}\big(C^\ast(X)\otimes_{\int_{S^{n-1}\times \R}C^\ast(X)}^{\mathbb{L}} C^\ast(X) , k \big)[\dim(M)]\\
\cong \RHom_{C^\ast(X)}^{left}\big(CH_{S^n}(C^{\ast}(X)), k\big)[\dim(M)].
\end{multline}
where the last equivalence follows from  Theorem~\ref{T:Fact=CH}.
\item By Theorem~\ref{T:IteratedIntegralsTop} relating Factorization homology of singular cochains with mapping spaces,  the above equivalence~\eqref{eq:equivBrane} induces an equivalence 
\begin{equation}\label{eq:equivsingularHoch}
C_\ast\Big( M^{S^{n}}\Big)[\dim(M)]\, \stackrel{\simeq}\longrightarrow \, HH_{E_n}(C^\ast(X), C^\ast(X)). 
\end{equation}
\item Now, one  uses the higher Deligne conjecture (Corollary~\ref{T:Deligne}) and the latter equivalence~\eqref{eq:equivsingularHoch} to get an $E_{n+1}$-algebra structure on $C_\ast\Big( M^{S^{n}}\Big)[\dim(M)]$. The explicit definition of the cup-product given by Proposition~\ref{T:EdHoch} (and Theorem~\ref{T:RelativeDeligne}) allows to describe explicitly  the $E_1$-algebra structure at the level of the cochains $C^\ast\Big( M^{S^n}\Big)$ through the equivalence~\eqref{eq:equivsingularHoch}, which, in turn allows to check it induces the product $\star_{S^n}$.
\end{itemize}
\qed\end{sketchofproof}

\begin{example} Let $M=G$ be a Lie group and $A=S(V)\stackrel{\simeq}\to \Omega_{dR}(G)$ be its minimal model.
 The graded space $V$ is  concentrated in positive odd degrees. If $G$ is $n$-connected,
by the generalized HKR Theorem~\ref{T:HKRhigher}, there is an equivalence 
\begin{equation}\label{eq:HKRS(V)}
S(V\oplus V^*[-n])\; \cong  \; CH^{S^n}(A,A )\;\cong \; C_{\ast}\Big(G^{S^n}\Big)[\dim(G)]
\end{equation}
in $\hkmod$. The \emph{higher formality conjecture} shows that the equivalence~\eqref{eq:HKRS(V)} is an equivalence of $E_{n+1}$-algebras. 
Here the left hand side is viewed as an $E_{n+1}$-algebra obtained by the formality of the $E_{n+1}$-operad from the $P_{n+1}$-structure on $S(V\oplus V^*[-n])$ whose multiplicative structure is the one given by the symmetric algebra and the bracket is given by the pairing between $V$ and $V^*$.
\end{example}

\subsection{Iterated loop spaces and Bar constructions}\label{S:BarConstruction}
In this section we apply the formalism of factorization homology to describe iterated Bar constructions \emph{equipped with} their algebraic structure and relate them to the $E_n$-algebra structure of $n^{\text{th}}$-iterated loop spaces. We follow the approach of~\cite{F, AF, GTZ3}.

Bar constructions have been introduced in topology as a model for the coalgebra structure of the cochains on $\Omega(X)$, the based loop space of a pointed space $X$. Similarly the cobar construction of coaugmented coalgebra has been studied originally as a model for the ($E_1$-)algebra structure of the chains on $\Omega(X)$. The Bar and coBar constructions also induce equivalences between algebras and coalgebras under sufficient nilpotence and degree assumptions~\cite{StasheffMoore, FHT, FG}.

\smallskip
Let $(A,d)$ be a differential graded unital associative algebra which is \emph{augmented}, that is equipped with an algebra homomorphism  $\varepsilon: A\to k$. 
\begin{definition}\label{D:standardBar}The  standard \emph{Bar functor} of the augmented algebra $(A,d,\varepsilon)$ is
$$Bar(A):= \; k\mathop{\otimes}^{\mathbb{L}}_{A} k .$$
If $A$ is flat over $k$, it is computed by the standard chain complex 
 $Bar^{std}(A)=\bigoplus_{n\geq 0} {\overline{A}}^{\otimes n} $
(where $\overline{A}=ker(A\stackrel{\epsilon}\to k)$  is the augmentation ideal of $A$) endowed with the differential
\begin{eqnarray*}
 b(a_1\otimes \cdots a_n)     & = &   \sum_{i=1}^n \pm a_1\otimes \cdots \otimes d(a_i)\otimes \cdots 
\otimes a_n\\ 
& & +  \sum_{i=1}^{n-1} \pm a_1\otimes \cdots \otimes (a_i\cdot a_{i+1})\otimes \cdots \otimes a_n
\end{eqnarray*}
see~\cite{FHT, Fre2} for details (and signs).
\end{definition} 
The Bar construction has a standard coalgebra structure. 
It is well-known that if 
$A$ is a commutative differential graded algebra, then the shuffle product makes  the Bar construction $Bar^{std}(A)$ a CDGA and  a bialgebra as well.
 It was proved by Fresse~\cite{Fre2} that Bar constructions of $E_\infty$-algebras have an (augmented) $E_\infty$-structure,  allowing to consider iterated Bar constructions and further that there is a canonical $n^{\text{th}}$-iterated Bar construction functor for augmented $E_n$-algebras as well~\cite{FresseBarEn}. 
            
\smallskip

Let us now describe the factorization homology/algebra point of view on Bar constructions.

An \emph{augmented $E_n$-algebra} is an  $E_n$-algebra $A$ equipped with an $E_n$-algebra map
$\varepsilon: A\to k$, called the augmentation.   We denote \emph{$E_n\textbf{-Alg}^{aug}$ the $\infty$-category of augmented $E_n$-algebras} (see Definition~\ref{D:overcategory}).
The augmentation makes $k$ an $E_n$-module over $A$.  

By  Proposition~\ref{P:EnAlgmapgivesfactonDn}, an augmented $E_n$-algebra defines naturally a locally constant factorization algebra on the closed unit disk (with its stratification given by its boundary) of dimension less than $n$. Indeed,
 we obtain  functors
                           \begin{equation} \label{eq:definefactoofBar}
\omega_{D^i}: E_n\textbf{-Alg}^{aug}\longrightarrow E_i\textbf{-Alg}^{aug}(E_{n-i}\textbf{-Alg}^{aug}) \longrightarrow \textbf{Fac}^{lc}_{D^i}(E_{n-i}\textbf{-Alg}^{aug}). 
\end{equation}
\begin{definition}\label{D:DefnBarEm} Let $A$ be an augmented $E_n$-algebra. Its Bar construction is
$$Bar(A) := \int_{I \times \mathbb{R}^{n-1}}A\mathop{\otimes}^{\mathbb{L}}_{\int_{S^0\times \mathbb{R}^{n-1}} A} k.$$
\end{definition}
This definition agrees with Definition~\ref{D:standardBar} for differential graded associative algebras and further we have equivalences
\begin{equation} \label{eq:DefnBarEm} Bar(A) := \int_{I \times \mathbb{R}^{n-1}}A\mathop{\otimes}^{\mathbb{L}}_{\int_{S^0\times \mathbb{R}^{n-1}} A} k\; \cong \; k\mathop{\otimes}_{A}^{\mathbb{L}} k\; \cong p_*(\omega_{D^1}(A))\end{equation}
where $I$ is the closed interval $[0,1]$ and $p:I\to pt$ is the unique map; in particular the right hand side of~\eqref{eq:DefnBarEm} is just the factorization homology of the associated factorization algebra on $D^1$.

The functor~\eqref{eq:definefactoofBar} shows that 
$ Bar(A)\cong  p_*(\omega_{D^1}(A))$ has an natural structure of augmented $E_{n-1}$-algebra (which can also be deduced from Lemma~\ref{L:homthMfldisEn}.2).

We can thus iterate  (up to $n$-times) the Bar constructions of an augmented $E_n$-algebra. 
 \begin{definition}\label{D:BariteratedforEm} Let $0\leq i\leq n$.
The $i^{\text{th}}$-iterated Bar construction of an augmented $E_n$-algebra $A$ is the augmented $E_{n-i}$-algebra  $$Bar^{(i)}(A) := Bar(\cdots (Bar(A))\cdots).$$  
\end{definition}
Using the excision axiom of factorization homology, one finds
 \begin{lemma}[Francis~\cite{F}, \cite{GTZ3}]\label{L:iteratedBarforEnisCH} Let $A$ be an $E_n$-algebra and $0\leq i\leq n$.
 There is a natural equivalence of $E_{n-i}$-algebras $$Bar^{(i)}(A) \cong \int_{D^i\times \mathbb{R}^{n-i}} A\mathop{\otimes}^{\mathbb{L}}_{\int_{S^{i-1}\times\mathbb{R}^{n-i+1}} A} k \; \cong \; p_*(\omega_{D^i}(A))$$
 \end{lemma}
In particular, taking $n=\infty$, we recover an $E_\infty$-structure on the iterated  Bar construction $Bar^{(i)}(A)$ of an augmented $E_\infty$-algebra.

\smallskip

We now describes the expected coalgebras structures. We start with the $E_\infty$-case, for which we can use the derived Hochschild chains from \S~\ref{S:comAlgebracase}. 
Then, Lemma~\ref{L:iteratedBarforEnisCH}, Theorem~\ref{T:Fact=CH} and the excision axiom  give natural equivalences of $E_\infty$-algebras (\cite{GTZ3}):
\begin{equation}\label{eq:nBarisCHIn}
CH_{S^i}(A,k)\; \cong \; Bar^{(i)}(A) .
\end{equation}

Recall the continuous map~\eqref{eq:pinchcube} 
$ pinch: \text{Cube}_d(r) \times S^d \longrightarrow  \bigvee_{i=1\dots r}\, S^d$. 
Similarly to the definition of the map~\eqref{eq:pinchSr}, applying the singular set 
functor to the map~\eqref{eq:pinchcube}  we get a morphism
\begin{multline}\label{eq:copinchSn}
pinch^{S^n,r}_*: C_{\ast}\big(\text{Cube}_d(r)\big) \otimes CH_{S^d}(A)  \mathop{\otimes}\limits_{A}^{\mathbb{L}} k \\
\stackrel{pinch_*\otimes_{A}^{\mathbb{L}} id} \longrightarrow  CH_{\bigvee_{i=1}^r S^d}(A)  \mathop{\otimes}\limits_{A}^{\mathbb{L}} k 
\cong \Big( CH_{\coprod_{i=1}^r S^d}(A)  \mathop{\otimes}\limits_{A^{\otimes r}}^{\mathbb{L}} A\Big) \mathop{\otimes}\limits_{A}^{\mathbb{L}} k \\
\cong \Big( CH_{\coprod_{i=1}^r S^d}(A) \Big)  \mathop{\otimes}\limits_{A^{\otimes r}}^{\mathbb{L}} k
\cong \Big( CH_{S^d}(A,k) \Big)^{\otimes r}
\end{multline}

\begin{proposition}[\cite{GTZ3}]\label{P:EncoAlgBar}Let $A$ be an augmented $E_\infty$-algebra. 
The maps~\eqref{eq:copinchSn} $pinch^{S^d,r}_*$ makes the iterated Bar construction  $Bar^{(d)}(A)\cong  CH_{S^d}(A,k)$ a natural $E_n$-coalgebra in the $\infty$-category of $E_\infty$-algebras.
\end{proposition}

If $Y$ is a pointed space, its $E_\infty$-algebra of cochains $C^\ast(Y)$ has a canonical augmentation $C^\ast(Y)\to C^\ast(pt)\cong k$ induced by the base point $pt\to Y$. By Theorem~\ref{T:IteratedIntegrals}, we have an 
$E_\infty$-algebra morphism
\begin{multline} 
\label{eq:iteratedlooptoBar}
 \mathcal{I}t^{\Omega^n}: Bar^{(n)}(C^{\ast}(Y))\cong CH_{S^n}(C^{\ast}(Y),k) 
 \\  \stackrel{\mathcal{I}t\otimes_{C^{\ast}(Y)}^{\mathbb{L}} k}\longrightarrow 
C^{\ast}\big(Y^{S^n}\big)\otimes_{C^{\ast}(Y)}^{\mathbb{L}} k
 \longrightarrow C^{\ast}\big(\Omega^n(Y) \big).
\end{multline}

We  now have a nice application of factorization homology   in algebraic topology.
\begin{corollary}\label{C:iteratedloopbar}
\begin{enumerate}
\item  The map~\eqref{eq:iteratedlooptoBar} 
$ \mathcal{I}t^{\Omega^n}: Bar^{(n)}(C^{\ast}(Y))
\to C^{\ast}\big(\Omega^n(Y)\big) $ 
is an $E_n$-coalgebra morphism in the category of $E_\infty$-algebras.
 It is further an equivalence 
if $Y$ is connected, nilpotent and has finite homotopy groups $\pi_i(Y)$  in degree $i\leq n$.
\item The dual of~\eqref{eq:iteratedlooptoBar}
$ C_{\ast}\big(\Omega^n(Y) \big)  \longrightarrow \big(C_{\ast}\big(\Omega^n(Y) \big)\big)^{\vee\vee} \stackrel{\mathcal{I}t_{\Omega^n}}\longrightarrow 
\Big(Bar^{(n)}(C^{\ast}(Y))\Big)^{\vee}$ is a morphism in $E_n\textbf{-Alg}$. If $Y$ is $n$-connected,  it
is an equivalence.
\end{enumerate}\end{corollary}

We now sketch the construction of the $E_i$-coalgebra structure of the $i^{\text{th}}$-iterated Bar construction of an $E_n$-algebra. 
To do so, we only need to define a locally constant \emph{cofactorization algebra} structure\footnote{that is a locally constant coalgebra over the $\infty$-operad $\text{Disk}(\R^i)$ see remark~\ref{R:Ndisk(M)}} whose global section is $Bar^{(i)}(A)$. By (the dual of) Proposition~\ref{P:extensionfrombasis}, it is enough to define such a structure on the basis of convex open disks of $\R^i$. Let $\mathcal{A}\in \textbf{Fac}^{lc}_{\R^i}(E_{n-i}\textbf{-Alg}^{aug})$ be the factorization algebra associated to $A$ (by Theorem~\ref{P:En=Fact} and Theorem~\ref{T:Dunn}).

Let $V$ be a convex open subset. 
By Corollary~\ref{C:EnAlgmapgivesfactonDnhat} (and Theorem~\ref{P:En=Fact}), the augmentation gives us a stratified locally constant factorization algebra  $\widehat{\omega}(\mathcal{A}_{|V})$ on $\widehat{V}=V\cup \{\infty\}$ (with values in $E_{n-i}\textbf{-Alg}^{aug}$).  
 
If $U\subset V$ is a convex open subset,  we have a continuous map $\pi_{U}:\widehat{V}\to \widehat{U}$ which maps the complement of $U$ to a single point. 
Further, the augmentation defines  maps of factorization algebras (on $\widehat{U}$) 
$$\epsilon_{U}: {\pi_{U}}_* (\widehat{\omega}(\mathcal{A}_{ |V}))\longrightarrow \widehat{\omega}(\mathcal{A}_{|U})$$  
which, on every open convex subset of  $U$ is the identity, and, on every open convex neighborhood of $\infty$ is given by the augmentation. 

Define $$Bar^{(i)}(A)(U) := \int_V \widehat{\omega}(\mathcal{A}_{|V})\cong \int_{U} {\pi_{U}}_*(\widehat{\omega}(\mathcal{A}_{|V})) $$
to be the factorization homology of $\widehat{\omega}(\mathcal{A}_{|V})$.
We finally get, for $U_1,\dots, U_s$  pairwise disjoint convex subsets of a convex open subset $V$, a structure map
\begin{multline}
\nabla_{U_1,\dots, U_s,V}: Bar^{(i)}(A)(V)=  \int_V (\widehat{\omega}(\mathcal{A}_{|V})\\
\stackrel{\bigotimes \int_{U_i}\varepsilon_{U_i}}\longrightarrow 
\bigotimes_{i=1\dots s} \int_{U_i}\widehat{\omega}(\mathcal{A}_{|U_i})   
\stackrel{\simeq}\longrightarrow Bar^{(i)}(A)(U_1)\otimes \cdots \otimes Bar^{(i)}(A)(U_s).
\end{multline}



 The maps $\nabla_{U_1,\dots, U_s,V}$ are the structure maps of a locally constant factorization coalgebras (see~\cite{GTZ3}) hence they make $Bar^{(i)}(A)$ into an $E_i$-coalgebra (with values in the category of $E_{n-i}$-algebras), naturally in $A$:
\begin{theorem}[\cite{F, GTZ3, AF}]\label{P:BarforEn}
 The iterated Bar construction lifts into an $\infty$-functor $$Bar^{(i)}: E_n\text{-}Alg^{aug} \longrightarrow E_i\text{-}coAlg\Big(E_{n-i}\text{-}Alg^{aug}\Big).$$
One has an natural equivalence\footnote{the relative tensor product being the tensor product of $E_i$-modules over $A$} $Bar^{(i)}(A) \cong k\mathop{\otimes}\limits^{\mathbb{L}}_{A} k$ in $E_{n-i}\textbf{-Alg}$. 

 Further this functor is equivalent to the one given by Proposition~\ref{P:EncoAlgBar} when restricted to augmented $E_\infty$-algebras.
\end{theorem}
\begin{example}Let $\text{Free}_n$ be the free $E_n$-algebra on $k$ as in Example~\ref{ex:FactofConf}. By Definition~\ref{D:DefnBarEm}, we have equivalences of $E_{n-1}$-algebras.
\begin{eqnarray*}
Bar(\text{Free}_n) & =& \int_{D^1\times \R^{n-1}}\text{Free}_n \mathop{\otimes}^{\mathbb{L}}_{\int_{S^0\times \R^{n-1}}\text{Free}_n} k \\
&\cong & \int_{S^1\times \R^{n-1}}\text{Free}_n \mathop{\otimes}^{\mathbb{L}}_{\int_{\R^{n-1}}\text{Free}_n} k\\
&\cong &\Big(\text{Free}_n\otimes \text{Free}_{n-1}(k[1])\Big)\mathop{\otimes}^{\mathbb{L}}_{\int_{\R^{n-1}}\text{Free}_n} k \qquad\mbox{(by Proposition~\ref{P:FactofConf})}\\
&\cong & \text{Free}_{n-1}(k[1]).
\end{eqnarray*}
The result also holds for $\text{Free}_n(V)$ instead of $Free_n$, see~\cite{F2}.  Iterating, one finds 
\begin{proposition}[\cite{F2}] There is a natural  equivalence of $E_{n-i}$-algebras
$$Bar^{(i)}(Free_n(V)) \, \cong\, \text{Free}_{n-i}(V[i]).$$
\end{proposition}
If one works in $\hTopp$ instead of $\hkmod$, then $\text{Free}_{n}(X)=\Omega^n\Sigma^n X$ and the above proposition boils down to 
$Bar^{(i)}(\Omega^n\Sigma^n X) \cong B^i \Omega^i (\Omega^{n-i}\Sigma^n X) \stackrel{\simeq}{\leftarrow} \Omega^{n-i}\Sigma^{n-i}(\Sigma^{i} X)$.
\end{example}

\subsection{$E_n$-Koszul duality and Lie algebras homology}

Let $A\stackrel{\varepsilon}\to k$ be an augmented $E_n$-algebra. 
The linear dual 
  $RHom(Bar^{(n)}(A), k)$ of the $n^{\text{th}}$-iterated Bar construction inherits an $E_n$-algebra structure (Theorem~\ref{P:BarforEn}).
\begin{definition}[\cite{L-HA, F}] The $E_n$-algebra $A^{(n)!}:=RHom(Bar^{(n)}(A), k)$ dual to the iterated Bar construction is called the (derived) \emph{$E_n$-Koszul dual of $A$}.
\end{definition}
The terminology is chosen because it agrees with the usual notion of Koszul duality for quadratic associative algebras but it is really more like a $E_n\textit{-Bar}$-duality.

Direct inspection of the $E_n$-algebra structures show that  the dual of the iterated Bar construction is equivalent to the centralizer of the augmentation. 
(see \S~\ref{S:maincentralizers}):
\begin{lemma}[\cite{L-HA, GTZ3}] Let $A\stackrel{\varepsilon}\to k$ be an augmented $E_n$-algebra.
There is an natural equivalence of $E_n$-algebras $A^{(n)!} \cong \mathfrak{z}_n(A\stackrel{\varepsilon}\to k)$.
\end{lemma}

Let $M$ be a dimension $m$ manifold endowed with a framing of $M\times \R^{n}$. By Proposition~\ref{P:EnAlgmapgivesfactonDn}, Theorem~\ref{T:Dunn} and Proposition~\ref{L:expFact}, we have the functor 
\begin{multline*}
\omega_{M\times D^n}: E_{n+m}\textbf{-Alg}^{aug}\to E_{n}\textbf{-Alg}^{aug}(E_{m}\textbf{-Alg})\\\stackrel{\omega_{D^n}}\longrightarrow \textbf{Fac}^{lc}_{D^n}(E_{m}\textbf{-Alg})
\cong \textbf{Fac}^{lc}_{ \R^{m}\times D^n} \longrightarrow \textbf{Fac}^{lc}_{M\times D^n} 
\end{multline*}
where the last map is induced by the framing of $M\times \R^n$ as in Example~\ref{ex:framedisEn}.
Let  $p$ be the map $p:M\times D^n\to pt$.
We can compute the factorization homology $p_*(\omega_{M\times D^n})$  by first pushing forward along the projection on $D^n$ and then applying $p_*$ or first pushing forward on $M$ and then pushing forward to the point.  By Theorem~\ref{P:BarforEn}, we thus obtain an equivalence 
\begin{equation}
\label{eq:KoszuldualofFact} p_*(\omega_{M\times D^n}) \, \cong \,   Bar^{(n)} \Big(\int_{M\times \R^n}A\Big) \,\cong\,\int_{M} \big(Bar^{(n)}(A)\big)
\end{equation}
where the right equivalence is an equivalence of $E_n$-coalgebras. 
When $M$ is further \emph{closed}, this result can be extended  to obtain :
\begin{proposition}[Francis~\cite{F2, AF}] \label{P:BarofFact} Let $A$ be an augmented $E_ {n+m}$-algebra,  $M\times \R^n$ be a framed closed manifold.
There is an natural equivalence of $E_n$-algebras $$\int_{M\times \R^n} A^{(n+m)!} \cong \Big(\int_{M\times \R^n} A\Big)^{(n)!}$$ if $Bar^{(n)}\Big(\int_{M\times \R^n}A\Big)$  has projective finite type homology groups in each degree.
In particular, $\int_{M} A^{(m)!} \cong \big(\int_{M}A\big)^{\vee}$ when $M$ is framed (and the above condition is satisfied).
\end{proposition}
 In plain english, we can say that the factorization homology over a closed framed manifold of an algebra and its $E_n$-Koszul dual are the same (up to finiteness issues). 
\begin{example}[Lie algebras and their $E_n$-enveloping algebras]Let $\textbf{Lie-Alg}$ be the $\infty$-category of Lie algebras\footnote{which is equivalent to the category of $L_{\infty}$-algebras}.
The forgetful functor $E_n\textbf{-Alg}\to \textbf{Lie-Alg}$, induced by $A\mapsto A[n-1]$, has a left adjoint $U^{(n)}: \textbf{Lie-Alg}\to E_n\textbf{-Alg}$, the $E_n$-enveloping algebra functor (see~\cite{Fresse-EnKoszulity, FG} for a construction). For $n=1$, this functor agrees with the standard universal enveloping algebra. 
\begin{proposition}[Francis~\cite{F2}]\label{P:Barofenveloping}Let $\mathfrak{g}$ be a (differential graded) Lie algebra. There is an natural equivalence of $E_n$-coalgebras
$$(U^{(n)}(\mathfrak{g}))^{(n)!} \; \cong \; C^\bullet_{Lie}(\mathfrak{g}) $$ where $C^\bullet_{\text{Lie}}(\mathfrak{g})$ is the usual Chevalley-Eilenberg cochain complex (endowed with its differential graded commutative algebra structure). 
\end{proposition}
Then using Proposition~\ref{P:BarofFact} and Proposition~\ref{P:Barofenveloping}, we obtain for $n=1,3,7$ that
\begin{equation*}
\int_{S^{n}} U^{(n)}(\mathfrak{g}) \, \cong \, \left(\int_{S^n}  C^\bullet_{Lie}(\mathfrak{g}) \right)^\vee
\end{equation*}
which for $n=1$ gives the following standard result computing the Hochschild homology groups of an universal enveloping algebra:
$HH_* (U(\mathfrak{g})) \cong HH_* (C^\bullet_{Lie}(\mathfrak{g}))^{\vee}   $. Applying the Fubini formula, we also find
$$\int_{S^1\times S^1} U^{(2)}(\mathfrak{g})\, \cong \,  CH_*\big(CH_*(C^\bullet_{Lie}(\mathfrak{g}))^{(1)!}\big).$$
\end{example}

\subsection{Extended topological quantum field theories}

In~\cite{L-TFT}, Lurie introduced factorization homology as a (generalization of) an invariant of an extended topological field theory and offshoot of  the cobordism hypothesis. We wish now to reverse this construction and explain very roughly how factorization homology can be used to produce an extended topological field theory.

Following~\cite{L-TFT},   there is an $\infty$-category\footnote{the cobordism hypothesis actually ensures that it is an $\infty$-groupoid} of \emph{extended topological field theories} with values in a symmetric monoidal $(\infty,n)$-category with duals $(\mathcal{C},\otimes)$. It is  the category of symmetric monoidal functors $\textbf{Fun}^{\otimes}\big((\text{Bord}_n^{fr},\coprod), (\mathcal{C},\otimes) \big)$ where  $\text{Bord}_n^{fr}$  is  the $(\infty,n)$-category of bordisms of framed manifolds with monoidal structure given by disjoint union. 
In~\cite{L-TFT},  $\text{Bord}_n^{fr}$ is defined as an $n$-fold Segal space which precisely  models the following  \emph{intuitive} notion of an $(\infty,n)$-category whose objects are framed compact $0$-dimensional manifolds. 
The morphisms between objects are framed $1$-bordism,
 that is $\Hom_{\text{Bord}_n^{fr}}(X,Y)$ consists of $1$-dimensional framed manifolds  $T$ with boundary $\partial T= Y\coprod X^{op}$ (where $X^{op}$ has the opposite framing to the one of  the object $X$). 
The $2$-morphisms in $\text{Bord}_n^{fr}$ are framed $2$-bordisms between $1$-dimensional framed manifolds (with corners) and so on. 
The $n$-morphisms are $n$-framed bordisms between $n-1$-dimensional framed manifolds with corners, its $n+1$-morphisms diffeomorphisms and the higher morphisms are isotopies.
Note that in the precise model of $\text{Bord}_n^{fr}$, the boundary component $N_1,\dots, N_r$ of  a manifold $M$ are represented by an open manifold with boundary components $N_1\times \R, \dots, N_r\times \R$ (in other words  are replaced by open collars).

There is an $(\infty,n+1)$-category $E_{\leq n}\textbf{-Alg}$ whose construction is  only sketched in~\cite{L-TFT} and detailled in~\cite{Sc-EnCat} using a model based on factorization algebras. The category    $E_{\leq n}\textit{-}Alg$  can be described informally as the $\infty$-category with objects the $E_n$-algebras, $1$-morphisms $\Hom_{E_{\leq n}\textbf{-Alg}}(A,B)$ is the space of all $(A,B)$-bimodules in $E_{n-1}\textbf{-Alg}$ and so on.
The ($\infty$-)category  $n\text{-}\Hom_{E_{\leq n}\textbf{-Alg}}(P,Q))$ of $n$-morphisms is the $\infty$-category of $(P,Q)$-bimodules where $P,Q$ are $E_1$-algebras (with additional structure). In other words  we have
\begin{eqnarray*} n\text{-}\Hom_{E_{\leq n}\textbf{-Alg}}(P,Q)) &\cong& \{P\}\mathop{\times}_{\textbf{Fac}^{lc}_{(-\infty,0)} }\textbf{Fac}^{lc}_{\R_*} \mathop{\times}_{\textbf{Fac}^{lc}_{(0, +\infty)}} \{R\} \\ 
&\cong & 
\{P\}\mathop{\times}_{E_1\textbf{-Alg}} \textbf{BiMod} \mathop{\times}_{E_1\textbf{-Alg}} \{R\}
\end{eqnarray*} see Example~\ref{ex:pointeddisk}.
 The composition 
$$ n\text{-}\Hom_{E_{\leq n}\textbf{-Alg}}(P,Q))\times n\text{-}\Hom_{E_{\leq n}\textbf{-Alg}}(Q,R)) \longrightarrow n\text{-}\Hom_{E_{\leq n}\textbf{-Alg}}(P,R))$$ 
is given by tensor products of bimodules: $ {}_{P}M_{Q} \otimes^{\mathbb{L}}_{Q} {}_{Q}N_{R}$ which in terms of factorization algebras is induced by the pushforward along the map $q:\R_*\times_{\R} \R_*  \to \R_*$ where $\R_*\times_{\R} \R_*$ is identified with $\R$ stratified in two points $-1$ and $1$ and $q$ is the quotient map identifying the interval $[-1,1]$ with the stratified point  $0\in \R_*$.

Let $E_{\leq n}\textbf{-Alg}_{(0)}$ be the $(\infty, n)$-category  obtained from $E_{\leq n}\textbf{-Alg}$ by discarding non-invertible $n+1$-morphisms, that is  $E_{\leq n}\textbf{-Alg}_{(0)}=Gr^{(n)}(E_{\leq n}\textbf{-Alg})$ where $Gr^{(n)}$ is the right adjoint of the forgetful functor  $(\infty,n+1)\textbf{-Cat} \to (\infty,n)\textbf{-Cat}$.

 The $(\infty,n)$-category $E_{\leq n}\textbf{-Alg}_{(0)}$ is fully dualizable (the dual of an algebra is its opposite algebra) hence every $E_n$-algebra determines in an unique way an extended topological field theory by the cobordism hypothesis.

In fact, this extended field theory can be constructed by factorization homology. Let $M$ be a $m$-dimensional  manifold. We say that $M$ is \emph{stably $n$-framed} if  $M\times \R^{n-m}$ is framed. 
Assume that $M$ has two ends which are trivialized as $L\times \R^{op} \subset M$ and $R\times \R \subset M$, where $L,R$ are   stably framed codimension $1$ closed sub-manifolds; here $\R^{op}$ means $\R$ endowed with the opposite framing to the standard one. 
For instance,    $M$ can be the interior of compact manifold $\overline{M}$ with two boundary component $L$, $R$ and  trivializations $L\times [0,\infty)\hookrightarrow \overline{M}$ and $R\times (-\infty, 0]\hookrightarrow \overline{M}$ where the trivialization on $L\times [0,\infty)$ has the opposite orientation as the one induced by $M$.

In that case, Lemma~\ref{L:homthMfldisEn} (and Proposition~\ref{P:BimodasFact}) imply that the factorization homology $\int_M A$ is an $E_{n-m}$-algebra which is also a bimodule over the $E_{n-m+1}$-algebras  $\big(\int_{L\times \R^{n-m+1}}A, \int_{R\times \R^{n-m+1}}A\big)$: $$ \int_{M\times \R^{n-m}} A \,\in \, \Big\{\int_{L\times \R^{n-m+1}}A\Big\}\mathop{\times}_{E_1-\textbf{-Alg}} \textbf{BiMod}\big(E_{n-m}\textbf{-Alg}\big)\mathop{\times}_{E_1-\textbf{-Alg}} \Big\{\int_{R\times \R^{n-m+1}}A\Big\}.$$
Thus $\int_M A$ is a $m$-morphism in $ E_{\leq n}\textbf{-Alg}_{(0)}$ from $R$ to $L$. In fact, one can prove
\begin{theorem}[\cite{Sc-EnCat}]\label{P:TCHasaTFT} Let $A$ be an $E_n$-algebra. The rule which, to a stably $n$-framed manifold $M$ of dimension $m$, associates 
$$Z_A(M):= \int_{M\times \R^{n-m}} A $$ extends as an extended field theory $Z_A\in \textbf{Fun}^{\otimes}\big(\text{Bord}_n^{fr}, E_{\leq n}\textbf{-Alg}_{(0)}\big)$.
\end{theorem}

\section{Commutative factorization algebras}\label{S:ComFactAlgebras}
In this Section, we explain in details the relationship in between classical homology theory \emph{\`a la} Eilenberg-Steenrod with factorization homology and more generally between  (co)sheaves and factorization algebras. The main point is that when $\mathcal{C}$ is endowed with the monoidal structure given by the coproduct, then Factorization algebras boils down to the usual theory of cosheaves. This is in particular the case of factorization algebras with values in a category of commutative algebras $E_\infty\textbf{-Alg}(\mathcal{C},\otimes)$ in $(\mathcal{C}, \otimes)$.
\subsection{Classical homology as factorization homology}
In this section we explain the relationship between factorization homology and singular homology (as well as generalized (co)homology theories for spaces).

Let $\dgZ$ be the ($\infty$-)category of differential graded abelian groups (i.e. chain complexes of $\mathbb{Z}$-modules). It has a symmetric monoidal structure given by the direct sum of chain complexes, which is the coproduct in $\dgZ$.
We can thus define homology theory for manifolds with values in $(\dgZ, \oplus)$. These are precisely (restrictions of) the (generalized) cohomology theories for spaces and nothing more. Recall that $\Mfldn^{or}$ is the ($\infty$-)category of oriented manifolds, Example~\ref{E:mainExofXstructure}. 

\begin{corollary}\label{C:ordinaryhomologyasFact}
Let $G$ be an abelian group\footnote{or even a graded abelian group or chain complex of abelian groups}. There is an unique  homology theory for oriented manifolds with coefficient in $G$ (Definition~\ref{D:facthomwithcoeff}), that is (continuous) functor $\mathcal{H}_{G}: \Mfldn^{or} \to  \dgZ$ satisfying the axioms 
\begin{itemize} 
\item \textbf{(dimension)}  $\mathcal{H}_{G}({\R^n})\cong G$ ;
\item \textbf{(monoidal)} the canonical map $\bigoplus\limits_{i\in I}\mathcal{H}_{G}({M_i})\stackrel{\simeq}\longrightarrow \mathcal{H}_{G}({\coprod\limits_{i\in I} M_i})$ is a quasi-isomorphism;
\item \textbf{(excision)} If $M$ is an oriented manifold  obtained as the gluing $M=R\cup_{N\times \R}L$  of two submanifolds along a a codimension $1$ submanifold  $N$ of $M$  with a trivialization $N\times \R$ of its tubular neighborhood in $M$,  there is an  natural  equivalence
$$\mathcal{H}_{G}(M) \stackrel{\simeq}\longleftarrow \textit{cone}\Big(\mathcal{H}_{G}({N\times \R})\stackrel{i_L-i_N}\longrightarrow \mathcal{H}_{G}({L})\oplus\mathcal{H}_{G}(R)\Big).$$ Here $\mathcal{H}_{G}({N\times \R})\stackrel{i_L}\longrightarrow \mathcal{H}_{G}({L})$ and $\mathcal{H}_{G}({N\times \R})\stackrel{i_R}\longrightarrow  \mathcal{H}_{G}({R})$ are the maps induced by functoriality by the inclusions of $N\times \R$ in $L$ and $R$.
\end{itemize}
Then, this homology theory is singular homology\footnote{or generalized exceptional homology when $G$ is graded or a chain complex} with coefficient in $G$. In particular, it extends as an homology theory for spaces.
\end{corollary}
The uniqueness means of course up to a contractible choice, meaning that any two homology  theory with coefficient in $G$ will be naturally equivalent  and any two choices of equivalences will  also be naturally equivalent and so on. 
\begin{proof}
This is a consequence of Proposition~\ref{P:Factwhencoproduct} below applied to $\mathcal{C}=\dgZ$ and the fact that the homotopy colimit $$ \hocolim\Big(\xymatrix{\mathcal{H}_{G}(R)& \mathcal{H}_{G}({N\times \R}) \ar[l]_{i_R} \ar[r]^{i_L }& \mathcal{H}_{G}({L})}\Big) $$ is precisely computed by the cone\footnote{note that in this case, we know \emph{a posteriori} that we can choose $\mathcal{H}_{G}$ to be singular chains, so that $\mathcal{H}_{G}(i_L) \oplus \mathcal{H}_{G}(i_R)$ is injective and the cone is equivalent to a quotient of chain complexes} of the map $i_L - i_R$.
\qed\end{proof}
\begin{remark}Let $H$ be a topological group, $f: H\to \Homeo(\R^n)$ a map of topological groups and $Bf: BH\to B\Homeo(\R^n)$ the induced map, so that we have the category $\Mfldn^{(BH,Bf)}$  of manifolds with $H$-structure, see Example~\ref{E:mainExofXstructure}.
As shown by its proof,  Corollary~\ref{C:ordinaryhomologyasFact} also holds with oriented manifolds replaced by manifold with a $H$-structure; in particular for\emph{ all manifolds or} a contrario for \emph{framed manifolds}.
\end{remark}
\begin{remark}
Corollary~\ref{C:ordinaryhomologyasFact} and Theorem~\ref{T:Fact=CH} implies that $\mathcal{H}_{G}(M)$ is computed by (derived) Hochschild homology $CH_{M}(G)$ (in $(\dgZ, \oplus)$). If $M_\bullet$ is a simplicial set model of $M$, then  $\mathcal{H}_{G}(M)\cong CH_{M_\bullet}(G)$ which is exactly (by \S~\ref{SS:explicitsimplicialmodelforHH}) the chain complex of the simplicial abelian group $G[M_\bullet]$. In particular, for $M_\bullet=\Delta_{\bullet}(M)=\Hom(\Delta^\bullet, M)$, one recovers \emph{exactly} the singular chain complex $C_\ast(M)$ of $M$.
\end{remark}

Corollary~\ref{C:ordinaryhomologyasFact} is a particular case of a more general result which we now describe. Let $\mathcal{C}$ be a category with coproducts. Then $(\mathcal{C}, \coprod)$ is symmetric monoidal, with unit given by its initial object $\emptyset$.
As we have seen in~\S~\ref{SS:AxiomsforHH}, any object $X$ of $\mathcal{C}$ carries a canonical (thus natural in $X$) structure of commutative algebra in  $(\mathcal{C}, \coprod)$ which is given by the \lq\lq{}multiplication\rq\rq{}  $X\coprod X \stackrel{\coprod id_X} \longrightarrow X$ induced by the identity map $id_X: X\to X$ on each component.
 This algebra structure is further unital, with unit given by the unique map $\emptyset \to X$. This defines a functor $\textit{triv}:\mathcal{C} \to E_\infty\textbf{-Alg}(\mathcal{C})$ which to an object associates its trivial commutative algebra structure. 
In fact, the latter algebras are the only only possible commutative and even associative ones in $(\mathcal{C}, \coprod)$. 
\begin{lemma}[Eckman-Hilton principle] \label{L:EnAlgincoprod} Let $\mathcal{C}$ be a category with coproducts and $(\mathcal{C}, \coprod)$ the associated symmetric monoidal category. Let   $H\stackrel{f}\to \Homeo(\R^n)$ be a topological group morphism and $\iota: \Disk^{(BH,Bf)}_n\textbf{-Alg}(\mathcal{C}) \to \mathcal{C}$ be the underlying object functor~\eqref{eq:Defunderlying} (Definition~\ref{D:DiskXstructured}).
We have a commutative diagram of equivalences 
$$ \xymatrix{E_\infty\textbf{-Alg}(\mathcal{C}) \ar[rr]^{\simeq} && \Disk^{(BH,Bf)}_n\textbf{-Alg}(\mathcal{C}) \ar[ld]_{\simeq}^{\iota} \\ 
& \mathcal{C} \ar[lu]_{\simeq}^{\textit{triv}} &  }$$ where the horizontal arrow is the canonical functor of Example~\ref{E:ComisDisk}.

In particular, any $E_n$-algebra ($n\geq 1$) in $(\mathcal{C},\coprod)$ is a (trivial) commutative algebra.
\end{lemma}
\begin{proof}
 Let $I=\{1,\dots, n\}$ be a finite set and $m_I: \coprod_{i\in I} C \to C$ be any map. The universal property of the coproduct yields a commutative diagram 
$$ \xymatrix{ \big(\coprod_{i\in I} C\big) \coprod \big(\coprod_{i\in I} C\big) \ar[rrr]^{m_I \coprod m_I} \ar[d]_{\big(\coprod_{i\in I}(\coprod id_C)\big)\circ \sigma_I}    & && C\coprod C \ar[d]^{\coprod id_C} \\ 
\coprod_{i\in I} C   \ar[rrr]_{m_I}&&& C }$$
where $\sigma_I$ is the permutation  induced by the bijection $(1,n+1)(2,n+2)\cdots (n,2n)$ of $I\coprod I=\{1,\dots, 2n\}$ on itself. 
Hence, if $C$ is an $E_n$-algebra,  it is naturally an object of $E_n\textbf{-Alg}(E_\infty\textbf{-Alg})$ where the commutative algebra structure is the trivial one $C\coprod C\stackrel{\coprod id_C}\to C$. By   Dunn Theorem~\ref{T:Dunn} (or Eckman-Hilton principle),
the forgetful map (induced by the pushforward of factorization algebras) $E_n\textbf{-Alg}((E_\infty\textbf{-Alg}))\longrightarrow E_0\textbf{-Alg}(E_\infty\textbf{-Alg})$ is an equivalence.   It follows that the $E_n$-algebra $C$
 is an $E_\infty$-algebra whose structure is  equivalent to $\textit{triv}(C)$.

The group map $\{1\}\to H$ induces a canonical functor $\Disk^{(BH,Bf)}_n\textbf{-Alg}\to E_n\textbf{-Alg}$ so that the above result implies that such a $\Disk^{(BH,Bf)}$-algebra with underlying object $C$ is necessarily of the form $\textit{triv}(C)$.
\qed\end{proof}

\begin{proposition}\label{P:Factwhencoproduct} Let $(\mathcal{C}, \coprod)$ be a $\infty$-category whose monoidal structure is given by the coproduct and $f:H\to \Homeo(\R^n)$ be  a topological group morphism .
\begin{itemize}
\item Any homology theory for $(BH,Bf)$-structured manifolds (Definition~\ref{D:HomologyforMfld}) \emph{extends uniquely} into an homology theory for spaces (Definition~\ref{D:axioms}). 
\item Any object $C\in \mathcal{C}$ determines a unique homology theory for $(BH,Bf)$-manifolds with values in $C$ (Definition~\ref{D:facthomwithcoeff}); further the evaluation map $\mathcal{H}\mapsto \mathcal{H}(\R^n)$ is an equivalence between the category of homology theories for $(BH,Bf)$-manifolds  in $(\mathcal{C}, \coprod)$ and $\mathcal{C}$. 
\end{itemize}
\end{proposition}
\begin{proof}
By Theorem~\ref{T:UniquenessFactHom},  homology theories for $(BH,Bf)$-structured manifolds are equivalent to    $\Disk^{(BH,Bf)}_n$-algebras which, by Lemma~\ref{L:EnAlgincoprod}, are equivalent to $\mathcal{C}$. In particular, any $\Disk^{(BH,Bf)}_n$-algebra  is  given by the commutative algebra associated to an object of $\mathcal{C}$ so that by Theorem~\ref{T:Fact=CH}, it extends to an homology theory for spaces.
\end{proof}

\subsection{Cosheaves as factorization algebras}\label{SS:ComFactAlgebras}

In this Section we  identify (pre-)cosheaves and (pre-)factorization algebras when the monoidal structure is given by the coproduct.

\smallskip

 Let $(\mathcal{C}, \otimes)$ be symmetric monoidal and  let $\mathcal{F}$
be in $\textbf{PFac}_X(\mathcal{C})$. Then the structure maps $\mathcal{F}(U)\to \mathcal{F}(V)$ for any open $U$ inside an open $V$ induces a functor $\gamma_{\mathcal{C}}:\textbf{PFac}_X(\mathcal{C})\longrightarrow \textbf{PcoShv}_X(\mathcal{C})$ where $\textbf{PcoShv}_X(\mathcal{C}):=\textbf{Fun}(\text{Open}(X), \mathcal{C})$
 is the $\infty$-category of \emph{precosheaves on $X$ with values in $\mathcal{C}$}. 

\begin{lemma}\label{L:cosheaves=Fact}
Let $(\mathcal{C}, \coprod)$ be an $\infty$-category with coproducts whose monoidal structure is given by the coproduct and $X$ be a topological space.
\begin{enumerate}
\item The functor $\gamma_{\mathcal{C}}:\textbf{PFac}_X(\mathcal{C})\longrightarrow \textbf{PcoShv}_X(\mathcal{C})$ is an equivalence.
\item If $X$ has a factorizing basis of opens\footnote{for instance when $X$ is Hausdorff}, then the functor $\gamma_{\mathcal{C}}:\textbf{PFac}_X(\mathcal{C})\longrightarrow \textbf{PcoShv}_X(\mathcal{C})$ restricts to an equivalence
$$\textbf{Fac}_X(\mathcal{C}) \stackrel{\simeq}\longrightarrow  \textbf{coShv}_X(\mathcal{C})$$ between factorization algebras on $X$ and the $\infty$-category $\textbf{coShv}_X(\mathcal{C})$ of (homotopy) cosheaves on $X$ with values in $\mathcal{C}$.
\item If $X$ is a manifold, then the above equivalence also induces an equivalence $\textbf{Fac}^{lc}_X(\mathcal{C}) \stackrel{\simeq}\longrightarrow  \textbf{coShv}^{lc}_X(\mathcal{C})$ between locally constant factorization algebras and locally constant cosheaves. 
\end{enumerate}
\end{lemma}
\begin{proof} Let $\mathcal{F}$ be in $\textbf{PFac}_X$ and  $U_1,\dots, U_r$ be open subsets of $V\in Open(X)$, which are pairwise disjoint.
 Let $\rho_{U_1,\dots, U_i,V}:\mathcal{F}(U_1)\coprod\dots \coprod \mathcal{F}(U_i)\to \mathcal{F}(V)$ be the structure map of $\mathcal{F}$. 
The associativity of the structure maps (diagram~\eqref{eq:associativityPFact}) shows that
the structure map $\rho_{U_j,V}: \mathcal{F}(U_j)\to \mathcal{F}(V)$ factors as 
$$ \xymatrix{ \mathcal{F}(U_j) \ar[rrrr]^{\rho_{U_j,V}} \ar[d]_{\cong} &&&& \mathcal{F}(V)  \\ 
\emptyset \coprod\dots \emptyset\coprod \mathcal{F}(U_j) \coprod \emptyset \dots \coprod\emptyset   \ar[rrrr]^{\quad\big(\coprod\limits_{k=1}^{j-1}\rho_{\emptyset,U_k}\big)\coprod id \coprod \big(\coprod\limits_{k=j+1}^{i}\rho_{\emptyset,U_k}\big) }  &&&& 
\mathcal{F}(U_1)\coprod\dots \coprod \mathcal{F}(U_i) \ar[u]_{\rho_{U_1,\dots, U_i,V}}.}$$
The universal property of the coproduct implies that the following diagram
\begin{equation}\label{eq:commutPfactPcosheaf}
\xymatrix{\mathcal{F}(U_1)\coprod\dots \coprod \mathcal{F}(U_i) \ar[rrr]^{\rho_{U_1,\dots, U_i,V}} \ar[d]_{\rho_{U_1,V}\coprod \dots \coprod \rho_{U_i,V}} &&&  \mathcal{F}(V)  \\
\mathcal{F}(V)\coprod\dots \coprod\mathcal{F}(V)  \ar[rrru]_{\; \coprod id_{\mathcal{F}(V)}} &&&}
\end{equation} is commutative. 
It follows that the structure maps are completely determined by the precosheaf structure. 
Conversely, any precosheaf on $\mathcal{C}$ gives rise functorially to a prefactorization algebra with structure maps given by the composition $$ \mathcal{F}(U_1)\coprod\dots \coprod \mathcal{F}(U_i)\longrightarrow \mathcal{F}(V)\coprod\dots \coprod\mathcal{F}(V)\stackrel{\coprod id_{\mathcal{F}(V)}}\longrightarrow \mathcal{F}(V)$$ yielding a functor $\theta_{\mathcal{C}}:\textbf{PcoShv}_X(\mathcal{C})\to \textbf{PFac}_X(\mathcal{C})$. 
The commutativity of diagram~\eqref{eq:commutPfactPcosheaf} implies that this functor $\theta_{\mathcal{C}}$ is inverse to $\gamma_{\mathcal{C}}:\textbf{PFac}_X(\mathcal{C})\to \textbf{PcoShv}_X(\mathcal{C})$. 

Note that both the cosheaf condition and the factorization algebra conditions implies that the canonical map $\mathcal{F}(U_1)\coprod \dots \coprod\mathcal{F}( U_i)\to \mathcal{F} (U_1\coprod\dots\coprod U_i)$ is an equivalence (of constant simplicial objects). 
Now, we can  identify the two gluing conditions. 
Since $\mathcal{U}$ is a (factorizing) cover of $V$, then $$\mathcal{P}(U):=\{\{U_1,\dots, U_k\}, \text{ which are pairwise disjoint}\}$$ is a cover of $V$.
 Further, if $\mathcal{F}\in \textbf{Fac}_X$,  the \v{C}ech complex $\check{C}(\mathcal{U},\mathcal{F})$ is precisely the (standard) \v{C}ech complex $\check{C}^{\text{cosheaf}}(P\mathcal{U},\gamma_{\mathcal{C}}(\mathcal{F}))$ of the cosheaf  $\gamma_{\mathcal{C}}(\mathcal{F})$ computed on the cover $P\mathcal{U}$ so that the map $\check{C}(\mathcal{U},\mathcal{F})\to \mathcal{F}(V)$ is an equivalence if and only if $\check{C}^{\text{cosheaf}}(P\mathcal{U},\gamma_{\mathcal{C}}(\mathcal{F}))\to \mathcal{F}(V)$ is an equivalence.

Assume $X$ has a factorizing basis of opens. Both factorization algebras and cosheaf are determined by their restriction on a basis of opens. It follows   that $\gamma_{\mathcal{C}}$ sends factorization algebras to cosheaves and $\theta_{\mathcal{C}}$ sends cosheaves to factorization algebras.

It remains to consider the locally constant condition when $X$ is a manifold, thus has a basis of euclidean neighborhood. 
On each euclidean neighborhood $D$,  by Theorem~\ref{P:En=Fact}, the restriction of $\mathcal{F}\in \textbf{Fac}^{lc}_X$ to $D$ is the factorization algebra given by an $E_n$-algebra $A\in E_n\textbf{-Alg}((\mathcal{C},\coprod))$ in $\mathcal{C}$.
  Lemma~\ref{L:EnAlgincoprod}  implies that $A$ is given by the trivial commutative algebra $\textit{triv}(C)$ associated to an object $C\in \mathcal{C}$. 
It follows from the identification of \v{C}ech complexes above, that $\mathcal{F}_{|D}$ is thus equivalent to the constant cosheaf on $D$ associated to the object $C$.
 The converse follows from the fact if $\mathcal{G}\in \textbf{coShv}^{lc}_X(\mathcal{C})$, then for any point $x\in X$, there is an euclidean neighborhood $D_x\cong \R^n$ on which $\mathcal{G}_{|D_x}$ is constant.
 The identification of the \v{C}ech complexes above implies that $\theta_{\mathcal{C}}(\mathcal{G}_{|D_x})$ is locally constant on $D_x$. Proposition~\ref{P:locallocallyconstant} implies that the factorization algebra  $\theta_{\mathcal{C}}(\mathcal{G})$ is locally constant on $X$, hence finishes the proof.
\qed\end{proof}

From the identification between cosheaves and factorization algebras, we deduce that factorization homology in $(\mathcal{C}, \coprod)$ agrees with homology with local coefficient:
\begin{proposition}\label{P:Factissheafhomology} Let $(\mathcal{C}, \coprod)$ be a $\infty$-category whose monoidal structure is given by the coproduct and $X$ a manifold. 
\begin{itemize}
\item There is an equivalence between  homology theories for $(X,TX)$-structured manifolds (Definition~\ref{D:HomologyforMfld}) and  $\textbf{coShv}^{lc}_X(\mathcal{C})$,  the ($\infty$-)category of locally constant cosheaves on $X$ with values in $\mathcal{C}$.
\item The above equivalence is given, for any $(X,TX)$-structured manifold $M$  and (homotopy) cosheaf $\mathcal{G}\in \textbf{coShv}^{lc}_X(\mathcal{C})$,   by $$\int_{M} \mathcal{G} := \mathbb{R}\Gamma(M, \mathcal{G})$$ the cosheaf homology of $M$ with values in the cosheaf $p^*(\mathcal{G})$ where $p:M\to X$ is the map defining the $(X,TX)$-structure.
\end{itemize}
\end{proposition}
\begin{proof}
The first claim is an immediate application of Lemma~\ref{L:cosheaves=Fact}.3 and Theorem~\ref{T:Theorem6GTZ2}. The latter result implies that factorization homology is computed by the \v{C}ech complex of the locally constant factorization algebra associated to a $\Disk_n^{(M,TM)}$-algebra given by the pullback along $p: M\to X$ of some $A\in \Disk_n^{(X,TX)}\textbf{-Alg}$. Now the second claims follows from Lemma~\ref{L:cosheaves=Fact}.2. 
\qed\end{proof}
\begin{remark} Factorization homology on a $(X,e)$-structured manifold  $M$  depends only on its value on open sub sets of $M$. Thus Proposition~\ref{P:Factissheafhomology}  
implies that, for any manifold $M$ and $A\in \Disk_n^{(X,e)}\textbf{-Alg}(\mathcal{C})$, factorization homology $\int_M A $ is given by cosheaf homology of the locally constant cosheaf $\mathcal{G}$ given by the image of $A$ under the functor $\Disk_n^{(X,e)}\textbf{-Alg}\to \Disk_n^{(M,TM)}\textbf{-Alg}$ (of Example~\ref{E:DisknAlg}) 
and the equivalence $\textbf{Fac}^{lc}_X(\mathcal{C}) \stackrel{\simeq}\longrightarrow  \textbf{coShv}^{lc}_X(\mathcal{C})$ of Lemma~\ref{L:cosheaves=Fact}. 
\end{remark}

Let $(\mathcal{C},\otimes)$ be a symmetric monoidal ($\infty$-)category. We say that a factorization algebra $\mathcal{F}$ on $X$  is \emph{commutative}
if each $\mathcal{F}(U)$ is given a structure of differential graded commutative (or $E_\infty$-) algebra and the structure maps are maps of algebras.
 In other words, the category of \emph{commutative factorization algebras} is $\textbf{Fac}_X(E_\infty\textbf{-Alg})$.

A peculiar property of (differential graded) commutative algebras is that their coproduct is given by their tensor product (that is the underlying tensor product in $\mathcal{C}$ endowed with its canonical algebra structure). 
From Lemma~\ref{L:cosheaves=Fact}, we obtain the following:
\begin{proposition}\label{P:ComFact=CoShvofAlg}
Let $(\mathcal{C},\otimes)$ be a symmetric monoidal ($\infty$-)category. The functor 
$$\gamma_{E_\infty\textbf{-Alg}(\mathcal{C})}: \textbf{Fac}_X(E_\infty\textbf{-Alg}(\mathcal{C})) \longrightarrow \textbf{coShv}_X(E_\infty\textbf{-Alg}(\mathcal{C}))$$ is an equivalence.
\end{proposition}
In other words, commutative factorization algebras are cosheaves (in $E_\infty\textbf{-Alg}$).
\begin{remark}\label{R:noncommutativecosheaf}
In view of Proposition~\ref{P:ComFact=CoShvofAlg}, one can think general  factorization algebras as non-commutative cosheaves. 
\end{remark}
Combining Proposition~\ref{P:ComFact=CoShvofAlg}, Proposition~\ref{P:Factissheafhomology} and Theorem~\ref{T:Fact=CH}, we obtain:
\begin{corollary} Let $\mathcal{F}\in \textbf{Fac}_X^{lc}(E_\infty\textbf{-Alg})$ 
be a locally constant commutative factorization algebra on $X$.
Then $$\int_X{A} \, \cong \, \CH_X(\mathcal{F})$$ where $\CH_X(\mathcal{F})$ is the (derived) global section of the cosheaf which to any open $U$ included in an euclidean Disk $D$ associates the derived Hochschild homology $CH_U(\mathcal{F}(D))$.
\end{corollary}

\section{Complements on factorization algebras}\label{S:SomeProofs}
In this Section, we give several proofs of results, some of them probably known by the experts, about factorization algebras that  we have postponed and for which we do not know any reference in the literature.  

\subsection{Some proofs related to the locally constant condition and the pushforward}
\label{S:SomeProofsGen}

\paragraph{\textbf{\emph{Proof of Propositions~\ref{P:locallocallyconstant} and~\ref{P:locallocallyconstantStrat}.}}}

Let $U\subset D\subset M$ be an inclusion of open disks; we need to prove that $\mathcal{F}(U)\to \mathcal{F}(D)$ is a quasi-isomorphism. We can assume $D =\R^n$ (by composing with a homeomorphism); the proof in the stratified case is similar to the non-stratified one by replacing $\R^n$ with $\R^i\times [0,+\infty)^{n-i}$.
We first consider the case where $\mathcal{F}_{|U}$ is locally constant and further, that $U$ is an euclidean disk (with center $x$ and radius $r_0$). Denote $D(y,r)$ an euclidean open disk of center $y$ and radius $r>0$ and let 
$$T_+:=\sup(t\in\R, \mbox{ such that $\forall \, \frac{r_0}{2}\leq s < t$,  $\mathcal{F}(D(x,\frac{r_0}{2}) \to \mathcal{F}(D(x,s))$  is an equivalence}).$$By assumption $T_+\geq r_0$.  We  claim that $T=+\infty$. Indeed, let $T$ be finite and such that ,   $\mathcal{F}(D(x,\frac{r_0}{2}) \to \mathcal{F}(D(x,s))$  is an equivalence for all $s< T$. We will prove that $T$ can not be equal to $T_+$. Every point $y$ on the sphere of center $x$ and radius $T$ has a neighborhood in which $\mathcal{F}$ is locally constant. In particular, there is a number $\epsilon_y>0$,  an open angular sector $S_{[0,T+\epsilon_y)}$ of length $T+\epsilon_y$ and angle $\theta_y$ containing $y$ such that $\mathcal{F}_{|S_{(T-\epsilon_y, T+\epsilon_y)}}$ is locally constant. Here,  $S_{(\tau, \gamma)}$ denotes the restriction of the angular sector to the band containing numbers of radius lying in $(\tau, \gamma)$.

 We first note that $S_{[0,T+\epsilon_y)}$ has a factorizing cover $\mathcal{A}_y$ consisting of open angular sectors of the form $S_{(T-\tau, T+\epsilon_y)}$ ($0<\tau \leq\epsilon_y$)
 and $S_{[0, \kappa)}$ ($0<\kappa <T$); there is an induced similar cover $\mathcal{A}_y \cap U$  of  $S_{[0,T)}$ given by the angular sectors of the form  $S_{(T-\tau, T)}$  ($0<\tau \leq\epsilon_y$) and $S_{[0, \kappa)}$ ($0<\kappa <T$). The structure maps $\mathcal{F}(S_{(T-\tau, T)})\to \mathcal{F}(S_{(T-\tau, T+\epsilon_y)})$ 
 induce a map of \v{C}ech complex $\psi_y:\check{C}(\mathcal{A}_y\cap U, \mathcal{F}) \to \check{C}(\mathcal{A}_y, \mathcal{F})$ so that the following diagram is commutative:
\begin{equation} \label{eq:diagchechsector} \xymatrix{\check{C}(\mathcal{A}_y\cap U, \mathcal{F}) \ar[rr]^{\psi_y} \ar[d]_{\simeq} &&  \check{C}(\mathcal{A}_y, \mathcal{F}) 
\ar[d]^{\simeq} \\ \mathcal{F}(S_{[0,T)}) \ar[rr] && \mathcal{F}(S_{[0,T+\epsilon_y)})}.\end{equation}
Since the  map  $S_{(T-\tau, T+\epsilon_y)} \to S_{(T-\tau, T+\epsilon_y)}$ is the inclusion of a sub-disk inside a disk in $S_{(T-\epsilon_y, T+\epsilon_y)}$, it is a quasi-isomorphism and, thus, so is the map  $\psi_y:\check{C}(\mathcal{A}_y\cap U, \mathcal{F}) \to \check{C}(\mathcal{A}_y, \mathcal{F})$. 
It follows from diagram~\eqref{eq:diagchechsector} that the structure map $\mathcal{F}(S_{[0,T)})\to \mathcal{F}(S_{[0,T+\epsilon_y)})$ is a quasi-isomorphism. In the above proof, we could also have taken any angle $\theta \leq \theta_y$ or replaced $\epsilon_y$ by any $\epsilon \leq \epsilon_y$ without changing the result.

By compactness of the sphere of radius $T$, we can thus find an $\epsilon >0$ and a $\theta >0$ such that the 
structure map $\mathcal{F}(S\cap U) \to \mathcal{F}(S)$ is a quasi-isomorphism for any angular sector $S$ around $x$ of radius $r=T+\epsilon' <T+\epsilon$ and arc length $\phi_S<\theta$.
 The collection of such angular sectors $S$ is a (stable by intersection) factorizing basis of the disk $D(x,T+\epsilon')$ while the collection of sectors $S\cap U$ is a  (stable by intersection) factorizing basis of the disk $D(x,T)$. Further, we have proved that the structure maps $\mathcal{F}(S\cap U) \to \mathcal{F}(S)$ is a quasi-isomorphism for any such $S$.
 It follows that the map $\mathcal{F}(D(x,T)) \to \mathcal{F}(D(x,T+\epsilon')$ is a quasi-isomorphism (since again the induced map  in between the \v{C}ech complexes associated to this two covers is  a quasi-isomorphism). 
 It follows that $T_+>T$ for any finite $T$ hence is infinite as claimed above. 
 In particular,  the canonical map $\mathcal{F}(D(x,T)) \to \mathcal{F}(D(x, T+r))$ is a quasi-isomorphism for any $r\geq 0$.
 
Now, since the collection of disks of radius $T>0$ centered at $x$ is a factorizing cover of $\R^n$, we deduce that $\mathcal{F}(D(x,T))\to \mathcal{F}(\R^n)$ is a quasi-isomorphism. Indeed,   fix some $R>0$ and let $j_R: x+y\mapsto x+R/(R-|y|) y$ be the homothety centered at $x$ mapping $D(x,R)$ homeomorphically onto $\R^n$. The map $j_R$ is a bijection between the set $\mathcal{D}_R$ of (euclidean) sub-disks of $D(x,R)$ centered at $x$ and the set $\mathcal{D}$ of all (euclidean) disks of $\R^n$ centered at $x$.  
For any disk centered at $x$, the  inclusion $D(x,T) \hookrightarrow D(x,j_R(T))$ yields a quasi-isomorphism 
 $\mathcal{F}(D(x,T)) \to \mathcal{F}(D(x,j_R(T))$. If $\alpha= \{D(x,r_0), \dots, D(x, r_i)\} \in (P\mathcal{D}_R)^{i+1}$, we thus get a quasi-isomorphism
 \begin{multline*}\mathcal{F}(\alpha)\cong \mathcal{F}\Big(D\big(x,\min(r_j, j=0\dots i)\big)\Big)\\ \stackrel{\simeq} \longrightarrow  \mathcal{F}\Big(D\big(x,\min(j_R(r_j), j=0\dots i)\big)\Big) \cong \mathcal{F}(j_R(\alpha)).\end{multline*}
 Assembling those for all $\alpha$'s yields a quasi-isomorphism
 $\check{C}(\mathcal{D}_R,\mathcal{F}) \stackrel{\simeq}\longrightarrow  \check{C}(\mathcal{D},\mathcal{F})$ which fit into a commutative diagram
 $$\xymatrix{\check{C}(\mathcal{D}_R,\mathcal{F}) \ar[d]_{\simeq} \ar[rr]^{\simeq} && \check{C}(\mathcal{D},\mathcal{F})\ar[d]^{\simeq}  \\ \mathcal{F}(D(x,R))  \ar[rr] && \mathcal{F}(\R^n)} $$ whose vertical arrows are quasi-isomorphisms since $\mathcal{F}$ is a factorization algebra. It follows that the lower horizontal arrow is a quasi-isomorphism as claimed.
 
 \smallskip
 
We are left to prove the result for $U\hookrightarrow D=\R^n$  when $U$ is not necessarily  an euclidean disk.  Choose an euclidean open disk $\tilde{D}$  inside $U$ small enough so that $\mathcal{F}_{|\tilde{D}}$ is locally constant. Let $h:\tilde{D}\cong \R^n$ be an homothety (with the same center as $\tilde{D}$) identifying $\tilde{D}$ and $\R^n$. Then $\tilde{U}:=h^{-1}(U)\subset D \subset U$ is an open disk homothetic to $U$. 
So that by the above reasoning (after using an homeomorphism between $U$ and  an euclidean disk $\R^n$) we have that the structure map $\mathcal{F}(\tilde{U})\to \mathcal{F}(U)$ is a quasi-isomorphism as well.
Since $\mathcal{F}_{|\tilde{D}}$ is locally constant,
 the structure map $\mathcal{F}(\tilde{U}) \to \mathcal{F}(\tilde{D})$ is a quasi-isomorphism. Now,  Proposition~\ref{P:locallocallyconstant} follows from the commutative diagram  
$$\xymatrix{\mathcal{F}(\tilde{U}) \ar[r]^{\simeq} \ar[d]_{\simeq} & \mathcal{F}(U) \ar[rd] & \\   
\mathcal{F}(\tilde{D}) \ar[rr]^{\simeq} \ar[ru] && \mathcal{F}(D)}$$
which implies that the structure map $\mathcal{F}(U) \to \mathcal{F}(D)$ is a quasi-isomorphism.
\qed 
\paragraph{\textbf{\emph{Proof of Proposition~\ref{L:expFact}.}}} \label{Proof:expFact}
 First we check that if $\mathcal{F}$ is a factorization algebra on $X\times Y$ and $U\subset X$ is open, then $\underline{\pi_1}_*(\mathcal{F}(U))$ is a factorization algebra over $Y$.
 If $\mathcal{V}$ is a factorizing cover of an open set $V\subset Y$, then $\{U\}\times \mathcal{V}$ is a factorizing cover of $U\times V$ and the \v{C}ech complex $\check{C}\big(\mathcal{V}, \underline{\pi_1}_*\mathcal{F}(U)\big)$ is equal to $\check{C}(\{U\}\times \mathcal{V}, \mathcal{F})$.
Hence the natural map $\check{C}\big(\mathcal{V}, \underline{\pi_1}_*\mathcal{F}(U)\big)\to \underline{\pi_1}_*(\mathcal{F})(U,V)$ factors as  $$\check{C}\big(\mathcal{V}, \underline{\pi_1}_*\mathcal{F}(U)\big)=\check{C}(\{U\}\times \mathcal{V}, \mathcal{F})\to \mathcal{F}(U\times V)=\underline{\pi_1}_*(\mathcal{F})(U,V).$$ It is a quasi-isomorphism since $\mathcal{F}$ is a factorization algebra. 
We have proved that $\underline{\pi_1}_*(\mathcal{F})\in \textbf{PFac}_X(\textbf{Fac}_Y)$.
 To show that  $\underline{\pi_1}_*(\mathcal{F})\in \textbf{Fac}_X(\textbf{Fac}_Y)$, we only need to check that for every open $V\subset Y$, and any factorizing cover $\mathcal{U}$ of $U$, the natural map 
$\check{C}\big(\mathcal{U}, \underline{\pi_1}_*(\mathcal{F})( -, V)\big) \to \underline{\pi_1}_*(\mathcal{F})( U, V) $ is a quasi-isomorphism, which follows by the same argument.  Hence  $\underline{\pi_1}_*$ factors as a functor $\underline{\pi_1}_*: {\textbf{Fac}}_{X\times Y}\longrightarrow {\textbf{Fac}}_X(\textbf{Fac}_Y)$.

\smallskip 

When $\mathcal{F}$ is locally constant,  Proposition~\ref{P:pushforwardlc} applied to the first and second projection implies that 
$\underline{\pi_1}_*(\mathcal{F}) \in {\textbf{Fac}}^{lc}_X(\textbf{Fac}_Y^{lc})$.

Now we build an inverse of $\underline{\pi_1}_*$ in the locally constant case. Let  $\mathcal{B}$ be in $\textbf{Fac}_X(\textbf{Fac}_Y)$. 
 A (stable by finite intersection) basis of neighborhood 
of $X\times Y$ is given by the products $U\times V$, with $(U,V)  \in \mathcal{CV}(X)\times \mathcal{CV}(Y)$ where $\mathcal{CV}(X)$, $\mathcal{CV}(Y)$ are bounded geodesically convex neighborhoods (for some choice of Riemannian metric on $X$ and $Y$). 
Thus by \S~\ref{SS:extensionfrombasis},  in order to extend $\mathcal{B}$ to a factorization algebra on $X\times Y$, it is enough to prove that the rule $(U\times V)\mapsto \mathcal{B}(U)(V)$ (where $U\subset X$, $V\subset Y$) defines an $\mathcal{CV}(X)\times \mathcal{CV}(Y)$-factorization algebra.
 If $U\times V \in \mathcal{CV}(X)\times \mathcal{CV}(Y)$, 
then $U$ and $V$ are canonically homeomorphic to  $\R^n$ and $\R^m$ respectively. Now, the construction of the structure maps for opens in $\mathcal{CV}(U)\times \mathcal{CV}(V)$ restricts to proving the result for $\mathcal{B}_{|U\times V}\in \textbf{Fac}^{lc}_{\R^n}(\textbf{Fac}^{lc}_{\R^m})$. 
By Theorem~\ref{P:En=Fact}, $\textbf{Fac}_{\R^d} \cong E_d\textbf{-Alg}$, hence by Dunn Theorem~\ref{T:Dunn} below, 
we have that $\textbf{Fac}^{lc}_{\R^{n+m}}\stackrel{\underline{\pi_1}_*}\longrightarrow\textbf{Fac}^{lc}_{\R^n}(\textbf{Fac}^{lc}_{\R^m})$ is an equivalence which allows to define a $\mathcal{CV}(X)\times \mathcal{CV}(Y)$-factorization algebra structure associated to $\mathcal{B}$. 
We denote $j(\mathcal{B})\in \textbf{Fac}_{X\times Y}$ the induced factorization algebra on $X\times Y$.
 Note that $j(\mathcal{B})$ is locally constant, since, again, the question reduces to Dunn Theorem. 

It remains to prove that $j:\textbf{Fac}^{lc}_X(\textbf{Fac}_Y^{lc})\to \textbf{Fac}_{X\times Y}$ is a natural inverse  to  $\underline{\pi_1}_*$. This follows by uniqueness of the factorization algebra extending a factorization algebra on a factorizing basis, that is, Proposition~\ref{P:extensionfrombasis}.  
\qed 

\subsection{Complements on \S~\ref{SS:Fachalfline}}
Here we collect the proofs of statements relating factorization algebras and intervals.
\paragraph{\textbf{\emph{Proof of Proposition~\ref{P:BarInterval}.}}}\label{Proof:BarInterval}
The (sketch of) proof is extracted from the excision property for factorization algebras in~\cite{GTZ2}. By definition of a $\Disk_1^{fr}$-algebra, a factorization algebra $\mathcal{G}$ on $\R$ carries a structure of $\Disk_1^{fr}$-algebra (by simply restricting the value of $\mathcal{G}$ to open sub-intervals, just as in Remark~\ref{R:En=Fact}). 

Similarly, if $\mathcal{G}$ is a factorization algebra on $[0,+\infty)$,
 it carries a structure of a  $\Disk_1^{fr}$-algebra and a (pointed) right module over it, while a factorization algebra on $(-\infty, 0]$ carries the structure of a left (pointed) module over a $\Disk_1^{fr}$-algebra (see Definition~\ref{D:LandRMod}).
 It follows that a factorization algebra over the closed interval $[0,1]$ determines  an $E_1$-algebra $\mathcal{A}$ and pointed left module $\mathcal{M}^{\ell}$ and pointed right module $\mathcal{M}^{r}$ over $\mathcal{A}$.
 
By strictification we can replace the $E_1$-algebra and modules by differential graded associative ones so that we are left to the case of a factorization algebra  $\mathcal{F}$ on $[-1,1]$ which,  on the factorizing basis $\mathcal{I}$ of $[0,1]$,  is precisely the $\mathcal{I}$-prefactorization algebra $\mathcal{F}$ defined before the Proposition~\ref{P:BarInterval}.

Now, we are left to prove that, for any $A, M^{\ell}, M^{r}, m^r, m^{\ell}$,  $\mathcal{F}$ is a $\mathcal{I}$-factorization algebra, and then to compute its global section $\mathcal{F}([0,1])$. Theorem~\ref{P:En=Fact} implies  that the restriction $\mathcal{F}_A$ of $\mathcal{F}$ to $(0,1)$ is a factorization algebra. In order to conclude we only need to prove that the canonical maps  
$$\check{C}\big(\mathcal{U}_{[0,1)},\mathcal{F})\big)\longrightarrow \mathcal{F}([0,1))=M^{r}, \quad  \check{C}\big(\mathcal{U}_{(0,1]},\mathcal{F})\big)\longrightarrow\mathcal{F}((0,1])=M^{\ell}$$ $$\mbox{and} \quad \check{C}\big(\mathcal{U}_{[0,1]},\mathcal{F})\big)\longrightarrow \mathcal{F}([0,1])\cong M^{r}\mathop{\otimes}\limits^{\mathbb{L}}_{A} M^{\ell}$$ are quasi-isomorphisms. Here,  $\mathcal{U}_{[0,1]}$ is the factorizing covers given by all opens $U_t:=[0,1]\setminus \{t\}$ where $t\in [0,1]$ (in other words by the complement of a singleton). 
Similarly,  $\mathcal{U}_{[0,1)}$, $\mathcal{U}_{(0,1]}$ are respectively covers given by all opens $U_t^\ell:=([0,1)\setminus \{t\}$ where $t\in [0,1)$ and all opens $U_t^r:=(0,1]\setminus \{t\}$ where $t\in (0,1]$.

The proof in the 3 cases are essentially the same so we only consider the case of the opens $U_t$.  Since $ M^{r}\mathop{\otimes}\limits^{\mathbb{L}}_{A} M^{\ell} \cong M^{r}\mathop{\otimes}\limits_A B(A,A,A) \mathop{\otimes}_{A} M^{\ell}$ where $B(A,A,A)$ is the two-sided Bar construction of $A$, it is enough to prove the result for $M^{r}=M^{\ell}=A$ in which case we are left to prove that the canonical map $$\check{C}\big(\mathcal{U}_{[0,1]},\mathcal{F})\big)\longrightarrow A\mathop{\otimes}\limits^{\mathbb{L}}_{A} A \cong B(A,A,A) \stackrel{\simeq}\longrightarrow A$$ is an equivalence. 

Any two  open sets in $\mathcal{U}_{[0,1]}$  intersect non-trivially so that the set $P\mathcal{U}$ are singletons. We have $\mathcal{F}(U_t)\cong \mathcal{F}([-1,t)) \otimes \mathcal{F}((t,1])$ which is $A\otimes A$ if $t\neq \pm 1$ and is $A\otimes k$ or $k\otimes A$ if $t=1$ or $t=-1$. 
More generally, $$\mathcal{F}(U_{t_0},\dots,U_{t_n}, U_{\pm 1})\cong \mathcal{F}(U_{t_0},\dots,U_{t_n})\otimes k.$$ 
 Further, if $0<t_0<\cdots<t_n<1$, then $\mathcal{F}(U_{t_0},\dots,U_{t_n})\cong A\otimes A^{\otimes n}\otimes A$ and the structure map $\mathcal{F}(U_{t_0},\dots,U_{t_n}) \to \mathcal{F}(U_{t_0},\dots,\widehat{U_{t_i}},\dots, U_{t_n})$ is given by the multiplication $$a_0\otimes \cdots \otimes a_{n+1} \mapsto a_0\otimes \cdots (a_i a_{i+1})\otimes \cdots\otimes a_{n+1}.$$ This identifies the  \v{C}ech complex $\check{C}\big(\mathcal{U},\mathcal{F}\big)$ with a kind of  parametrized analogue of the standard two sided Bar construction with coefficients in $A$. 
We have  canonical maps $$\phi_{t}:\mathcal{F}(U_t)\cong \mathcal{F}([-1,t)) \otimes \mathcal{F}((t,1]) \to \mathcal{F}([-1,1))\otimes \mathcal{F}((-1,1])\cong A\otimes A\to A$$ 
induced by the multiplication in $A$.
The composition $\bigoplus\limits_{U_r, U_s \in P\mathcal{U}} \mathcal{F}(U_r, U_s)[1]\to \bigoplus\limits_{U_t \in P\mathcal{U}} \mathcal{F}(U_t)[0]\to A $ is the zero map so that we have a map of (total) chain complexes: $\eta:\check{C}\big(\mathcal{U},\mathcal{F}\big)\to A$.  In order to prove that $\eta$ is an equivalence, we consider the retract $\kappa: A\cong \mathcal{F}(U_1)\hookrightarrow \bigoplus\limits_{U_t \in P\mathcal{U}} \mathcal{F}(U_t)[0] \hookrightarrow \check{C}\big(\mathcal{U},\mathcal{F}\big)$ which satisfies $\eta\circ \kappa= id_A$. Let $h$ be the homotopy operator on $
\check{C}\big(\mathcal{U},\mathcal{F}\big)$ defined, on  $\mathcal{F}(U_{t_0},\dots,U_{t_n})[n]$, by
$$\sum_{i=0}^n (-1)^i s^{t_0,\dots, t_n}_i: \mathcal{F}(U_{t_0},\dots,U_{t_n})[n] \longrightarrow \bigoplus\limits_{U_{r_0},\dots U_{r_{n+1}} \in P\mathcal{U}} \mathcal{F}(U_{r_0},\dots, U_{r_{n+1}})[n+1] $$ 
where, for $0\le i\le n-1$, $s^{t_0,\dots, t_n}_i$ is defined as the suspension of the identity map $$\mathcal{F}(U_{t_0},\dots,U_{t_n})[n]\to \mathcal{F}(U_{t_0},\dots,U_{t_n})[n+1]\cong \mathcal{F}(U_{t_0},\dots,U_{t_i}, U_{t_i},\dots, U_{t_{n}})[n+1]$$ followed by the inclusion in the \v{C}ech complex.  

Similarly, the map $s^{t_0,\dots, t_n}_n$ is defined as  the suspension of the identity map 
$\mathcal{F}(U_{t_0},\dots, U_{t_n})[n]\to \mathcal{F}(U_{t_0},\dots,U_{t_n})[n+1]\cong \mathcal{F}(U_{t_0},\dots, U_{t_{n}}, U_{1})[n+1]$ (followed by the inclusion in the \v{C}ech complex). 
Note that $ d h+ hd = id -\kappa \circ \eta$ where $d$ is the total differential on $\check{C}\big(\mathcal{U},\mathcal{F}\big)$. It follows   that $\eta:\check{C}\big(\mathcal{U},\mathcal{F}\big) \to A $ is an equivalence. 
\qed 

\paragraph{\textbf{\emph{Proof of Corollary~\ref{C:FacXhalfline}. }}} The functor $\underline{\pi_1}_*:\textbf{Fac}^{lc}_{X\times [0, +\infty)} \to E_1\textbf{-RMod}(\textbf{Fac}^{lc}_X)$  is well-defined by Corollary~\ref{L:expFactStrat} and Proposition~\ref{P:halfline}. In order to check it is an equivalence, as in the proof of Proposition~\ref{L:expFact}, we only need to prove it when $X=\R^n$, that is that, if $\mathcal{F}\in \textbf{Fac}^{lc}_{[0,+\infty)}(\textbf{Fac}^{lc}_X)$, then it is in the essential image of  $\underline{\pi_1}_*$. 
By Proposition~\ref{L:expFact}, we can also assume that the restriction $\mathcal{F}_{|(0,+\infty)}$ is in $\underline{\pi_1}_*(\textbf{Fac}^{lc}_{\R^n \times (0,+\infty)})$.

Let $\mathcal{I}_{\varepsilon}$ be the factorizing cover of $[0,+\infty)$ consisting of all intervals with the restriction  that intervals containing $0$ are included in $[0,\varepsilon)$ note that $I_{\tau}\subset I_{\varepsilon}$ if $\epsilon>\tau$. We can replace $\mathcal{F}$ by its \v{C}ech complex on $\mathcal{I}_{\varepsilon}$ (for any $\varepsilon$) and thus by its limit over all $\varepsilon >0$, which we still denote $\mathcal{F}$. 
As in the proof of Proposition~\ref{L:expFact}, we only need to prove that  $(U,V)\mapsto \big(\mathcal{F}(V)\big)(U)$ extends as a factorization algebra relative to the factorizing basis of $\R^n \times [0,+\infty)$ consisting of cubes (with sides parallel to the axes).
The only difficulty is to define the prefactorization algebra structure on this basis (since we already know it is locally constant, and thus will extend into a factorization algebra). 
As noticed above, we already have such structure when no cubes intersect $\R^n \times \{0\}$.
 Given a  finite family of pairwise disjoint cubes  lying in a bigger cube $K\times [0, R)$ intersecting $\{0\}$, we can find $\varepsilon >0$ such that no cubes of the family lying in $\R^n\times (0,+\infty)$ lies in the band $\R^n\times [0,\varepsilon)$. The value of $\mathcal{F}$ on each square containing  $\R^n \times \{0\}$ can be computed using the \v{C}ech complex associated to $ I_{\varepsilon}$. This left us, in every such cube, with one term containing a summand  $[0,\tau)$ ($\tau\leq \epsilon$) and cubes in the complement. Now choosing the maximum of the possible $\tau$ allows to first maps $(\mathcal{F}(c))(d)$ to $\big(\mathcal{F}(\tau,R)\big)(K)$ for every cube $c\times d$  in $\R^n\times (\tau, R)$ (since we already have a factorization algebra on $\R^n\times (0,+\infty)$). Then to maps all other summands to terms of the form $\mathcal{F}([0,\tau))(d)$, then all of them in $\mathcal{F}([0,\tau))(K)$ and finally to evaluate the last two remaining summand in
 $\mathcal{F}( [0, R))(K)$ using the prefactorization algebra structure of  $\mathcal{F}$ with respect to intervals in $[0,\infty)$. This is essentially the same argument as in the proof of Corollary~\ref{C:EnModasFact}.
\qed

\subsection{Complements on \S~\ref{SS:EnModasFac}} \label{SS:AppEnModasFac}
Here we collect proofs of statements relating factorization algebras and $E_n$-modules.
\paragraph{\textbf{\emph{Proof of Theorem~\ref{P:EnModasFact}}}}
 The functoriality  is immediate from the construction.
Let ${\mathbf{Fin}_*}$ be the $\infty$-category associated to the category $Fin_*$ of pointed finite sets. If $\mathcal{O}$ is an operad, the $\infty$-category $\mathcal{O}\textbf{-Mod}_{A}$ of $\mathcal{O}$-modules\footnote{in $\hkmod$} over an $\mathcal{O}$-algebra $A$ is the category of $\mathcal{O}$-linear functors $\mathcal{O}\textbf{-Mod}_{A}:= \Map_{\textbf{O}}(\textbf{O}_*, \hkmod)$.  
Here, following the notations of \S~\ref{S:EnAlg}, $\textbf{O}$ is the $\infty$-categorical enveloppe of $\mathcal{O}$ as in~\cite{L-HA} and $\textbf{O}_* := \textbf{O}\times_{{\mathbf{Fin}}}{\mathbf{Fin}_*}$. 
There is an natural fibration $\pi_{\mathcal{O}}:\mathcal{O}\textbf{-Mod}\longrightarrow \mathcal{O}\textbf{-Alg}$ whose fiber at $A\in\mathcal{O}\textbf{-Alg}$ is $\mathcal{O}\textbf{-Mod}_{A}$.

Let $\mathcal{D}_{\textit{isk}}$ be the set of all open disks in $\R^n$. Recall from Remark~\ref{R:Ndisk(M)} that $\mathcal{D}_{\textit{isk}}$-prefactorization algebras are exactly algebras over the operad  $ \Disk(\R^n)$ and that locally constant $\mathcal{D}_{\textit{isk}}$-prefactorization algebras are the same as  locally constant factorization algebras on $\R^n$ (Proposition~\ref{P:Ndisk(M)}).
 The map of operad $\Disk(\R^n)\to \mathbb{E}_{\R^n}$ of~\cite[\S 5.2.4]{L-HA} induces a fully faithful functor $E_n\textbf{-Alg}\to \Disk(\R^n)\textbf{-Alg} $ and thus a  functor  
$$\tilde{\psi}:E_n\textbf{-Mod} \longrightarrow \Disk(\R^n)\textbf{-Mod}\longrightarrow  \Disk(\R^n)\textbf{-Alg}.$$  

The map $\tilde{\psi}$ satisfies that, for every convex subset $C\subset \R^n$, one has  $$\tilde{\psi}(M)(C)= M(C)=\mathcal{F}_M(C).$$ 
By definition, $\tilde{\psi}\circ can : E_n\textbf{-Alg} \to \Disk(\R^n)\textbf{-Alg}$ is the composition $$E_n\textbf{-Alg} \longrightarrow \textbf{Fac}_{\R^n}^{lc}\longrightarrow  \Disk(\R^n)\textbf{-Alg}.$$ Hence,  the commutativity of the diagram in the Theorem will follow automatically once we have proved  that $\tilde{\psi}$ factors as a composition of functors \begin{equation}\label{eq:psifactorsthrough}\tilde{\psi}:E_n\textbf{-Mod} \stackrel{\psi}\longrightarrow \textbf{Fac}^{lc}_{\R^n_*} \longrightarrow \Disk(\R^n)\textbf{-Alg}.\end{equation} 
Assuming for the moment that we have proved that $\tilde{\psi}$ factors through $\textbf{Fac}^{lc}_{\R^n_*}$, let us show that $\psi$ is fully faithful.  By definition of categories of modules, we have a commutative diagram 
$$ \xymatrix{ E_n\textbf{-Mod}\ar[d]_{\pi_{E_n}} \ar[rr] &&  \Disk(\R^n)\textbf{-Mod} \ar[d]^{\pi_{\Disk(\R^n)}} \\
E_n\textbf{-Alg}\ar[rr] &&\Disk(\R^n)\textbf{-Alg} }$$ whose bottom arrow is a fully faithful embedding by~\cite[\S 5.2.4]{L-HA}.
Since the mapping spaces of $\mathcal{F}\in\textbf{Fac}_X^{lc}$ are the mapping spaces of  the  underlying prefactorization algebra, the map $\textbf{Fac}_{\R^n_*}^{lc} \to  \Disk(\R^n)\textbf{-Alg}$ is fully faithful, 
and we are left to prove that $$\tilde{\psi}: \Map_{E_n\textbf{-Mod}}(M,N) \to \Map_{\Disk(\R^n)\textbf{-Alg}}(\tilde{\psi}(M), \tilde{\psi}(N))$$ is an equivalence for all $M\in E_n\textbf{-Mod}_A$ and $N\in E_n\textbf{-Mod}_B$. 
The fiber at (the image of) an $E_n$-algebra $A$ of $\Disk(\R^n)\textbf{-Mod} \to \Disk(\R^n)\textbf{-Alg} $ is the (homotopy) pullback $$\Disk(\R^n)\textbf{-Mod}_A:= \Disk(\R^n)\textbf{-Alg}^{\slash A} \times^{h}_{\Disk(\R^n\setminus\{0\})\textbf{-Alg}^{\slash A}} \text{Iso}_{\Disk(\R^n\setminus\{0\})\textbf{-Alg}}(A).  $$
Here we write $\Disk(\R^n)\textbf{-Alg}^{\slash A}$ for the $\infty$-category of $\Disk(\R^n)$-algebras under $A$ and $\text{Iso}_{\Disk(\R^n)\textbf{-Alg}}(A)$ its subcategory of objects $A\stackrel{f}\to B$ such that $f$ is an equivalence. 
 In plain english, $\Disk(\R^n)\textbf{-Mod}_A$ is the $\infty$-category  of maps  $A\stackrel{f}\to B$ (where $B$ runs through $\Disk(\R^n)\textbf{-Alg}$) whose restriction to $\R^n\setminus\{0\}$ is an equivalence. 

It is now sufficient to prove, given $E_n$-algebras $A$ and $B$ (identified with objects of $\textbf{Fac}_{\R^n}^{lc}$) and
 two locally constant  factorization algebras $\tilde{\psi}(M)$,   $\tilde{\psi}(N)$ on $\R^n_*$ together with two maps of factorizations algebras $f:A\to \tilde{\psi}(M)$, $g: B\to \tilde{\psi}(N)$ whose restrictions to $\R^n\setminus\{0\}$ are quasi-isomorphisms, that the canonical map 
\begin{multline}\label{eq:fullyfaithfulpsi} \Map_{\Disk(\R^n)\textbf{-Alg}}(A,B) {\times}^{h}_{\Map_{\Disk(\R^n)\textbf{-Alg}}(A, \tilde{\psi}(N))} \Map_{\Disk(\R^n)\textbf{-Alg}}( \tilde{\psi}(M), \tilde{\psi}(N))\\
 \longrightarrow\Map_{\Disk(\R^n)\textbf{-Alg}}( \tilde{\psi}(M), \tilde{\psi}(N)) \end{multline}
is an equivalence. This  pullback is the mapping space
 $ \Map_{\Disk(\R^n)\textbf{-Mod}}( \tilde{\psi}(M), \tilde{\psi}(N))$ and the maps to  $ \Map_{\Disk(\R^n)\textbf{-Alg}}(A, \tilde{\psi}(N))$ are induced by post-composition by $g$  and precomposition by $f$.
 
The fiber of the map~\eqref{eq:fullyfaithfulpsi} at $\Theta:  \tilde{\psi}(M) \to \tilde{\psi}(N))$ is the mapping space of $\Disk(\R^n)$-algebras  $A\stackrel{\tau}\to B$ such that, for any disk $U\subset \R^n\setminus\{0\}$, which is a sub-disk of a disk $D$ containing $0$,  the following diagram is commutative:
\begin{equation}\label{eq:diagramfullyfaithfulpsi} \xymatrix{A(D)  \ar[d]_{\tau(D)}& & A(U)  \ar[ll]^{\rho_{U,D}^A} \ar[rr]^{\simeq}_{f} \ar[d]_{\tau(U)} & & \tilde{\psi}(M)(U) \ar[d]^{\Theta(U)} \\
B(D) && B(U) \ar[ll]^{\rho_{U,D}^B} \ar[rr]^{\simeq}_{g} &&  \tilde{\psi}(N)(U) . } 
\end{equation}
Here $\rho^A_{U,D}$ and $\rho^B_{U,D}$ are the structure maps of the factorization algebras associated to $A$ and $B$.
The right hand square of Diagram~\eqref{eq:diagramfullyfaithfulpsi} shows that $\tau$ is uniquely  determined by $\Theta$ on every open disk in $\R^n\setminus \{0\}$. 

Since $A$ and $B$ are locally constant factorization algebras on $\R^n$,  the maps $\rho^A_{U,D}$ and $\rho^B_{U,D}$ are natural quasi-isomorphisms. It follows from the left hand square in Diagram~\eqref{eq:diagramfullyfaithfulpsi} that the restriction of $\tau$ to  $\R^n\setminus \{0\}$ also determines the map $\tau$ on $\R^n$.   Hence the map~\eqref{eq:fullyfaithfulpsi} is an equivalence which concludes the proof that $\psi$ is fully faithful.

\smallskip

It remains to  prove that $\tilde{\psi}$ factors through a functor $\psi$, that is that we have a composition as written in~\eqref{eq:psifactorsthrough}.  
This amounts to prove that  for any $M\in {E_n}\textbf{-Mod}_{A}$ (that is $M$ is an $E_n$-module over $A$), $\tilde{\psi}(M)$ is a locally constant factorization algebra on the stratified manifold given by the pointed disk $\R^n_*$.
Since the convex subsets are a factorizing basis stable by finite intersection, we only have to prove this result on the cover $\mathcal{CV}(\R^n)$ (by Proposition~\ref{P:extensionfrombasis} and Proposition~\ref{P:locallocallyconstant}).

 Note that if $V\in \mathcal{CV}(\R^n)$ is a subset of $\R^n\setminus\{0\}$, then $\psi(M)_{|V}$ lies in the essential image of $\psi\circ can (M)_{|V}$ where $\psi\circ can(M)$ is the functor inducing the equivalence between $E_n$-algebras and locally constant factorization algebras on $\R^n$ (Theorem~\ref{P:En=Fact}). 
We denote $\mathcal{F}_A:=\psi\circ can (A)$ the  locally constant factorization algebra on $\R^n$ induced by $A$.
In particular,  the canonical map  $$\check{C}(\mathcal{CV}(V), \mathcal{F}_M)= \check{C}(\mathcal{CV}(V), \mathcal{F}_A)\to \mathcal{F}_A(V)\cong\mathcal{F}_M(V)$$ is a quasi-isomorphism and further,  if $U\subset V$ is a sub disk, then $\mathcal{F}_M(U) \to \mathcal{F}_M(V)$ is a quasi-isomorphism. 

We are left to consider the case where $V$ is a convex set containing $0$. Let $\mathcal{U}_V$ be the cover of $V$ consisting of all open sets which contains $0$ and are a finite union of disjoint convex subsets of $V$.
 It is a factorizing cover, and, by construction, two open sets in $\mathcal{U}_V$ intersects non-trivially since they contain $0$. Hence $P\mathcal{U}_V=\mathcal{U}_V$. 
Since 
$\mathcal{U}_V\subset P\mathcal{CV}(V)$, we have a  diagram of short exact sequences of chain complexes
{\small$$\xymatrix{0\ar[r] & \check{C}(\mathcal{U}_V, \mathcal{F}_M)\ar[r]^{{ i_M\quad}} & \check{C}(\mathcal{CV}(V),\mathcal{F}_M) \ar[r] & \check{C}(\mathcal{CV}(V),\mathcal{F}_M)/\check{C}(\mathcal{U}_V,\mathcal{F}_M)\ar[r] \ar[d]^{\simeq}& 0 \\
0\ar[r] & \check{C}(\mathcal{U}_V, \mathcal{F}_A)\ar[r]^{i_A\quad} & \check{C}(\mathcal{CV}(V),\mathcal{F}_A) \ar[r] & \check{C}(\mathcal{CV}(V),\mathcal{F}_A)/\check{C}(\mathcal{U}_V,\mathcal{F}_A)\ar[r]& 0 } $$ }
where the vertical equivalence follows from the fact that $\mathcal{F}_M(U)\cong \mathcal{F}_A(U)$ if $U$ is a convex set not containing $0$. 
Moreover, since $\mathcal{U}_V$ is a factorizing cover of $V$ and $\mathcal{F}_A$ a factorization algebra,   $i_A$ is a quasi-isomorphism, hence $\check{C}(\mathcal{CV}(V),\mathcal{F}_A)/\check{C}(\mathcal{U}_V,\mathcal{F}_A)$ is acyclic. It follows that $i_M: \check{C}(\mathcal{U}_V, \mathcal{F}_M)\to \check{C}(\mathcal{CV}(V),\mathcal{F}_M) $ is a quasi-isomorphism as well.

We are left to prove that the canonical map $\check{C}(\mathcal{U}_V, \mathcal{F}_M)\to \mathcal{F}(M)\cong M$ is a quasi-isomorphism. 
Note that for any $U\in \mathcal{U}_V$, we have $\mathcal{F}_M(U) \cong M\otimes_{A}^{\mathbb{L}} \mathcal{F}_A(U)$.
We deduce that $\check{C}(\mathcal{U}_V, \mathcal{F}_M)\cong M\otimes_{A}^{\mathbb{L}} \check{C}(\mathcal{U}_V, \mathcal{F}_A)$ as well.
  The chain map $$ M\otimes_{A}^{\mathbb{L}} \check{C}(\mathcal{U}_V, \mathcal{F}_A)\cong \check{C}(\mathcal{U}_V, \mathcal{F}_M)\longrightarrow \mathcal{F}(M)\cong M\otimes_A^{\mathbb{L}} A$$ is an equivalence since it is obtained by tensoring (by $M$ over $A$) the quasi-isomorphism $\check{C}(\mathcal{U}_V, \mathcal{F}_A)\to \mathcal{F}_A(V)\cong A$ (which follows from the fact that $\mathcal{F}_A$ is a factorization algebra).  

It remains to prove that $\mathcal{F}_M(U)\to \mathcal{F}_M(V)$ is a quasi-isomorphism if both $U$, $V$ are convex subsets containing $0$. This is immediate since $M(U)\to M(V)$ is an equivalence by definition of an $E_n$-module over $A$.
\qed 

\paragraph{\textbf{\emph{Proof of Corollary~\ref{C:EnModasFact}.}}}
By Theorem~\ref{P:EnModasFact}, we have a commutative diagram 
$$  \xymatrix{ E_n\textbf{-Mod}\ar[d]_{\pi_{E_n}} \ar[rr] && \textbf{Fac}^{lc}_{\R^n_*}\times^{h}_{\textbf{Fac}^{lc}_{\R^n\setminus\{0\}}} \textbf{Fac}^{lc}_{\R^n}\ar[d] \\
E_n\textbf{-Alg}\ar[rr] && \textbf{Fac}^{lc}_{\R^n} }   $$ with fully faithful horizontal arrows.
Since $E_n\textbf{-Alg}\to \textbf{Fac}_{\R^n}^{lc}$ is an equivalence,
 we only need to prove that, for any $E_n$-algebra $A$, the induced fully faithful functor $E_n\textbf{-Mod}_A \longrightarrow \textbf{Fac}^{lc}_{\R^n_*}\times_{\textbf{Fac}^{lc}_{\R^n\setminus\{0\}}}\{A\}$ between the fibers is essentially surjective\footnote{that is we are left to prove Corollary~\ref{C:AEnModasFact}}.

 Let $\mathcal{M}$ be a locally constant factorization algebra on $\R^n_{*}$ such that $\mathcal{M}_{|\R^\setminus\{0\}}$ is equal to  $\mathcal{A}_{|\R^n\setminus\{0\}}$ where $\mathcal{A}$ is the  factorization algebra associated to $A$ (by Theorem~\ref{P:En=Fact}). 
 Recall that $N:\R^n\to [0,+\infty)$ is the euclidean norm map.
 Lemma~\ref{L:Normpreserveslc} implies that $N_*(\mathcal{M})$ is  is locally constant on the stratified half-line $[0,+\infty)$ 
and thus equivalent to a right module over the $E_1$-algebra $N_*(\mathcal{M})(\R^n\setminus\{0\})\cong \mathcal{A}(\R^n\setminus\{0\}) \cong \int_{S^{n-1}\times \R} A$.

By homeomorphism invariance of (locally constant) factorization algebras, we can replace $\R^n$ by the unit open disk $D^n$ of $\R^n$ in the  above analysis. We also denote $D^n_*$ the disk $D^n$ viewed as a pointed space with base point $0$.
We now use this observation to define a structure of $E_n$-module over $A$ on $M:=\mathcal{M}(D^n)=N_*(\mathcal{M})\big([0,1)\big)$. It amount to define, for any finite set $I$, continuous maps (compatible with the structure of the operad of little disks of dimension $n$)
\begin{multline*}
\text{Rect}_*\Big( D^n_*\coprod\Big(\coprod_{i\in I} D^n\Big), D^n_* \Big)   \longrightarrow \Map_{\hkmod}\big(M\otimes A^{\otimes I}, M\big)\\ \stackrel{\simeq}\longrightarrow \Map_{\hkmod}\Big(M\otimes \Big(\int_{D^n}A\Big)^{\otimes I}, M\Big)
\end{multline*}
where $\text{Rect}_*$ is the space of  rectilinear embeddings which maps the center of the first copy  $D^n_*$ to the center of $D^n_*$ (i.e. preserves the base point of $D^n$).
Let $I_N$ be the map that sends an element $f\in \text{Rect}_*\big( D^n\coprod\big(\coprod_{i\in I} D^n\big), D^n \big)$   to the smallest open sub-interval $I_N(f)\subset (0,1)$ which contains $N\big(f\big(\coprod_{i\in I} D^n\big)\big)$, that is the smallest interval that contains the image of the non-pointed disks. 
By definition $I_N$ is continuous (meaning the lower and the upper bound of $I_N(f)$ depends continuously of $f$) and its image is disjoint from the image $N(D^n_*)$ of the pointed copy of $D^n$.
 Similarly we define $r(f)$ to be the radius of $f(D^n_*)$. 
We have a continuous map $$\tilde{N}:\text{Rect}_*\Big( D^n_*\coprod\big(\coprod_{i\in I} D^n\big), D^n_* \Big) \longrightarrow \text{Rect}\Big([0,1)\coprod (0,1), [0,1)\Big) $$ 
given by $\tilde{N}(f)\big((0,1)\big)= I_N(f)$ and $\tilde{N}(f)\big([0,1)\big)= [0,r(f))$. 
Since $f(\coprod_{i\in I}D^n)\subset S^{n-1}\times (0,1)$ , we have the composition
\begin{multline*}\Upsilon:\text{Rect}_*\Big( D^n_*\coprod\big(\coprod_{i\in I} D^n\big), D^n_* \Big)    \longrightarrow\text{Rect}\Big(\coprod_{i\in I} D^n, S^{n-1}\times (0,1) \Big) \\  \longrightarrow \Map_{\hkmod}\left( \Big(\int_{D^n} A\Big)^{\otimes I}, \int_{S^{n-1}\times (0,1)} A\right)  \end{multline*}where the first map is  induced by the restriction  to $\coprod_{i\in I} D^n$ and the last one by functoriality of factorization homology with respect to embeddings.
We finally define 
\begin{multline*}
\mu:\text{Rect}_*\Big( D^n_*\coprod\big(\coprod_{i\in I} D^n\big), D^n_* \Big) \stackrel{\tilde{N}\times \Upsilon}\longrightarrow\\
 \text{Rect}\Big([0,1)\coprod (0,1), [0,1)\Big) \times \Map_{\hkmod}\left( \Big(\int_{D^n} \hspace{-.2pc}A\Big)^{\otimes I}, \int_{S^{n-1}\times (0,1)} \hspace{-.5pc} A\right)\longrightarrow\\
 \Map_{\hkmod}\left( M \otimes \int_{S^{n-1}\times (0,1)} \hspace{-.5pc}A , \, M\right)\times \Map_{\hkmod}\left( \Big(\int_{D^n} \hspace{-.2pc}A\Big)^{\otimes I}, \int_{S^{n-1}\times (0,1)} \hspace{-.5pc}A\right)\\
 \stackrel{id_M\otimes \circ}\longrightarrow\Map_{\hkmod}\left( M \otimes\Big(\int_{D^n} A\Big)^{\otimes I} , \,M\right)
\end{multline*} 
where the second  map is induced by the $E_1$-module structure of $M=N_*(\mathcal{M})([0,1))$ over $\int_{S^{n-1}\times (0,1)} A $ and the last one by composition.
That $\mu$ is compatible with the action of the little disks operad follows from the fact that  $\Upsilon$ is induced by the $E_n$-algebra structure of $A$ and $M$ is an $E_1$-module over $\int_{S^{n-1}\times (0,1)}A$.
Hence,  $M$ is in $E_n\textbf{-Mod}_A$.

 We now prove that the factorization algebra 
 $\psi(M)$ is $\mathcal{M}$.  For all euclidean disks $D$ centered at $0$, one has $\psi(M)(D)= \mu(\{D\hookrightarrow D^n\})(M)=\mathcal{M}(D)$ and further $\psi(M)(U)= \mathcal{A}(U)$ if $U$ is a disk that does not contain $0$. 
The $\mathcal{D}$-prefactorization algebra structure of $\psi(M)$ (where $\mathcal{D}$ is  the basis of opens consisting of all euclidean disks centered at $0$ and all those who do not contain $0$) is precisely given by $\mu$ according to the construction of $\psi$ (see Theorem~\ref{P:EnModasFact}). Hence, by Proposition~\ref{P:extensionfrombasis}, $\psi(M) \cong \mathcal{M}$ and the essential surjectivity follows.
\qed

\paragraph{\textbf{\emph{Proof of Proposition~\ref{P:BimodasFact}.}}}
 By~\cite[\S 4.3]{L-HA}, we have two functors $i_{\pm}:\textbf{BiMod}\to E_1\textbf{-Alg}$ and the (homotopy) fiber of $\textbf{BiMod}\stackrel{(i_{-}, i_+)}\longrightarrow E_1\textbf{-Alg}\times E_1\textbf{-Alg}$ at a point $(L,R)$ is the category of $(L,R)$-bimodules which is equivalent to the category $E_1\textbf{-LMod}_{L\otimes R^{op}}$. We have a factorization 
\begin{equation}\label{eq:factorizationofjv}\xymatrix{ \textbf{Fac}_{\R_{*}}^{lc}  \ar@/_/[rrrd]_{(j_{-}^*,j_+^*)\quad}\ar[rr]^{(j_{\pm}^*,(-N)_*)\qquad} && \textbf{BiMod}\ar[r]^{(i_{-}, i_+)}&  E_1\textbf{-Alg}\times E_1\textbf{-Alg} \\ && & \ar[u]_{\simeq}\textbf{Fac}^{lc}_{(-\infty,0)}\times \textbf{Fac}^{lc}_{(0,+\infty)}.}\end{equation}
We can assume that $L,R$ are strict and consider the fiber $$\big(\textbf{Fac}_{\R_{*}}^{lc}\big)_{L,R}:=\{\mathcal{F}_L, \mathcal{F}_R\}\times_{\textbf{Fac}^{lc}_{(-\infty,0)}\times \textbf{Fac}^{lc}_{(0,+\infty)}}\textbf{Fac}_{\R_{*}}^{lc} $$ of  $(j_{-}^*,j_+^*)$ at the pair of  factorization algebras $(\mathcal{F}_L, \mathcal{F}_R)$ on $(-\infty,0)$, $(0,+\infty)$ corresponding  to $L$, $R$ respectively (using Proposition~\ref{P:BarInterval}).
The  pushforward along the opposite of the euclidean norm map gives the functor
$(-N)_*:\big(\textbf{Fac}_{\R_{*}}^{lc}\big)_{L,R}\to E_1\textbf{-LMod}_{L\otimes R^{op}}$. 

We further have a locally constant factorization algebra $\mathcal{G}^{L,R}_{M}$ on $\R_*$ which is defined on the basis of disks by the same rule as for the open interval (for disks included in a component $\mathcal{R}\setminus\{0\}$) together with $\mathcal{G}_{M}^{L,R}(\alpha, \beta)=M$ for $\alpha<0<\beta$. 
For  $r<t_1<u_1\cdots <t_n<u_n<\alpha<0<\beta< x_{1}<y_{1}<\cdots <x_m<y_{m}<s$, the structure maps \begin{multline*}\Big(\bigotimes_{i=1\dots n}\mathcal{G}_{M}^{L,R}\big( (u_i, t_i)\big)\Big) \otimes \mathcal{G}_{M}^{L,R}\big((\alpha,\beta)\big) \otimes \Big(\bigotimes_{i=1\dots n}\mathcal{G}_{M}^{L,R}\big( (u_i, t_i)\big) \Big)\\
\cong\; L^{\otimes n} \otimes M\otimes R^{\otimes m}\longrightarrow M\cong  \mathcal{G}_{M}^{L,R}\big((r,s)\big)\end{multline*} are given by 
$\ell_1\otimes \cdots \otimes \ell_n\otimes a \otimes r_1\otimes \cdots \otimes r_n \mapsto (\ell_1\cdots \ell_n)\cdot a \cdot (r_1\cdots r_n)$.

One checks as in Proposition~\ref{P:BarInterval} that $\mathcal{G}_{M}^{L,R}$ is a locally constant factorization algebra on $\R_*$.
The induced functor $E_1\textbf{-LMod}_{L\otimes R^{op}}\to \textbf{Fac}_{\R_*}^{lc}$ is an inverse of $(-N)_*$. Thus the fiber $\big(\textbf{Fac}_{\R_{*}}^{lc}\big)_{L,R}$ of $(j_{-}^*,j_+^*)$ is equivalent to  $E_1\textbf{-LMod}_{L\otimes R^{op}}$. It now follows from diagram~\eqref{eq:factorizationofjv} that the functor $(j_{\pm}^*,(-N)_*):\textbf{Fac}_{\R_{*}}^{lc} \cong \textbf{BiMod}$  is an equivalence.
\qed 

\paragraph{\textbf{\emph{Proof of Proposition~\ref{P:envelopFac}.}}}

We define a functor $G: E_1\textbf{-RMod}_{\mathcal{A}(S^{n-1}\times \R)}\to\textbf{Fac}^{lc}_{\R^n_*}\times_{\textbf{Fac}^{lc}_{\R^n\setminus\{0\}}} \{\mathcal{A}\}$ (which will be an inverse of $N_*$) as follows. 
By Proposition~\ref{P:halfline} we have an equivalence $$E_1\textbf{-RMod}_{\mathcal{A}(S^{n-1}\times \R)}\,\cong\, \textbf{Fac}^{lc}_{[0,+\infty)}\times_{\textbf{Fac}^{lc}_{(0,+\infty)}} \{N_*(\mathcal{A})\}.$$
It is enough to define $G$  as a functor from $\textbf{Fac}^{lc}_{[0,+\infty)}\times_{\textbf{Fac}^{lc}_{(0,+\infty)}} \{N_*(\mathcal{A})\}$ to locally constant $\mathcal{U}$-factorization algebras, where $\mathcal{U}$ is  a (stable by finite intersections) factorizing basis of $\R^n$  (by Proposition~\ref{P:extensionfrombasis}).
We choose $\mathcal{U}$ to be the basis consisting of all euclidean disks centered at $0$ and all convex open subsets not containing $0$. 
  Let $\mathcal{R} \in \textbf{Fac}^{lc}_{[0,+\infty)}\times_{\textbf{Fac}^{lc}_{(0,+\infty)}} \{N_*(\mathcal{A})\}$. If $U\in \mathcal{U}$ does not contains $0$, then we set $G(\mathcal{R})(U)=\mathcal{A}(U)$ and structure maps on open sets in $\mathcal{U}$ not containing $0$ to be the one of $\mathcal{A}$; this defines  a  locally constant factorization algebra on $\R^n\setminus\{0\}$ since $\mathcal{A}$ does. 

We denote
 $D(0,r)$ the euclidean disk of radius $r>0$ and  set $G(\mathcal{R})(D(0,r))=\mathcal{R}([0,\epsilon))$.
Let  $D(0,r), U_1,\dots, U_i$ be pairwise subsets of $\mathcal{U}$ which are sub-sets  of an euclidean disk $D(0,s)$. Then, $U_1,\dots, U_i$  lies in $ S^{n-1}\times (r,s)$. Denoting respectively $\rho^{\mathcal{A}}$, $\rho^{\mathcal{R}}$ the structure maps of the factorization algebras $\mathcal{A}\in \textbf{Fac}^{lc}_{S^{n-1}\times (0,+\infty)}$ and $\mathcal{R}\in \textbf{Fac}^{lc}_{[0,+\infty)}$,   we have the following composition
\begin{multline}\label{eq:structuremapforG}
G(\mathcal{R})\big(D(0,r)\big)\otimes G(\mathcal{R})(U_1)\otimes \cdots\otimes G(\mathcal{R})(U_i) 
\cong
 \mathcal{R}\big([0,r)\big) \otimes \mathcal{A}(U_1) \otimes \cdots \otimes \mathcal{A}(U_i) 
\\ \stackrel{id\otimes {\rho}_{U_1,\dots, U_i, S^{n-1}\times (r,s)}^{A}} \longrightarrow \mathcal{R}\big([0,r)\big) \otimes \mathcal{A}\big(S^{n-1}\times (r,s)\big)  
\cong \mathcal{R}\big([0,r)\big)\otimes N_*(\mathcal{A})\big( (r,s)\big) \\
\stackrel{\rho^{\mathcal{R}}_{[0,r), (r,s), [0,s)}}\longrightarrow  \mathcal{R}\big([0,s)\big)
=G(\mathcal{R})\big(D(0,s)\big).
\end{multline}
The maps~\eqref{eq:structuremapforG} together with the structure maps of $\mathcal{A}_{|\R^n\setminus\{0\}}\cong \mathcal{R}_{|\R^n\setminus\{0\}}$ define the structure of a $\mathcal{U}$-factorization algebra since $\mathcal{R}$ and $\mathcal{A}$ are factorization algebras. 

The maps $G(\mathcal{R})\big(D(0,r)\big)\to G(\mathcal{R})\big(D(0,s)\big)$ are quasi-isomorphisms since $\mathcal{R}$ is locally constant.  Since the maps~\eqref{eq:structuremapforG} only depend on the structure maps of $\mathcal{R}$ and $\mathcal{A}$, the rule $\mathcal{R}\mapsto G(\mathcal{R})$ extends into a functor $$G:E_1\textbf{-RMod}_{\mathcal{A}(S^{n-1}\times \R)}\cong \textbf{Fac}^{lc}_{[0,+\infty)}\times_{\textbf{Fac}^{lc}_{(0,+\infty)}} \{N_*(\mathcal{A})\}\to\textbf{Fac}^{lc}_{\R^n_*}\times_{\textbf{Fac}^{lc}_{\R^n\setminus\{0\}}} \{\mathcal{A}\} .$$
In order to check that $N_*\circ G$ is equivalent to the identity functor of $E_1\textbf{-RMod}_{\mathcal{A}(S^{n-1}\times \R)}$ it is sufficient to check it on the basis of opens of $[0,+\infty)$ given by the open intervals and the half-closed intervals $[0,s)$ for which the result follows from the definition of the maps~\eqref{eq:structuremapforG}. Similarly, one can check that $G\circ N_*$ is  equivalent to the identity of $\textbf{Fac}^{lc}_{\R^n_*}\times_{\textbf{Fac}^{lc}_{\R^n\setminus\{0\}}} \{\mathcal{A}\} $ by checking it on the open cover $\mathcal{U}$. 
\qed

\section{Appendix}
 \label{S:operad}
In this appendix, we briefly  collect  several notions and results about
$\infty$-categories and ($\infty$-)operads and in particular the $E_n$-operad and its algebras and their modules.

\subsection{ $\infty$-category overview}\label{S:DKL}
There are several equivalent (see~\cite{Be1}) notions of (symmetric monoidal) $\infty$-categories and the reader shall feel free to use its favorite ones in these notes though we choose
\begin{definition}
In this paper, an $\infty$-category means a  complete Segal spaces~\cite{Re, L-TFT}. 
\end{definition}
Other appropriate models\footnote{Depending on the context some models are more natural to use than others} are given by Segal category~\cite{HiSi, ToVe-Segaltopoi} or Joyal quasi-categories~\cite{L-HT}. 
Almost all $\infty$-categories  in these notes  arise as some (derived) topological (or simplicial or dg) categories or localization of a category with weak equivalences. They carry along derived functors (such as derived homomorphisms) lifting the usual derived functors of usual derived categories. 
We recall below (examples~\ref{ex:inftyfromodel} and~\ref{ex:topologicalasinfty}) how to go  from a model or topological category to an $\infty$-category.

Following~\cite{Re, L-TFT}, a
 \emph{\textbf{Segal space}} is a functor $X_\bullet: \Delta^{op}\to \Top$, that is a simplicial  space\footnote{here a space can also mean a simplicial set and it is often technically easier to work in this setting}, which is Reedy fibrant (see~\cite{Ho}) and satisfies the  the condition that for every integers $n\geq 0$, the natural map  (induced by the face maps)
\begin{equation}\label{eq:defnofSegalcond} X_n \longrightarrow X_1\times_{X_0}\times X_1 \times_{X_0} \cdots\times_{X_0} X_1\end{equation}(where there is $n$ copies of $X_1$) is a weak homotopy equivalence. \footnote{Alternatively, one can work out an equivalent notion fo Segal spaces which forget about the Reedy fibrancy condition and replace  condition~\eqref{eq:defnofSegalcond} by the condition that the following natural
map is  is a weak homotopy equivalence:
 $$ X_n\longrightarrow \holim \big(X_1 \stackrel{d_0}\longrightarrow X_0 \stackrel{d_1}\longleftarrow X_1 \stackrel{d_0}\longrightarrow \dots
\stackrel{d_1}\longleftarrow X_1 \stackrel{d_0}\longrightarrow X_0 \stackrel{d_1}\longleftarrow X_1 \big).$$} 

Associated to a Segal space $X_\bullet$ is a (discrete) category $\textit{ho}(X_\bullet)$ with objects the points of $X_0$ and morphisms $\textit{ho}(X_\bullet)(a,b)=\pi_0\big(\{a\}\times_{X_0} X_1 \times_{X_0}\{b\}\big)$. We \emph{call $\textit{ho}(X_\bullet)$ the homotopy category of $X_\bullet$}.

A Segal space $X_\bullet$ is \emph{complete} if the canonical map $X_0\to \text{Iso}(X_1)$ is a weak equivalence, where $\text{Iso}(X_1)$ is the subspace of $X_1$ consisting of maps $f$ whose class $[f]\in \textit{ho}(X_\bullet)$ is invertible.

There is a simplicial closed model category structure, denoted $\SeSp$ on the category of simplicial spaces such that a fibrant object in  $\SeSp$ is precisely a Segal space. The category of simplicial spaces has another simplicial closed model structure, denoted $\CSS$, whose fibrant objects are precisely complete Segal spaces ~\cite[Theorem 7.2]{Re}.   Let $\mathbb{R}:\SeSp \to \SeSp$ be a fibrant replacement functor and $\widehat{\cdot}: SeSp \to \CSS $ be the completion functor that assigns to a Segal space $X_\bullet$ an equivalent complete Segal space $\widehat{X_\bullet}$.
 The composition $X_\bullet \mapsto \widehat{\mathbb{R}(X_\bullet)}$ gives  a fibrant replacement functor $L_{\CSS}$ from simplicial spaces to complete Segal spaces. 

\begin{example}[Discrete categories] \label{ex:discreteasinfty}
Let $\mathcal{C}$ be an ordinary category (which we also referred to as a discrete category since its $\Hom$-spaces are discrete). Its nerve is a Segal space which is not complete in general. However, one can form its \emph{classifying diagram}, abusively denoted $N(\mathcal{C})$ which is a complete Segal space~\cite{Re}. This is the $\infty$-category associated to $\mathcal{C}$.

By definition, the classifying diagram is the simplicial space 
 $n\mapsto \big(N(\mathcal{C})\big)_n:= N_\bullet(\text{Iso}(\mathcal{C}^{[n]}))$ given by the ordinary nerves (or classifying spaces) of   $\text{Iso}(\mathcal{C}^{[n]})$ the subcategories of isomorphisms of the categories of $n$-composables arrows in $\mathcal{C}$.
\end{example}

\begin{example}[\textbf{Topological category}] \label{ex:topologicalasinfty}
Let $T$ be a topological (or simplicial) category. Its  nerve $N_\bullet(T)$ is a simplicial space. Applying the complete Segal Space replacement functor we get the \emph{$\infty$-category  $T_\infty:= L_{\CSS}(N_\bullet(T))$ associated to $T$}.  

Note that there is a  model category structure on topological category which is Quillen equivalent\footnote{more precisely there is a zigzag of Quillen equivalences in between them; zigzag which goes through  the model category structure of Segal categories.} to $\CSS$, (\cite{Be1}). The functor $T\mapsto T_\infty$ realizes this equivalence. 
 If $T$ is a discrete topological category (in other words an usual category viewed as a topological category), then $ T_\infty$ is equivalent to the $\infty$-category $N(T)$ associated to $T$ in Example~\ref{ex:discreteasinfty} (\cite{Be2}). 
It is worth mentioning that the functor $T\mapsto T_\infty$ is the analogue for complete Segal spaces of the \emph{homotopy coherent nerve} (\cite{L-HT}) for quasi-categories, see~\cite{Be3} for a comparison.
\end{example}

\begin{example}[\textbf{The $\infty$-category of a model category}]\label{ex:inftyfromodel}  Let $\mathcal{M}$ be a model category and $\mathcal{W}$ be its subcategory of weak-equivalences. We denote $L^H(\mathcal{M},\mathcal{W})$ its \emph{hammock localization}, see \cite{DK}. 
One of the main property of $L^H(\mathcal{M},\mathcal{W})$ is that it is a simplicial category and that the (usual) category $\pi_0(L^H(\mathcal{M},\mathcal{W}))$ is the homotopy category of $\mathcal{M}$. Further,   every weak equivalence has a (weak) inverse in $L^H(\mathcal{M},\mathcal{W})$. 
If $\mathcal{M}$ is   a simplicial model category, then for every pair $(x,y)$ of objects the simplicial set of morphisms $\mathop{Hom}_{L^H(\mathcal{M},\mathcal{W})}(x,y)$ is naturally homotopy equivalent to the function complex $ \Map_{\mathcal{M}} (x,y)$. 

By construction, the nerve $N_\bullet(L^H(\mathcal{M},\mathcal{W}))$ is  a simplicial space. Applying  the complete Segal Space replacement functor we get  
\begin{proposition}\cite{Be1}The simplicial space $ \bm{L}_\infty(\mathcal{M}):= L_{\CSS}(N_\bullet(L^H(\mathcal{M},\mathcal{W})))$  is a complete Segal space, which is the $\infty$-category associated to $\mathcal{M}$.
\end{proposition}
Note that the above construction extends to any category with weak equivalences.

Also, the \emph{limit and colimit}  in the $\infty$-category $\bm{L}_{\infty}(\mathcal{M})$ associated to a closed model category $\mathcal{M}$ can be computed by the\emph{ homotopy limit and homotopy colimit} in $\mathcal{M}$, that is by using fibrant and cofibrant resolutions. The same is true for derived functors. For instance a right Quillen functor $f: \mathcal{M}\to \mathcal{N}$ has a lift $\mathbb{L}f:  \bm{L}_\infty(\mathcal{M})\to \bm{L}_\infty(\mathcal{N})$.
\end{example}
\begin{remark} \label{R:non-unique} There are other functors that yields a complete Segal space out 
of a model category. For instance, one can generalize the construction of Example~\ref{ex:discreteasinfty}. For $\mathcal{M}$ a model category and any integer $n$, let $\mathcal{M}^{[n]}$ be 
the (model) category of $n$-composables morphisms, that is the category of functors from the poset 
$[n]$ to $\mathcal{M}$. The \emph{classification diagram} of $\mathcal{M}$ is the simplicial space
 $n\mapsto N_\bullet(\mathcal{W}e(\mathcal{M}^{[n]}))$ where $\mathcal{W}e(\mathcal{M}^{[n]})$ is 
the subcategory of weak equivalences of $\mathcal{M}^{[n]}$. Then taking a \emph{Reedy} fibrant 
replacement yields another complete Segal space
 $N_\bullet(\mathcal{W}e(\mathcal{M}^{[n]}))^{f}$ (\cite[Theorem 6.2]{Be2}, \cite[Theorem 8.3]{Re}).
 It is known that the Segal space $N_\bullet(\mathcal{W}e(\mathcal{M}^{[n]}))^{f}$ is equivalent 
to $\bm{L}_\infty(\mathcal{M})=L_{\CSS}(N_\bullet(L^H(\mathcal{M},\mathcal{W})))$ (\cite{Be2}).
\end{remark}
\begin{definition}\label{Def:MapSpaceinftycat} The \emph{objects} of an $\infty$-category $\bm{C}$ are the points of $\bm{C}_0$.
By definition, an $\infty$-category has a \emph{space} (and not just a set) of morphisms $$\Map_{\bm{C}}(x,y):=  \{x\}\times^{h}_{\bm{C}_0}\bm{C}_1 \times^{h}_{\bm{C}_0} \{y\}$$ between two objects $x$ and $y$.
A morphism \emph{$f\in \Map_{\bm{C}}(x,y)$ is called an equivalence} if its image $[f]\in \Map_{\textit{ho}(\bm{C})}(x,y)$ is an isomorphism.
\end{definition}
From Example~\ref{ex:inftyfromodel}, we get  an \emph{$\infty$-category of $\infty$-categories}, denoted $\infty\textbf{-Cat}$, whose morphisms are  called $\infty$-\emph{functors} (or just functors for short). An equivalence of $\infty$-categories is an equivalence in $\infty\textbf{-Cat}$ in the sense of Definition~\ref{Def:MapSpaceinftycat}.

 The model category of complete Segal spaces is cartesian closed~\cite{Re} hence so is the $\infty$-category $\infty\textbf{-Cat}$. In particular, given two  $\infty$-categories $\mathcal{C}$, $\mathcal{D}$ we have an $\infty$-category $ \textbf{Fun}(\mathcal{C}, \mathcal{D})$ of functors\footnote{computed from the $\Hom$-space in the category of simplicial spaces using the fibrant replacement functor   $L_{\CSS}$} from $\mathcal{C}$ to 
$\mathcal{D}$.
There is an natural weak equivalence of spaces:
\begin{equation} \Map_{\infty\textbf{-Cat}}\big(\mathcal{B}, \textbf{Fun}(\mathcal{C}, \mathcal{D}\big)\stackrel{\simeq}\to \Map_{\infty\textbf{-Cat}}\big(\mathcal{B}\times \mathcal{C}, \mathcal{D}\big).
\end{equation} 
\begin{remark}[the case of simplicial model categories] \label{R:Top=Model} When $\mathcal{M}$ is a simplicial closed model category, there are natural equivalences (\cite{Re}) of spaces 
$$\Map_{\bm{L}_\infty(\mathcal{M})}(x,y) \cong \Map_{\big(L^H(\mathcal{M},\mathcal{W})\big)_\infty}(x,y)\cong \Map_{L^H(\mathcal{M},\mathcal{W})}(x,y)\cong \Map_{\mathcal{M}}(x,y)$$ where the right hand side is the function complex of $\mathcal{M}$ and $x,y$ two objects.  The first two equivalences also hold for general model categories (\cite{Be1}).  In particular, the two constructions of an $\infty$-category associated to a simplicial model category, either viewed as topological category as in Example~\ref{ex:topologicalasinfty}, or as a model category as in Example~\ref{ex:inftyfromodel}, are equivalent:
\begin{proposition}\label{P:Top=Model} Let $\mathcal{M}$ be a simplicial model category. Then $\mathcal{M}_\infty \cong \bm{L}_\infty(\mathcal{M})$.
\end{proposition}
\end{remark}

Let $\bm{I}$ be the $\infty$-category associated to the trivial category $\Delta^1= \{0\}\to \{1\}$ which has two objects and only one non-trivial morphism. 
We have two maps $i_0, i_1: \{pt\} \to \bm{I}$ from the trivial category to $\bm{I}$ which respectively maps the object $pt$ to $0$ and $1$.
\begin{definition}\label{D:overcategory} Let  $A$  be an object of an $\infty$-category $\mathcal{C}$.  The \emph{$\infty$-category $\mathcal{C}_{A}$} of objects over $A$  is the pullback
$$\xymatrix{\mathcal{C}_{A}\ar[r] \ar[d] & \textbf{Fun}(\bm{I}, \mathcal{C}) \ar[d]^{i_1^*} \\ \{A\} \ar[r] & \mathcal{C}. } $$
The \emph{$\infty$-category ${}_{A}\mathcal{C}$} of objects  under $A$ is the pullback
$$\xymatrix{\mathcal{C}_{A}\ar[r] \ar[d] & \textbf{Fun}(\bm{I}, \mathcal{C}) \ar[d]^{i_0^*} \\ \{A\} \ar[r] & \mathcal{C}. }$$
\end{definition}
Informally, the $\infty$-category $\mathcal{C}_{A}$ is just the category of objects $B\in \mathcal{C}$ equipped with a map $f:B\to A$ in $\mathcal{C}$.

\smallskip

There is a notion of symmetric monoidal $\infty$-category generalizing the classical notion for discrete categories. There are several equivalent way to define this notion, see~\cite{L-HA, L-HT, ToVe3} for details.
 Let $\Gamma$ be the skeleton of the category $\textit{Fin}_*$ of finite pointed sets, that is the subcategory spanned by the objects $n_{+}:=\{0,\dots, n\}$, $n\in \mathbb{N}$. For $i=1\dots n$, let $s_{i}: n_{+}\to 1_{+}$ be the map sending $i$ to $1$ and everything else to $0$.
\begin{definition}A symmetric monoidal $\infty$-category is a functor $T\in \textbf{Fun}(\Gamma, \infty\textbf{-Cat})$ such that the canonical map 
$T( n_{+})\stackrel{\prod_{i=0}^n s_i}\longrightarrow \big(T(1_{+})\big)^{n}$ is an equivalence.
The full sub-category of $\textbf{Fun}(\Gamma, \infty\textbf{-Cat})$ spanned by the symmetric monoidal categories is denoted $\infty\textbf{-Cat}^{\otimes}$. Its morphisms are called \emph{symmetric monoidal functors}. Further  $\infty\textbf{-Cat}^{\otimes}$  is enriched over $\infty\textbf{-Cat}$.

 A symmetric monoidal category $T:\Gamma\to \infty\textbf{-Cat}$ will usually be denoted as $(\bm{T}, \otimes)$ where $\bm{T}:=T(1)$.
 If $C:\Gamma\to \infty\textbf{-Cat}$ and $D:\Gamma\to \infty\textbf{-Cat}$ are symmetric monoidal categories, we will \emph{denote 
$\textbf{Fun}^{\otimes}(\mathcal{C},\mathcal{D})$ the $\infty$-categories of symmetric monoidal functors}.
\end{definition}
Equivalently,  a symmetric monoidal category is an $E_\infty$-algebra object in the $\infty$-category $\infty\textbf{-Cat}$.
An $(\infty$-)category with finite coproducts has a canonical structure of symmetric monoidal $\infty$-category and so does a category with finite products. 

\begin{example}[The $\infty$-category $\hTop$] \label{E:hsset} Applying the above procedure(Example~\ref{ex:inftyfromodel}) to the model category of simplicial sets, we obtain the $\infty$-category $\hsset$. Similarly, the model category of topological spaces yields the $\infty$-category $\hTop$ of topological spaces. By Remark~\ref{R:Top=Model}, we can also apply Example~\ref{ex:topologicalasinfty} to the standard enrichment of these categories into topological (or simplicial) categories to construct (equivalent) models of $\hsset$ and $\hTop$.

 Since the model categories $\sset$ and $\Top$ are Quillen equivalent~\cite{GoJa, Ho}, their associated $\infty$-categories are equivalent.  The left and right equivalences $|-|: \sset \stackrel{\sim}{\underset{\sim}{\rightleftarrows}} \Top :\Delta_\bullet(-)$  are respectively induced by the singular set and geometric realization functors. 
The \emph{disjoint union of simplicial sets and topological spaces make $\hsset$ and $\hTop$ into symmetric monoidal $\infty$-categories}.

The above analysis also holds for the pointed versions ${\hsset}_*$ and $\hTopp$ of the above $\infty$-categories (using the model categories of these pointed versions~\cite{Ho}).
\end{example}

\begin{example}[Chain complexes]\label{Ex:Chains(k)}  The model category of (unbounded) chain complexes over $k$ (say with the projective model structure)~\cite{Ho} yields the $\infty$-category of chain complexes $\hkmod$(Example~\ref{ex:inftyfromodel}). 
The mapping space between two chain complex $P_*, Q_*$  is equivalent to  the geometric realization of the simplicial set 
$n\mapsto \Hom_{\hkmod}(P_*\otimes C_\ast(\Delta^n), Q_*)$ where $\Hom_{\hkmod}$ stands for morphisms of chain complexes. 
It follows from Proposition~\ref{P:Top=Model}, that one can also use Example~\ref{ex:topologicalasinfty} applied to the category of chain complexes endowed with the above topological space of morphisms to define $\hkmod$.
In particular a chain homotopy between two chain maps $f,g\in \Map_{\hkmod}(P_*,Q_*)$ is a path in $\Map_{\hkmod}(P_*,Q_*)$. 

In fact,  $\textit{ho}(\hkmod)\cong \mathcal{D}(k)$ is the usual \emph{derived category of $k$-modules}.
 The (derived) tensor product over $k$ yields a symmetric monoidal structure to $\hkmod$ which will usually simply denote by $\otimes$. Note that $\hkmod$ \emph{is enriched over itself}, that is, for any $P_*, Q_* \in \hkmod$, there is an object $\mathbb{R}\Hom_{k}(P_*,Q_*)\in \hkmod$ together with an adjunction 
$$ \Map_{\hkmod}\big(P_*\otimes Q_*, R_*\big)  \, \cong\, \Map_{\hkmod}\big( P_*, \mathbb{R}\Hom_{k}(Q_*,R_*)\big).$$
The interested reader can refer to~\cite{GH, L-HA} for details  of $\infty$-categories enriched over $\infty$-categories and to~\cite{GM, DS} for  model categories enriched over symmetric monoidal closed model categories  (which is the case of the category of chain complexes). 
\end{example}

\begin{example}\label{E:CDGAMod} In characteristic zero, there is a standard closed model category structure on the category of commutative differential graded algebras (CDGA for short), see~\cite[Theorem 4.1.1]{Hi}. Its fibrations are epimorphisms and (weak) equivalences are quasi-isomorphisms (of CDGAs). We thus get the $\infty$-category $\hcdga$ of CDGAs.
The category $\hcdga$ also has a monoidal structure given by the (derived) tensor product (over $k$) of differential graded commutative algebras, which makes $\cdga$ a symmetric monoidal model category. 
Given $A,B \in \cdga$, the mapping space $\Map_{\hcdga}(A,B)$ is the (geometric realization of the) simplicial set of maps $[n]\mapsto \mathop{Hom}_{\text{dg-Algebras}}(A, B\otimes \Omega^*(\Delta^n))$ (where $\Omega^*(\Delta^n)$ is the CDGA of forms on the $n$-dimensional standard simplex and $\Hom_{\text{dg-Algebras}}$ is the module of differential graded algebras maps). 
It has thus a canonical enrichment over chain complexes. 

 The model categories  of left modules and commutative algebras over a CDGA $A$ yield the $\infty$-categories
$E_1\textbf{-LMod}_{A}$ and $\hcdga_{A}$ . The base change functor 
lifts to a functor of $\infty$-categories. Further, if $f: A\to B$ is a weak equivalence, the natural 
functor $f_*:E_1\textbf{-LMod}_{B} \to:E_1\textbf{-LMod}_{A}$ induces an equivalence $E_1\textbf{-LMod}_{B} \stackrel{\sim}\to E_1\textbf{-LMod}_{A}$ 
of $\infty$-categories since it is induced by a Quillen equivalence. 

Moreover, if $f:A\to B$ is a morphism of CDGAs, we get a natural functor 
$f^*:E_1\textbf{-LMod}_{A}\to E_1\textbf{-LMod}_{B},\, M\mapsto M\otimes^{\mathbb{L}}_A B$, which is an equivalence of $\infty$-categories 
when $f$ is a quasi-isomorphism, and is a (weak) inverse of $f_*$ (see~\cite{ToVe} or~\cite{KM}). 
The same results applies to monoids in
 $E_1\textbf{-LMod}_{A}$  that is to the categories of commutative differential graded $A$-algebras.
\end{example}


\subsection{$E_n$-algebras and $E_n$-modules} \label{S:EnAlg}

The classical definition of an $E_n$-algebra (in chain complexes) is an algebra over any $E_n$-operad in chain complexes, that is an operad weakly homotopy equivalent to the chains on the little ($n$-dimensional)  cubes operad $(\text{Cube}_{n}(r))_{r\geq 0}$ (\cite{May-gils}). 
Here $$\text{Cube}_n(r):=\text{Rect}\Big(\coprod_{i=1}^r (0,1)^n, (0,1)^n\Big)$$ is the space  of rectilinear embeddings of $r$-many disjoint copies of the unit open cube in itself. 
It is topologized as the subspace of the space of all continuous maps.
 By a\emph{ rectilinear embedding}, we mean a composition of a translation and dilatations in the direction given by a vector of the canonical basis of $\R^n$.
 In other words,  $\text{Cube}_{n}(r)$ is the configuration space of $r$-many disjoint open rectangles\footnote{more precisely rectangular parallelepiped in dimension bigger than 2} parallel to the axes lying in the unit open cube. The operad structure
$ \text{Cube}_{n}(r)\times \text{Cube}_{n}(k_1)\times \cdots \times \text{Cube}_{n}(k_r)\to \text{Cube}_{n}(k_1+\cdots+k_r)$ is simply given by composition of embeddings.

An \emph{$E_n$-algebra in chain complexes} is thus a chain complex $A$ together with  chain maps $\gamma_r: C_\ast(\text{Cube}_{n}(r))\otimes A^{\otimes r}\longrightarrow A$ compatible with the composition of operads~\cite{BV, May-gils, Fre-grtbook}. By definition of the operad $\text{Cube}_{n}$,  we are only considering (weakly) unital versions of $E_n$-algebras.

The model category of $E_n$-algebras gives rises to \emph{the $\infty$-category $ E_n\textbf{-Alg}$} of $E_n$-algebras in the symmetric monoidal $\infty$-category of differential graded $k$-modules. The symmetric structure of $\hkmod$ lifts to a a symmetric monoidal 
structure on $(E_n\textbf{-Alg},\otimes)$ given by the tensor product of the underlying chain complexes\footnote{other possible models for the symmetric monoidal $\infty$-category  $(E_n\textbf{-Alg},\otimes)$ are given by algebraic Hopf operads such as those arising from the filtration of the Barratt-Eccles operad in~\cite{BF}}. 

\medskip

One can extend the above notion  to define $E_n$-algebras with coefficient in any  symmetric monoidal $\infty$-category following~\cite{L-HA}. One way is to rewrite it  in terms of symmetric monoidal functor as follows. 
Any  topological (resp. simplicial) operad $\mathcal{O}$ defines a symmetric monoidal category, denoted $\mathbf{O}$, fibered over the category of pointed finite sets $\text{Fin}_*$. This category  $\mathbf{O}$ has the finite sets for objects.
 For any sets $n_{+}:=\{0,\dots n\}$, $m_{+}:=\{0,\dots, m\}$ (with base point $0$), its morphism space  $\mathbf{O}(n_{+},m_{+})$  (from $n_+$ to $m_+$) is   the disjoint union $\coprod_{f: n_{+}\to m_{+}}\prod_{i\in m_{+}}\mathcal{O}((f^{-1}(i))_{+})$ and the composition is induced by the operadic structure.
 The rule $n_+\otimes m_+= (n+m)_+$ makes canonically  $\mathbf{O}$ into a symmetric monoidal topological (resp. simplicial) category. We abusively denote $\mathbf{O}$ its associated $\infty$-category. Note that this construction extends to  colored operad and is a  special case of an $\infty$-operad \footnote{\emph{An $\infty$-operad} $\mathcal{O}^{\otimes}$ is a  $\infty$-category together with a functor $\mathcal{O}^\otimes \to N(Fin_\ast)$  satisfying a list of axioms, see \cite{L-HA}. It is to colored topological operads what $\infty$-categories are to topological categories.}.  
 
  Then, if $(\mathcal{C}, \otimes)$ is a symmetric monoidal $\infty$-category, \emph{a $\mathcal{O}$-algebra in $\mathcal{C}$ is a symmetric monoidal functor $A\in  \textbf{Fun}^{\otimes}(\mathbf{O},\mathcal{C})$}.
 We call $A(1_+)$  the underlying algebra object of $A$ and we usually  denote it simply by $A$.

\begin{definition}\cite{L-HA, F}\label{D:EnAlgebras} Let $(\mathcal{C}, \otimes)$ be  symmetric monoidal ($\infty$-)category. The $\infty$-category of \emph{$E_n$-algebras with values in $\mathcal{C}$} is
 $$E_n\textbf{-Alg}(\mathcal{C}):=  \textbf{Fun}^{\otimes}(\textbf{Cube}_n, \mathcal{C}).$$
Similarly  $E_n\textbf{-coAlg}(\mathcal{C}):=\textbf{Fun}^{\otimes}(\textbf{Cube}_n, \mathcal{C}^{op})$ is the category of \emph{$E_n$-coalgebras in $\mathcal{C}$}. We denote \emph{$Map_{E_{n}\textbf{-Alg}}(A,B)$ the mapping space  of  $E_n$-algebras maps from $A$ to $B$}.
\end{definition}
Note that Definition~\ref{D:EnAlgebras} is a definition of categories of (weakly) \emph{(co)unital} $E_n$-coalgebra objects.

One has an equivalence $E_n\textbf{-Alg}\cong E_n\textbf{-Alg}(\hkmod)$ of symmetric monoidal $\infty$-categories  (see~\cite{L-HA, F2}) where $E_n\textbf{-Alg}$ is the $\infty$-category associated to  algebras over the operad $\text{Cube}_{n}$ considered above.
It is clear from the above definition  that any ($\infty$-)operad $\mathbb{E}_n$ weakly homotopy equivalent (as an operad) to $\text{Cube}_n$ gives rise to an equivalent $\infty$-category of algebra. 
In particular, the inclusion of rectilinear embeddings into all framed embeddings gives us
 an alternative definition for $E_n$-algebras:
\begin{proposition}[\cite{L-HA}]\label{D:EnasDisk} Let $\text{\emph{Disk}}^{fr}_n$ be the category with objects the integers and morphism the spaces $\text{\emph{Disk}}^{fr}_n(k,\ell):= \text{\emph{Emb}}^{fr}(\coprod_{k} \R^n, \coprod_{\ell} \R^n)$  of framed embeddings of $k$ disjoint copies of a disk $\R^n$ into $\ell$ such copies (see Example~\ref{E:DisknAlg}).
 The natural map
$\textbf{Fun}^\otimes(\text{\emph{Disk}}^{fr}_n,\,\hkmod) \stackrel{\simeq}\longrightarrow E_n\textbf{-Alg} $
is an equivalence.
\end{proposition}

\begin{example}(Iterated loop spaces)\label{ex:loopspaces}
The standard examples\footnote{May\rq{}s recognition principle~\cite{May-gils} actually asserts that any $E_n$-algebra in $(\Top, \times)$ which is group-like is homotopy equivalent to such an iterated loop space} of $E_n$-algebras are given by iterated loop spaces. If $X$ is a pointed space, we denote $\Omega^n (X):=\Map_{*}(S^n, X)$ the set of all pointed maps from $S^n\cong I^n/\partial I^n$ to $X$, equipped with the compact-open topology.  
The pinching map~\eqref{eq:pinchcube} $pinch: \text{Cube}_n(r) \times S^n \longrightarrow  \bigvee_{i=1\dots r}\, S^n$ induces an $E_n$-algebra structure (in $(\Top, \times)$) given by 
$$\text{Cube}_n(r) \times\big(\Omega^n(X))\big)^r  \cong \text{Cube}_n(r) \times \Map_*(\bigvee_{i=1\dots r} S^n, X) \stackrel{pinch^*}\longrightarrow \Omega^n(X).$$
Since the construction is functorial in $X$, the singular chain complex $C_\ast(\Omega^n(X))$ is also an $E_n$-algebra in chain complexes, and further this structure is compatible with the $E_\infty$-coalgebra structure of $C_\ast(\Omega^n(X))$  (from Example~\ref{E:singularchainasEinfty}).
 Similarly, the singular cochain complex $C^\ast(\Omega^n(X))$ is an $E_n$-coalgebra in a way compatible with its $E_\infty$-algebra structure; that is an object of $E_n\textbf{-coAlg}(E_\infty\textbf{-Alg})$.
\end{example}

\begin{example}[$P_n$-algebras] \label{Ex:PnAlg} A standard result of Cohen~\cite{Co-littledisks} shows that, for $n\geq 2$, the homology of an $E_n$-algebra is a 
$P_n$-algebra (also see~\cite{Si, Fre-grtbook}). \emph{A $P_n$-algebra} is a graded vector space $A$  endowed with a degree $0$ multiplication with unit which makes $A$ a graded commutative algebra,  and a (cohomological) degree 1-n operation $[-,-]$ which makes $A[1-n]$  a graded Lie algebra. These operations are also required to satisfy the Leibniz rule
$[a\cdot b, c] = a[b,c] + (-1)^{|b|(|c|+1-n)}[a,c]\cdot b $.

 For $n=1$, $P_n$-algebras are just usual Poisson algebras while for $n=2$, they are Gerstenhaber algebras. 

In characteristic $0$, the operad $\text{Cube}_n$ is formal, thus equivalent as an operad to the operad governing $P_n$-algebras (for $n\geq 2$). It follows that $P_n$-algebras gives rise to $E_n$-algebras in that case, that is there is a functor\footnote{which is \emph{not} canonical, see \cite{Ta-formalityactionAssoc}}  $P_n\textbf{-Alg}\to E_n\textbf{-Alg}$.
\end{example}

 There  are natural maps (sometimes called the stabilization functors)
\begin{equation}\label{eq:towerofEnoperad}
 \text{Cube}_0 \longrightarrow \text{Cube}_1 \longrightarrow \text{Cube}_2 \longrightarrow \cdots 
\end{equation} (induced by taking products of cubes with the interval $(0,1)$).
It is a fact (\cite{May-gils, L-HA}) that the colimit of this diagram, denoted by $\mathbb{E}_\infty$ 
is equivalent to the commutative operad $\textit{Com}$ (whose associated symmetric monoidal $\infty$-category is $\mathbf{Fin}_*$).
\begin{definition}\label{D:EinftyAlg}
The($\infty$-) category of $E_\infty$-algebras with value in $\mathcal{C}$ is $E_\infty\textbf{-Alg}(\mathcal{C}):=\textbf{Fun}^{\otimes}(\textbf{Cube}_\infty, \mathcal{C})$. It is simply denoted $E_\infty\textbf{-Alg}$ if $(\mathcal{C},\otimes) =(\hkmod, \otimes)$.
 Similarly, the category of $E_\infty$-coalgebras. is $E_\infty\text{-coAlg}:=\textbf{Fun}^{\otimes}(\textbf{Cube}_\infty, \mathcal{C}^{op})$.
\end{definition}
Note that Definition~\ref{D:EinftyAlg} is a definition of (weakly) unital $E_\infty$-algebras.

\smallskip

The category $E_\infty\textbf{-Alg}$ is (equivalent to) the $\infty$-category associated to the model category of $\mathbb{E}_\infty$-algebras for any $E_\infty$-operad $\mathbb{E}_\infty$.

The natural map $\textbf{Fun}^{\otimes}(\mathbf{Fin}_*, \mathcal{C})\longrightarrow \textbf{Fun}^{\otimes}(\textbf{Cube}_\infty, \mathcal{C})=E_\infty\textbf{-Alg}$ is also an equivalence. 

For any $n\in \mathbb{N}-\{0\} \cup \{+\infty\}$,
 the map $\text{Cube}_1 \to \text{Cube}_n$ (from the nested sequence~\eqref{eq:towerofEnoperad}) induces a \emph{functor} $E_n\textbf{-Alg} \longrightarrow E_1\textbf{-Alg}$ 
which associates to  an $E_n$-algebra its underlying $E_1$-algebra structure.

\begin{example}[Singular (co)chains]\label{E:singularchainasEinfty} Let $X$ be a topological space. 
Its \emph{singular cochain complex} $C^\ast(X)$ \emph{has a natural structure of $E_\infty$-algebra}, 
whose underlying $E_1$-structure is given by the usual (strictly associative) cup-product (for instance see~\cite{M2}). 
The \emph{singular chains $C_\ast(X)$ have a natural structure of $E_\infty$-coalgebra} which is the predual of 
$(C^\ast(X),\cup)$.
 There are similar  constructions for simplicial sets $X_\com$ instead of spaces, 
see~\cite{BF}. We recall that $C^\ast(X)$ is the linear dual of the singular chain complex $C_\ast(X)$ with coefficient in $k$ which is the geometric realization (in the ordinary category of chain complexes) 
of the simplicial $k$-module $k[ \Delta_\bullet(X)]$ spanned by the singular set  
$\Delta_\bullet(X):=\{\Delta^\bullet\stackrel{f}\to X, f \mbox{ continuous}\}$. Here $\Delta^n$ is the standard $n$-dimensional simplex. 
\end{example}

\begin{remark}\label{R:Einftyistensored} The mapping space $Map_{E_{\infty}\textbf{-Alg}}(A,B)$ of  two $E_\infty$-algebras $A$, $B$ (in the model category of  $E_\infty$-algebras) is the (geometric realization of the) simplicial set $[n]\mapsto Hom_{E_\infty\textbf{-Alg}}\big(A, B\otimes C^\ast(\Delta^n)\big)$. 

 The $\infty$-category $E_\infty\text{-}Alg$ is enriched over $\hsset$ (hence $\hTop$ as well by Example~\ref{E:hsset})  and has all ($\infty$-)colimits.
 In particular, it is \emph{tensored} over $\hsset$, see~\cite{L-HT, L-HA} for details on tensored $\infty$-categories  or~\cite{EKMM, MCSV} in the context of topologically enriched model categories. 
We recall that it means that there is a functor $E_\infty\textbf{-Alg} \times \hsset \to E_\infty\textbf{-Alg}$, denoted $(A, X_\bullet)\mapsto A\boxtimes X_\bullet$, together with natural equivalences
 $$Map_{E_\infty\textbf{-Alg}}\big(A\boxtimes X_\bullet, B\big) \; \cong \; Map_{\hsset}\big(X_\com, Map_{E_\infty\textbf{-Alg}}\big(A, B\big)\big). $$

To compute explicitly this tensor, it is useful to know the following proposition.
 \begin{proposition}\label{P:tensor=coprod} Let $(\mathcal{C},\otimes)$ be a symmetric monoidal $\infty$-category. In the symmetric monoidal $\infty$-category ${E_\infty}\textbf{-Alg}(\mathcal{C})$, the tensor product is a coproduct. 
 \end{proposition}
 For a proof see Proposition 3.2.4.7 of \cite{L-HA} (or \cite[Corollary 3.4]{KM}); for $\mathcal{C}= \hkmod$, this essentially follows from the observation that an $E_\infty$-algebra is a commutative monoid in $(\hkmod,\otimes)$, see~\cite{L-HA} or \cite[Section 5.3]{KM}. In particular, Proposition~\ref{P:tensor=coprod} implies that,   for any finite set $I$,  $A^{\otimes I}$ has a natural structure of $E_\infty$-algebra. 
\end{remark}

\paragraph{\textbf{Modules over $E_n$-algebras.}} 

In this paragraph, we give a brief account of various categories of modules over $E_n$-algebras. Note that by definition (see below), the categories we considered are categories of \emph{pointed} modules. Roughly, an $A$-modules $M$ being pointed means it is equipped with a map $A\to M$. 

\smallskip

Let $Fin$, (resp. $Fin_\ast$) be the category of  (resp. pointed) finite sets. There is a forgetful functor $Fin_\ast \to Fin$ forgetting which point is the base point. There is also a functor $Fin \to Fin_\ast$ which adds an extra point called the base point. 
We write $\mathbf{Fin}$, ${\mathbf{Fin}_*}$  for the associated $\infty$-categories (see Example~\ref{ex:discreteasinfty}). 
Following~\cite{L-HA, F}, if $\mathcal{O}$ is a (coherent)  operad, \emph{the $\infty$-category $\mathcal{O}\textbf{-Mod}_{A}$ of $\mathcal{O}$-modules}\footnote{in $\hkmod$. 
Of course, similar construction hold with $\hkmod$ replaced by a symmetric monoidal $\infty$-category} \emph{over an $\mathcal{O}$-algebra $A$} is the category of $\mathcal{O}$-linear functors $\mathcal{O}\textbf{-Mod}_{A}:= \Map_{\mathbf{O}}(\mathbf{O}_*, \hkmod)$  where $\mathbf{O}$ is the ($\infty$-)category associated\footnote{in the paragraph above Definition~\ref{D:EnAlgebras}} to the operad $\mathcal{O}$   and $\mathbf{O}_* := \mathbf{O}\times_{{\mathbf{Fin}}}{\mathbf{Fin}_*}$ (also see~\cite{Fre-Mod} for similar constructions in the model category setting of topological operads).

The categories $\mathcal{O}\textbf{-Mod}_A$ for $A\in\mathcal{O}\textbf{-Alg}$ assemble to form an $\infty$-category $\mathcal{O}\textbf{-Mod}$ describing   pairs consisting of an $\mathcal{O}$-algebra and a module over it. More precisely,
there is an natural fibration $\pi_{\mathcal{O}}:\mathcal{O}\textbf{-Mod}\longrightarrow \mathcal{O}\textbf{-Alg}$ whose fiber at $A\in\mathcal{O}\textbf{-Alg}$ is $\mathcal{O}\textbf{-Mod}_{A}$.
When $\mathcal{O}$ is an $\mathbb{E}_n$-operad (that is an operad equivalent to $\text{Cube}_n$), we simply write $E_n$ instead of $\mathcal{O}$:
\begin{definition}\label{DEnModoverA}
Let $A$ be an $E_n$-algebra (in $\hkmod$). We denote  $E_n\textbf{-Mod}_A$ the $\infty$-category of (pointed) $E_n$-modules over $A$.  Since $\hkmod$ is bicomplete and enriched over itself,    $E_n\textbf{-Mod}_A$  is naturally enriched over $\hkmod$ as well.

We \emph{denote\footnote{The $\mathbb{R}$ in the notation is here to recall that this corresponds to a functor that can be computed as a derived functor associated to ordinary model categories using standard techniques of homological/homotopical algebras} $\mathbb{R}Hom_{A}^{E_n}(M,N) \in \hkmod$ the enriched mapping space of morphisms of $E_n$-modules over $A$.} Note that if $\mathbb{E}_n$ is  a cofibrant $E_n$-operad and further $M$, $N$ are modules over an $\mathbb{E}_n$-algebra $A$, then $\mathbb{R}Hom_{A}^{E_n}(M,N)$ is computed by $Hom_{\text{Mod}^{\mathbb{E}_n}_A}(Q(M), R(N))$. 
Here $Q(M)$ is a cofibrant replacement of $M$ and $R(N)$ a fibrant replacement of $N$ in the model category $\text{Mod}^{\mathbb{E}_n}_A$ of modules over the  $\mathbb{E}_n$-algebra $A$. In particular,  
$\mathbb{R}Hom_{A}^{E_n}(M,N)\cong Hom_{E_n\textbf{-Mod}_A}(M,N)$.
This follows from the fact that $E_n\textbf{-Mod}_A$ is equivalent to the $\infty$-category associated to the model category $\text{Mod}^{\mathbb{E}_n}_A$.

\smallskip

If $(\mathcal{C}, \otimes)$ is a symmetric monoidal ($\infty$-)category and $A\in E_n\textbf{-Alg}$, then we \emph{denote $E_n\textbf{-Mod}_A(\mathcal{C})$ the $\infty$-category of $E_n$-modules over $A$ (in $\mathcal{C}$)}.

We \emph{denote respectively  $E_n\textbf{-Mod}$} the $\infty$-category of all $E_n$-modules in $\hkmod$ \emph{and $E_n\textbf{-Mod}(\mathcal{C})$} the $\infty$-category of all $E_n$-modules in $(\mathcal{C}, \otimes)$.
\end{definition}
 By definition, the canonical functor\footnote{which essentially forget the module in the pair $(A,M)$} $\pi_{E_n}:E_n\textbf{-Mod}(\mathcal{C})\to E_n\textbf{-Alg}(\mathcal{C})$  gives rise, 
for any $E_n$-algebra $A$,  to a (homotopy) pullback square: 
\begin{equation}\label{eq:pullbackModAlg}
 \xymatrix{  E_n\textbf{-Mod}_{A}(\mathcal{C}) \ar[r] \ar[d] & E_n\textbf{-Mod}(\mathcal{C}) \ar[d]^{\pi_{E_n}} \\ \{A\}\ar[r] & E_n\textbf{-Alg}(\mathcal{C})}
\end{equation}
  Note that the functor $\pi_{E_n}$ is monoidal.
  
 We also have a canonical functor $can:E_n\textbf{-Alg} \to E_n\textbf{-Mod}$ induced by the tautological  module structure that any algebra has over itself. 
\begin{example}\label{E:E1andEinftyModules} 
If $A$ is a differential graded algebra,  $E_1\textbf{-Mod}_{A}$ is equivalent to the $\infty$-category of (pointed) $A$-bimodules. 
If $A$ is a CDGA,  $E_\infty\textbf{-Mod}_{A}$ is equivalent to the $\infty$-category of (pointed) left $A$-modules.
\end{example}
 
\begin{example}[\textbf{left and right modules}]\label{ex:LRModA} If $n=1$, we also have naturally defined $\infty$-categories of left and right modules over an $E_1$-algebra $A$ (as well as $\infty$-categories of all right modules and left modules). They are the immediate generalization of the ($\infty$-categories associated to the model) categories of pointed left and right differential graded modules over a differential graded associative unital algebra. We refer to~\cite{L-HA} for details.

\begin{definition}\label{D:LandRMod}{We write respectively $E_1\textbf{-LMod}_A(\mathcal{C})$, $E_1\textbf{-RMod}_A(\mathcal{C})$, $E_1\textbf{-LMod}(\mathcal{C})$ and $E_1\textbf{-RMod}(\mathcal{C})$ for the $\infty$-categories of left modules over a fixed $A$, right modules over $A$, and all left modules and all right modules (with values in $(\mathcal{C}, \otimes)$).

 If $\mathcal{C}=\hkmod$, we simply write $E_1\textbf{-LMod}_A$, $E_1\textbf{-RMod}_A$, $E_1\textbf{-LMod}$, $E_1\textbf{-RMod}$.}
Further, we will \emph{denote $\mathbb{R}Hom_{A}^{left}(M,N) \in \hkmod$ the enriched mapping space of morphisms of  left modules over $A$} (induced by the enrichment of $\hkmod$). In particular $\mathbb{R}Hom_{A}^{left}(M,N)\cong Hom_{E_1\textbf{-LMod}_A}(M,N)$.
 \end{definition}
There are  standard models for these categories. 
For instance, the  category of right modules over an $E_1$-algebra can be obtained by considering a colored operad $\text{Cube}^{right}_1$ obtained from the little interval operad  $\text{Cube}_1$ as follows. 
Denote $c$, $i$ the two colors. We define $\text{Cube}^{right}_1(\{X_j\}_{j=1}^r, i):= \text{Cube}_1(r)$  if all $X_j=i$. If $X_1=c$ and all others $X_j=i$, we  set $\text{Cube}^{right}_1(\{X_j\}_{j=1}^{r}, c):= \text{Rect}\big([0,1)\coprod \big(\coprod_{i=1}^r (0,1)\big), [0,1)\big)$ where $\text{Rect}$ is  the space of rectilinear embeddings (mapping $0$ to itself).  All other spaces of maps are empty.
Then the $\infty$-category associated to the category of $\text{Cube}^{right}_1$-algebras  is equivalent to $E_1\textbf{-RMod}$.
\end{example}

Let $A$ be an $E_1$-algebra, then the usual tensor product of right and left $A$-modules has a canonical lift
$$ -\mathop{\otimes}_{A}^{\mathbb{L}} - :  E_1\textbf{-RMod}_A \times E_1\textbf{-LMod}_A \longrightarrow \hkmod$$
which, for a differential graded associative algebra over a field $k$ is computed by the two-sided Bar construction.
 There is a similar derived functor $E_1\textbf{-RMod}_A(\mathcal{C}) \times E_1\textbf{-LMod}_A(\mathcal{C})  \longrightarrow \mathcal{C} $, still denoted $(R,L)\mapsto R\otimes_{A}^{\mathbb{L}}L$, whenever $(\mathcal{C},\otimes)$ is  a symmetric monoidal $\infty$-category with geometric realization and such that $\otimes$ preserves geometric realization in both variables, see~\cite{L-HA}. There are (derived) adjunction 
$$   \Map_{E_1\textbf{-LMod}_{A}}\big( P_*\otimes L, N\big)\;  \cong\;  \Map_{\hkmod} \big(P_*,\mathbb{R}\Hom_{A}^{left}(L,N) \big), $$  $$ \Map_{\hkmod}\big( R\mathop{\otimes}_{A}^{\mathbb{L}}L, N\big) \; \cong \;  \Map_{E_1\textbf{-LMod}_{A}}\big(L,\mathbb{R}\Hom_{k}(R,N) \big) $$
which relates the tensor product with the enriched mapping spaces of modules.

\end{document}